\documentclass[alpha-refs,serif]{wiley-article}

\usepackage{siunitx}
\usepackage[utf8]{inputenc}
\RequirePackage{amsmath,amsfonts,amssymb}
\usepackage{pictexwd}
\usepackage{url}
\usepackage[colorlinks,citecolor=blue,urlcolor=blue]{hyperref}
\usepackage{bm}
\usepackage{color}
\usepackage{here,latexsym}
\usepackage{curves}
\usepackage{graphicx}
\usepackage{caption}
\usepackage{subcaption}
\usepackage{float}
\usepackage[T1]{fontenc}
\usepackage{newunicodechar}  
\usepackage{akkmathset}

\def\argmin{\mbox{argmin}}

\def\a{\alpha}
\def\b{\beta}
\def\R{\mathbb R}
\def\dd{\Delta}

\def\d{\delta}

\def\B{{\cal B}}

\def\mcX{\mathcal{X}}
\def\E{{\mathbb E}}

\def\mbS{{\mathbb S}}

\def\M{{\cal M}}

\def\P{{\mathbb P}}
\def\g{\gamma}
\def\l{\lambda}
\def\cS{\mathcal{S}}

\def\labda1{\lambda_1}
\def\labda2{\lambda_2}
\def\bma{\bm\alpha}
\def\bmb{\bm\beta}
\def\m{\mu}

\def\e{\varepsilon}
\def\f{\phi}
\def\t{\tau}

\def\argmin{\mbox{argmin}}

\def\comment#1{\relax}

\def\=in{\mathop{\rm =}}

\usepackage{xr}
\papertype{Original Article}
\paperfield{Journal Section}

\title{Score estimation in the monotone single index mode}

\abbrevs{ESE, efficient score estimator; LSE, least squares estimator; MLE, maximum likelihood estimator; MRCE, maximum rank correlation estimator; SSE, simple score estimator.}

\def\c{{\mathbf c}}
\def\a{\gamma}
\def\m{\mu}

\def\GG{\Gamma}

\def\P{{\mathbb P}}

\def\a{\alpha}
\def\ee{\epsilon}
\def\bmS{\bm S}
\def\bmSigma{\bmSigma}
\def\bmU{\bm U}
\def\bmV{\bm V}

\def\bmx{\bm x}
\def\bmX{\bm X}

\def\bmZ{\bm Z}

\definecolor{pink}{RGB}{255,0,255}
\definecolor{green2}{RGB}{0,153,0}
\definecolor{grey}{RGB}{128,128,128}
\definecolor{red2}{RGB}{204,0,0}

\author[1, 2]{Fadoua Balabdaoui} 
\author[3]{Piet Groeneboom}
\author[4]{Kim Hendrickx}

\affil[1]{Universit{\'e} Paris-Dauphine, PSL Research University, Paris, 75016, France}
\affil[2]{Seminar f{\"u}r Statistik, ETH Z{\"u}rich, 8092, Z{\"u}rich, Schweiz}
\affil[3]{Delft University of Technology, Van Mourik Broekmanweg 6, 2628 XE Delft, The Netherlands}
\affil[4]{Hasselt University, I-BioStat, Agoralaan, BA 3590 Diepenbeek, Belgium}

\corraddress{Piet Groeneboom, DIAM, Delft University,  Van Mourik Broekmanweg 6, 2628 XE Delft,  The Netherlands}
\corremail{P.Groeneboom@tudelft.nl}

\fundinginfo{The work of Kim Hendrickx was supported by the Research Foundation Flanders (FWO) [grant number 11W7315N], IAP Research Network P7/06 of the Belgian State (Belgian Science Policy), Flemish Supercomputer Center, funded by the Hercules Foundation and the Flemish Government - department EWI.}

\runningauthor{Balabdaoui et al.}

\begin{document}

\maketitle

\begin{abstract}
		We consider estimation in the single index model where the link function is monotone. For this model a profile least squares estimator has been proposed to estimate the unknown link function and index.  
Although it is natural to propose this procedure, it is still unknown whether it produces index estimates which converge at the parametric rate.  We show that this holds if we solve a score equation corresponding to this least squares problem. Using a Lagrangian formulation, we show how one can solve this score equation without any reparametrization. This makes it easy to solve the score equations in high dimensions.
We also compare our method with the Effective Dimension Reduction (EDR) and the Penalized Least Squares Estimator (PLSE) methods, both available on CRAN as  R packages, and compare with link-free methods, where the covariates are ellipticallly symmetric.

\end{abstract}

\section{Introduction}
Single index models are flexible models used in regression analysis of the type $\E(Y|\bm X) =\psi_0(\bma_0^T \bm X)$, where $\psi_0$ is an unknown link function and $\bma_0$ is an unknown regression parameter. By lowering the dimensionality of the classical linear regression problem, determined by the number of covariates, to a univariate $\bma_0^T \bm X$ index, single index models do not suffer from the ``curse of dimensionality''. They also provide an advantage over the generalized linear regression models by overcoming the risk of misspecifying the link function $\psi_0$. To ensure identifiability of the single index model, one typically assumes that the Euclidean norm $\|\bma_0\|$ equals one with the first non-zero element of $\bma_0$ being positive.  
\\

Several estimation approaches have been considered in the literature of single index models. These methods can be classified into two groups: M-estimators and direct estimators. In the first approach, one considers a non-parametric regression estimate for the infinite dimensional link function $\psi_0$ and then estimates $\bma_0$ by minimizing a certain criterion function, where $\psi_0$ is replaced by its estimate. Examples of this type are the semi-parametric least squares estimators of \cite{ichimura1993} and \cite{hardle:93} and the pseudo-maximum likelihood estimator of \cite{delecroix2003}, all using kernel regression estimates for the unknown link functions. An example of an M-estimator that does not depend on an estimate of the link function $\psi_0$ is Han's maximum rank correlation estimator \citep{han:87}.
\\

Direct estimators, such as the average derivative estimator of \cite{hardle1989} or the slicing regression method proposed in \cite{duan1991}, avoid solving an optimization problem and are often computationally more attractive than M-estimators. 
\\

In this paper we focus on estimating the regression parameter $\bma_0$ under the constraint that $\psi_0$ is monotone. Shape constrained inference arises naturally in a variety of fields. For example in economics where a concavity restriction is assumed in utility theory to indicate the exhibition of risk conversion in economic behavior. Convex optimization problems also appear frequently and often allow for straightforward computation and optimization.  The single index model with convex link has been studied in \cite{Kuchibhotla_patra_sen:17}, where the authors consider estimation of an efficient penalized least squares estimator.  An efficient estimate for the single index with smooth link function is proposed in \cite{Kuchibhotla_patra:16}.
\\

A special case of the monotone singe index model is the widely used econometric binary choice model where interest is in estimating a choice probability based on a binary response variable $Y$ and one or more covariates $\bmX$. Whether or not the outcome is zero or one depends on an underlying utility score, i.e. $Y=1$ if $\bma_0^T\bmX \ge \varepsilon$, where $\varepsilon$ is an unobserved disturbance term with unknown distribution function $F_0$. The binary choice model therefore belongs to the class of monotone single index models since $\E(Y | \bmX) = F_0(\bma_0^T\bmX)$. Estimation of the regression parameters in the binary choice model is among others considered in \cite{cosslett:87}, \cite{cosslett:07} and \cite{klein_spady:93}.
\\

\cite{balabdaoui2016} considered a global least squares estimator for the pair ($\bma_0,\psi_0$) in the general single index model under monotonicity of the function $\psi_0$. They derived an $n^{1/3}$ convergence rate, but the asymptotic limiting distribution for their estimator of $\bma_0$ has not been derived. A conjecture is made in \cite{tanaka2008} that this rate is too slow. In this paper, we will give simulation results on the asymptotic variance of the least squares estimator and investigate its rate of convergence numerically.
\\

Recently,  \cite{kim_piet:18} developed several score estimators for the current status linear regression model $Y = \bmb_0^T \bm Z + \varepsilon$,  where the distribution function $F_0$ of $\varepsilon$ is left unspecified. Instead of observing the response $Y$, a censoring variable $T$ and censoring indicator $\dd=1_{Y\le T}$ are observed. This model is a special case of the monotone single index model and can be formulated as $\E(\dd |T, \bm Z) =  F_0(T - \bmb_0^T \bm Z) = F_0( \bma_0^T \bm X)$ where $ \bma_0 = (1,-\bmb_0)^T$ and $\bm X = (T, \bm Z)^T$. The estimators in \cite{kim_piet:18} are obtained by the root of a score function involving the maximum likelihood estimator (MLE) of the distribution function for fixed $\bmb$. The authors prove $\sqrt n$-consistency and asymptotic normality of their estimators and show that under certain smoothness assumptions, the limiting variance of a score estimator is arbitrarily close to the efficient variance. Their result is remarkable since it is the first time in the current status regression model that a $\sqrt n$-consistent estimate for $\bmb_0$ is proposed based on the MLE for $F_0$ which only converges at $n^{-1/3}$-rate to the true distribution function $F_0$.  
\\

We consider extending the estimators in \cite{kim_piet:18} to the more general single index regression problem and propose two different score equations involving the least squares estimator (LSE) $\hat \psi_{n\bma}$ of the unknown link function $\psi_0$ which minimizes
\begin{align}
\label{def:sum_of_squares_LSE}
	\frac{1}{n} \sum_{i=1}^n \Big \{  Y_i - \psi(\bma^T\bm X_i)  \Big \}^2,
\end{align}
over all monotone increasing functions $\psi$ for fixed $\bma$. We establish an $n^{1/3}/\log n$-rate for the estimator $\hat \psi_{n\bma}$ and propose a single index score estimator of $\bma_0$ that converges at the parametric $\sqrt n$-rate to the true regression parameter $\bma_0$.

\section{The single-index model with monotone link}
\label{section:model}
Consider the following regression model  
\begin{eqnarray}\label{Model}
Y = \psi_0(\bma_0^T\bm X)  + \e,
\end{eqnarray}
where $Y $ is a one-dimensional random variable, $\bm X = (X_1,\ldots, X_d)^T$ is a $d$-dimensional random vector with distribution $G$ and $\e$ is a one-dimensional random variable such that $\E[\e | \bm X] = 0$ $G$-almost surely. The function $\psi_0$ is a monotone link function in $\mathcal{M}$, where $\mathcal{M}$ is the set of monotone increasing functions defined on $\R$ and $\bma_0$ is a vector of regression parameters belonging to the $d-1$ dimensional sphere $\cS_{d-1}:=\{ \bma \in \R^d : \| \bma \| = 1\}$,  where $\|\,  \cdot \|$ denotes the Euclidean norm in $\R^d$.

\section{The least squares estimator (LSE) for the link function $\psi$}
\label{section:LSE}
Let $(\bm X_1, Y_1), \ldots, (\bm X_n, Y_n)$ denote $n$ random variables which are i.i.d. like $(\bm X, Y)$ in (\ref{Model}), i.e. $\E(Y|\bm X) = \psi_0(\bma_0^T\bm X)$ $G$-almost surely and consider the  sum of squared errors
\begin{align*}
S_n(\psi, \bma)= \frac{1}{n} \sum_{i=1}^n \Big \{  Y_i - \psi(\bma^T\bm X_i)  \Big \}^2,
\end{align*} 
which can be computed for any pair $(\psi,\bma) \in \mathcal{M}\times \cS_{d-1}$.
For a fixed $\bma$, order the values $\bma^T \bm X_1,   \ldots, \bma^T \bm X_n$ in increasing order and arrange $Y_1, \ldots, Y_n$ accordingly. As ties are not excluded, let  $m= m_{\bma}$ be the number of distinct projections among $Z_i = \bma^T \bm X_i$ and $Z^{\bma}_1 <  \ldots < Z^{\bma}_m$   the corresponding ordered values.   For $i=1, \ldots, m$, let   
$$n^{\bma} _i  = \sum_{j=1}^n  {1}_{\{\bma^T \bm X_j = Z^{\bma}_i\}}\quad \text{ and } \quad Y^{\bma}_i  =   \sum_{j=1}^n Y_j  {1}_{\{\bma^T \bm X_j = Z^{\bma}_i\}}  /n^{\bma}_i.$$
Then, well-known results from  isotonic regression theory imply that the functional $\psi\mapsto S_n( \bma,\psi)$ is minimized by the left derivative of the greatest convex minorant of the cumulative sum diagram  
\begin{align*}
\Big \{(0,0), \Big(\sum_{j=1}^i  n^{\bma}_j, \sum_{j=1}^i  n^{\bma}_j Y^{\bma}_j  \Big), i=1, \ldots, m \Big \}.
\end{align*}
See for example Theorem 1.1 in \cite{b4:72} or  Theorem 1.2.1 in \cite{rwd:88}.
By strict convexity of  $\psi \mapsto S_n(\psi, \bma)$, the minimizer is unique at the distinct projections.  We denote by  $\hat \psi_{n\bma}$  the monotone function which takes the values of this minimizer at the distinct projections and is a stepwise and right-continuous function outside the set of those projections. 
\\

In Section \ref{section:score-estimators} we illustrate how to derive score estimators, based on solving a score equation derived from the sum of squares $S_n$. We propose two different techniques to ensure that the estimator has length one. In the first approach, we consider a parametrization $\mbS: \R^{d-1} \mapsto \cS_{d-1}$ of the unit sphere and solve a set of $d-1$ score equations. In the second approach we add a Lagrange penalty to the sum of squared errors and differentiate the corresponding minimization criterion w.r.t the $d$ components of $\bma$.

\section{Score estimators for the regression parameter $\a$}
\label{section:score-estimators}

\subsection{The score estimator on the unit sphere}
\label{subsection:score-estimator1}
Consider the problem of minimizing
\begin{align}
\label{sum-of-squares0}
\frac1n\sum_{i=1}^n\left\{Y_i -\hat\psi_{n\bm\a}( \bma^T\bm X_i)\right\}^2,
\end{align}
over all $\bma \in \cS_{d-1}$ , where $\hat\psi_{n\bm\a}$ is the LSE of $\psi_{\bma}$. 
Let $\mbS$ be a local parametrization mapping $\R^{d-1}$ to the sphere $\cS_{d-1}$, i.e., for each $\bma \in \B(\bma_0,\d_0)$ on the sphere $\cS_{d-1}$, there exists a unique vector $\bmb \in \R^{d-1}$ such that
$$\bma = \mbS\left( \bmb\right).$$
The minimization problem given in (\ref{sum-of-squares0}) is equivalent to minimizing 
\begin{align}
\label{sum-of-squares}
\frac1n\sum_{i=1}^n\left\{Y_i -\hat\psi_{n\bm\a}\left( \mbS(\bmb)^T\bm X_i\right)\right\}^2,
\end{align}
over all $\bmb$ where $\hat\psi_{n\bm\a}$ is the LSE of the link function with $\bma = \mbS(\bmb)$. Analogously to the treatment of the score approach in the current status regression model proposed by \cite{kim_piet:18}, we consider the derivative of (\ref{sum-of-squares}) w.r.t. $\bmb$, where we ignore the non-differentiability of the LSE $\hat\psi_{n\bm\a}$. This leads to the set of $d-1$ equations,
\begin{align}
\label{score1}
\frac1n\sum_{i=1}^n \left(\bm J_{\mbS}(\bmb)\right)^T\bm X_i\left\{Y_i -\hat\psi_{n\bm\a}\left(\mbS(\bmb)^T\bm X_i\right)\right\} = \bm0,
\end{align}
where $\bm J_{\mbS}$ is the Jacobian of the map $\mbS$ and where $\bm 0 \in \R^{d-1}$ is the vector of zeros.
Just as in the analogous case of the ``simple score equation'' in \cite{kim_piet:18}, we cannot hope to solve equation (\ref{score1}) exactly due to the discrete nature of the score function in (\ref{score1}). Instead, we define the solution in terms of a ``zero-crossing'' of the above equation. The following definition is taken from \cite{kim_piet:18}.

\begin{definition}[zero-crossing]
	\label{zero-crossing}
	{
		We say that $\bmb_*$ is a crossing of zero of a real-valued function $\zeta: {\cal A}\mapsto \R: \bmb \mapsto \zeta(\bmb)$ if each open neighborhood of $\bmb_*$ contains points $\bmb_{1},\bmb_{2} \in {\cal A}$ such that $ \bar{\zeta}(\bmb_{1})\bar{\zeta}(\bmb_{2}) \le 0$, where $\bar\zeta$ is the closure of the image of the function (so contains its limit points). 
		We say that an $m$-dimensional function $\zeta:{\cal A}\mapsto \R^m: \bm\b \mapsto \zeta(\bm\b) = (\zeta_1(\bm\b),\ldots\zeta_m(\bm\b))'$ has a crossing of zero at a point $\bm\b_*$, if $\bm\b_*$ is a crossing of zero of each component $\zeta_j: {\cal A} \mapsto \R, j=1\ldots,m$.
	}
\end{definition}  

Our simple score estimator $\hat \bma_n$ (SSE) is defined by,
\begin{align}
\label{def:estimator1}
\hat\bma_n := \mbS(\hat \bmb_n),
\end{align}
where $\hat \bmb_n$ is a zero crossing of the function
\begin{align}
\label{def:phi_n}
\f_n(\bmb) :=\int \left(\bm J_{\mbS}(\bmb)\right)^T\bm x\left\{y -\hat\psi_{n\bm\a}\left(\mbS(\bmb)^T\bm x\right)\right\}\,d\P_n(\bm x,y),
\end{align}
and $\P_n$ denotes the empirical probability measure of $(\bm X_1, Y_1), \ldots, (\bm X_n, Y_n)$. The probability measure of $(\bm X, Y)$ will be denoted by $P_0$ in the remainder of the paper. For another formulation, directly in terms of $\bm\alpha$, without reparametrization, see the Lagrangian formulation in Section \ref{subsect:score-Lagrange} below.
\\

The SSE is based on a simplified version of the derivative of the sum of squared errors $S_n$ w.r.t. the components of $\bmb$, where we ignored the non-differentiability of the discrete LSE $\hat \psi_{n \bma}$. As a consequence, the limiting variance of this SSE, given in Section \ref{section:Asymptotics}, is not the efficient variance for the single index model. We can improve the SSE and extend this simplified score approach by incorporating an estimate of the derivative of the link function to obtain an efficient estimator of $\bma_0$. 
\\

Let $\hat \psi_{n \bma}$ denote again the LSE of the link function for fixed $\bma$ defined in Section \ref{section:LSE} and define the estimate $\tilde \psi_{nh,\bma}'$ by
\begin{align*}
\tilde\psi_{nh,\bm\a}'(u)=\frac1h\int K\left(\frac{u-x}h\right)\,d \hat\psi_{n\bm\a}(x),
\end{align*}
where $h$ is a chosen bandwidth. Here  $d \hat\psi_{n\bma}$ represents the jumps of the discrete function $\hat \psi_{n\bma}$ and $K$ is one of the usual symmetric twice differentiable kernels with compact support $[-1,1]$, used in density estimation.
The estimator $\tilde \bma_n$ is given by
\begin{align}
\label{def:estimator-efficient}
\tilde \bma_n := \mbS(\tilde \bmb_n),
\end{align}
where $\tilde \bmb_n$ is a zero crossing of $\xi_{nh}$ (see Definition \ref{zero-crossing}) defined by
\begin{align}
\label{def:xi_nh}
\xi_{nh}(\bmb) :=\int \left(\bm J_{\mbS}(\bmb)\right)^T \bm x \, \tilde\psi_{nh,\bm\a}'\bigl(\mbS(\bmb)^T\bm x\bigr) \left\{y-\hat\psi_{n\bm\a}\bigl(\mbS(\bmb)^T\bm x\bigr)\right\}\,d\P_n(\bm x,y).	
\end{align}
Again another formulation, directly in terms of $\bm\alpha$ is given in Section \ref{subsect:score-Lagrange} below. A picture of the estimates $\hat\psi_{n,\hat\bma_n}$ and $\tilde\psi'_{n,\hat\bma_n}$ for $n=1000$, $d=3$ and a sample for the model used in the simulation study of Table \ref{table:simulation1b} below, is shown in Figure \ref{fig:psi_and_psi'}. For the derivative a local linear extension of the function is used at the boundary for points with a distance to the boundary smaller than the bandwidth. Note that we only need one bandwidth choice for the derivative and that this is only needed for the ESE and not for the SSE.

\begin{center}
	[Figure \ref{fig:psi_and_psi'} here]
\end{center}

In the remainder of this section we illustrate how the score estimators can be obtained in practice using a local coordinate system representing the unit sphere. An example of such a parametrization is the spherical coordinate system $\mbS: [0,\pi]^{(d-2)}\times [0,2\pi]\mapsto \cS_{d-1}:$
\begin{align*}
(\b_1,\b_2,\ldots,\b_{d-1}) \mapsto
& (\cos(\b_1), \sin(\b_1)\cos(\b_2),\sin(\b_1)\sin(\b_2)\cos(\b_3),\ldots,\\ &\sin(\b_1)\ldots\sin(\b_{d-2})\cos(\b_{d-1}), \sin(\b_1)\ldots\sin(\b_{d-2})\sin(\b_{d-1}))^T.
\end{align*}
The map parameterizing the positive half of the sphere $\mbS: \{(\b_1,\b_2,\ldots,\b_{d-1})\in [0,1]^{(d-1)}: \|\bmb \|\le1\} \mapsto \cS_{d-1}:$
\begin{align*}
(\b_1,\b_2,\ldots,\b_{d-1}) \mapsto
& \left(\b_1,\b_2,\ldots,\b_{d-1}, \sqrt{1-\b_1^2-\ldots -\b_{d-1}^2}\right)^T,
\end{align*}
is another example that can be used provided $\a_d$ is positive. Prior knowledge about the position of $\bma_0$ can be derived from an initial estimate such as the LSE proposed in \cite{balabdaoui2016}.
We illustrate the set of equations for the SSE corresponding to (\ref{score1}) for dimension $d=3$ and consider the parametrization
\begin{align}
\label{parametrization-dim3}
\cS_3= \{(\a_1,\a_2,\a_3) = (\cos(\b_{1})\sin(\b_{2}),\sin(\b_{1})\sin(\b_{2}),\cos(\b_{2})) : 0\le \b_{1}\le 2\pi, 0\le \b_{2}\le \pi\}\subset \R^2.
\end{align}
The SSE can be obtained by solving the problem
\begin{align}
\label{score-dimension3}
\left \{
\begin{array}{l}
s_1(\b_1,\b_2) = \frac1n\sum_{i=1}^n (-\sin(\b_{1})\sin(\b_{2})X_{i1} + \cos(\b_{1})\sin(\b_{2})X_{i2} )\left\{Y_i -\hat\psi_{n\bm\a}(\bma^T\bm X_i)\right\} = 0, \\
s_2(\b_1,\b_2)=\frac1n\sum_{i=1}^n (\cos(\b_{1})\cos(\b_{2})X_{i1} + \sin(\b_1)\cos(\b_{2})X_{i2} - \sin(\b_{2})X_{i3})\left\{Y_i -\hat\psi_{n\bm\a}(\bma^T\bm X_i)\right\} = 0.
\end{array}
\right.
\end{align}

For the parameterizations discussed in this manuscript we have that for each map $\mbS$ and each parameter vector $\bmb$, 
$$\left(\mbS(\bmb) \right)^T \mbS(\bmb) =  1.$$ 
Taking derivatives w.r.t. $\bmb$, we get 
$$\left(\mbS(\bmb) \right)^T\bm J_{\mbS}(\bmb) = \bm 0^T, $$
so that the columns of $\bm J_{\mbS}(\bmb)$ belong to the space
$$\{\bma \}^\perp \equiv \left\{\mbS(\bmb)\right\}^\perp \equiv \left\{\bm z \in \R^d: \bma^T z= 0  \right\}\equiv \left\{\bm z \in \R^d: \left(\mbS(\bmb)  \right)^T z= 0  \right\}.$$
\\
Note that for $\mbS(\bmb) = \left( \cos(\b_{1})\sin(\b_{2}),\sin(\b_{1})\sin(\b_{2}),\cos(\b_{2})\right)^T $,  
\begin{align}
\label{bmJ_s}
\bm J_{\mbS}(\bmb) =
\begin{bmatrix}
-\sin(\b_{1})\sin(\b_{2}) & \cos(\b_{1})\cos(\b_{2}) 
\\[0.3em]
\cos(\b_{1})\sin(\b_{2}) &\sin(\b_{1})\cos(\b_{2}) \\[0.3em]
0 & - \sin(\b_{2}) 
\end{bmatrix},
\end{align}
such that we indeed have
\begin{align}
\label{property-StJ}
\mbS(\bmb) ^T\bm J_{\mbS}(\bmb) = \left( 0,0\right),
\end{align}
for all $\bmb$. 
This again implies that the columns of $\bm J_{\mbS}(\bmb) $ are perpendicular to the vector $\bma =\mbS(\bmb)$. Note moreover that the columns are linearly independent and hence form a basis for $\{\bma\}^\perp$.
\\

It is shown in Lemma 1 of \cite{Kuchibhotla_patra_sen:17} that it is possible to construct a set of ``local parametrization matrices'' $H_{\bma}$ for each $\bma \in \B(\bma_0,\d_0)$ with $\|\bma\| = 1$ satisfying
\begin{align*}
\bma^T H_{\bma} = \bm 0^T \quad \text{ and }\quad \left(H_{\bma}\right)^T  H_{\bma} = \bm I_{d-1}
\end{align*}
Their matrix $\left(H_{\bma} \right)^T$ corresponds to the Moore-Penrose pseudo-inverse of the matrix $H_{\bma}$ and is the analogue of our matrix $\left(\bm J_{\mbS}(\bmb)\right)^T$  in the proof of asymptotic normality of their estimator.  We will however show that the orthonormality assumption is not needed in the proofs.
\\

\subsection{The score estimator with Lagrange penalty}
\label{subsect:score-Lagrange}

Instead of tackling the fact that our parameter space is essentially of dimension $d-1$ by the parametrization $\bma = \mbS(\bmb)$ which locally maps $\R^{d-1}$ into the sphere $\cS_{d-1}$, one can  introduce the restriction $\|\bma\|=1$ via a Lagrangian term. We then consider the problem of minimizing
\begin{align}
\label{Lagrangian-sum-of-squares}
\frac1n\sum_{i=1}^n\left\{Y_i -\hat\psi_{n\bm\a}( \bma^T\bm X_i)\right\}^2+\l\left\{\|\bm\a\|^2-1\right\},
\end{align}
where $\hat\psi_{n\bm\a}$ is the LSE defined in Section \ref{section:LSE} and $\l$ is a Lagrange parameter which we add to the sum of squared errors to deal with the identifiability of the single-index model. 

Analogously to the treatment given in Section \ref{subsection:score-estimator1} for the SSE, we consider the derivative of (\ref{Lagrangian-sum-of-squares}) w.r.t. $\bma$, where we ignore the non-differentiability of the LSE $\hat\psi_{n\bm\a}$. This leads to the equation,
\begin{align}
\label{lagrange-score1}
-\frac1n\sum_{i=1}^n\bm X_i\left\{Y_i-\hat\psi_{n\bm\a}(\bm\a^T\bm X_i)\right\}+\l\bma=0.
\end{align}
Here $\l$ has to satisfy
\begin{align}
\label{definition_lambda}
\l=\l\bma^T\bma=\frac1n\sum_{i=1}^n\bm\a^T\bm X_{i}\left\{Y_i-\hat\psi_{n\bm\a}(\bm\a^T\bm X_i)\right\}.
\end{align}
Plugging in the above expression for $\l$ in (\ref{lagrange-score1}), we consider the simple score equation
\begin{align}
\label{score-equation}
\bm 0 &=\frac1n\sum_{i=1}^n\bm X_i\left\{Y_i -\hat\psi_{n\bm\a}(\bm\a^T\bm X_i)\right\} - \bm\a^T\left (\frac1n\sum_{i=1}^n\bm X_i \left\{Y_i -\hat\psi_{n\bm\a}(\bm\a^T\bm X_i)\right\} \right) \bm\alpha\nonumber\\
&=\frac1n\left(\bm I-{\bm\a}\bma^T\right)\sum_{i=1}^n\bm X_i \left\{Y_i -\hat\psi_{n\bm\a}(\bm\a^T\bm X_i)\right\},
\end{align}
where $\bm I$ is the $d\times d$ identity matrix. The same procedure can be derived for the ESE defined in (\ref{def:estimator-efficient}) and we define the simple and efficient score estimators of the regression parameter $\bma_0$ in model (\ref{Model}), referred to as the score estimators using a Lagrange penalty, by zero crossing of the corresponding  score functions 
\begin{align}
\label{score-equation_Lagrange1}
\bma \mapsto \frac1n\left(\bm I-{\bm\a}\bma^T\right)\sum_{i=1}^n\bm X_i \left\{Y_i -\hat\psi_{n\bm\a}(\bm\a^T\bm X_i)\right\},
\end{align}
and
\begin{align}
\label{score-equation_Lagrange2}
\bma \mapsto \frac1n\left(\bm I-{\bm\a}\bma^T\right)\sum_{i=1}^n\bm X_i \tilde \psi_{nh,\bm\a}'\bigl(\bm\a^T\bm X_i\bigr)  \left\{Y_i -\hat\psi_{n\bm\a}(\bm\a^T\bm X_i)\right\},
\end{align}
respectively.

The Lagrange approach has the advantage that we do not have to deal with the reparametrization $\bma = \mbS(\bmb)$, but has the disadvantage that we cannot assume that $\hat\bma_n$ has exactly norm $1$ because we again have to deal with crossings of zero instead of exact equality to zero.  One way to circumvent this problem is to normalize the solution of the right-hand side of (\ref{score-equation_Lagrange1}) or (\ref{score-equation_Lagrange2}) at the end of the iterations. This technique gave approximately the same solutions as the approach via reparametrization.

In Section \ref{section:Asymptotics} we will only derive the limiting behavior of the score estimators using the parametrization of the unit sphere, but we conjecture that both estimators, using the parametrization or the Lagrange penalty, have the same asymptotic properties. A conjecture that is further motivated by a simulation study presented in Section \ref{section:simulations}.  Since the Lagrange approach avoids the mapping into the parameter space (which depends on the dimension $d$), this technique can easily be adapted for different dimensions and is favored over the parametrization score approach from a practical point of view, especially if the dimension is large. Examples of {\tt R} scripts, using {\tt Rcpp}, of simulation runs with this method are given in \cite{github:18}.

\section{Linear estimates if the covariates have an elliptically symmetric distribution}
\label{sec:basic_estimate}
It is well-known that if $\bm X$ has an elliptically symmetric distribution there exist link-free ordinary linear regression estimates of $\bma_0$ which are $\sqrt{n}$-convergent and have an asymptotic normal distribution, see \cite{duan1991}. The following estimator of this type is defined in \cite{balabdaoui2016} for the monotone single index model. Let $\hat\bma_n$ be defined by
\begin{align}
\label{LS_BDJ}
\hat{\bm\alpha}_n=\argmin_{\bm\a\in\R^d}\sum_{i=1}^n\left\{Y_i-\bm\a^T(\bm X_i-\bar{\bm X}_n)\right\}^2,
\end{align}
with $\bar{\bm X}_n =1/n \sum_{i=1}^{n}\bm X_i$ being the sample mean of the covariate vector. The estimate of the regression parameter $\bma_0$ in model (\ref{Model}) is now given by $\tilde{\bm\a}_n=\hat{\bm\alpha}_n/\|\hat{\bm\alpha}_n\|$. Note that $\hat \bma_n$ in (\ref{LS_BDJ}) is the estimator one would use if the link function is known to be linear and one would not make the restriction that the estimator has norm 1. The following result is proved in \cite{balabdaoui2016}.

\begin{theorem}
\label{Th:H-LFLSE}
Let $(\bmX_1,Y_1),\dots,(\bm X_n,Y_n)$ be an i.i.d.\ sample from $(\bmX,Y)$ such that $E(Y|\bmX)=\psi_0(\bma_0\bmX)$ almost surely, where $\psi_0$ is non-decreasing and $\bma_0^T\bma_0=1$. Suppose that $\bm X$ has an elliptically symmetric distribution with finite mean $\bm\mu\in\R^d$ and a positive definite covariance matrix $\bm\Sigma$. Assume, moreover, that $\E\|Y\bmX\|<\infty$ and that there exists a nonempty interval $[a,b]$ on which $\psi_0$ is strictly increasing. Then, as $n\to\infty$, the estimator $\tilde{\bma}_n=\hat{\bm\alpha}_n/\|\hat{\bm\alpha}_n\|$, where $\hat\bma_n$ is defined by (\ref{LS_BDJ}), converges in probability to $\bma_0$. If, moreover, $\E Y^2\|\bmX\|^2<\infty$, $\sqrt{n}\{\tilde\bma_n-\bma_0\}$ converges in distribution to a normal distribution with mean $\bm 0$ and covariance matrix
\begin{align*}
\frac1{c^2}\left(\bm I-\bm\a_0\bm\a_0^T\right)\bm\Sigma^{-1}\bm \Gamma\bm\Sigma^{-1}\left(\bm I-\bm\a_0\bm\a_0^T\right),\qquad
c=\text{\rm Cov}\left( \psi_0(\bma_0^T\bm X), \bma_0^T\bmX  \right)/\bma_0^T\bm \Sigma\bma_0,
\end{align*}
where $\bm \Gamma$ is the covariance matrix of $\left\{Y-\E Y -(\bmX-\bm\mu)^T\bm\Sigma^{-1}\text{\rm cov}(\bmX,Y)\right\}(\bmX-\bm\mu)$.
\end{theorem}

\begin{remark}
{ To avoid unnecessarily heavy notation, we denote in this section and the corresponding Appendix \ref{supA:LFLSE} the covariance matrix of $\bmX$ by $\bm\Sigma$ and the corresponding sample covariance matrix by $\bm S_n$, but note that $\bm\Sigma$ and $\bm S_n$ have a different meaning in the rest of the paper (see, e.g., (\ref{def_Sigma}) and (\ref{def_bmS})).
}
\end{remark}

Instead of first calculating the estimate $\hat\bma_n$ in (\ref{LS_BDJ}) and then dividing by the norm of this estimate to get norm 1, one can also consider the estimate $\hat{\bm\a}_{n,\bmS_n norm1}$, defined by
\begin{align}
\label{LS_norm1}
\hat{\bm\a}_{n,\bmS_n norm1}=\argmin_{\bma\in\R^d,\,\bm\a^T\bm S_n\bma=1}\sum_{i=1}^n\left\{Y_i-\bm\a^T(\bm X_i-\bar{\bm X}_n)\right\}^2,
\end{align}
which is more in line with the estimators of our paper, where we compute the estimators under the condition that the norm is equal to 1. Here $\bmS_n$, defined by
\begin{align*}
\bmS_n=\frac1n\sum_{i=1}^n\left\{\bmX_i-\bar\bmX_n\right\}\left\{\bmX_i-\bar\bmX_n\right\}^T,
\end{align*}
estimates the covariance $\bm\Sigma$, and the ordinary inner product of $\bm x$ and $\bm y$ is replaced by $\bm x^T\bmS_n\bm y$. Instead of the restriction $\bma_0^T\bma_0=1$, we now use the restriction $\bma_0^T\bm\Sigma\bma_0=1$ in the underlying model, which is estimated by the restriction $\bma^T\bmS_n\bma=1$ in the sample.\\

 We will call this estimator the {\it link-free least squares estimator} (LFLSE). Since the estimator discussed above first solves another minimization problem (without a restriction on the norm), and then makes the (ordinary) norm equal to 1 by dividing by the norm, we call this the {\it hybrid link-free least squares estimator} (H-LFLSE).
The two estimators are not the same, even if $\bm\Sigma$ is a multiple of the identity and if we use the ordinary norm for the second estimator.\\

Suppose now that the true index $\bma_0$ satisfies $\bma^T_0\bm\Sigma\bma_0 =1$.  Note that such normalization of the true parameter $\bma_0$ is always possible  since it does not alter the direction of monotonicity of the link function nor the identifiability of the model. Then, following the Lagrange approach of Section  \ref{subsect:score-Lagrange}, we can define the estimator $\hat{\bm\a}_{n,\bmS_n norm1}$ as the minimizer of
\begin{align*}
\frac1n\sum_{i=1}^n\left\{Y_i-\bm\a^T(\bm X_i-\bar{\bm X}_n)\right\}^2 + \l \{\bm\a^T\bmS_n\bma - 1\},
\end{align*}
for $\bma \in \R^d$ and for suitably chosen $\l\ge0$.
This time, the optimization criterion does not depend on the LSE $\hat \psi_{n \bma}$ of the link function $\psi_0$ and we do no longer have the crossing of zero difficulty. 
Note that in this case (\ref{lagrange-score1}) is replaced by
\begin{align}
\label{lagrange-score1a}
-\frac1n\sum_{i=1}^n\bm X_i\left\{Y_i-\bm\a^T\left(\bm X_i-\bar{\bm X}_n\right)\right\}+\l\bmS_n\bma=\bm0.
\end{align}
Here $\l$ has to satisfy
\begin{align}
\label{definition_lambda_LFLSE}
\l=\l\bma^T\bmS_n\bma=\frac1n\sum_{i=1}^n\bm\a^T\left(\bm X_i-\bar{\bm X}_n\right)\left\{Y_i-\bm\a^T\left(\bm X_i-\bar{\bm X}_n\right)\right\}.
\end{align}
Since therefore
\begin{align*}
\l\bmS_n\bma&=\frac1n\sum_{i=1}^n\bm\a^T\left(\bm X_i-\bar{\bm X}_n\right)\left\{Y_i-\bm\a^T\left(\bm X_i-\bar{\bm X}_n\right)\right\}\bmS_n\bma\\
&=\frac1n\bmS_n\bma\bma^T\sum_{i=1}^n\bm\a^T\left(\bm X_i-\bar{\bm X}_n\right)\left\{Y_i-\bm\a^T\left(\bm X_i-\bar{\bm X}_n\right)\right\},
\end{align*}
we get the equation
\begin{align}
\label{Lagrange_norm1}
\frac1n\left(\bm I-\bmS_n\bma\bma^T\right)\sum_{i=1}^n\left(\bm X_i -\bar{\bm X}_n\right)\left\{Y_i-\bm\a^T\left(\bm X_i -\bar{\bm X}_n\right)\right\}=\bm0,
\end{align}
where, if more solutions are found, the one giving the smallest criterion is chosen.
For this estimator we have the following result.
\begin{theorem}
\label{Th:LFLSE}
Let $(\bmX_1,Y_1),\dots,(\bm X_n,Y_n)$ be an i.i.d.\ sample from $(\bmX,Y)$ such that $E(Y|\bmX)=\psi_0(\bma_0\bmX)$ almost surely, where $\psi_0$ is non-decreasing. Suppose that $\bm X$ has an elliptically symmetric distribution with finite mean $\mu\in\R^d$ and a positive definite covariance matrix $\bm\Sigma$ satisfying $\bma_0^T\bm\Sigma\bma_0=1$. Assume, moreover, that $\E|Y|<\infty$ and $\E\|\bmX\|^2<\infty$ and that there exists a nonempty interval $[a,b]$ on which $\psi_0$ is strictly increasing. Then, as $n\to\infty$, the estimator $\hat{\bm\a}_{n,\bmS_n norm1}$, defined by (\ref{LS_norm1}), converges in probability to $\bma_0$. If, moreover, $\E Y^2<\infty$ and $\E\|\bmX\|^4<\infty$, then $\sqrt{n}\{\hat{\bm\a}_{n,\bmS_n norm1}-\bma_0\}$ converges in distribution to a normal distribution with mean $\bm 0$ and covariance matrix which can be computed from relation (\ref{asymptotic_relation_elliptic_sym}) given in Appendix \ref{supA:LFLSE} in the Supplementary Material.
\end{theorem}

The proof of Theorem \ref{Th:LFLSE} is given in the Supplementary Material.
As an example of how one can compute the asymptotic covariance matrix, we also compute in the Supplementary Material, the asymptotic covariance matrix for the simulation setting corresponding to Table \ref{table:simulation2b}, given in Section \ref{subsec:simulation1}. The solution of equation (\ref{Lagrange_norm1}) was done by a {\tt{C++}} program, using Broyden's algorithm. It is very fast and produces a norm $\{\hat{\bm\a}_{n,\bmS_n norm1}^T\bm S_n\hat{\bm\a}_{n,\bmS_n norm1}\}^{1/2}$ of the solution which is equal to 1 in 10 decimals, without any need of renormalization, but just by solving in $\R^d$. This illustrates the soundness of the Lagrange inspired method of estimating the parameter $\bma_0$ by solving (\ref{Lagrange_norm1}) in $\R^d$ as an alternative to shifting to a lower dimensional parametrization. Since equation  (\ref{score-equation}) for the score estimates of Section \ref{subsect:score-Lagrange} is discontinuous and cannot be solved exactly, we use a derivative free optimization algorithm proposed by \cite{hooke1961} instead of Broyden's algorithm to obtain the score estimates as zero crossings of equation (\ref{score-equation}). More information on the computation of the score estimates in Section \ref{section:score-estimators} is given in Section \ref{section:simulations}. {\tt R} scripts for simulations with the estimator of Theorem \ref{Th:LFLSE}, using the derivative free optimization algorithm, are also made available in \cite{github:18}.

\begin{remark}
\label{remark_spher_sym}
{For the situation where the covariance matrix of $\bm X$ is assumed to be a multiple of the identity (so $\bm X$ has a {\it spherically} symmetric distribution), we also have a consistent estimate of $\bma_0$ if we define $\hat{\bm\a}_{n,norm1}$ by
\begin{align}
\label{LS_norm1a}
\hat{\bm\a}_{n,norm1}=\argmin_{\bma\in\R^d,\,\bm\a^T\bma=1}\sum_{i=1}^n\left\{Y_i-\bm\a^T(\bm X_i-\bar{\bm X}_n)\right\}^2,
\end{align}
under the condition that $\bma_0^T\bma_0=1$. In this case the limiting distribution of $\sqrt{n}\{\hat{\bm\a}_{n, norm1}-\bma_0\}$ is degenerate,  as is the case for the score estimators of Section \ref{section:score-estimators} (see Theorems \ref{theorem:asymptotics}-\ref{theorem:asymptotics-efficient} and Remark \ref{remark:degenerate} in Section \ref{section:Asymptotics}). The derivation of the consistency and normal limit distribution proceeds along the same lines as the proof of Theorem \ref{Th:LFLSE}, but is somewhat easier since we do not have to deal with the behavior of $\bmS_n$. However, this estimate will be inconsistent for the situation that $\bm\Sigma$ is not a multiple of the identity. We give the asymptotic covariance matrix in Appendix \ref{supA:LFLSE} for the simulation setting of Table \ref{table:simulation2b}. Remarkably, the asymptotic variances are bigger than for the estimate $\hat{\bm\a}_{n,\bmS_n norm1}$ of Theorem \ref{Th:LFLSE} in this situation.
}
\end{remark}

\section{Asymptotic behavior of the score estimators}
\label{section:Asymptotics}
In this section we first give results on the behavior of the LSE $\hat \psi_{n \bma}$ of the monotone link function and next derive the limiting distribution of the SSE and the ESE introduced in Section \ref{subsection:score-estimator1}. The results stated in this Section are all proved in the Supplementary Material of this article. 

\begin{proposition}\label{prop:maximizer-population}
	Let the function $\psi_{\bm\a}$ be defined by
	\begin{eqnarray}
	\label{def_psi_alpha}
	\psi_{\bm\a}(u)  := \E\big[\psi_0(\bma_0^T\bm X) |  \bma^T\bm X = u \big].
	\end{eqnarray}
	Suppose that  
	\begin{description}
		\item A1. The space $\mathcal{X}$ is convex, with a nonempty interior. There exists also $R > 0$ such that $\mathcal{X}  \subset \mathcal{B}(0, R)$.
		
		\item A2. There exists $K_0 > 0$ such that the true link function $\psi_0$ satisfies  $|\psi_0(u)| \le K_0$ for all $u$ in $\{\bma^T\bm  x, \bm x \in \mathcal{X},\bma \in \cS_{d-1}   \}$. 

         \item A3.  There exists $\delta_0 > 0$ such that the function $\psi_{\bma}$ defined in (\ref{def_psi_alpha}) is monotone increasing on $\mathcal{I}_{\bma}  :=  \{\bma^T\bm  x, \bm x \in \mathcal{X}   \}  $   for all  $\bma  \in \B(\bma_0, \delta_0):= \{\bma : \|\bma -\bma_0\| \le \delta_0 \}$. 		
	\end{description}

	Then,  the functional $L_{\bma}$ given by,
	\begin{eqnarray}
	\label{Distpsipsi0}
	\psi \mapsto L_{\bma}(\psi):=\int_{\mathcal{X}}  \Big(  \psi_0(\bma_0^T\bm x) -\psi(\bma^T\bm x)   \Big)^2 dG(\bm x) ,
	\end{eqnarray}
	admits a minimizer ${\psi}^{\bma}$,  over the set of monotone increasing functions defined on $\R$, denoted by $\M$,  such that $\psi^{\bma}$ is uniquely given  by the function ($\psi_{\bm\a}$) in (\ref{def_psi_alpha}) on $\mathcal{I}_{\bma}  = \Big\{ \bma^T \bm x: \bm x \in \mathcal{X}   \Big \} $. 
\end{proposition}

\begin{proposition}
	\label{prop:L_2-psi-psi_n-alpha}
	Under Assumptions A1-A3 and Assumptions
	\begin{description}	
		
    \item A4. Let $a_0$ and $b_0$ denote the infimum and supremum of the interval $\mathcal{I}_{\bma_0}= \big \{ \bma_0^T  \bm x,  \ \bm x \in \mathcal{X}  \big \}$.  Then, the true link function $\psi_0$ is continuously differentiable on $(a_0 - \delta_0 R, b_0 + \delta_0  R)$, where $R$ is the same radius of assumption A1 above. 
		
		\item A5. The distribution of $\bm X$ admits a density $g$, which is differentiable on $\mathcal{X}$. Also, there exist  positive constants $\underline{c}_0 $,  $\bar{c}_0$,  $\underline{c}_1 $ and   $\bar{c}_1 $ such that $\underline{c}_0 \le g \le \bar{c}_0$ and $\underline{c}_1 \le \partial g/\partial x_i  \le \bar{c}_1$  on $\mcX$  for all $i=1, \ldots, d$.
		
		\item A6.  There exist $c_0 > 0$ and $M_0 > 0$ such that $\E\Big[\vert Y  \vert^m \ | \bm X=\bm x \Big] \le m! M_0^{m-2} c_0$  for all integers $m \ge 2$ and $\bm x \in \mcX$  $G$-almost surely.  		
	\end{description}
	 we have,
	\begin{align*}
	\sup_{\bma \in \B(\bma_0,\delta_0)} \int\left\{\hat \psi_{n\bma}(\bma^T\bm x)-\psi_{\bm\a}(\bma^T\bm x)\right\}^2\,dG(\bm x)=O_p\left((\log n) ^2n^{-2/3}\right).		
	\end{align*}	
\end{proposition}

\paragraph{Discussion of the Assumptions A1-A6}  Before presenting our main theorem, we would like to first comment on the different assumptions made so far. Convexity in Assumption A1 is satisfied by a wide range of distributions, and hence is not very restrictive. It implies that the support of the linear predictor $\bma^T \bm X$ is an interval for all $\bma \in \B(\bma_0,\delta_0)$, which makes things easier to visualize. This can be however generalized by assuming that $\mathcal{X}$ is the union of convex sets; a very related  assumption was made in \cite{hardle:93}. Boundedness in Assumption A1 can be relaxed and replaced by sub-Gaussianity of the distribution of $\bm X$; see Remark \ref{RemarkBounded}.

In Assumption A2 we only  impose boundedness on the true regression function on $\{\bma_0^T \bm x, \bm x \in \mathcal{X}\}$, whereas other estimation procedures require boundedness on the second derivative of $\psi_0$, as done for example in  \cite{hardle:93} and \cite{hristache01}. 

Assumption A3 is made to enable deriving the explicit limit of  the LSE $\hat \psi_{n\bma}$  for all $\bma \in \B(\bma_0,\delta_0)$.	
In the Supplementary Material it is shown that Assumption A3 is plausible by proving that for $\bma$ in a neighborhood of $\bma_0$ the derivative of the function $\psi_{\bma}$ is indeed strictly positive provided the derivative of the true link function $\psi_0$ stays away from zero; i.e., there exists $C > 0$ such that $\psi'_0 \ge C $ on $\{\bma_0^T \bm x, \bm x \in \mathcal{X}\}$. However, it can be shown that the latter condition  can be made without loss of generality. A proof can be found in the Supplementary Material. The idea is to artificially add to both sides of the  regression model a function of $\bma_0^T\bm X$ with a strictly positive derivative without violating the remaining assumptions. Based on several numerical experiments it seems that Assumption $A3$ remains plausible in practice even if $\bma$ is not not necessarily in the neighborhood of $\bma_0$.  In Figure \ref{fig:psi_alpha} we compare the true link function $\psi_0$ with the function $\psi_{\bma}$ for the model $E(Y| \bm X) = \psi_0(\a_{01}X_1 + \a_{02}X_1)$, where $X_1, X_2 \stackrel{i.i.d}{\sim} U[0,1]$, $\psi_0(\bm x)= x^3$ and $\a_{01} = \a_{02} =  1/\sqrt 2$ for $\a_{1} = 1/2, \a_{2} =  \sqrt 3/2$. Figure \ref{fig:psi_alpha} shows how the function $\psi_{\bma}$ defined in (\ref{def_psi_alpha}) inherits the monotonicity of the true link function $\psi_0$. 

\begin{center}
	[Figure \ref{fig:psi_alpha}  here]
\end{center}

The Assumptions A4 and A5, and sometimes stronger versions therereof, are made in many references on single index models. Here, they are mainly needed to be able to control the conditional expectation of  $\bm X$ or $\bma^T_0 \bm X$ given $\bma^T \bm X$ when $\bma$ is in a small neighborhood  of $\bma_0$.  Finally,  Assumption A6 is needed to show that $\max_{1\le i\le n} |Y_i| = O_p(\log n)$. As noted in \cite{balabdaoui2016}, such an assumption is satisfied in the special case where the conditional distribution of $Y|\bm X =\bm x$ belongs to an exponential family.

\begin{remark}
\label{RemarkBounded}
{
Boundedness in A1 can be relaxed if we assume that $\bm X$ has a sub-Gaussian distribution. We recall that $\bm X $ is sub-Gaussian if there exists $\sigma > 0$ such that for all $\bm u \in \mathcal{S}_{d-1}$ and $t \in \RR$ 
	\begin{eqnarray*}
		P(\bm u^T (\bm X  -  E(\bm X)) > t)  \le \exp(-t^2/(2\sigma^2))  \ \ \ \textrm{and} \ \  P(\bm u^T (\bm X  -  E(\bm X)) <- t)  \le \exp(-t^2/(2\sigma^2)).
	\end{eqnarray*}
	Following similar arguments as in \cite{balabdaoui2016} in the proof of Proposition \ref{prop:L_2-psi-psi_n-alpha} given in the Supplementary Material, we can easily show that sub-Gaussianity of $\bm X$ implies that  
	$$\sup_{\bma \in \B(\bma_0,\delta_0)} \int\left\{\hat \psi_{n\bma}(\bma^T\bm x)-\psi_{\bm\a}(\bma^T\bm x)\right\}^2\,dG(\bm x)  =  O_p\left((\log n) ^{5/2} n^{-2/3}\right).$$  
	The larger power in the logarithmic factor, in comparison with the power $2$ obtained in Proposition \ref{prop:L_2-psi-psi_n-alpha},  does not affect the conclusion about the asymptotic behavior of the score estimator since the proof still works for any uniform convergence rate of the form $(\log n)^\gamma n^{-2/3}$.

An important special case of sub-Gaussian distributions is that of normal distributions. Since these are also elliptically symmetric (spherically symmetric if they are centered with a diagonal covariance matrix) then it is known that an ordinary least squares estimator can consistently estimate the true direction of $\bma_0$, provided that $cov(\psi_0(\bma_0 \bm X), \bma^T_0 X) \ne 0$.  The latter condition can be easily shown to hold true when $\psi_0$ is monotone and non-flat. This result, which goes back to \cite{brillinger83}, can be extended to any elliptically symmetric distribution, a fact that has been exploited by \cite{duan1991} in inverse regression.  In Section \ref{sec:basic_estimate}, we discussed this simple estimator after normalization in \cite{balabdaoui2016} and introduced a new estimator for $\bma_0$. Both estimators are asymptotically normal.
}
\end{remark}

\medskip

\begin{remark}
{
We would like to note that in our formulation of the model, we do assume that the predictor $\bm X$ and error $\varepsilon$ are independent; the assumption has also been made in e.g.  \cite{duan1991}, \cite{sherman-cavanagh98}, \cite{hristache01}, where the $\e$ is moreover assumed to be normally distributed with mean $0$ and some finite variance (not depending on $\bm X$).   Such an assumption is unfortunately violated by many statistical models.  For illustration, take the logistic regression with a Bernoulli response $Y$ and success probability $\pi(\bm X) = \exp(\bma_0^T \bm X)/ (1+ \exp(\bma_0^T \bm X))$. Then, the conditional variance $\text{Var}(Y | \bm X) = \pi(\bm X) (1- \pi(\bm X)) = \exp(\bm{ \alpha_0}^T \bm X) / (1+ \exp(\bm{\alpha_0}^T X))^2$, which depends on $\bm X$.
}
\end{remark}

Proposition \ref{prop:L_2-psi-psi_n-alpha} is used to derive the following result on the asymptotic distribution of the SSE defined in (\ref{def:estimator1}) in Section \ref{subsection:score-estimator1}.

\begin{theorem}
	\label{theorem:asymptotics}
	Let Assumptions A1-A6 be satisfied and assume that
	\begin{description}
		\item A7. For all $\bmb \ne \bmb_0$ such that $\mbS\left(\bmb\right) \in \B(\bma_0, \delta_0)$, the random variable 
		\begin{align*}
		\text{\rm Cov}\Big[(\bmb_0-\bmb)^T\bm J_{\mbS}(\bmb)^T\bm X,  \psi_0\big(\mbS(\bmb_0)\big)  | \ \mbS(\bmb)^T\bm X \Big],
		\end{align*}
		is not equal to $0$ almost surely.  
		\item A8. The functions $\bm J_{\mbS}^{ij}(\bmb)$, where $\bm J_{\mbS}^{ij}(\bmb)$ denotes the $i\times j$ entry of $ \bm J_{\mbS}(\bmb)$ for $i=1,\ldots, d$ and $j=1,\ldots,d-1$ are $d-1$ times continuously differentiable on $\mathcal{C} := \{\bmb \in \R^{d-1}:\mbS\left(\bmb\right) \in \B(\bma_0, \delta_0) \}$ and there exists $M > 0$ satisfying 
		\begin{eqnarray}\label{DiffAssump}
		\max_{k. \le d-1} \sup_{\bmb \in \mathcal{C} } \vert D^{k} \bm J_{\mbS}^{ij}(\bmb)  \vert  \le M
		\end{eqnarray}
		where $k = (k_1, \ldots, k_d)$  with $k_j$ an integer $\in \{0, \ldots, d-1 \}$, $k. = \sum_{i=1}^{d-1} k_i$ and 
		\begin{eqnarray*}
			D^k s(\bmb)\equiv \frac{\partial^{k.} s(\bmb)}{\partial \beta_{k_1} \ldots  \partial \beta_{k_d}}.
		\end{eqnarray*}
		We also assume that $\mathcal{C}$ is a convex and bounded set in $\R^{d-1}$ with a nonempty interior.
		\item A9. $\left(\bm J_{\mbS}(\bmb_0)\right)^T \E\Bigl[\psi_0'(\bma_0^T\bm X)\,\text{  Cov}(\bm X|\bma_0^T\bm X)\Bigr]  \left(\bm J_{\mbS}(\bmb_0)\right)$  is nonsingular.  \end{description}
	 
	 Let  $\hat \bma_n$ be defined by (\ref{def:estimator1}). Then:
	\begin{enumerate}
		\item[(i)]
		\text{[Existence of a root]} A crossing of zero $\hat\bmb_n$ of $\f_n(\bmb)$ exists with probability tending to one.
		\item[(ii)]
		\text{[Consistency]}
		\begin{align*}
		\hat\bma_n \stackrel{p}{\rightarrow}\bma_0,\qquad n\to\infty.
		\end{align*}
		\item[(iii)]
		\text{[Asymptotic normality]}	
		Define the matrices,
		\begin{align}
		\label{def_A}
		\bm A:=\E\Bigl[\psi_0'(\bma_0^T\bm X)\,\text{\rm Cov}(\bm X|\bm\a_0^T\bm X)\Bigr],
		\end{align}
		and
		\begin{align}
		\label{def_Sigma}
		\bm \Sigma:=\E\left[\left\{Y -\psi_0(\bm\a_0^T\bm X)\right\}^2\,\left\{\bm X -\E(\bm X|\bm\a_0^T\bm X) \right\}\left\{\bm X -\E(\bm X|\bm\a_0^T\bm X) \right\}^T\right].
		\end{align}
		Then 
		\begin{align*}
		\sqrt n (\hat \bma_n - \bma_0) \to_d N_{d}\left(\bm 0,  \bm A^- \bm \Sigma \bm A^-\right),
		\end{align*}
		where $\bm A^{-}$ is the Moore-Penrose inverse of $\bm A$.
	\end{enumerate}
\end{theorem}

It can be easily seen from expression (\ref{bmJ_s}) for the matrix $\bm J_{\mbS}$ that the spherical coordinate system satisfies Assumption A8. In Section \ref{section:simulations}, we calculate the matrix specified in Assumptions A9 for the simulation model considered in this section and show that this assumption is indeed satisfied in the corresponding model.

\begin{remark}
\label{remark:degenerate}
{
Note that $\bma_0^T\bm A = \bm 0$ and that the normal distribution $N_{d}\left(\bm 0,  \bm A^- \bm \Sigma \bm A^-\right)$ is concentrated on the $(d-1)$-dimensional subspace, orthogonal to $\bma_0$ and is therefore degenerate, as is also clear from its covariance matrix $\bm A^- \bm \Sigma \bm A^-$, which is a matrix of rank $d-1$.
}
\end{remark}

For the ESE, we designed the function $\xi_{nh}$ by representing the sum of squares
\begin{align*}
\frac1n\sum_{i=1}^n\left\{Y_i -\psi_{\bma}( \bma^T\bm X_i)\right\}^2,
\end{align*}
in a local coordinate system with $d-1$ unknown parameters $\bmb = (\b_1,\ldots, \b_{d-1})^T$ followed by differentiation of the reparametrized sum of squares w.r.t. $\bmb$ where we also consider differentiation of the function $\psi_{\bma}$. 

\begin{theorem}
	\label{theorem:asymptotics-efficient}
	Let Assumptions A1-A8 be satisfied. Furthermore  assume that the following conditions hold:
	\begin{description}
		\item A10. The function $\psi_{\bma}$ is two times continuously differentiable on ${\cal I}_{\bma}$ for all $\bma$.
		\item A11. $ \left(\bm J_{\mbS}(\bmb)\right)^T \E\Bigl[\psi_0'(\bm\a_0^T\bm X)^2\,\text{ Cov}(\bm X|\bm\a_0^T\bm X)\Bigr]\bm J_{\mbS}(\bmb)$ is nonsingular. 
	\end{description}
	 Let $\tilde \bma_n$ be defined by (\ref{def:estimator-efficient}) and suppose $h \asymp n^{-1/7}$.	Then:
	\begin{enumerate}
		\item[(i)]
		\text{[Existence of a root]} A crossing of zero $\tilde\bmb_n$ of $\xi_{nh}(\bmb)$ exists with probability tending to one.
		\item[(ii)]
		\text{[Consistency]}
		\begin{align*}
		\tilde \bma_n \stackrel{p}{\rightarrow}\bma_0,\qquad n\to\infty.
		\end{align*}
		\item[(iii)]
		
		\text{[Asymptotic normality]}
		Define the matrices,
		\begin{align}
		\label{def:I1}
		\tilde {\bm A}:=\E\Bigl[\psi_0'(\bm\a_0^T\bm X)^2\,\text{ Cov}(\bm X|\bm\a_0^T\bm X)\Bigr],
		\end{align}
		and
		\begin{align}
		\label{def:I2}
		\tilde {\bm \Sigma}:=\E\left[\left\{Y -\psi_0(\bm\a_0^T\bm X)\right\}^2\psi_0'(\bm\a_0^T\bm X)^2\left\{\bm X -\E(\bm X|\bm\a_0^T\bm X) \right\}\left\{\bm X -\E(\bm X|\bm\a_0^T\bm X) \right\}^T\right].
		\end{align}
		Then 
		\begin{align*}
		\sqrt n (\tilde \bma_n - \bma_0) \to_d N_{d}\left(\bm 0,  \tilde {\bm A}^- \tilde {\bm \Sigma} \tilde {\bm A}^-\right),
		\end{align*}
		where $\tilde {\bm A}^{-}$ is the Moore-Penrose inverse of $\tilde {\bm A}$.	
		
	\end{enumerate}
\end{theorem}

\begin{remark}
{
		The asymptotic variance of the estimator $\tilde \bma_n$ is similar to that obtained for the ``efficient'' estimates proposed in \cite{xia2006} and in \cite{Kuchibhotla_patra:16}. 
		The efficient score function for the semi-parametric single index model is 
		\begin{align*}
		\tilde \ell_{\bma_0,\psi_0}(\bm x, y) =  \frac{y - \psi\left(\bma_0^T\bm x\right)}{\sigma^2(\bm x)}\psi'\left(\bma_0^T\bm x\right) \left\{\bm x -\frac{\E\left\{\sigma^{-2}(\bm X)\bm X | \bma_0^T\bm X =\bma_0^T\bm x \right\}}{\E\left\{\sigma^{-2}(\bm X)| \bma_0^T\bm X=\bma_0^T\bm x\right\}} \right\}. 
		\end{align*}
		More details on the efficiency calculations can be found in e.g. \cite{vaart:98}, chapter 25 for a general description of the efficient score functions and in \cite{delecroix2003} or \cite{Kuchibhotla_patra:16} for the efficient score in the single index model. 
		
		In a homoscedastic model with var$(Y|\bm X=\bm x) = \sigma^2$, where $\sigma^2$ is independent of covariates $\bm x$, the asymptotic variance equals $\sigma^2 \tilde {\bm A}^-$ which is the same as the inverse of $\E(\tilde \ell_{\bma_0,\psi_0}(\bm X, Y)\tilde \ell_{\bma_0,\psi_0}(\bm X, Y)^T)$. This indeed shows that our estimate defined in (\ref{def:estimator-efficient}) is efficient in the homoscedastic model. As also explained in Remark 2 of \cite{Kuchibhotla_patra:16}, our estimator has also a high relative efficiency with respect to the optimal semi parametric efficiency bound if the constant variance assumption provides a good approximation to the truth.
}
\end{remark}

\subsection{The asymptotic relation for the score estimators}
\label{subsec:asymptotic-relation}
To obtain the asymptotic normality result of the SSE $\hat \bma_n$ given in Theorem \ref{theorem:asymptotics}, we prove in the Supplementary Material that the following asymptotic relationship holds for $\hat \bmb_n$:
\begin{align*}
\bm B\left(\hat\bmb_n-\bmb_0\right)&=\int \left(\bm J_{\mbS}(\bmb_0)\right)^T\left\{\bm x-\E(\bm X| \mbS(\bmb_0)^T\bm X =  \mbS(\bmb_0)^T\bm x)\right\}\left\{y-\psi_0\bigl(\mbS(\bmb_0)^T\bm x\bigr)\right\}\,d\bigl(\P_n-P_0\bigr)(\bm x, y) \nonumber\\
&\qquad +o_p\left(n^{-1/2}\right).
\end{align*}
where 
\begin{align}
\label{def_B}
\bm B= \left(\bm J_{\mbS}(\bmb_0)\right)^T \E\Bigl[\psi_0'(\mbS(\bmb_0)^T\bm X)\,\text{ Cov}(\bm X|\mbS(\bmb_0)^T\bm X)\Bigr]  \left(\bm J_{\mbS}(\bmb_0)\right) = \left(\bm J_{\mbS}(\bmb_0)\right)^T \bm A \, \bm J_{\mbS}(\bmb_0),
\end{align}
in $\R^{(d-1)\times(d-1)}$. We assume in Assumption A9 that $\bm B$ is invertible so that
\begin{align*}
&\sqrt n \left(\hat\bmb_n-\bmb_0\right) \nonumber\\
&=  \sqrt n\bm B^{-1}\int \left(\bm J_{\mbS}(\bmb_0)\right)^T\left\{\bm x-\E(\bm X| \mbS(\bmb_0)^T\bm X =  \mbS(\bmb_0)^T\bm x)\right\}\left\{y-\psi_0\bigl(\mbS(\bmb_0)^T\bm x\bigr)\right\}\,d\bigl(\P_n-P_0\bigr)(\bm x, y)+o_p(1)\nonumber\\
&\qquad \to_d N(\bm 0,\bm  \Pi),
\end{align*}
where 
\begin{align}
\label{def_Pi}
\bm \Pi = \bm B^{-1} \left(\bm J_{\mbS}(\bmb_0)\right)^T \bm \Sigma \, \bm J_{\mbS}(\bmb_0) \, \bm B^{-1} \in \R^{(d-1)\times(d-1)}.
\end{align}
The limit distribution of the single index score estimator $\hat \bma_n$ defined in (\ref{def:estimator1}) now follows by an application of the delta-method and we conclude that
\begin{align*}
\sqrt n (\hat \bma_n -\bma_0) &=\sqrt n \left(\mbS\big(\hat \bmb_n\big) -\mbS\left(\bmb_0\right)\right) =  J_{\mbS}(\bmb_0) \sqrt n (\hat \bmb_n -\bmb_0)+ o_p(1)\\
&\to_d  N_{d}\left(\bm 0,  J_{\mbS}(\bmb_0)\bm  \Pi \left(J_{\mbS}(\bmb_0)\right)^T\right)
=N_{d}\left(\bm 0,  {\bm A}^-{\bm \Sigma}{\bm A}^-\right),
\end{align*}
where the last equality follows from the following lemma.
\medskip

\begin{lemma}
	\label{lemma:Moore-Penrose}
	Let the matrix $\bm A$ be defined  by (\ref{def_A}) and let $\bm A^-$ be the Moore-Penrose inverse of $\bm A$. Then
	\begin{align*}
	{\bm A}^-=\bm J_{\mbS}(\bmb_0) \left\{\left(\bm J_{\mbS}(\bmb_0)\right)^T{\bm A}\bm J_{\mbS}(\bmb_0)\right\}^{-1}\left(\bm J_{\mbS}(\bmb_0)\right)^T =\bm J_{\mbS}(\bmb_0) \bm B^{-1}\left(\bm J_{\mbS}(\bmb_0)\right)^T .
	\end{align*}
\end{lemma}
The proof of Lemma \ref{lemma:Moore-Penrose} is given the Supplementary Material. 
\\

The asymptotic variance of the ESE can be obtained similarly to the derivations of the asymptotic limiting distribution for the SSE. First the asymptotic variance is expressed in terms of the parametrization $\mbS$ as in (\ref{def_Pi}) and next, similar to Lemma \ref{lemma:Moore-Penrose}, equivalence to the expression $\tilde {\bm A}^-\tilde {\bm \Sigma} \tilde {\bm A}^-$ given in Theorem \ref{theorem:asymptotics-efficient} is proved.
 
\section{Computation and finite sample behavior of the score estimators}
\label{section:simulations}
In this section we investigate the applicability and the performance of the score estimates of Section \ref{section:score-estimators} in practice. We first describe the optimization algorithm used to obtain the score estimates and next include two simulation studies and compare the score estimates with alternative estimates for the monotone single index model. In the first simulation setting, given in Section \ref{subsec:simulation1}, we illustrate that the variance of the score estimates converges to the asymptotic variances derived in Section \ref{section:Asymptotics}. We also compare the score estimates with the maximum rank correlation estimate (MRCE) proposed by \cite{han:87}, and the Effective Dimension Reduction Estimate (EDRE), proposed in \cite{hristache01} and the Penalized Least Squares Estimate, proposed in \cite{Kuchibhotla_patra_sen:17}.  To compute the latter two estimates, we used the R packages {\tt{EDR}} and {\tt{simest}}, respectively. As for the alternative $\sqrt n$-consistent ``link-free" estimates of Section \ref{sec:basic_estimate}, the MRCE also does not depend on an estimate of the link function $\psi_0$. Finally, we discuss the convergence of the variances for the least squares estimate (LSE) minimizing the sum of squared errors $S_n$ defined in (\ref{def:sum_of_squares_LSE})  w.r.t. $(\psi, \bma)$, for which the asymptotic distribution is still an open problem.
	\\
	
In the second simulation setting in Section \ref{subsec:simulation2}, we investigate the quality of the score estimates if the dimension of the parameter space increases and investigate the finite sample behavior of different efficient estimates in the single index model. 
\\

By the discontinuous nature of the score functions given in Section \ref{section:score-estimators}, we introduced the concept of a zero-crossing in Definition \ref{zero-crossing}. It is not possible to solve the score equations exactly and we therefore search the crossing of zero, by minimizing the sum of squared component score functions over all possible values of $\bmb$ (parametrization approach of Section \ref{subsection:score-estimator1}), respectively $\bma$ (Lagrange approach of Section \ref{subsect:score-Lagrange}).  Note that the crossing of zero of the score function is equivalent to the minimizer of the sum of squared component scores so that the minimization procedure is justified. Due to the non-convex nature of the optimization function, standard optimization approaches based on a convex loss function cannot be used to obtain the score estimates. 
\\

We use a derivative free optimization algorithm proposed by \cite{hooke1961} to obtain the score estimates. The method is a pattern-search optimization method that does not require the objective function to be continuous. The algorithm starts from an initial estimate of the minimum and looks for a better nearby point using a set of $2d$ equal step sizes along the coordinate axes in each direction, first making a step in the direction of the previous move. For the object function we take the sum of the squared values of the component functions, which achieves a minimum at a crossing of zero. If in no direction an improvement is found, the step size is halved, and a new search for improvement is done, with the reduced step sizes. This is repeated until the step size has reached a prespecified minimum. A very clear exposition of the method is given in \cite{torczon:97}, Section 4.3. In this paper also convergence proofs for the optimization algorithm are presented. The optimization algorithm depends on a starting value for the regression parameters. In our simulations we used the true parameter values as starting values. In practice, we propose to search over a random grid of starting values on the unit sphere and select the estimate that results the smallest prediction error $\sum_{i=1}^{n}\{Y_i- \hat \psi_{n \hat\bma_n}(\hat{\bma_n}^T\bmX_i)\}^2$ among all different initial searches as a final starting value to obtain the score estimate.

\subsection{Simulation 1: The asymptotic properties of the score estimators}
\label{subsec:simulation1}
In this section, we illustrate the asymptotic properties of the SSE and the ESE given in Section \ref{section:Asymptotics} in the model
\begin{align*}
Y = \psi_0(\bma_0^T \bm X) + \varepsilon, \quad, \psi_0(\bm x)= x^3, \quad \a_{01} = \a_{02} =\a_{03} = 1/\sqrt 3, \quad  \quad X_1,X_2, X_3 \stackrel{i.i.d}{\sim} U[1,2], \quad  \varepsilon \sim N(0,1),
\end{align*}
where $\varepsilon$ is independent of the covariate vector $\bm X = (X_1,X_2, X_3)^T$. For this model, we have
$$\bm A = \frac{17}{15}\begin{bmatrix}
2 & -1 & -1
\\[0.3em]
-1 & 2 & -1 \\[0.3em]
-1 & -1 & 2
\end{bmatrix}
\quad \text{,} \quad
\bm \Sigma = \frac{1}{36}\begin{bmatrix}
2 & -1 & -1
\\[0.3em]
-1 & 2 & -1 \\[0.3em]
-1 & -1 & 2
\end{bmatrix}
\quad \text{ and} \quad
\tilde {\bm A} = \tilde {\bm \Sigma} =\frac{89953}{7560}\begin{bmatrix}
2 & -1 & -1
\\[0.3em]
-1 & 2 & -1 \\[0.3em]
-1 & -1 & 2
\end{bmatrix},
$$
where the matrices $\bm A, \bm \Sigma,\tilde {\bm A}$ and $\tilde {\bm \Sigma}$ are defined in (\ref{def_A}), (\ref{def_Sigma}), (\ref{def:I1}) and (\ref{def:I2}) respectively. Note that the rank of the matrices is equal to $d-1 = 2$. For this model, using the spherical coordinate system in three dimension introduced in (\ref{parametrization-dim3}), we get for the matrix specified in Assumption A9 that
$$
\left(\bm J_{\mbS}(\bmb_0)\right)^T \E\Bigl[\psi_0'(\bma_0^T\bm X)\,\text{Cov}(\bm X|\bma_0^T\bm X)\Bigr]  \left(\bm J_{\mbS}(\bmb_0)\right) = 	\begin{bmatrix}
\frac{17} {15} &0		\\[0.3em]
0  &\frac{17} {10} 
\end{bmatrix},
$$
which illustrates that the nonsingularity in Assumption A9 is indeed satisfied in this simulation setup. The same holds for the matrix given in Assumption A10.

The asymptotic variance of $\hat \bma_n$ resp. $\tilde \bma_n$  defined in Theorem \ref{theorem:asymptotics} resp. Theorem \ref{theorem:asymptotics-efficient} is equal to
\begin{align}
\label{matrix-simulation-normal}
\bm A^- \bm \Sigma  \bm A^-=
\frac{25}{2601}\begin{bmatrix}
2 & -1 & -1
\\[0.3em]
-1 & 2 & -1 \\[0.3em]
-1 & -1 & 2
\end{bmatrix}
\quad 
\text{ resp.  }\quad 
\tilde {\bm A}^-\tilde {\bm \Sigma} \tilde {\bm A}^- = \tilde {\bm A}^-= 
\frac{840}{89953}\begin{bmatrix}
2 & -1 & -1
\\[0.3em]
-1 & 2 & -1 \\[0.3em]
-1 & -1 & 2
\end{bmatrix}.
\end{align}
\\

We compare the estimates with the least squares estimates minimizing the sum of squared errors $S_n$ defined in (\ref{def:sum_of_squares_LSE})  w.r.t. $(\psi, \bma)$ and with the MRCE proposed by \cite{han:87}. This estimator is defined by the maximizer of
\begin{align*}
H_n(\bma):=\frac1{n(n-1)} \sum_{i\ne j}\{Y_i > Y_j\}\{\bma^T\bmX_i > \bma^T\bmX_j \},
\end{align*}
over all $\bma \in \cS_{d-1}$. The MRCE is proved to be a $\sqrt n$-consistent and asymptotically normal estimator of the regression parameter in the monotone single index model by \cite{sherman:93}, who gave an expression for the asymptotic covariance matrix in an (implicit) $(d-1)$-dimensional representation in his Theorem 4 on p.\ 133. If, in accordance with the parametrization methods of our paper, we turn this into an expression in terms of our $d$-dimensional representation, we obtain as the asymptotic covariance matrix of the MRCE $\bmV^-\bm S\bmV^-$
where
\begin{align}
\label{def_bmS}
\bm S=\E\left[\left\{\bm X -\E(\bm X|\bm\a_0^T\bm X) \right\}\left\{\bm X -\E(\bm X|\bm\a_0^T\bm X) \right\}^T
S(Y,\bma_0^T\bmX)^2g_0(\bma_0^T\bmX)^2\right],
\end{align}
and $\bmV^-$ is the Moore-Penrose inverse of
\begin{align*}
\bm V=\E\left[\left\{\bm X -\E(\bm X|\bm\a_0^T\bm X) \right\}\left\{\bm X -\E(\bm X|\bm\a_0^T\bm X) \right\}^T
S_2(Y,\bma_0^T\bmX)g_0(\bma_0^T\bmX)^2\right],
\end{align*}
and $g_0$ is the density of $\bm\a_0^T\bm X$ and $S$ and $S_2$ are defined by
\begin{align*}
S(y,u)=E\left[1_{\{y>Y\}}-1_{\{y<Y\}}\bigm| \bma_0^T\bmX=u\right],\qquad S_2(y,u)=\frac{\partial}{\partial u}S(y,u).
\end{align*}
It is clear from our simulations that the factor $2$ in front of $\bm V$ in (20) of \cite{sherman:93} cannot be correct and indeed \cite{sherman-cavanagh98} have a note on p.\ 361 of their paper, attributed to Myoung-Jae Lee that this factor $2$ should not be there.
\\
To obtain the LSE and MRCE under the identifiability restriction $\|\bma_0 \| = 1$, we also consider the parametrization of the unit sphere and first rewrite the optimization function
 \begin{align}
\tilde S_n(\bmb) :=  \frac1n \sum_{i=1}^n \Big \{  Y_i -\hat \psi_{n\mbS(\bmb)}(\mbS(\bmb)^T\bm X_i)  \Big \}^2,
 \end{align}
for the LSE and
\begin{align}
\tilde H_n(\bmb){:=}\frac1{n(n-1)} \sum_{i\ne j}\{Y_i > Y_j\}\{\mbS(\bmb)^T\bmX_i > \mbS(\bmb)^T\bmX_j \},
\end{align}
for the MRCE in terms of the $(d-1) = 2$ dimensional vector $\bmb$ using the spherical coordinate system in three dimensions. Next we use the optimization algorithm by \cite{hooke1961}, discussed above, to minimize $\tilde S_n$ respectively maximize $\tilde H_n$ w.r.t. $\bmb$ to end up with a LSE respectively MRCE of the regression parameter that has length one and hence satisfies our identifiability restriction.
\\

To illustrate the link-free least squares estimates H-LFLSE and LFLSE in Section \ref{sec:basic_estimate}, we also consider normally distributed covariates $\bm X$, $X_i \stackrel{i.i.d}{\sim}  N(0,1)$ for $i = 1,\ldots,d$. Since the asymptotic results for the score estimates of section \ref{section:score-estimators} are proven under the assumption of bounded covariates only, this simulation provides further insight in the convergence of the variances of our score estimates in a model where not all Assumptions given in Section \ref{section:Asymptotics} are satisfied. Since the LFLSE does not depend on the behavior of the LSE $\hat \psi_{n \bma}$ and no longer suffer from the crossing of zero difficulties, we used Broyden's method for solving nonlinear equations in higher dimensions to obtain the LFLSE, which is very fast and of quasi Newton type.
\\

For sample sizes $n=100,500,1000,2000,5000$ and $n=10000$ we generated $N= 5000$ datasets and show, in Tables \ref{table:simulation1b} and \ref{table:simulation2b}, the mean and $n$ times the covariance of the estimates. Tables \ref{table:simulation1b} and \ref{table:simulation2b} also show the asymptotic values to which the results for the SSE and the ESE should converge based on Theorem \ref{theorem:asymptotics} and Theorem \ref{theorem:asymptotics-efficient} respectively. For the limiting variance of the MRCE, we used the description above. The asymptotic distributions of the H-LFLSE and LFLSE, on the other hand, are only derived for the normally distributed and not the uniformly distributed covariate setting, since the latter setting does not satisfy the condition of elliptic symmetry. The variances to which the H-LFLSE and the LFLSE should converge are given in Section \ref{sec:basic_estimate}. The limiting distribution of the LSE is still unknown and therefore no asymptotic results are provided for the LSE in Tables \ref{table:simulation1b} and \ref{table:simulation2b}.
\\

\begin{center}
[Tables \ref{table:simulation1b} and \ref{table:simulation2b}] 
 \end{center}

For the two simulation studies, the results shown in Tables \ref{table:simulation1b} and \ref{table:simulation2b} show convergence of $n$ times the variance-covariance matrices towards the asymptotic values. The performance of the ESE is slightly better than the performance of the SSE; the difference between the asymptotic limiting variances is smaller in the model with uniform$[1,2]$ covariates $X_i$  than the difference in the model with standard normal covariates $X_i$. Although the model with standard normal covariates violates Assumptions A1, A2 and A4 given in Section \ref{section:Asymptotics}, our proposed score estimates perform reasonably well. We added In Table \ref{table:simulation1b} the values for the estimators EDRE (``Effective Dimension Reduction Estimate'') and PLSE (``Penalized Least Squares Estimate'') that are further studied in Subsection \ref{subsec:simulation2}.
\\

Figure \ref{fig:Lagrange-sim} illustrates the similarity for $n\cdot$var$(\hat\a_{3n})$ between the the score estimates obtained with either the parametrization or Lagrange approach. Similar results are obtained for the other variances reported in Table \ref{table:simulation1b}, which supports the conjecture that the asymptotic properties of the score estimates with Lagrange penalty term (Section \ref{subsect:score-Lagrange}) are equivalent to the asymptotic results presented in Section \ref{section:Asymptotics} for the estimates obtained via a parametrization of the unit sphere (Section \ref{subsection:score-estimator1}).
\\

\begin{center}
	[Figure \ref{fig:Lagrange-sim}  here]
\end{center}

The performance of the link-free estimates MRCE, H-LFLSE and LFLSE is considerably worse than the performances of our proposed score estimates in all simulation settings. In the model with standard normal covariates, the variances of these link-free estimates are remarkably larger than the variances of the score estimates and the LSE.
This might be caused by the fact that these estimates are not based on an estimate of the link function and hence do not take information about this link function into account. It is clear from our experiments that the link-free estimates are for sure not the most preferred estimates to use in the monotone single index model, even if the conditions for their use are satisfied.
\\

Estimation of the link function $\hat \psi_{n \bma}$ is very straightforward. Since the number of jump points of the LSE $\hat \psi_{n \bma}$ is of the order $n^{1/3}$, estimation of the smooth derivate estimate $\tilde \psi_{nh \bma}'$  only requires one additional summation over these  $O(n^{1/3})$ jump points for each of the $n$ observations. The computation time for the score estimates is relatively fast (a more in depth study of the computation time is given in Section \ref{subsec:simulation2}). Although the MRCE does not depend on an estimate of the link function, a double sum is needed for the calculation of the criterion function $\tilde H_n$ which increases the computation time considerably when the sample size is large. Since the H-LFLSE only depends on an ordinary least squares algorithm  and since we can use Broyden's algorithm for the LFLSE, these estimates do not require a hard optimization algorithm and the computation of the LFLSE and the H-LFLSE requires less than a second in all our simulations.
\\

The behavior of the LSE is rather remarkable. Table \ref{table:simulation1b} suggests an increase of $n$ times the covariance matrix, whereas  Table \ref{table:simulation2b} suggests a decrease. 
The results presented in Table \ref{table:simulation2b} show that the performance of the LSE is better than the performance of the SSE when $X_i \sim N(0,1)$. For the model with uniform covariates, summarized in Table \ref{table:simulation1b}, our proposed score estimates are better than the LSE. The variances for the LSE presented in Tables \ref{table:simulation1b} and \ref{table:simulation2b} suggest that the rate of convergence for the LSE is faster than the cube-root $n$ rate proved in \cite{balabdaoui2016}.

\subsection{Simulation 2: Further comparisons and the behavior of the estimators if the covariates have a higher dimension}
\label{subsec:simulation2}
In this section we illustrate the applicability of the score estimates given in Section \ref{subsect:score-Lagrange} when the dimension of the covariate space increases. We also compare our estimates with the  Effective Dimension Reduction Estimate (EDRE) (see \cite{hristache01}) and the Smooth Penalized Least Squares estimate (PLSE) (see \cite{Kuchibhotla_patra:16}). Here we use again the \verb|R| packages \verb|EDR| and \verb|simest| available on CRAN, just as in the  previous simulation. We use the Lagrange approach to the computation of the SSE and ESE. The computation relies on \verb|C++| programs which are used in \verb|R| (via \verb|Rcpp|) scripts, see \cite{github:18}. The following results can be reproduced by running the R scripts, given there.
\\

We consider the model of Table \ref{table:simulation2b} more generally:
\begin{align*}
Y = \psi_0(\bma_0^T \bm X) + \varepsilon, \quad, \psi_0(\bm x)= x^3, \quad \bma =d^{-1/2}(1,\dots,1)^T, \quad  \quad X_i\stackrel{i.i.d}{\sim} N(0,1), \quad  \varepsilon \sim N(0,1),
\end{align*}
where (dimension) $d = 5,10,15,25$ (the case $d=3$ was considered in Table \ref{table:simulation2b}). The estimation error is measured via $\sqrt{n/d}\,\|\hat\bma_n-\bma_0\|_2$, where $\|\cdot\|_2$ is the Euclidean norm. The results are compared with what the asymptotic distribution of the efficient estimate would give. The asymptotic distribution of efficient estimates of $\bma_0$ is, for general $d\ge2$,  given by a degenerate normal distribution with mean zero and covariance matrix
\begin{align}
\label{Sigma_d}
\bm\Sigma_d=\frac1{27d}\left\{d\bm I-\bm1\cdot\bm1^T\right\},
\end{align}
where $\bm I$ is the $d\times d$ identity matrix and $\bm1$ is the column vector with $d$ components, equal to $1$.
The special case $d=3$ was used in Table \ref{table:simulation2b}. Results of these experiments are given in Figures \ref{boxplots1} to \ref{boxplots3}. Figures \ref{boxplots1} and \ref{boxplots2} give the behavior of the $L_2$-distance for different dimensions and sample sizes $n=100$ and $n=1000$, respectively. Figure \ref{boxplots3} gives the computing time in seconds for sample size $n=1000$. Clearly EDR needs the longest computing time.
\\

\begin{center}
[Figures \ref{boxplots1} to \ref{boxplots3} here]
\end{center}

All algorithms depend on starting values for the regression parameter. For the EDRE and SSE we do not have to specify tuning parameters, although for the EDRE there are if fact tuning parameters, hidden in the package EDR. Some of the algorithms have more need for a reasonable starting value than others. For example, one can start SSE and LSE at a starting value having a larger distance to the real value of $\bma_0$ than others, such as the ESE. One can solve this isssue for example by starting the ESE algorithm from the value obtained by the LSE, which is itself started from arbitrary starting values, such as $(1,0,\dots,0)$ or from a value, found by a preliminary search on starting values of the LSE algorithm, using the sum of the squared errors as criterion (see the remark on this issue just before section \ref{subsec:simulation1} above). Clearly more research for this selection procedure is necessary.\\

The bandwidth for the computation of the estimate of the derivative $\psi_0'$ of $\psi_0$ in the algorithm for the ESE is set equal to $h =\tfrac12\tilde c n^{-1/7}$ where $\tilde c$ equals the range $\bma^T\bmX_1,\ldots,\bma^T\bmX_n$ and $\bma$ is the current estimate of $\bma_0$ during the iterations. This choice gave satisfactory results in all our experiments.  We do not discuss bandwidth selection procedures in this manuscript, but  note that the bootstrap techniques discussed in \cite{kim_piet:18} and \cite{kim_piet:17EJS} for the current status model, can also be investigated further to select the bandwidth of the ESE in the monotone single index model in practice. For the PLSE (generalized) cross-validation can be used to select the smoothing parameter. In our experiments; we took the smoothness penalty for the PLSE equal to $0.1$ (after some preliminary experimentation). The EDR method uses an average derivative  estimate (derivative w.r.t.\ the covariate $X$) as starting value, but computing this estimate is done within the package.
\\

The results in Figures \ref{boxplots1} to \ref{boxplots3} suggest that the asymptotically efficient estimates ESE and PLSE have the best behavior. The results for the EDRE deteriorate significantly with increasing dimension, both in $L_2$-error and computing time. The LSE has remarkably good behavior and there is certainly the suggestion that its rate of convergence is faster than $n^{1/3}$ for the present model.\\

\section{Summary}
\label{sec:discussion}
In this paper we introduce estimates for the regression parameter in the monotone single index model. Our estimates are obtained via the zero-crossing of an unsmooth score equation derived from the sum of squared errors and depend on the behavior of the LSE of the underlying monotone link function. We prove $\sqrt n$-consistency and asymptotic normality of our estimates and therefore, for the first time, define estimates that depend on the cube-root-$n$ consistent LSE of the link function which still converge at the parametric rate to the true regression parameter in the monotone single index model.  By introducing a score approach similar to the M-approach for the profile LSE, where simultaneous minimization is over $\bm\alpha$ and the link function, we avoid the difficulties that arise when analyzing the limiting behavior of the profile LSE. This novel result in the field of shape constrained statistics will hopefully help us to further understand the behavior of the profile LSE in the monotone single index model, for which the limiting distribution is still unknown.
\\

We consider two different score estimates, one very simple one that does not require any smoothing technique and one efficient estimate that is based on a smooth estimate of the derivative of the link function. This derivative estimate depends again only on the LSE of the link function, but, in contrast with the LSE itself, also on a kernel which is integrated w.r.t.\ the jumps of the LSE. We use two techniques to ensure that the norm of the regression parameter estimate is one. The first approach uses a parametrization of the unit sphere in $d-1$ dimensions. In the second method, motivated by the Lagrange approach, we directly solve an equation in dimension $d$ for the parameter $\bm\a$ and divide by the norm at the end of the iterations.\\

Since our score functions depend on the piecewise constant LSE of the link function, we obtain unsmooth score functions that might not have an exact root. We therefore work with zero-crossings instead of exact zeros and prove that there indeed always exists a value for the regression parameter where the score functions cross zero.
\\

We compare our score estimates with link-free estimates of the single index parameters that avoid estimation of the link function. To that end, we also introduce a link-free least squares estimate, conditioned to have norm one and derive the asymptotic variance of this link-free estimate for the situation where we have elliptically symmetric distributions for the covariates $\bm X$ (like  the normal $\bm X$ in the simulation setting of Table \ref{table:simulation2b}). Since this estimate no longer depends on the LSE of the link function, the crossing of zero issues disappear and we can solve the corresponding score equation exactly. This also  illustrates the applicability of the method motivated by the Lagrangian formulation, which avoids the reparametrization. Our simulations clearly show that our score estimates have a better behavior than the link-free estimates, even if the conditions for application of the latter methods are fulfilled.\\

A numerical comparison between the score estimate and other estimates for the single index model reveals that our score estimates perform well in higher dimensions.  Our computer experiments moreover point out that the Lagrange score approach  can easily be used in higher dimensions.

\begin{figure}[!ht]
	\centering
	\begin{subfigure}{0.45\linewidth}
		\includegraphics[width=0.95\textwidth]{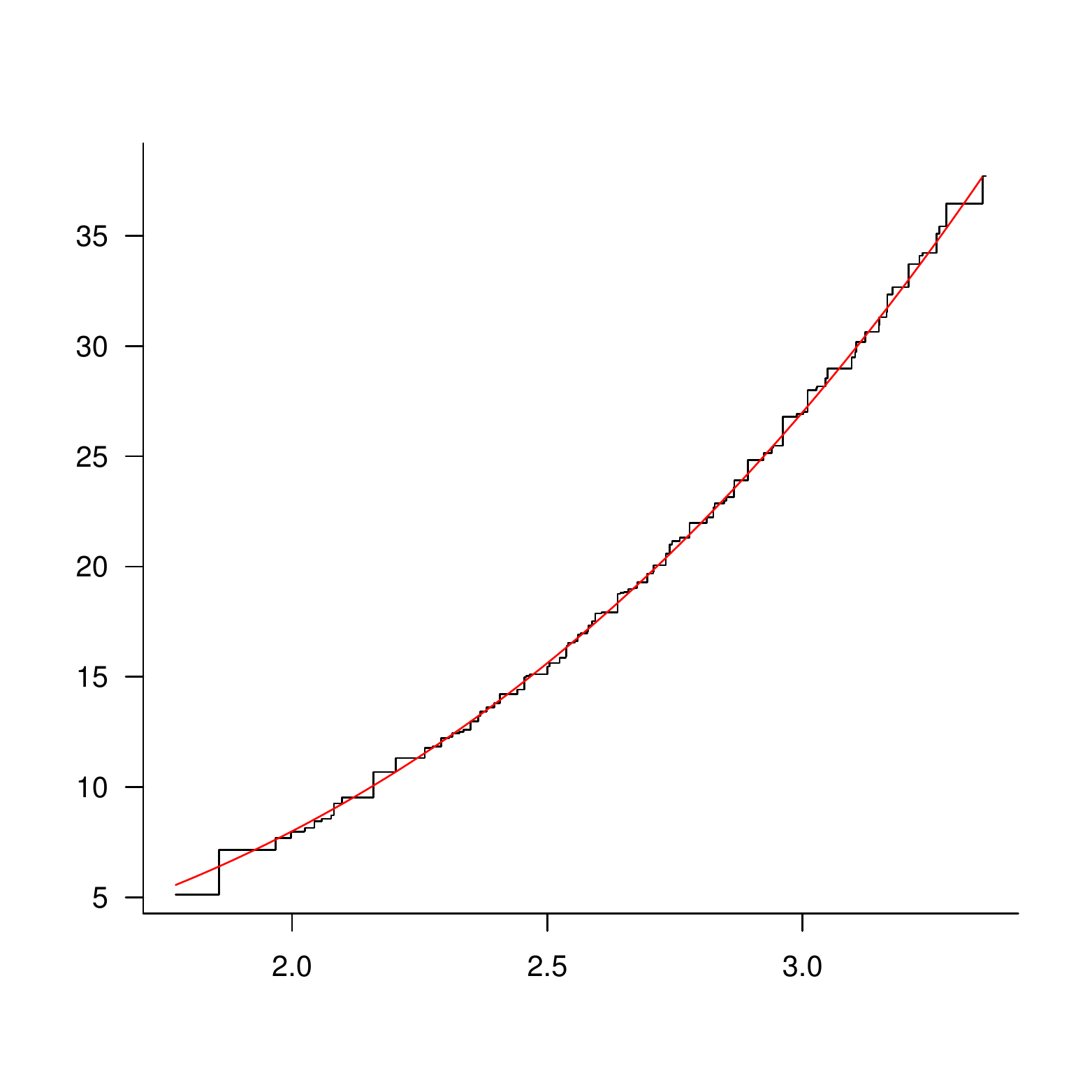}
		\caption{$\hat\psi_{n,\hat\bma_n}$}
	\end{subfigure}
	\begin{subfigure}{0.45\linewidth}
	\includegraphics[width=0.95\textwidth]{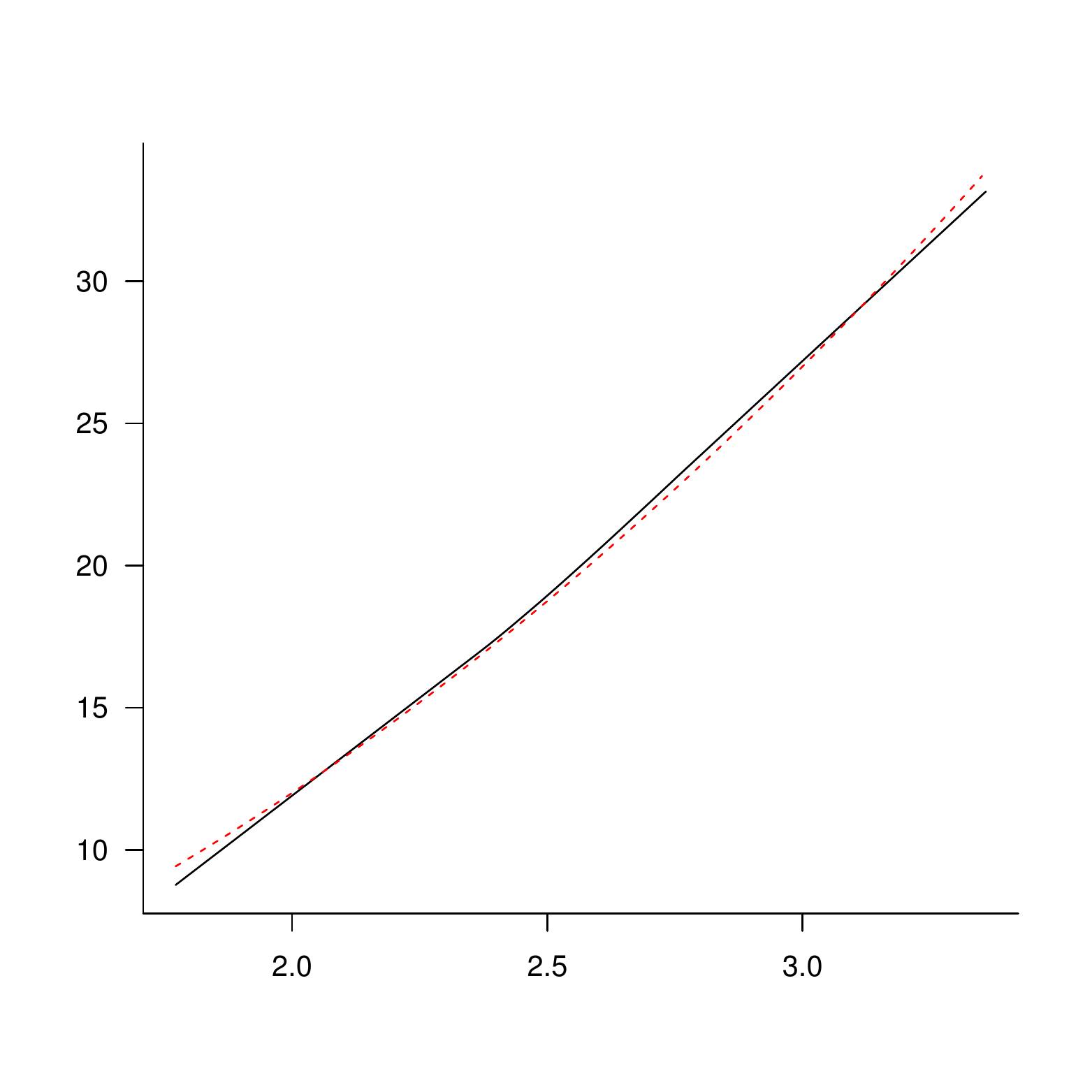}
	\caption{$\tilde\psi'_{n,\hat\bma_n}$}
	\end{subfigure}
	\caption{The estimates  $\hat\psi_{n,\hat\bma_n}$ and $\tilde\psi'_{n,\hat\bma_n}$ of $\psi_0$ and $\psi'_0$, respectively, for $n=1000$, $d=3$ and a sample for the model used in the simulation study of Table \ref{table:simulation1b}. The red curves are $\psi_0$ and $\psi_0'$. The bandwidth chosen in the estimate $\tilde\psi'_{n,\hat\bma_n}$ was equal to the range of the values $\bmX^T\hat\bma_n$ times $n^{-1/7}$.}
	\label{fig:psi_and_psi'}
\end{figure}

\begin{figure}[!ht]
	\centering
	\includegraphics[width=0.35\textwidth]{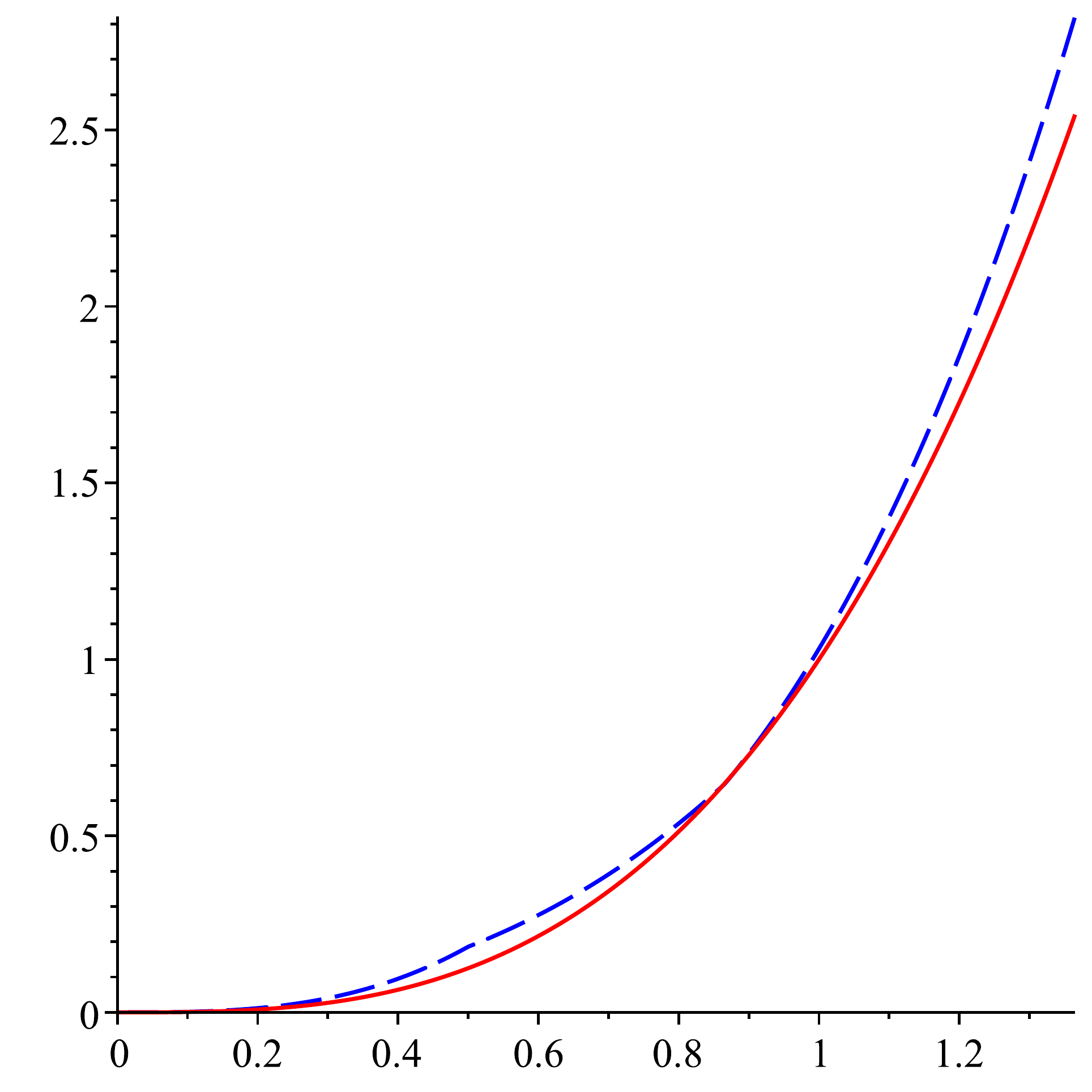}
	\caption{ The real $\psi_0$ (red, solid) and the function $\psi_{\bma}$ (blue, dashed) for  $\psi_0(x)= x^3$,  $\a_{01} = \a_{02} =  1/\sqrt 2$ and $\a_{1} = 1/2, \a_{2} =  \sqrt 3/2$, with $X_1, X_2 \stackrel{i.i.d}{\sim} U[0,1]$. }
	\label{fig:psi_alpha}
\end{figure}

\begin{figure}[!ht]
	\centering
	\begin{subfigure}{0.45\linewidth}
		\includegraphics[width=0.95\textwidth]{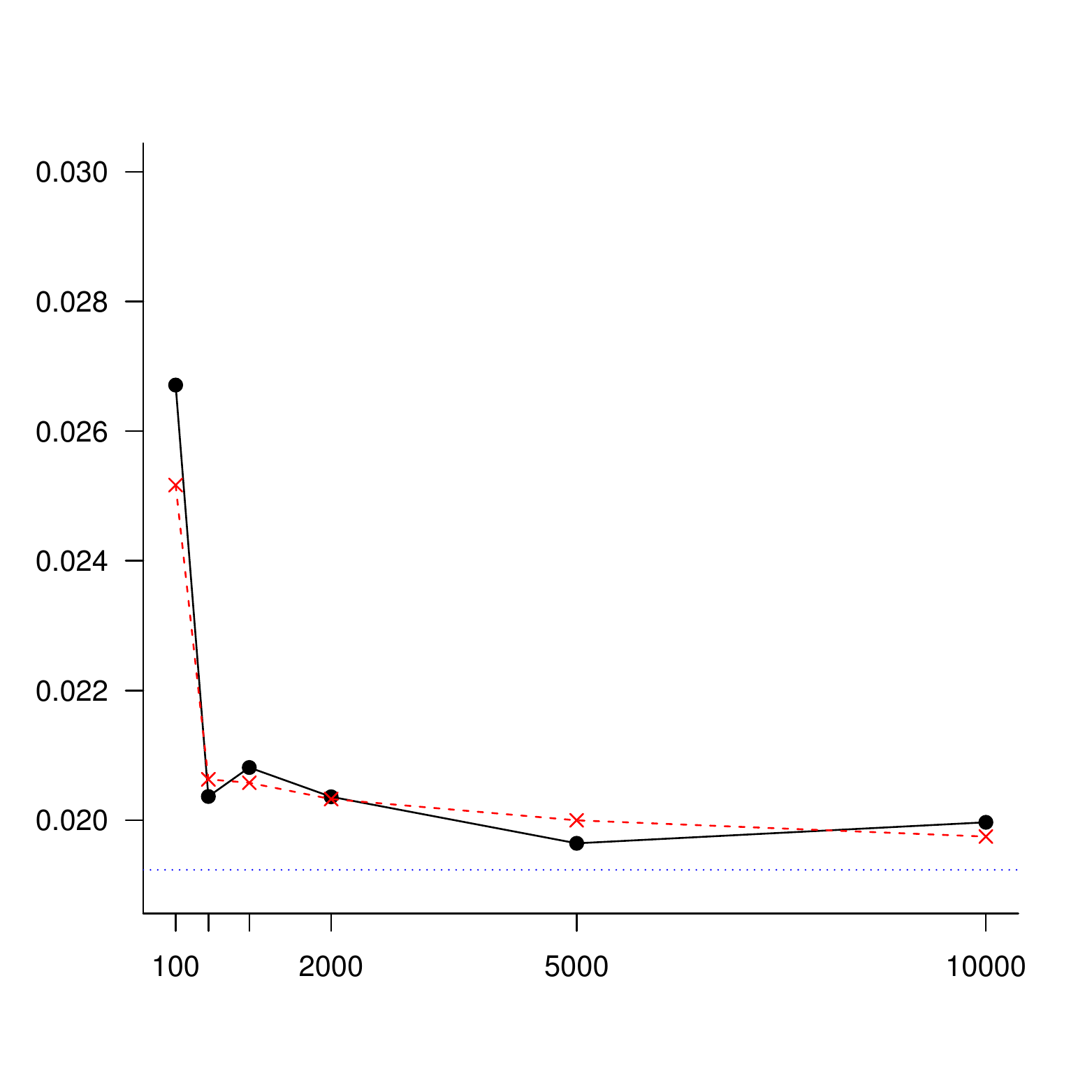}
		\caption{SSE}
	\end{subfigure}
	\begin{subfigure}{0.45\linewidth}
	\includegraphics[width=0.95\textwidth]{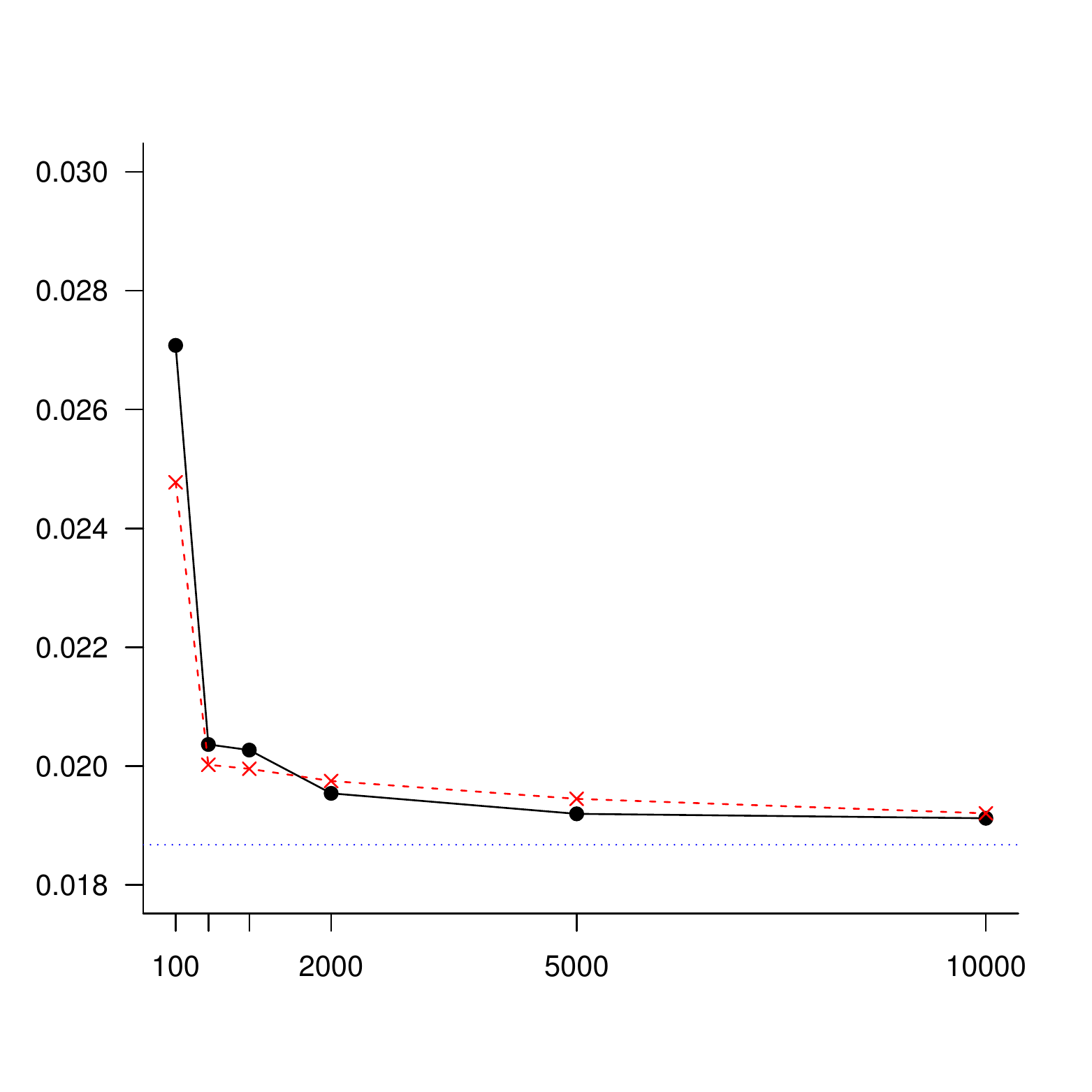}
	\caption{ESE}
	\end{subfigure}
	\caption{Simulation model  ($ X_i\sim  U[1,2], d=3$): $n\mapsto n\cdot$var$(\hat\a_{3n})$ for (a) the SSE and (b) the ESE using the parametrization approach (red,dashed, $\times$) Lagrange approach (black, solid,  $\bullet$). The blue, dotted line indicates the asymptotic variance of the score estimates. }
	\label{fig:Lagrange-sim}
\end{figure}

\begin{figure}[!h]
	\centering
	\begin{subfigure}{0.485\linewidth}
		\includegraphics[width=0.95\textwidth]{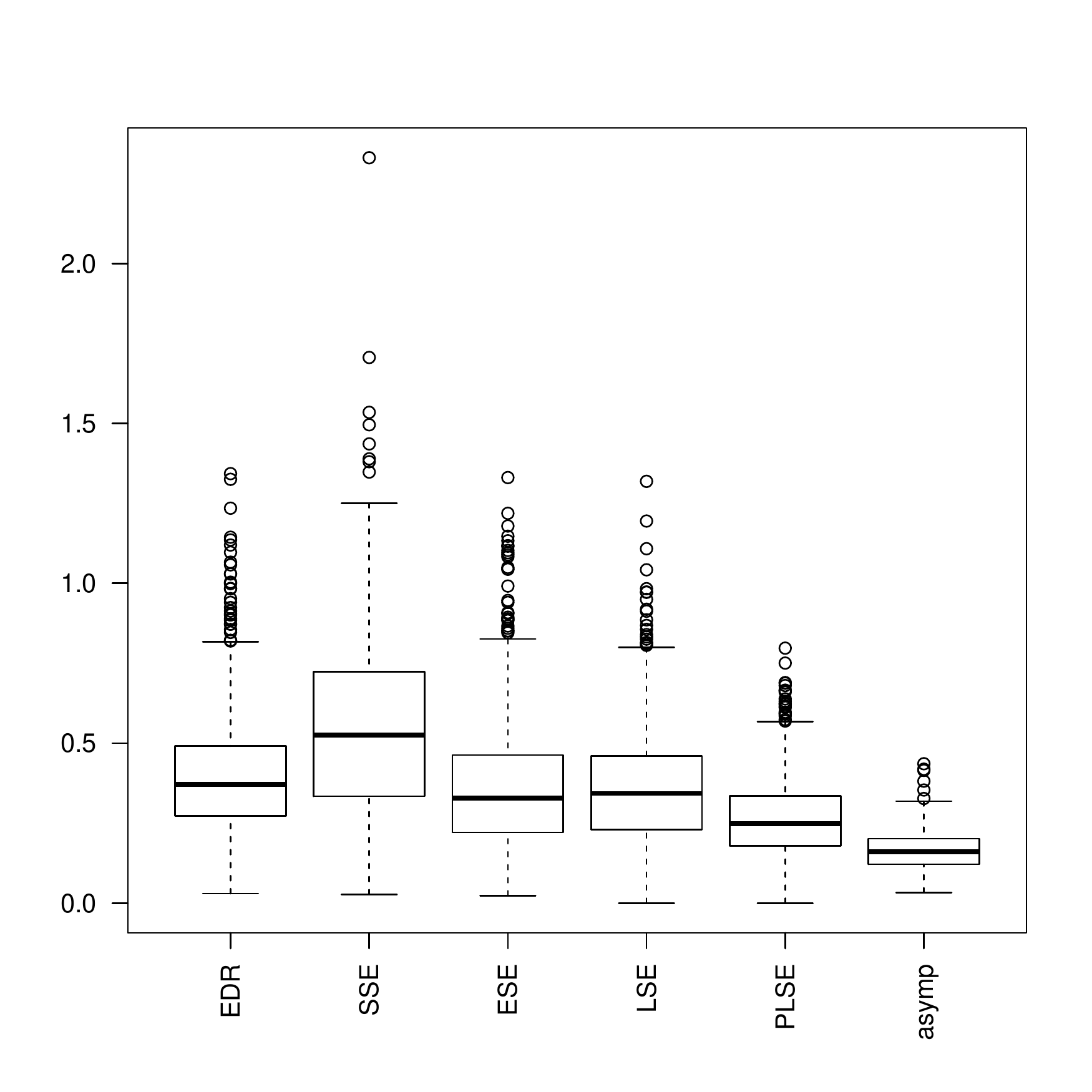}
		\caption{$d = 5$}
	\end{subfigure}
	\begin{subfigure}{0.485\linewidth}
	\includegraphics[width=0.95\textwidth]{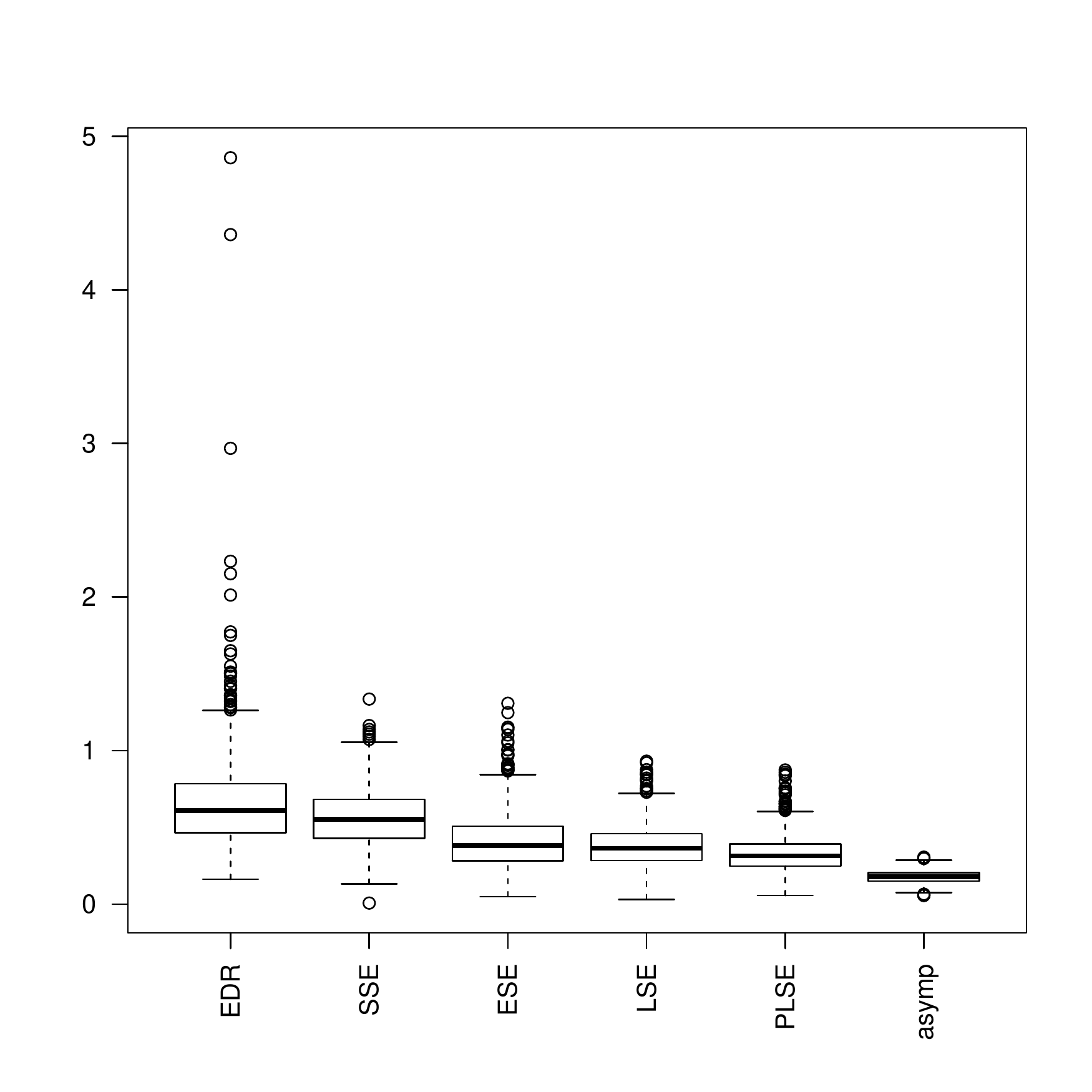}
	\caption{$d = 10$}
\end{subfigure}\\
	\begin{subfigure}{0.485\linewidth}
		\includegraphics[width=0.95\textwidth]{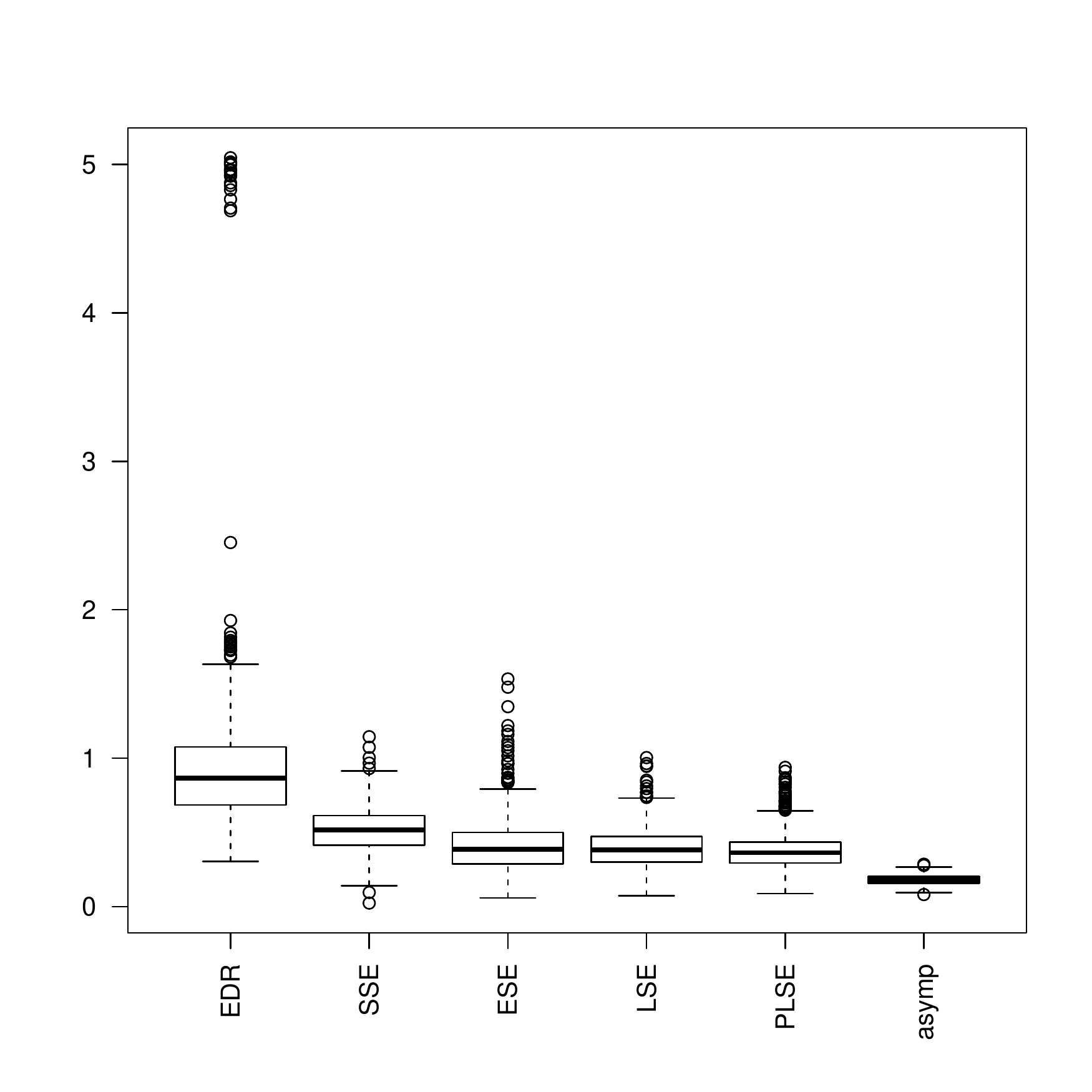}
		\caption{$d=15$}
	\end{subfigure}
\begin{subfigure}{0.485\linewidth}
	\includegraphics[width=0.95\textwidth]{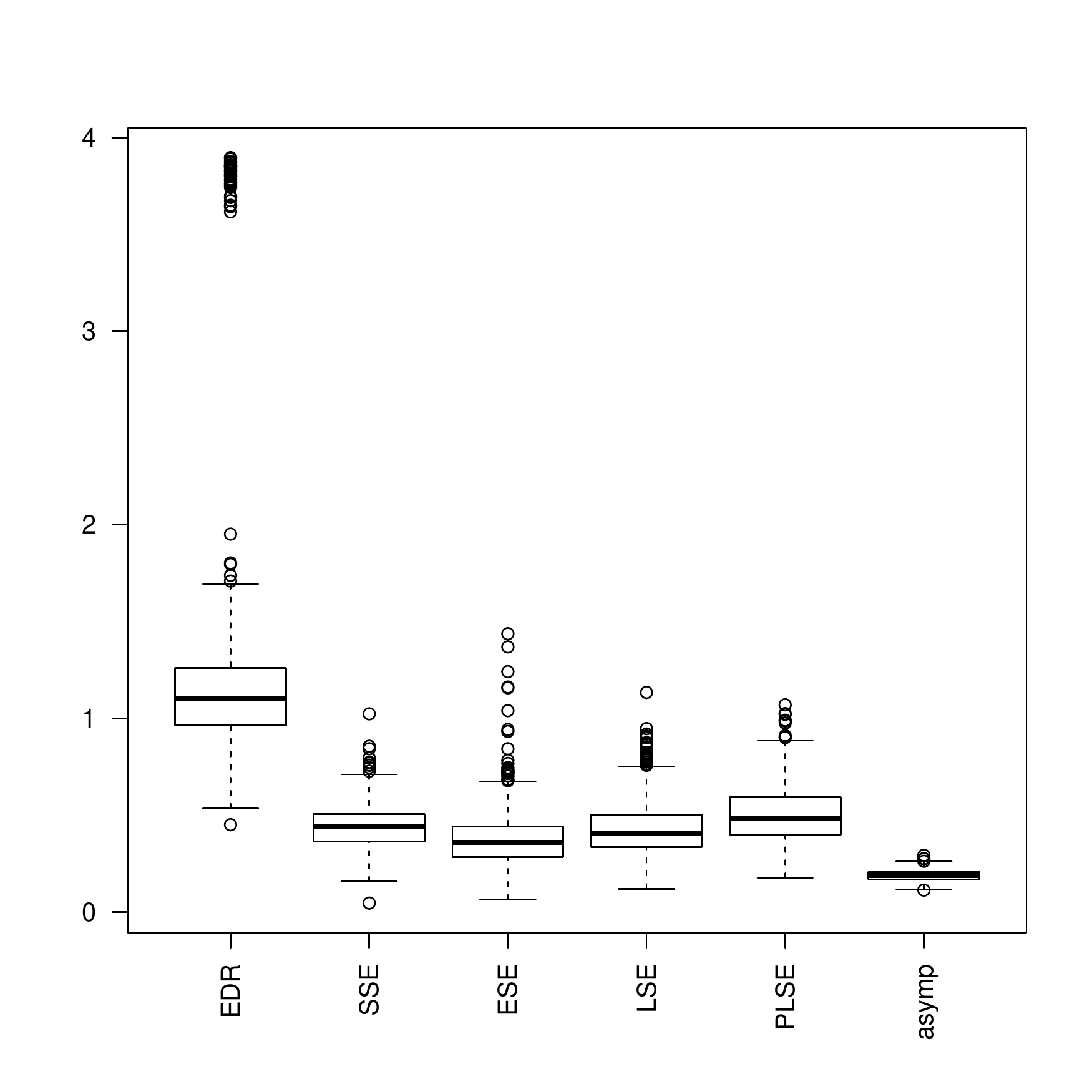}
	\caption{$d = 25$}
\end{subfigure}
	\caption{Boxplots of $\sqrt{n/d}\,\|\hat\bma_n-\bma_0\|_2$ for $n=100$ and for (a) $d = 5$, (b) $d=10$, (c) $d=15$ and (d) $d=25$ and 1000 replications for the EDR, SSE, ESE, LSE and PLSE.
	The algorithms for SSE, ESE, LSE and PLSE were started at $\bma_0$. The asymptotic distribution is generated via $1000$ draws from the degenerate limiting normal distribution for the efficient estimates, with mean zero and covariance matrix $\bm\Sigma_d$, defined by (\ref{Sigma_d}).}
\label{boxplots1}
\end{figure}

\begin{figure}[!h]
	\centering
	\begin{subfigure}{0.485\linewidth}
		\includegraphics[width=0.95\textwidth]{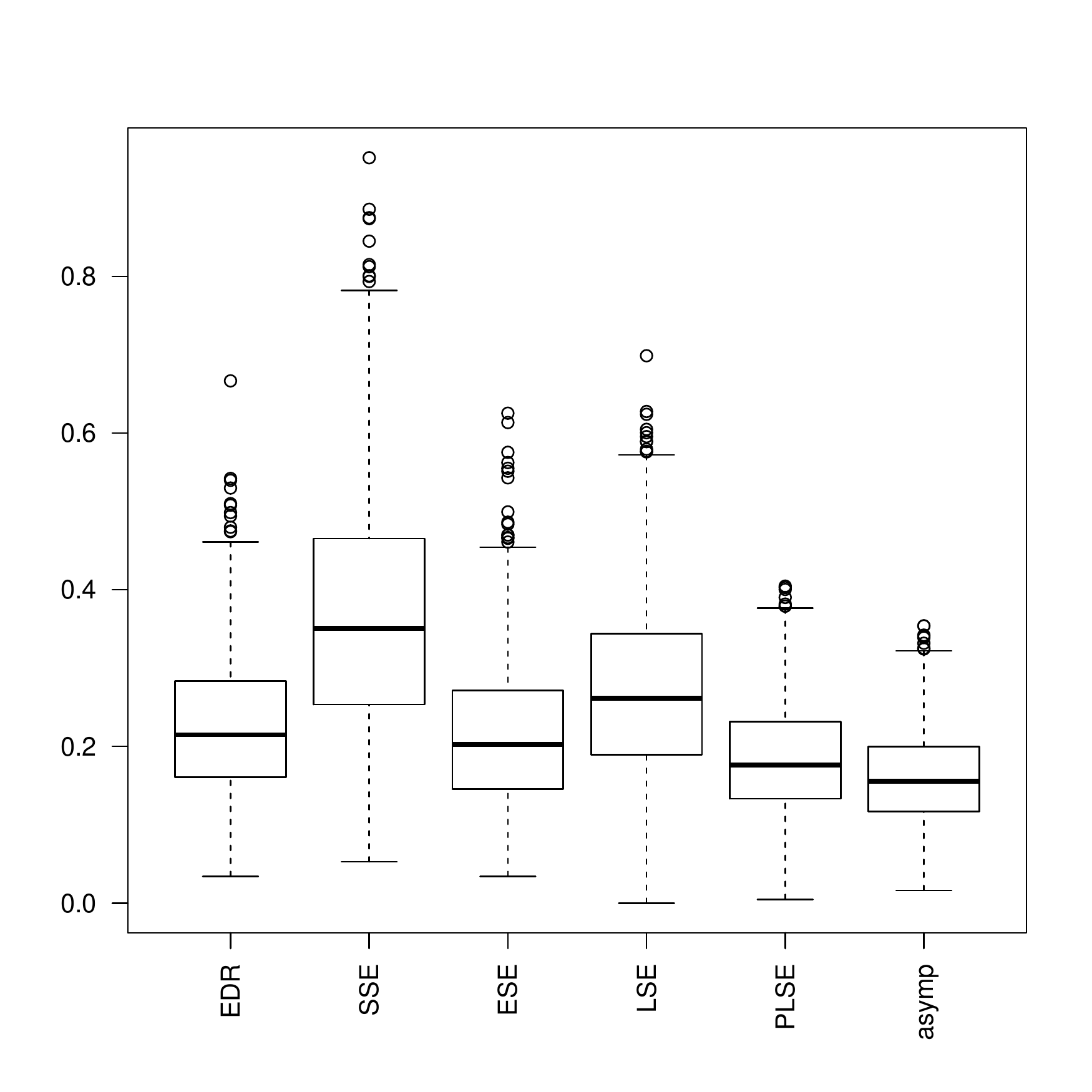}
		\caption{$d = 5$}
	\end{subfigure}
	\begin{subfigure}{0.485\linewidth}
	\includegraphics[width=0.95\textwidth]{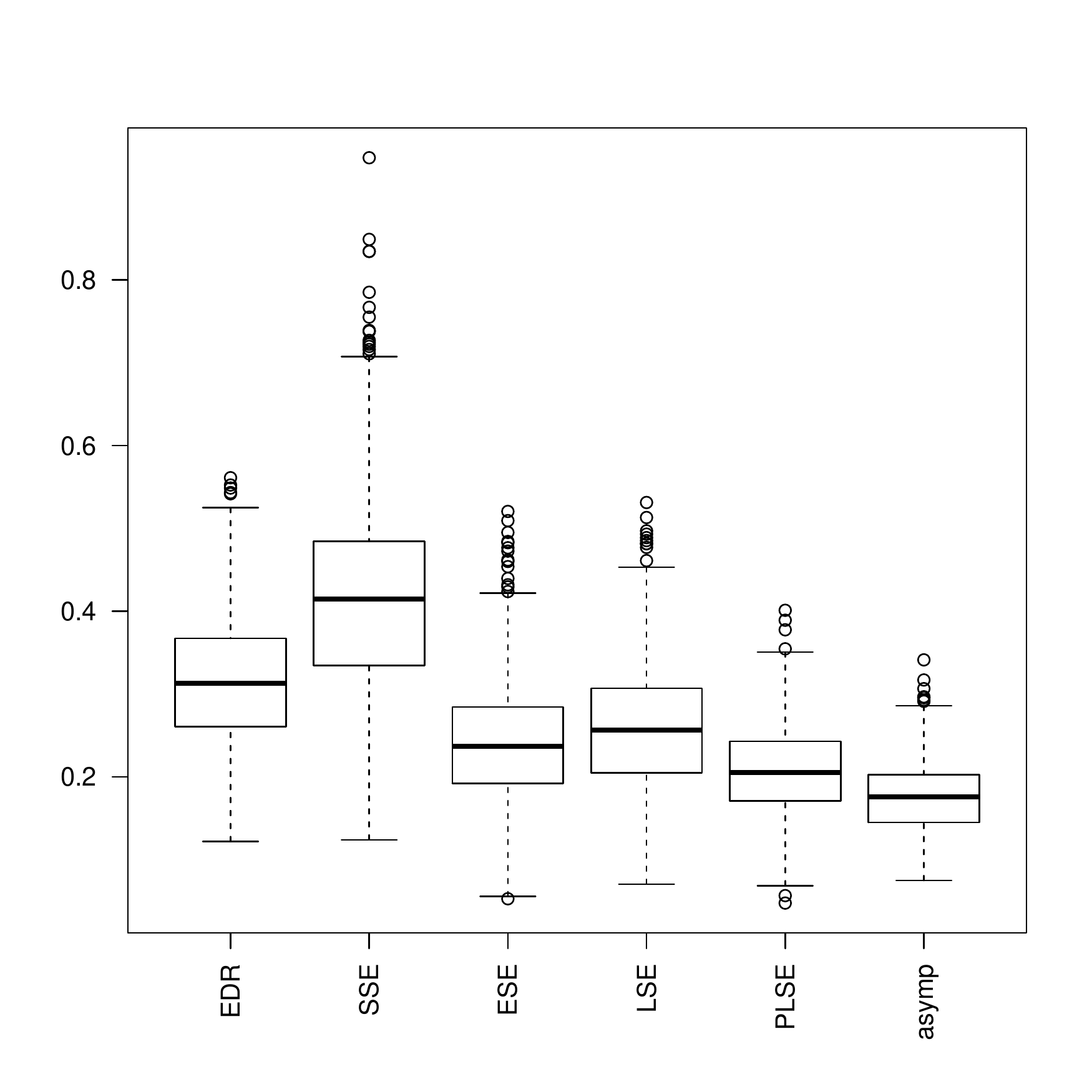}
	\caption{$d = 10$}
\end{subfigure}\\
	\begin{subfigure}{0.485\linewidth}
		\includegraphics[width=0.95\textwidth]{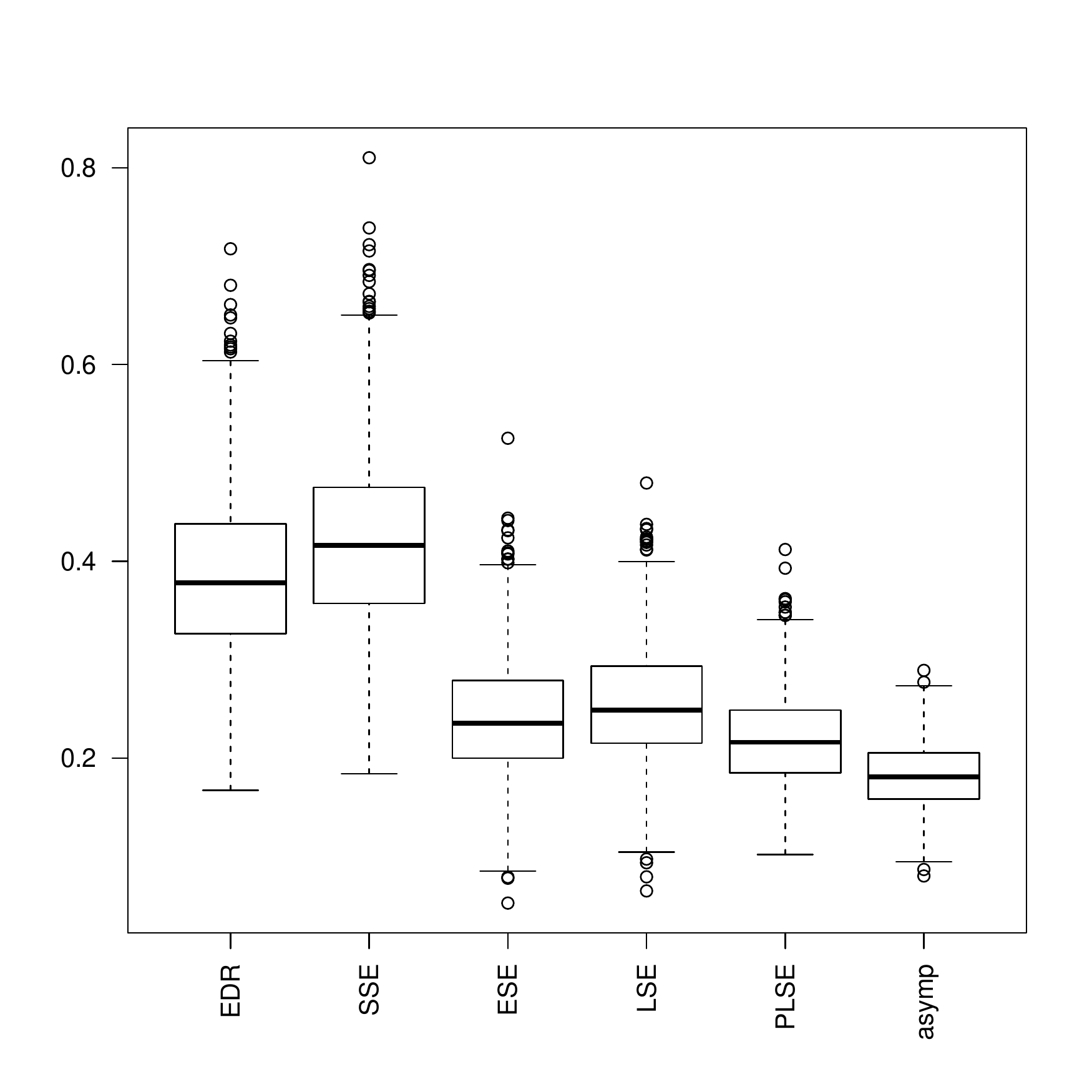}
		\caption{$d=15$}
	\end{subfigure}
\begin{subfigure}{0.485\linewidth}
	\includegraphics[width=0.95\textwidth]{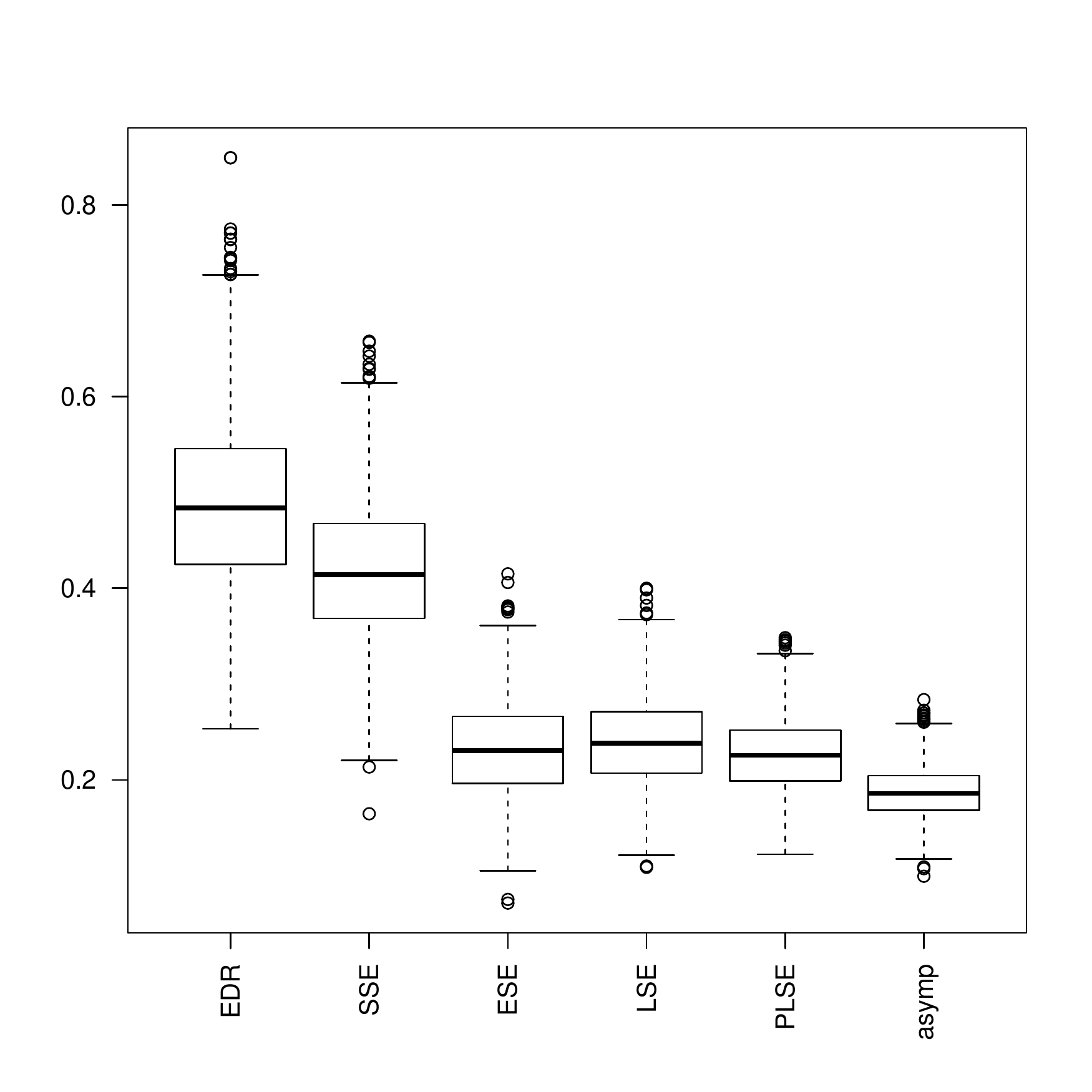}
	\caption{$d = 25$}
\end{subfigure}
	\caption{Boxplots of $\sqrt{n/d}\,\|\hat\bma_n-\bma_0\|_2$ for $n=1000$ and for (a) $d = 5$, (b) $d=10$, (c) $d=15$ and (d) $d=25$ and $1000$ replications for the EDRE, SSE, ESE, LSE and PLSE.
	The algorithms for SSE, ESE, LSE and PLSE were started at $\bma_0$. The asymptotic distribution is generated via $1000$ draws from the degenerate limiting normal distribution for the efficient estimates, with mean zero and covariance matrix $\bm\Sigma_d$, defined by (\ref{Sigma_d}).}
\label{boxplots2}
\end{figure}

\begin{figure}[!h]
	\centering
	\begin{subfigure}{0.485\linewidth}
		\includegraphics[width=0.95\textwidth]{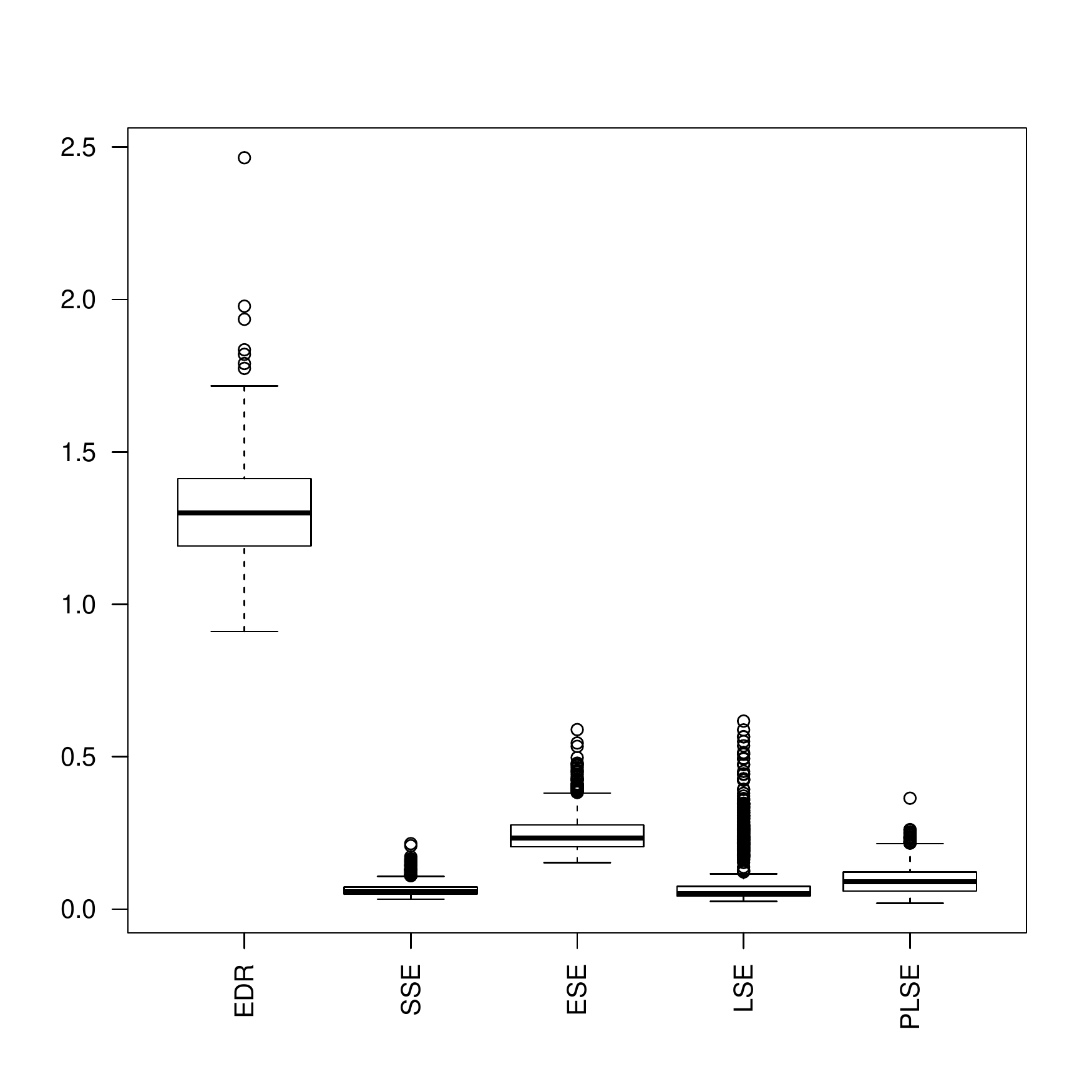}
		\caption{$d = 5$}
	\end{subfigure}
	\begin{subfigure}{0.485\linewidth}
	\includegraphics[width=0.95\textwidth]{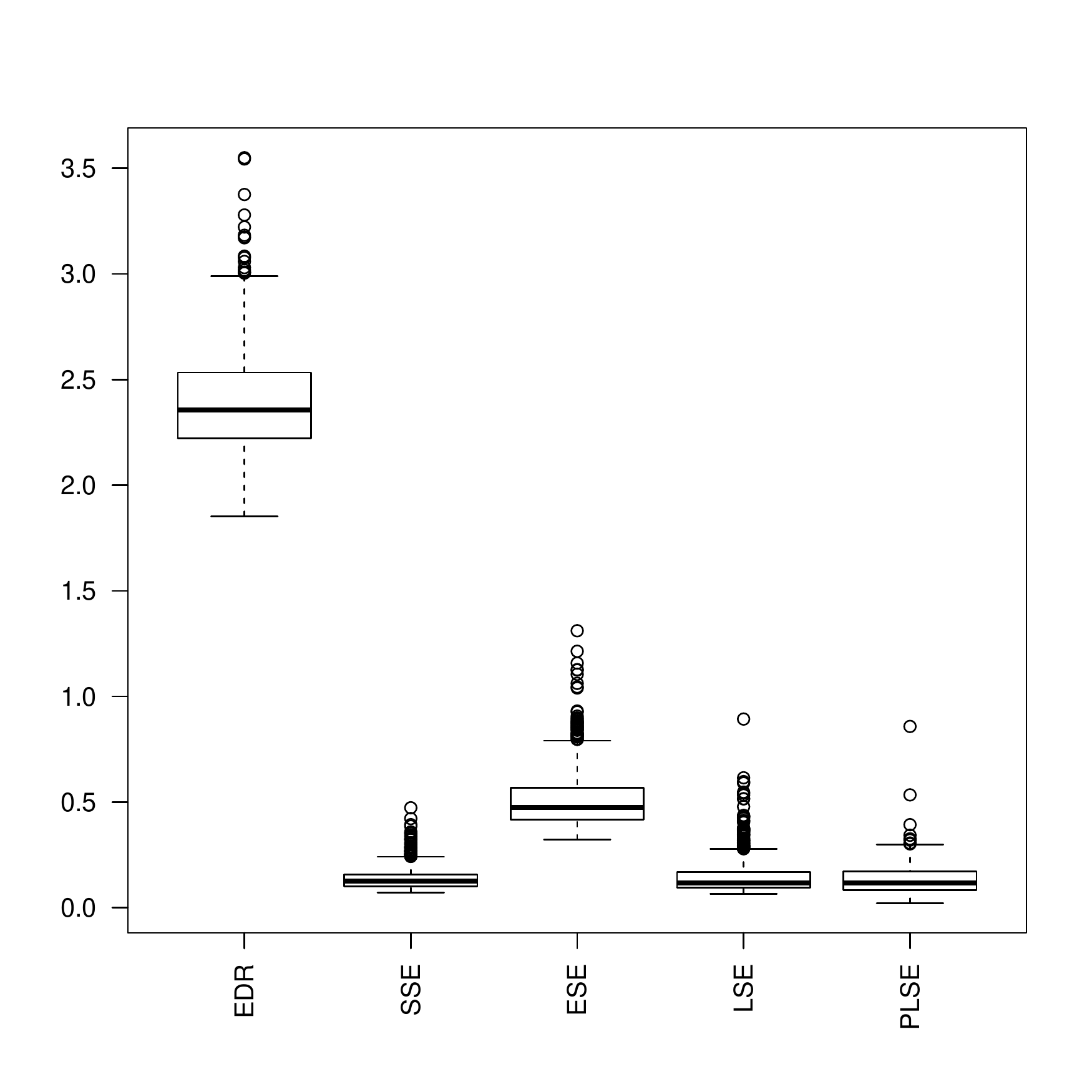}
	\caption{$d = 10$}
\end{subfigure}\\
	\begin{subfigure}{0.485\linewidth}
		\includegraphics[width=0.95\textwidth]{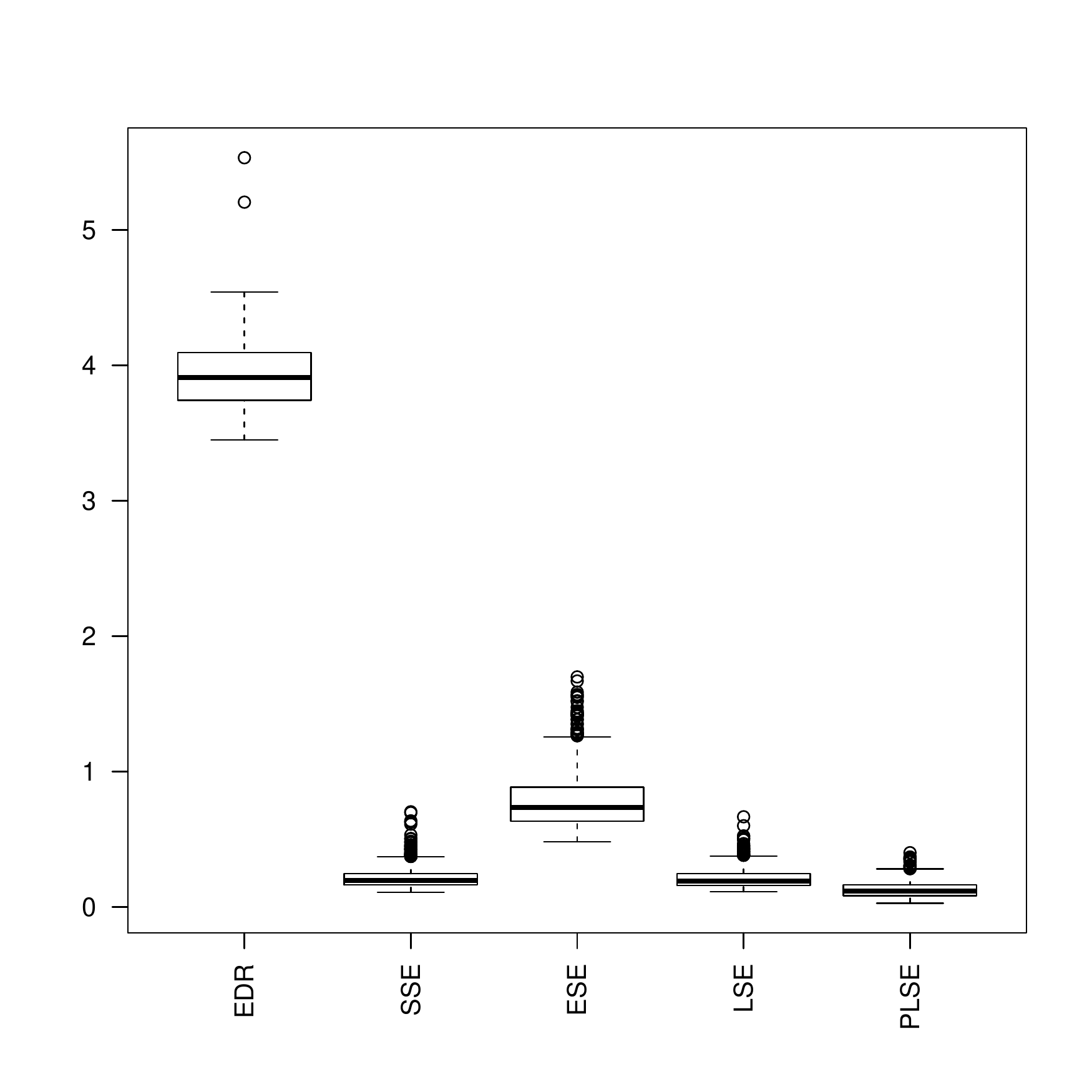}
		\caption{$d=15$}
	\end{subfigure}
\begin{subfigure}{0.485\linewidth}
	\includegraphics[width=0.95\textwidth]{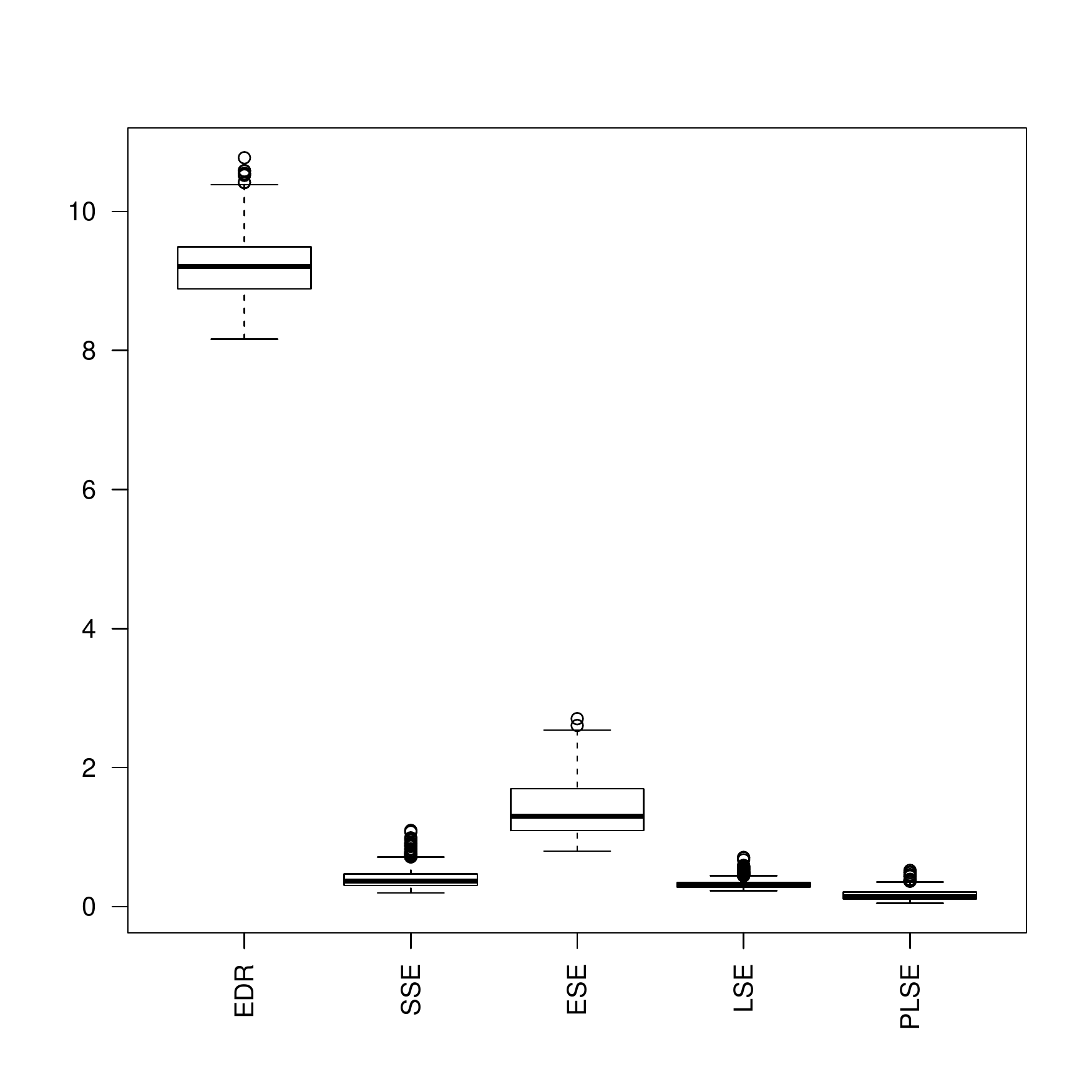}
	\caption{$d = 25$}
\end{subfigure}
	\caption{Boxplots of computing times in seconds for $n=1000$ and for (a) $d = 5$, (b) $d=10$, (c) $d=15$ and (d) $d=25$ and $1000$ replications for the EDRE, SSE, ESE, LSE and PLSE.}
\label{boxplots3}
\end{figure}

\begin{table}[!ht]
	\centering
	\caption{ Simulation 1, model  ($ X_i\sim  U[1,2], d=3$): The mean value ($\hat\mu_i$ =  mean($\hat\a_{in}), i=1,2,3$) and $n$ times the variance-covariance ($\hat\sigma_{ij} = n\cdot$cov$(\hat\a_{in},\hat\a_{jn})$,$i,j=1,2,3$) of the simple score estimate (SSE), the efficient score estimate (ESE), the least squares estimate (LSE), the maximum rank correlation estimate (MRCE), the penalized least squares estimate (PLSE) and the effective dimension reductiosn estimate (EDRE), for different sample sizes $n$ with  $X_i\sim  U[1,2]$. The line, preceded by $\infty$, gives the asymptotic values.}
	\label{table:simulation1b}
	\vspace{0.5cm}
	\scalebox{0.6}{
		\begin{tabular}{|lr|ccc|cccccc |}
			\hline
			Method & $n$ & $\hat\mu_1$ & $\hat\mu_2$ & $\hat\mu_3$ & $\hat \sigma_{11}$& $\hat\sigma_{22}$& $\hat\sigma_{33}$& $\hat\sigma_{12}$&$\hat\sigma_{13}$&$\hat\sigma_{23}$\\
			\hline
			&&&&&&&&&&\\
			SSE& 100 &0.5770 & 0.5768 & 0.5775 & 0.0260 & 0.0265 & 0.0252 &-0.0137& -0.0124 &-0.0128\\
			&500 &0.5771 & 0.5774 & 0.5775 & 0.0209 & 0.0214 & 0.0207& -0.0100& -0.0100 &-0.0106 \\
			&1000&0.5771 & 0.5773 & 0.5775 & 0.0204 & 0.0209 & 0.0206& -0.0104 &-0.0101 &-0.0105  \\
			&2000&0.5772 & 0.5773 & 0.5775 & 0.0201 & 0.0205 & 0.0203& -0.0101 &-0.0100 &-0.0103 \\
			&5000&0.5773 & 0.5774 & 0.5774 & 0.019 & 0.0198 & 0.0200& -0.0097 &-0.0099& -0.0101\\
			&10000&  0.5773 & 0.5774 & 0.5774 & 0.0192 & 0.0197 & 0.0197 &-0.0096& -0.0096 & -0.0101\\
			\hline
			&$\infty$&0.5774 & 0.5774 & 0.5774 & 0.0192  & 0.0192 & 0.0192  &-0.0096  &-0.0096  &-0.0096 \\
			&&&&&&&&&&\\
			&&&&&&&&&&\\
			ESE&100 &0.5761 & 0.5770 & 0.5783 & 0.0256 & 0.0265 & 0.0248& -0.0136& -0.0119& -0.0129  \\
			&500 &0.5767 & 0.5774 & 0.5779 & 0.0204 & 0.0208 & 0.0200& -0.0106& -0.0098& -0.0103 \\
			&1000&0.5769 & 0.5774 & 0.5778 & 0.0199 & 0.0203 & 0.0200& -0.01001& -0.0099& -0.0102\\
			&2000&0.5771 & 0.5778 & 0.5777 & 0.0195 & 0.0199 & 0.0197&-0.0098& -0.0097& -0.0101 \\
			&5000&0.5772 & 0.5774 & 0.5775 & 0.0191 & 0.0193 & 0.0194 &-0.0094& -0.0096& -0.0098 \\
			&10000&0.5773 & 0.5774 & 0.5774 & 0.0187 & 0.0192 & 0.0192 &-0.0093& -0.0094 &-0.0098\\
			\hline
			&$\infty$&0.5774 & 0.5774 & 0.5774 &0.0187 &0.0187 &0.0187 & -0.0093 & -0.0093 & -0.0093\\
			&&&&&&&&&&\\
			&&&&&&&&&&\\
			LSE&100 &0.5769 & 0.5772& 0.5767 & 0.0467 & 0.0474 & 0.0460& -0.0240 &-0.0226 &-0.0234  \\
			&500 &0.5773 & 0.5773& 0.5772 & 0.0478 & 0.0480 & 0.0474& -0.0243 &-0.0237&-0.0237	  \\
			&1000&0.5773 & 0.5773& 0.5773 & 0.0496 & 0.0500 & 0.0496& -0.0250 &-0.0246& -0.0250  \\
			&2000& 0.5774 & 0.5772&  0.5773 & 0.0504 & 0.0517 & 0.0517& -0.0252 &-0.0252& -0.0265\\
			&5000&0.5774 & 0.5773&  0.5773 & 0.0549 & 0.0553 & 0.0541& -0.0280 &-0.0268& -0.0273 \\
			&10000&0.5773 & 0.5774 & 0.5773 & 0.0583 & 0.0579&  0.0587& -0.0287 &-0.0295& -0.0291\\
			\hline
			&$\infty$&0.5774 & 0.5774 & 0.5774 & ? & ?& ?& ?& ?& ?\\
			&&&&&&&&&&\\
			&&&&&&&&&&\\
			MRCE&100&  0.5770&  0.5770  &0.5769&  0.0465&  0.0463&  0.0448& -0.0241& -0.0224&-0.0223\\
			&500 & 0.5773 & 0.5774  &0.5773 & 0.0171&  0.0167&  0.0170& -0.0084& -0.0087& -0.0082\\
			&1000& 0.5773  &0.57741 &0.5773 & 0.0343&  0.0333&  0.0339& -0.0168& -0.0174& -0.0165\\
			&2000 &0.5773  &0.5773  &0.5773 & 0.0302&  0.0303&  0.0316& -0.0145& -0.0157& -0.0158\\
			&5000 &0.5774 &0.5773  &0.5773 & 0.0288&  0.0288&  0.0292& -0.0142 &-0.0146& -0.0146\\
			&10000 &0.5774 & 0.5774&  0.5773 & 0.0266&  0.0276&  0.0277& -0.0133& -0.0134& -0.0143	\\
			\hline
			&$\infty$&0.5774 & 0.5774 & 0.57740 &0.0214 &0.0214 &0.0214 & -0.0107  & -0.0107  & -0.0107 \\		
			&&&&&&&&&&\\
			&&&&&&&&&&\\
			EDRE&100&0.5772&0.5771&0.5772&0.0215&0.0201&0.0208&-0.0105&-0.0111&-0.0096\\
			&500&0.5773&0.5774&0.5772&0.0198&0.0195&0.0195&-0.0099&-0.0099&-0.009\\
			&1000&0.5771&0.5777&0.5771&0.0208&0.0212&0.0207&-0.0107&-0.0101&-0.0106\\
			&2000&0.5772&0.5774&0.5774&0.0222&0.0225&0.0209&-0.0119&-0.0103&-0.0106\\
			&5000&0.5773&0.5774&0.5773&0.0218&0.0236&0.0240&-0.0107&-0.0111&-0.0129\\
			&10000&0.5772&0.5774&0.5774&0.0239&0.0246&0.0249&-0.0118&-0.0121&-0.0128\\
			\hline
			&$\infty$&0.5774 & 0.5774 & 0.5774 &? &?&?&?&?&?\\	
			&&&&&&&&&&\\
			&&&&&&&&&&\\
			PLSE&100&0.5772&0.5771&0.5772&0.0215&0.0201&0.0208&-0.0105&-0.0111&-0.0096\\
			&500&0.5774&0.5774&0.5771&0.0198&0.0194&0.0198&-0.0097&-0.0101&-0.0097\\
			&1000&0.5772&0.5777&0.5771&0.0206&0.0214&0.0211&-0.0105&-0.0101&-0.0110\\
			&2000&0.5773&0.5773&0.5774&0.0233&0.0235&0.0217&-0.0125&-0.0107&-0.0110\\
			&5000&0.5774&0.5773&0.5773&0.0268&0.0287&0.0297&-0.0129&-0.0139&-0.0158\\
			&10000&0.5769&0.5776&0.5776&0.0517&0.0489&0.0566&-0.0219&-0.0296&-0.027\\			
			\hline
			&$\infty$&0.5774 & 0.5774 & 0.5774 &0.0187 &0.0187 &0.0187 & -0.0093 & -0.0093 & -0.0093\\
		&&&&&&&&&&\\
		&&&&&&&&&&\\	
		\hline		
		\end{tabular}}
\end{table}

\begin{table}[!ht]
	\centering
	\caption{ Simulation 1, model  ($X_i\sim  N(0,1), d=3$): The mean value ($\hat\mu_i$ =  mean($\hat \a_{in}), i=1,2,3$) and $n$ times the variance-covariance ($\hat\sigma_{ij} = n\cdot$cov$(\hat\a_{in},\hat\a_{jn})$,$i,j=1,2,3$) of the simple score estimate (SSE), the efficient score estimate (ESE), the least squares estimate (LSE), the maximum rank correlation estimate (MRCE), the hybrid link-free least squares estimate (H-LFLSE) and the link-free least squares estimate (LFLSE), for different sample sizes $n$ with  $X_i\sim  N(0,1)$. The line, preceded by $\infty$, gives the asymptotic values.}
	\label{table:simulation2b}
	\vspace{0.5cm}
	\scalebox{0.6}{
		\begin{tabular}{|lr|ccc|cccccc |}
			\hline
			Method&$n$ & $\hat\mu_1$ & $\hat\mu_2$ & $\hat\mu_3$ & $\hat \sigma_{11}$& $\hat\sigma_{22}$& $\hat\sigma_{33}$& $\hat\sigma_{12}$&$\hat\sigma_{13}$&$\hat\sigma_{23}$\\
			\hline
			&&&&&&&&&&\\
			SSE&100 & 0.5710 & 0.5756 & 0.5780 & 0.2638 & 0.3093 & 0.2828 & -0.1445 & -0.1141& -0.1657 \\
			&500 & 0.5757 & 0.5771 & 0.5785 & 0.1414 & 0.1612 & 0.1498 & -0.0761 & -0.0641 & -0.0856 \\
			&1000& 0.5764 & 0.5772 & 0.5781 & 0.1234 & 0.1248 & 0.1213 & -0.0631 & -0.0600 & -0.0617 \\
			&2000& 0.5768 & 0.5771 & 0.5779 & 0.1044 & 0.1049 & 0.1037 & -0.0527 & -0.0517 & -0.0522\\
			&5000& 0.5770 & 0.5773 & 0.5776 & 0.0936 & 0.0972 & 0.0939 & -0.0484 & -0.0452 & -0.0488\\
			&10000&0.5771 & 0.5774 & 0.5775 & 0.0878 & 0.0926 & 0.0896 & -0.0454 & -0.0424 & -0.0473\\
			\hline
			&$\infty$&0.5774 & 0.5774 & 0.5774 & 0.0741 & 0.0741 & 0.0741 & -0.0370 & -0.0370 & -0.0370\\
			&&&&&&&&&&\\
			&&&&&&&&&&\\
			ESE&100 & 0.5718 & 0.5770 & 0.5799 & 0.1233 & 0.1410 & 0.1218 & -0.0701 & -0.0495 & -0.0719\\
			&500 & 0.5758 & 0.5775 & 0.5785 & 0.0565 & 0.0591 & 0.0513 & -0.0321 & -0.0233 & -0.0278\\
			&1000& 0.5764 & 0.5774 & 0.5781 & 0.0433 & 0.0432 & 0.0418 & -0.0223 & -0.0209 & -0.0210\\
			&2000& 0.5768 & 0.57730 & 0.5779 & 0.0366 & 0.0362 & 0.0365 & -0.0181 & -0.0184 & -0.0181\\
			&5000& 0.5770 & 0.5774 & 0.5776 & 0.0304 & 0.0321 & 0.0320 & -0.0152 & -0.0152 & -0.0168\\
			&10000&0.5771 & 0.5774 & 0.5775 & 0.0296 & 0.0297 & 0.0303 & -0.0145 & -0.0151 & -0.0152\\
			\hline
			&$\infty$&0.5774 & 0.5774 & 0.5774 & 0.0247 & 0.0247 & 0.0247 & -0.0123 & -0.0123 & -0.0123\\
			&&&&&&&&&&\\
			&&&&&&&&&&\\
			LSE&100 & 0.5751 & 0.5748 & 0.5776 & 0.1737 & 0.1749 & 0.1731 &-0.0858& -0.0869& -0.0866\\
			&500 & 0.5768 & 0.57737 & 0.5773 & 0.1072 & 0.1046 & 0.1069 &-0.0523& -0.0545& -0.0523\\
			&1000& 0.5770 & 0.5774 & 0.5774 & 0.1011 & 0.0982 & 0.1004 &-0.0494 &-0.0516 &-0.0489 \\
			&2000& 0.5773 & 0.5774 & 0.5772 & 0.0921 & 0.0914 & 0.0895 & -0.0470& -0.0451 &-0.0444\\
			&5000& 0.5773 & 0.5775 & 0.5772 & 0.0904 & 0.0887 & 0.0899 & -0.0447 & -0.0457 &-0.0441 \\
			&10000& 0.5773 & 0.5775 & 0.5773 & 0.0890&  0.0852 & 0.0898 &-0.0421& -0.0467 &-0.0431\\
			\hline
			&$\infty$&0.5774 & 0.5774 & 0.5774 & ? & ?& ? & ? & ? &? \\
			&&&&&&&&&&\\
			&&&&&&&&&&\\
			MRCE&100 &0.5687&  0.5707 & 0.5718 & 0.7942 & 0.7960&  0.8022& -0.3899 &-0.3841& -0.3878\\
			&500 &0.5766 & 0.5766  &0.5762 & 0.5132 & 0.5209 & 0.5243 &-0.2537 &-0.2573 &-0.2659\\
			&1000 &0.5774 & 0.5767  &0.5767 & 0.4875 & 0.4826 & 0.4753 &-0.2467 &-0.2408 &-0.2343\\
			&2000 &0.5772 & 0.5770  &0.5773 & 0.4363  &0.4347 & 0.4365 &-0.2173 &-0.2189 &-0.2172\\
			&5000 &0.5773 & 0.5773  &0.5773 & 0.4169  &0.4303 & 0.4268 &-0.2102 &-0.2068 &-0.2198\\
			&10000 &0.5772 & 0.5774 & 0.5773 & 0.3985 & 0.4182 & 0.4109 &-0.2029 &-0.1956 &-0.2153\\
			\hline
			&$\infty$&0.5774 & 0.5774 & 0.5774 & 0.3583 & 0.3583& 0.3583 & -0.1791 & -0.1791 & -0.1791\\
			&&&&&&&&&&\\
			&&&&&&&&&&\\
			H-LFLSE & 100 & 0.5733 & 0.5725 & 0.5737 & 0.4727 & 0.4710 & 0.4960 &-0.2213 &-0.2477& -0.2429 \\
			& 500 & 0.5762 & 0.5770  &0.5762 & 0.5117  &0.5031  &0.5192& -0.2468 &-0.2623& -0.2560 \\
			& 1000 & 0.5769&  0.5771 & 0.5767 & 0.5205 & 0.5130 & 0.5080 &-0.2634 &-0.2564& -0.2500  \\
			& 2000 & 0.5771&  0.5773 & 0.5770 & 0.5193 & 0.5284 & 0.5190 &-0.2640 &-0.2541& -0.2648 \\
			& 5000 & 0.5773  &0.5773 & 0.5772 & 0.5099 & 0.5291 & 0.5135 &-0.2629& -0.2467 &-0.2665  \\
			& 10000 & 0.5773 & 0.5774 & 0.5772 & 0.5267 & 0.5194&  0.5211 &-0.2626 &-0.2641 &-0.2569 \\
			\hline
			&$\infty$&0.5774 & 0.5774 & 0.5774 & 0.5185 & 0.5185 & 0.5185 & -0.2593 & -0.2593 & -0.2593\\
		&&&&&&&&&&\\
		&&&&&&&&&&\\			
		LFLSE & 100&0.5788&0.5808&0.5798&0.6789&0.6717&0.6580&-0.0775&-0.0597&-0.0763\\
		&500&0.5786&0.5782&0.5775&0.6640&0.6975&0.6432&-0.0913&-0.0551&-0.0903\\
		&1000&0.5780&0.5775&0.5776&0.6763&0.6945&0.6467&-0.0857&-0.0775&-0.0838\\
		&2000&0.5776&0.5774&0.5777&0.7053&0.6825&0.6771&-0.0975&-0.0866&-0.0850\\
		&5000&0.5777&0.5773&0.5773&0.6957&0.6998&0.6659&-0.0848&-0.0735&-0.1040\\
		&10000&0.5774&0.5773&0.5774&0.6812&0.6716&0.6893&-0.0925&--0.0889&-0.0872\\
		\hline
		&$\infty$&0.5774 & 0.5774 & 0.5774 & 0.6852 &0.6852 & 0.6852 & -0.0926 & -0.0926 & -0.0926\\
		&&&&&&&&&&\\
		\hline
		\end{tabular}}
	\end{table}
\end{document}


\maketitle

\newpage
\appendix
\begin{center}
	\Large{\textbf{Supplementary Material}}
\end{center}
\setcounter{page}{1}
\numberwithin{equation}{section}
\setcounter{equation}{0}

\section{Preliminaries}
\label{preliminaries}
In this supplementary material, we first prove Theorem \ref{Th:LFLSE} stated in Section \ref{sec:basic_estimate} for the link-free least squares estimator when the covariate is elliptically symmetric. We next give the proofs of the results stated in Section \ref{section:Asymptotics} for the score estimators using a parametrization of the unit sphere, introduced in Section \ref{subsection:score-estimator1}. The results for the simple and efficient score estimators together with additional technical lemmas needed for proving our main results are given in different sections in this supplement. We will write ``Lemma X in Appendix Y" when we refer to these results in the remainder of the Appendix.
Entropy results are used in our proofs. Before we start the proofs, we first introduce some notations and definitions used in the remainder of the Appendix.

We will denote the $L_2$-norm of a function $f$ defined on $\mathcal{X}$  with respect to some probability measure $\P$ by $\Vert \cdot \Vert_{\P}$ ; i.e.,
\begin{eqnarray*}
\Vert f \Vert_{\P} =  \P(f^2)^{1/2} =  \left(\int_{\mcX}  f^2(x)  d\P(x)\right)^{1/2} .
\end{eqnarray*}
Also, we will denote by $\|\cdot \|_{B, \P} $ the Bernstein norm of a function $f$  defined on $\mathcal{X} \times \R$ which is given by 
\begin{eqnarray*}
\Vert f \Vert_{\P, B}  =  \bigg(2\P\Big(e^{\vert f \vert}  - \vert f \vert - 1  \Big)\bigg)^{1/2}.
\end{eqnarray*}
For both norms,  $\P$ will be taken to be  $P_0$, the true joint probability measure of the $(\bm X, Y)$.  Note that when $f$ is only a function of $\bm x \in \mathcal{X}$ then 
\begin{eqnarray*}
	\Vert f \Vert_{P_0}  =  \left(\int_{\mcX}  f^2(\bm x)  dG(\bm x)\right)^{1/2}  = \left(\int_{\mcX}  f^2(\bm x)  g(\bm x) d\bm x\right)^{1/2},
\end{eqnarray*}
by Assumption A5.

For a class of functions ${\cal F}$ on ${\cal R}$ equipped with a norm $\| \cdot\|$, we let the bracketing number $N_{B}(\zeta, {\cal F}, \Vert \cdot\Vert)$ denote the minimal number $N$ for which there exists pairs of functions $\{[g_j^L,g_j^U], j=1,\ldots, N\}$ such that $\|g_j^U-g_j^L\|\le \zeta$ for all $j=1,\ldots,N$ and such that for each $g \in {\cal F}$ there is a $j\in \{1,\ldots,N\}$ such that $g_j^L\le g \le g_j^U$. The $\zeta-$entropy with bracketing of ${\cal F}$ is defined as $H_{B}(\zeta, {\cal F}, \| \cdot\|) = \log(N_{B}(\zeta, {\cal F},\| \cdot\|))$.

Results on entropy calculations used in proving our main results are given in the Appendix \ref{supD:entropy}.  
Our proofs use inequalities for empirical processes described in Lemma 3.4.2 and Lemma 3.4.3 of \cite{vdvwe:96}.

\textbf{Lemma 3.4.2 (\cite{vdvwe:96})}
Let ${\cal F}$ be a class of measurable functions such that $\|f \|_{\P}\le \delta$ and $\|f \|_{\infty} \le M$ for every $f$ in ${\cal F}$. Then
\begin{align*}
\E_{\P} \Big[\Vert \mathbb{G}_n \Vert_{{\cal F}} \Big] &\lesssim   J_n(\delta, {\cal F}, \|\cdot \|_{\P})  \left( 1 +\frac{J_n(\delta, {\cal F}, \|\cdot \|_{\P})}{\sqrt{n} \delta^2 }M  \right),
\end{align*}
where $\mathbb{G}_n  =\sqrt n (\P_n-\P)$ and 
\begin{align*}
J_n(\delta, {\cal F}, \|\cdot \|) = \int_{0}^{\delta}\sqrt{1 + H_B(\epsilon, {\cal F}, \|\cdot \|)}d\epsilon
\end{align*}

\textbf{Lemma 3.4.3 (\cite{vdvwe:96})}
Let ${\cal F}$ be a class of measurable functions such that $\|f \|_{\P,B}\le \delta$ for every $f$ in ${\cal F}$. Then
\begin{align*}
\E_{\P} \Big[\Vert \mathbb{G}_n \Vert_{{\cal F}} \Big] &\lesssim   J_n(\delta, {\cal F}, \|\cdot \|_{\P,B})  \left( 1 +\frac{J_n(\delta, {\cal F}, \|\cdot \|_{\P,B})}{\sqrt{n} \delta^2 } \right),
\end{align*}

In the sequel, and whenever the $\epsilon$-bracketing entropy of some class $\mathcal{F}$ with respect to some norm $\Vert\cdot \Vert $ is bounded above by $ C \epsilon^{-1}$ for some constant $C > 0$ (which may depend on $n$), we will write for all $e> \epsilon$
\begin{eqnarray}\label{Jn}
J_n(e)  = \int_0^e (1+  C / \epsilon)^{1/2} d\epsilon
\end{eqnarray}  
Moreover, we will use the inequality 
\begin{eqnarray}\label{IneqJn}
J_n(e)   \le e + 2 C^{1/2} e^{1/2}
\end{eqnarray}
which is an immediate consequence of the fact that $\sqrt{x+y } \le \sqrt{x} + \sqrt{y}$ for all $x, y \ge 0$.

\section{The link-free estimator when the covariate is elliptically symmetric}
\label{supA:LFLSE}
In this section we prove Theorem \ref{Th:LFLSE}.\\
\begin{proof}
	Since $\hat{\bm\a}_{n,\bmS_n norm1}^T\bmS_n\hat{\bm\a}_{n,\bmS_n norm1}=1$, we get for each converging subsequence  $(\hat{\bm\a}_{n_k,\bmS_{n_k} norm1})$ of
	the sequence $(\hat{\bm\a}_{n,\bmS_n norm1})$, and taking $\bma=\hat{\bm\a}_{n_k,\bm S_{n_k} norm1}$ in (\ref{Lagrange_norm1}), convergence to:
	\begin{align}
	\label{limit_eq}
	\left(\bm I-\bm\Sigma\tilde{\bma}\tilde\bma^T\right)\left\{\text{Cov}(\bmX,Y)-\bm\Sigma\tilde\bma\right\}=0.
	\end{align}
	where $\tilde\bma$ is the limit of the subsequence $(\hat{\bm\a}_{n_k,\bmS_{n_k} norm1})$ and therefore satisfies $\tilde{\bma}^T\bm\Sigma\tilde\bma=1$. Arguing as in \cite{balabdaoui2016} in their proof of Theorem 3.1, we have:
	\begin{align*}
	\text{Cov}(\bmX,Y)=\E\left\{\psi_0(\bma_0^T\bmX)\E\left\{\bmX-\bm\mu|\bm\a_0^T\bmX\right\}\right\},
	\end{align*}
	where $\E\left\{\bmX-\bm\mu|\bm\a_0^T\bmX\right\}$ is linear in $\bma_0^T\bmX$ (because of the elliptic symmetry of $\bmX$), and where $\bm\m=\E\bmX$. It follows that
	\begin{align*}
	\text{Cov}(\bmX,Y)=\left(\bma_0^T\bm\Sigma\bma_0\right)^{-1}\text{Cov}\left\{\psi_0(\bma_0^T\bmX),\bma_0^T\bmX\right\}\bm\Sigma\bma_0
	=\text{Cov}\left\{\psi_0(\bma_0^T\bmX),\bma_0^T\bmX\right\}\bm\Sigma\bma_0,
	\end{align*}
	where we use $\bma_0^T\bm\Sigma\bma_0=1$ in the last equality. Hence (\ref{limit_eq}) turns into
	\begin{align*}
	c\left(\bm I-\bm\Sigma\tilde{\bma}\tilde\bma^T\right)\bm\Sigma \bma_0=
	\left(\bm I-\bm\Sigma\tilde{\bma}\tilde\bma^T\right)\bm\Sigma\tilde\bma=\bm0,
	\qquad c=\text{Cov}\left\{\psi_0(\bma_0^T\bmX),\bma_0^T\bmX\right\}.
	\end{align*}
	Assuming $c\ne0$, using the condition of $\psi_0$ being strictly increasing on a subinterval, we get:
	\begin{align}
	\label{orthogonality_elliptic}
	\bm\Sigma\bma_0=\left(\tilde\bma^T\bm\Sigma\bma_0\right)\bm\Sigma\tilde\bma.
	\end{align}
	Premultiplying on the left with $\bma_0^T$, we get:
	\begin{align*}
	\bma_0^T\bm\Sigma\bma_0=1=\left(\tilde\bma^T\bm\Sigma\bma_0\right)^2,
	\end{align*}
	implying $\tilde\bma^T\bm\Sigma\bma_0=\pm1$, and hence $\bm\Sigma\tilde\bma=\pm\bm\Sigma\bma_0$, which leads to $\tilde\bma=\pm\bma_0$.
	So, restricting the estimators to a sufficiently small neighborhood of $\bma_0$, excluding $-\bma_0$, we get that all convergent subsequences of $(\hat{\bm\a}_{n,\bmS_n norm1})$ converge to $\bma_0$, and hence $(\hat{\bm\a}_{n,\bmS_n norm1})$ itself converges (in distribution) to $\bma_0$.
	
	For the remaining part of the proof, we first note that we just showed that $\bma_0$ satisfies (\ref{limit_eq}).
	We now write (\ref{Lagrange_norm1}) in the following form:
	\begin{align*}
	&\frac1n\left(\bm I-\bmS_n{\bm\a}\bma^T\right)\sum_{i=1}^n\left(\bm X_i -\bar{\bm X}_n\right)\left\{Y_i-\bm\a^T\left(\bm X_i -\bar{\bm X}_n\right)\right\}\\
	&=\frac1n\left(\bm I-\bmS_n{\bm\a}\bma^T\right)\sum_{i=1}^n\left(\bm X_i -\bar{\bm X}_n\right)\left\{Y_i-\bm\a_0^T\left(\bm X_i -\bar{\bm X}_n\right)\right\}\\
	&\qquad-\frac1n\left(\bm I-\bmS_n\bm\a\bma^T\right)\sum_{i=1}^n\left(\bm X_i -\bar{\bm X}_n\right)\left\{\bm\a-\bm\a_0\right\}^T\left(\bm X_i -\bar{\bm X}_n\right)\\
	&=\frac1n\left(\bm I-\bmS_n{\bm\a_0}\bma_0^T\right)\sum_{i=1}^n\left(\bm X_i -\bar{\bm X}_n\right)\left\{Y_i-\bm\a_0^T\left(\bm X_i -\bar{\bm X}_n\right)\right\}\\
	&\qquad-\frac1n\left\{\bmS_n\left(\bma-\bma_0\right)\bma^T+\bmS_n\bma_0\left(\bm\a-\bma_0\right)^T\right\}\sum_{i=1}^n\left(\bm X_i -\bar{\bm X}_n\right)\left\{Y_i-\bm\a_0^T\left(\bm X_i -\bar{\bm X}_n\right)\right\}\\
	&\qquad-\frac1n\left(\bm I-\bmS_n\bm\a\bma^T\right)\sum_{i=1}^n\left(\bm X_i -\bar{\bm X}_n\right)\left(\bm X_i -\bar{\bm X}_n\right)^T\left\{\bm\a-\bm\a_0\right\}
	\end{align*}
	\begin{align*}
	&=\frac1n\left(\bm I-\bm\Sigma{\bm\a_0}\bma_0^T\right)\sum_{i=1}^n\left(\bm X_i -\bar{\bm X}_n\right)\left\{Y_i-\bm\a_0^T\left(\bm X_i -\bar{\bm X}_n\right)\right\}\\
	&\qquad-\frac1n\left\{\bmS_n-\bm\Sigma\right\}\bm\a_0\bma_0^T\sum_{i=1}^n\left(\bm X_i -\bar{\bm X}_n\right)\left\{Y_i-\bm\a_0^T\left(\bm X_i -\bar{\bm X}_n\right)\right\}\\
	&\qquad-\frac1n\left\{\bmS_n\left(\bma-\bma_0\right)\bma^T+\bmS_n\bma_0\left(\bm\a-\bma_0\right)^T\right\}\sum_{i=1}^n\left(\bm X_i -\bar{\bm X}_n\right)\left\{Y_i-\bm\a_0^T\left(\bm X_i -\bar{\bm X}_n\right)\right\}\\
	&\qquad-\frac1n\left(\bm I-\bmS_n\bma\bma^T\right)\sum_{i=1}^n\left(\bm X_i -\bar{\bm X}_n\right)\left(\bm X_i -\bar{\bm X}_n\right)^T\left\{\bm\a-\bm\a_0\right\}\\
	&=\bm0.
	\end{align*}
	Define
	\begin{align}
	\label{def_Z_n}
	\bmZ_n=\sqrt{n}\left\{\hat{\bm\a}_{n,\bmS_n norm1}-\bma_0\right\}
	\end{align}
	Then, multiplying the terms in the preceding relation with $\sqrt{n}$, we get:
	\begin{align*}
	&\left\{\left(I-\bm\Sigma\bma_0\bma_0^T\right)\bm\Sigma\bmZ_n+\bm\Sigma\left( \bmZ_n\bma_0^T+\bma_0 \bmZ_n^T\right)\left\{\text{Cov}(\bmX,Y)-\bm\Sigma\bma_0\right\}\right\}\left\{1+o_p(1)\right\}\\
	&=\frac1{\sqrt{n}}\left(\bm I-\bm\Sigma{\bm\a_0}\bma_0^T\right)\sum_{i=1}^n\left(\bm X_i -\bar{\bm X}_n\right)\left\{Y_i-\bm\a_0^T\left(\bm X_i -\bar{\bm X}_n\right)\right\}\\
	&\qquad-\sqrt{n}\left\{\bmS_n-\bm\Sigma\right\}\bm\a_0\bma_0^T\left\{\text{Cov}(\bmX,Y)-\bm\Sigma\bma_0\right\}\left\{1+o_p(1)\right\}.
	\end{align*}
	So the asymptotic equation becomes:
	\begin{align}
	\label{asymptotic_relation_elliptic_sym}
	&\left(I-\bm\Sigma\bma_0\bma_0^T\right)\bm\Sigma\bmZ_n+c'\bm\Sigma\left( \bmZ_n\bma_0^T+\bma_0 \bmZ_n^T\right)\bm\Sigma\bma_0\nonumber\\
	&=\frac1{\sqrt{n}}\left(\bm I-\bm\Sigma{\bm\a_0}\bma_0^T\right)\sum_{i=1}^n\left(\bm X_i -\bar{\bm X}_n\right)\left\{Y_i-\bm\a_0^T\left(\bm X_i -\bar{\bm X}_n\right)\right\}\nonumber\\
	&\qquad\qquad-\c'\sqrt{n}\left\{\bmS_n-\bm\Sigma\right\}\bm\a_0+o_p(1),\qquad\qquad
	c'=\text{Cov}\left(\psi_0(\bma_0^T),\bma_0^T\bm X\right)-1.
	\end{align}
	Since the right-hand side tends to a normally distributed random vector, the covariance matrix of the (normal) limit distribution of $\bmZ_n$ can be deduced from this relation.
	Note that multiplying this relation with $\bma_0^T$ yields:
	\begin{align}
	\label{asymptotic_relation_elliptic_sym2}
	2\bma_0^T\bm\Sigma\bmZ_n=- \sqrt{n}\bma_0^T\left\{\bmS_n-\bm\Sigma\right\}\bm\a_0+o_p(1),
	\end{align}
	showing that the limit distribution of $\bmZ_n$  has a nondegenerate component in the direction of $\Sigma\bma_0$ if $c'\ne0$.
\end{proof}

We compute the covariance matrix of the limit distribution for the simulation setting of Table \ref{table:simulation2b}. In this case $\bm\Sigma=\bm I$, so (\ref{asymptotic_relation_elliptic_sym}) boils down to
\begin{align}
\label{asymptotic_relation_elliptic_sym_simulation}
&\left(I-\bma_0\bma_0^T\right)\bmZ_n+c'\left( \bmZ_n\bma_0^T+\bma_0 \bmZ_n^T\right)\bma_0\nonumber\\
&=\frac1{\sqrt{n}}\left(\bm I-{\bm\a_0}\bma_0^T\right)\sum_{i=1}^n\left(\bm X_i -\bar{\bm X}_n\right)\left\{Y_i-\bm\a_0^T\left(\bm X_i -\bar{\bm X}_n\right)\right\}\nonumber\\
&\qquad\qquad-c'\sqrt{n}\left\{\bmS_n-\bm I\right\}\bm\a_0+o_p(1),\qquad\qquad
c'=\text{Cov}\left(\psi_0(\bma_0^T),\bma_0^T\bm X\right)-1=2.
\end{align}

We can rewrite the left-hand side of (\ref{asymptotic_relation_elliptic_sym_simulation}) in the following way:
\begin{align}
\label{asymptotic_relation_left}
&\left(I-\bma_0\bma_0^T\right)\bmZ_n+2 \left(\bmZ_n\bma_0^T+\bma_0 \bmZ_n^T\right)\bma_0
= \bmZ_n  -  \bma_0\bma_0^T  \bmZ_n   + 2 \bmZ_n + 2 \bma_0 \bmZ_n^T  \bma_0\nonumber\\
& =3 \bmZ_n + \bma_0\bma_0^T\bmZ_n\nonumber\\
&=3\left(I-\bma_0\bma_0^T\right)\bmZ_n+4\bma_0\bma_0^T\bmZ_n.
\end{align}
We write the dominant term of the right-hand side of (\ref{asymptotic_relation_elliptic_sym_simulation}) in a similar decomposition:
\begin{align*}
&\frac1{\sqrt{n}}\left(\bm I-{\bm\a_0}\bma_0^T\right)\sum_{i=1}^n\left(\bm X_i -\bar{\bm X}_n\right)\left\{Y_i-\bm\a_0^T\left(\bm X_i -\bar{\bm X}_n\right)\right\}
-2\sqrt{n}\left\{\bmS_n-\bm I\right\}\bm\a_0\\
&=\left(\bm I-{\bm\a_0}\bma_0^T\right)\left\{\frac1{\sqrt{n}}\sum_{i=1}^n\left(\bm X_i -\bar{\bm X}_n\right)\left\{Y_i-\bm\a_0^T\left(\bm X_i -\bar{\bm X}_n\right)\right\}-2\sqrt{n}\left\{\bmS_n-\bm I\right\}\bm\a_0\right\}\\
&=\left(\bm I-{\bm\a_0}\bma_0^T\right)\left\{\frac1{\sqrt{n}}\sum_{i=1}^n\left(\bm X_i -\bar{\bm X}_n\right)\left\{Y_i-3\bm\a_0^T\left(\bm X_i -\bar{\bm X}_n\right)\right\}\right\}
-2\bma_0\bma_0^T\sqrt{n}\left\{\bmS_n-\bm I\right\}\bm\a_0
\end{align*}
Moreover, we write:
\begin{align*}
\bm Z_n=\bm U_n+\bm V_n,
\end{align*}
where $\bm V_n=(\bma_0^T\bm Z_n) \bma_0$, and $\bm U_n=\bm Z_n-\bm V_n$. Then
\begin{align*}
\bma_0^T\bmU_n=\bma_0^T\bmZ_n-\bma_0^T(\bma_0^T\bmZ_n)\bma_0=\bma_0^T\bmZ_n-\bma_0^T\bma_0(\bma_0^T\bmZ_n)=0,
\end{align*}
so we get from (\ref{asymptotic_relation_left}) that the left side of (\ref{asymptotic_relation_elliptic_sym_simulation}) is given by:
\begin{align*}
3\bmU_n+4\bm V_n.  
\end{align*}
So we get, omitting the $o_p(1)$ term on the right of  (\ref{asymptotic_relation_elliptic_sym_simulation}) now, the equations
\begin{align}
\label{Z_n-equation1}
\bmU_n=\frac13\left(\bm I-{\bm\a_0}\bma_0^T\right)\left\{\frac1{\sqrt{n}}\sum_{i=1}^n\left(\bm X_i -\bar{\bm X}_n\right)\left\{Y_i-3\bm\a_0^T\left(\bm X_i -\bar{\bm X}_n\right)\right\}\right\},
\end{align}
and, defining $W_n=\bma_0^T\bmZ_n$,
\begin{align*}
4\bmV_n=4W_n\bma_0=4\bma_0\bma_0^T\bm V_n
=-2\bma_0\bma_0^T\sqrt{n}\left\{\bmS_n-\bm I\right\}\bm\a_0,
\end{align*}
implying, using $\bma_0\ne\bm0$:
\begin{align}
\label{Z_n-equation2}
W_n=-\frac12\bma_0^T\sqrt{n}\left\{\bmS_n-\bm I\right\}\bma_0.
\end{align}
So
\begin{align}
\bm Z_n=\bmU_n+W_n\bma_0,
\end{align}
where $\bmU_n$ satisfies equation (\ref{Z_n-equation1}) and $W_n$ satisfies (\ref{Z_n-equation2}).

The random vector $\bmU_n$ converges in distribution to a normal random vector $\bmU$ with mean $\bm0$ and covariance matrix
\begin{align*}
\text{Cov}(\bmU)=\text{Cov}\left(\frac13\left(\bm I-{\bm\a_0}\bma_0^T\right)\left\{\left(\bma_0^T\bmX)^3+\e-3\bma_0^T\bmX\right)\bmX\right\}\right),
\end{align*}
which can be computed to be:
\begin{align*}
\frac7{27}\begin{bmatrix}
\phantom{-}2  &-1 & -1\\[0.3em]
-1 & \phantom{-}2 & -1 \\[0.3em]
-1 & -1 & \phantom{-}2
\end{bmatrix}
,
\end{align*}
and $\bmV_n=W_n\bma_0$ converges in distribution to a vector $\bmV$ with covariance matrix:
\begin{align*}
\frac16\begin{bmatrix}
1 &1 &1\\[0.3em]
1 & 1 & 1 \\[0.3em]
1 & 1 & 1
\end{bmatrix}
.
\end{align*}
Finally, the pair $(\bmU_n,\bmV_n)$ converges to $(\bmU,\bmV)$, where $\bm U$ and $\bmV$ satisfy $\E\,\bmU\bmV^T=\bm0$.
It now follows that the asymptotic covariance matrix of $\bmZ_n=\bmU_n+\bmV_n$ is given by:
\begin{align*}
\frac7{27}\begin{bmatrix}
\phantom{-}2  &-1 & -1\\[0.3em]
-1 & \phantom{-}2 & -1 \\[0.3em]
-1 & -1 & \phantom{-}2
\end{bmatrix}
+
\frac16\begin{bmatrix}
1 &1 &1\\[0.3em]
1 & 1 & 1 \\[0.3em]
1 & 1 & 1
\end{bmatrix}
=
\frac1{45}\begin{bmatrix}
37  &-5 & -5\\[0.3em]
-5 & 37 & -5 \\[0.3em]
-5 & -5 & 37
\end{bmatrix}
\approx
\begin{bmatrix}
\phantom{-}0.68518  &-0.09259 & -0.09259\\[0.3em]
-0.09259 & \phantom{-}0.68518 & -0.09259 \\[0.3em]
-0.09259 & -0.09259 & \phantom{-}0.68518
\end{bmatrix}
.
\end{align*}

Similar computations yield that the asymptotic covariance matrix of $\sqrt{n}\left\{\hat{\bm\a}_{n,norm1}-\bm\a_0\right\}$ , where $\hat{\bm\a}_{n,norm1}$ is defined by (\ref{LS_norm1a}) in Remark \ref{remark_spher_sym}, is given by the covariance matrix of
\begin{align*}
\frac13\left(\bm I-{\bm\a_0}\bma_0^T\right)\bm X\left\{Y-\bma_0^T\bm X\right).
\end{align*}
which is given by:
\begin{align}
\label{asymp_covar-spher-sym}
\frac1{27}\begin{bmatrix}
22 & -11 &-11
\\[0.3em]
-11  &22  &-11  \\[0.3em]
-11  &-11  &22
\end{bmatrix}
\approx
\begin{bmatrix}
\phantom{-}0.814815 & -0.407407 &-0.407407
\\[0.3em]
-0.407407  &\phantom{-}0.814815  &-0.407407  \\[0.3em]
-0.407407  &-0.407407  &\phantom{-}0.814815
\end{bmatrix}
.
\end{align}
Note that the asymptotic variances are larger than those of $\hat{\bm\a}_{n,\bmS_n norm1}$.

\section{The least squares estimator (LSE) of the link function $\psi$}
\label{supA:LSE}
In this section we first prove Proposition \ref{prop:maximizer-population}. We next show in Lemma \ref{lemma: maxorderLSE} that the LSE  $\hat \psi_{n\bma}$ is of order $O_p(\log n)$ uniformly in $\B(\bma_0,\d_0)$. This result is used in the proof of Proposition \ref{prop:L_2-psi-psi_n-alpha}, given at the end of this section.

\begin{proof}[Proof of Proposition \ref{prop:maximizer-population}]
	Note that we can write
	\begin{align*}
	L_{\bma}(\psi) &=  \E\Big[ \Big( \psi_0(\bma_0^T\bm X)-\psi(\bma^T\bm X)  \Big)^2 \Big]. 
	\end{align*}
	Thus, 
	\begin{align*}
	\E\Big[ \Big(\psi_0(\bma_0^T\bm X)- \E(\psi_0(\bma_0^T\bm X)  \   | \ \bma^T\bm X) \Big)^2 \Big]  = \min_{\psi \in \mathcal{M}'} 	L_{\bma}(\psi) ,
	\end{align*}
	with $\mathcal{M}'$ the set of all bounded Borel-measurable function defined on $\mathcal{I}_{\bma}$, since the conditional expectation  $\E(\psi_0(\bma_0^T\bm X)  \   | \ \bma^T\bm X)$ is the least squares projection of $\psi_0(\bma_0^T\bm X)$ on the $\sigma$-algebra $\sigma\{\bma^T\bmX\}$ for the underlying probability measure.

	 Therefore, if the minimizing function $u \mapsto \E(\psi_0(\bma_0^T\bm X)  \   | \ \bma^T\bm X =u) $ is monotone increasing on $\mathcal{I}_{\bma}$, then this implies that it necessarily minimizes $L_{\bma}$ over $\mathcal{M}$. Furthermore, such a minimizer is unique by strict convexity of $L_{\bma}$.
\end{proof}  
\begin{lemma}\label{lemma: maxorderLSE}
	$$\max_{\bma \in \mathcal{B}(\bma_0, \delta_0)} \sup_{\bm x \in \mathcal{X}} \Big \vert \hat\psi_{n \bma}(\bma^T\bm x)\Big \vert = O_p(\log n ).$$
\end{lemma}

\begin{proof}
	The proof of this lemma is similar to that of Lemma 4.4 of  \cite{balabdaoui2016}.   For a fixed $\bma$  it follows from the min-max formula of an isotonic regression that we have for all $\bm x \in \mathcal{X}$
	\begin{eqnarray*}
		\min_{ 1\le k \le  n}   \frac{\sum_{i=1}^k Y^{\bma}_i}{k} \le \hat\psi_{n \bma}(\bma^T\bm  x)   \le   \max_{ 1\le k \le  n}   \frac{\sum_{i=k}^n Y^{\bma}_i}{n-k+1}.
	\end{eqnarray*}
	Hence,
	\begin{eqnarray*}
		\min_{1 \le i \le n}  Y_i  \le \hat\psi_{n \bma}(\bma^T\bm  x)   \le \max_{1 \le i \le n}  Y_i
	\end{eqnarray*}
	and this in turn implies that
	\begin{eqnarray*}
		\max_{\bma \in \mathcal{B}(\bma_0, \delta_0)} \sup_{\bm x \in \mathcal{X}} \Big \vert \hat\psi_{n \bma}(\bma^T\bm x)\Big \vert \le \max_{1 \le i \le n}  \vert Y_i  \vert.
	\end{eqnarray*}
	Using similar arguments as in \cite{balabdaoui2016},  we use Assumption A7 to show that $ \max_{1 \le i \le n}  \vert Y_i  \vert = O_p(\log n)$, which completes the proof.
\end{proof}

\medskip

\begin{proof}[Proof of Proposition \ref{prop:L_2-psi-psi_n-alpha}]
	By the definition of the LSE of the unknown monotone link function, $\hat \psi_{n \bma}$  maximizes the map $\psi \mapsto \mathbb{M}_n$ over $\M$ where,
	\begin{align}
	\label{def:M_n}
	\mathbb{M}_n(\psi, \bma) =  \int_{\mathcal{X} \times \R}  \Big(  2 y \psi(\bma^T\bm x)  - \psi^2(\bma^T\bm x)  \Big)  d \P_n(\bm x, y).
	\end{align}
	Moreover, $\psi_{\bma}$ maximizes the map  $\psi \mapsto \mathbb{M}$ over $\M$, where
	\begin{align}
	\label{def:M}
	\mathbb{M}(\psi, \bma) =  \int_{\mathcal{X} \times \R}  \Big(  2 y \psi(\bma^T\bm x)  - \psi^2(\bma^T\bm x)  \Big)  d P_0(\bm x, y).
	\end{align}
	Define the function $f_{\psi,\bma}$ by,
	\begin{align*}
	f_{\psi,\bma}(\bm x, y) = 2 y \psi(\bma^T\bm x)  - \psi^2(\bma^T\bm x)
	\end{align*}
	Note that by definition of the LSE as the maximizer of (\ref{def:M_n}), we have
	\begin{align*}
	\int_{\mathcal{X} \times \R}  \Big(  f_{\hat\psi_{n \bma},\bma}(\bm x, y) - f_{\psi_{\bma},\bma}(\bm x, y)  \Big)  d\P_n(\bm x, y) \ge 0.
	\end{align*}
	Moreover for all $\bma \in \B(\bma_0, \delta_0)  $ and $\psi \in \M$, we have,
	\begin{align*}
	\int_{\mathcal{X} \times \R}  \Big(  f_{\psi,\bma}(\bm x, y) - f_{\psi_{\bma},\bma}(\bm x, y)  \Big)  dP_0(\bm x, y) = -d^2_{\bma}(\psi, \psi_{\bm\a}),
	\end{align*}
	where,  for any $\bma \in \B(\bma_0,\delta_0)$ and for any two elements  $\psi_1$ and $\psi_2$ in $\mathcal{M}$, we define the squared distance 
	\begin{align*}
	d^2_{\bma}(\psi_1, \psi_2)  = \int_{\mathcal{X}}  \Big( \psi_2(\bma^T\bm x) - \psi_1(\bma^T\bm x)  \Big)^2 g(\bm x) d\bm x.
	\end{align*}
	This can be seen as follows:
	\begin{align*}
	& \int_{\mathcal{X} \times \R}  \Big(  f_{\psi,\bma}(\bm x, y) - f_{\psi_{\bma},\bma}(\bm x, y)  \Big)  dP_0(\bm x, y) \\
	&= \int_{\mathcal{X} \times \R}  \Big(  2 \psi_{\bma}(\bma^T\bm x) (\psi(\bma^T\bm x) -\psi_{\bma}(\bma^T\bm x))   - \psi^2(\bma^T\bm x) + \psi_{\bma}^2(\bma^T\bm x)  \Big)  d P_0(\bm x, y)\\
	&= -\int_{\mathcal{X}}  \Big( \psi(\bma^T\bm x) - \psi_{\bma}(\bma^T\bm x)  \Big)^2 g(\bm x) d\bm x
	= -d^2_{\bma}(\psi, \psi_{\bm\a}),
	\end{align*}
	where we use that $\E\{Y|\bma^T\bm X = u\} = \psi_{\bma}(u)$.
	This implies that, for all $\bma \in \B(\bma_0, \delta_0)  $ and $\psi \in \M$, we have,
	\begin{align*}
	\int_{\mathcal{X} \times \R}  \Big(  f_{\psi,\bma}(\bm x, y) - f_{\psi_{\bma},\bma}(\bm x, y)  \Big)  d(\P_n-P_0)(\bm x, y) \ge d^2_{\bma}(\psi, \psi_{\bm\a}).
	\end{align*}
	We write,
	\begin{align*}
	&\P\left\{ \sup_{\bma\in \B(\bma_0,\delta_0) }   d_{\bma}(\hat\psi_{n \bma}, \psi_{\bma})\ge\ee\right\}\\
	&\le\P\left\{\sup_{\bma\in \B(\bma_0,\delta_0),d_{\bma}(\hat\psi_{n \bma}, \psi_{\bma})\ge\ee}\left\{
	\int_{\mathcal{X} \times \R}  \Big(  f_{\hat\psi_{n \bma},\bma}(\bm x, y) - f_{\psi_{\bma},\bma}(\bm x, y)  \Big)  d(\P_n-P_0)(\bm x, y) - d^2_{\bma}(\hat\psi_{n \bma}, \psi_{\bma})\right\}\ge0,\right.\\
	&\left. \hspace{9.5cm}\sup_{\bma\in \B(\bma_0,\delta_0) }   d_{\bma}(\hat\psi_{n \bma}, \psi_{\bma})\ge\ee\right\}\\
	&\le\P\left\{\sup_{\bma\in \B(\bma_0,\delta_0), \psi \in \M_{RK}, d_{\bma}(\psi, \psi_{\bma})\ge\ee}\left\{
	\int_{\mathcal{X} \times \R}  \Big(  f_{\psi,\bma}(\bm x, y) - f_{\psi_{\bma},\bma}(\bm x, y)  \Big)  d(\P_n-P_0)(\bm x, y) - d^2_{\bma}(\psi_{\bma}, \psi)\right\}\ge0,\right.\\
	&\left. \hspace{1.50cm} \max_{\bma \in \mathcal{B}(\bma_0, \delta_0)} \sup_{\bm x \in \mathcal{X}} \Big \vert \hat\psi_{n \bma}(\bma^T\bm x)\Big \vert  \le K\right\} 
	+ \P\left\{\max_{\bma \in \mathcal{B}(\bma_0, \delta_0)}\sup_{\bm x \in \mathcal{X}} \Big \vert \hat\psi_{n \bma}(\bma^T\bm x)\Big \vert  > K \right\}.\\
	\end{align*}
	Fix $\nu > 0$. Since 
	$$\max_{\bma \in \mathcal{B}(\bma_0, \delta_0)} \sup_{\bm x \in \mathcal{X}} \Big \vert \hat\psi_{n \bma}(\bma^T\bm x)\Big \vert = O_p(\log n ), $$
	by Lemma \ref{lemma: maxorderLSE},  
	we can find $K_1 > 0$ large enough such that  
	$$ \P\left\{\max_{\bma \in \mathcal{B}(\bma_0, \delta_0)}\sup_{\bm x \in \mathcal{X}} \Big \vert \hat\psi_{n \bma}(\bma^T\bm x)\Big \vert  > K_1 \log n \right\} < \nu/2.$$
	Define 
	\begin{eqnarray}\label{MRK}
	\mathcal{M}_{RK}  =  \Big \{ \psi  \  \text{monotone non-decreasing on $[-R,R]$ and bounded by $K$} \Big\},
	\end{eqnarray}
	and consider the related class 
	\begin{align}
	\label{FRK}
	{\cal F}_{RK} &= \bigg \{f(\bm x, y)  =  2y \Big(\psi (\bma^T\bm x)  -  \psi_{\bma} (\bma^T\bm x)  \Big)  -  \psi(\bma^T\bm x)^2  +   \psi_{\bma}(\bma^T\bm x)^2 , \nonumber\\
	& \hspace{4cm}  (\bma, \psi) \in \mathcal{B}(\bma_0, \delta_0)  \times  \M_{RK}  \ \text{and} \  (\bm x, y) \in \mcX \times \R \bigg\},
	\end{align}
	and for some $v > 0$
	\begin{align}
	\label{FRKv}
	{\cal F}_{RKv}: & =    \bigg \{f \in  {\cal F}_{RK}:  d_{\bma}(\psi, \psi_{\bma})  \le v  \  \text{for all $\bma \in \mathcal{B}(\bma_0, \delta_0)$} \bigg\}.  
	\end{align}
	Note now that the class ${\cal F}_{RKv}$ is included in the class $\mathcal{H}_{RC\delta }$ defined in Lemma  \ref{lemma: BernEntropyGeneral} of Appendix \ref{supD:entropy}  with 
	$C=2 K^2$ and $\delta = 2K v$. This holds true provided that $K_0 \le K$, and $K \ge 1$ which we can assume for $n$ large enough since $K$ will be chosen to be of order $\log n$.  To see the claimed inclusion, it is enough to show that if $m$ is a nondecreasing function $[-R, R]$ then $m^2$ can be written as the difference of two monotone functions. This is true because $m^2 =  m^2 \mathbb{I}_{m \ge 0}   -  (-m^2) \mathbb{I}_{m < 0}$,  and  $m^2$ and $-m^2$ are nondecreasing on the subsets $\{m \ge 0\}$  and $\{m < 0\}$ respectively. When restricting attention to the event that $\hat{\psi}_{n \bma}$ is bounded by $K$  for $n$ large enough, we can consider only monotone functions $\psi \in \mathcal{M}_{RK}$. Using the expression of $\psi_{\bma}$ the latter is bounded by $K_0 \le K$.  On the other hand, for any function $f  \in {\cal F}_{RKv}$, there exist nondecreasing monotone functions $f_1$ and $f_2$ such that $\psi^2 -   \psi_{\bma}^2  = f_2 - f_1 $,  such that $\Vert f_1\Vert_\infty, \Vert f_2 \Vert_\infty \le  K^2  + K^2_0  \le 2 K^2$. Using that $K \ge 1$ implies that $  \Vert 2 \psi \Vert_\infty , \Vert2  \psi_{\bma}  \Vert_\infty \le  2  K \le 2 K^2$.  To finish, note that for any $\bma$ we have that $\int_{\mathcal{X}}  \left(\psi (\bma^T\bm x)  -  \psi_{\bma} (\bma^T\bm x)\right)^2 dG(\bm x) \le v^2  $ we also have that 
	\begin{eqnarray*}
		\int_{\mathcal{X}}  \left(\psi^2(\bma^T\bm x)  -  \psi^2_{\bma} (\bma^T\bm x)\right)^2 dG(\bm x)  & \le   & (2K)^2 \int_{\mathcal{X}}  \left(\psi (\bma^T\bm x)  -  \psi_{\bma} (\bma^T\bm x)\right)^2 dG(\bm x) \\
		\le 4 K^2  v^2. 
	\end{eqnarray*}
	The calculation above implies that we can take $\delta = 2K v $.   Using the result of Lemma \ref{lemma: BernEntropyGeneral} in Appendix \ref{supD:entropy}, it follows that the related class  $\widetilde{\mathcal{F}}_{RKv}  = \tilde{D}^{-1} {\cal F}_{RKv}$   with $\tilde{D}  =  16 M_0 C = 32 M_0 K^2 $  and a  given  $v > 0$  satisfies 
	\begin{eqnarray*}
		H_B\big(\epsilon, \widetilde{\mathcal{F}}_{RKv},  \Vert \cdot \Vert_{B, P_0}  \big)  &\le  & H_B\big(\epsilon, \widetilde{\mathcal{H}}_{RKv},  \Vert \cdot \Vert_{B, P_0}  \big) \\
		& \le   & \frac{A}{\epsilon} 
	\end{eqnarray*}
	for some constant $A > 0$ (depending only on $d$ and the other parameters of the problem), and that for all $\tilde{f} \in \widetilde{F}_{RKv}  $ we have  $\Vert \tilde{f} \Vert_{B, P_0} \le  \tilde{D}^{-1} \delta   = (32 M_0  K^2)^{-1}  2 K v  = (16 M_0)^{-1}  K^{-1} v   \equiv A_0 K^{-1} v$.   It follows from Lemma 3.4.3 of \cite{vdvwe:96} that 
	\begin{eqnarray*}
		\E\Big[\Vert \mathbb{G}_n \Vert_{\widetilde{F}_{RKv}} \Big] &\lesssim   & J_n(A_0 K^{-1} v),  \left( 1 + K^2 \frac{J_n(A_0 K^{-1} v)}{\sqrt{n} A^2_0 v^2 }  \right), \  \  \text{ where $J_n$ is defined in (\ref{Jn})} \\
		& \le &  A_0 K^{-1} v   +  2  A^{1/2}_0  K^{-1/2} v^{1/2}  A^{1/2} , \ \  \text{using the inequality in (\ref{IneqJn})}  \\
		& \le & B_0 K^{-1/2}  (v + v^{1/2})
	\end{eqnarray*}
	for some constant $B_0 > 0$, where we used the fact that $K^{-1/2}\ge K^{-1}$.  Therefore, 
	\begin{eqnarray*}
		\E \Big[\Vert \mathbb{G}_n \Vert_{\widetilde{F}_{RKv}} \Big] &\lesssim   &  B_0 K^{-1/2}  (v + v^{1/2})   \left( 1 + K^2 \frac{B_0 K^{-1/2}  (v + v^{1/2})}{\sqrt{n} A^2_0 v^2 }  \right) \\ 
		& \le &  C_0 K^{-1/2}  (v + v^{1/2})   \left(  1+  C_0 K^{3/2}  \frac{1 + v^{1/2}}{\sqrt n v^{3/2}}  \right).
	\end{eqnarray*}
	Using the definition of the class $\widetilde{F}_{RKv}$, the preceding display implies that 
	\begin{eqnarray}\label{ControlExp}
	\E \Big[\Vert \mathbb{G}_n \Vert_{\mathcal{F}_{RKv}} \Big] &\lesssim   & C_0  K^{3/2}  (v + v^{1/2})   \left(  1+  C_0 K^{3/2}  \frac{1 + v^{1/2}}{\sqrt{n} v^{3/2}}  \right).
	\end{eqnarray}
	Now with $K =K_1\log n,$ we have that
	\begin{align} 
	\label{proof:L2-a}
	&\P\left\{ \sup_{\bma\in \B(\bma_0,\delta_0) }   d_{\bma}(\hat\psi_{n \bma}, \psi_{\bma})\ge\ee\right\}\nonumber\\
	&\le\P\left\{\sup_{\substack{\bma\in \B(\bma_0,\delta_0), \psi \in \M_{RK} ,\\d_{\bma}(\psi, \psi_{\bma})\ge\ee}}\left\{
	\int_{\mathcal{X} \times \R}  \Big(  f_{\psi,\bma}(\bm x, y) - f_{\psi_{\bma},\bma}(\bm x, y)  \Big)  d(\P_n-P_0)(\bm x, y) - d^2_{\bma}(\psi, \psi_{\bma})\right\}\ge0,\right.\nonumber\\
	&\left. \hspace{8.50cm} \max_{\bma \in \mathcal{B}(\bma_0, \delta_0)} \sup_{\bm x \in \mathcal{X}} \Big \vert \hat\psi_{n \bma}(\bma^T\bm x)\Big \vert  \le K\right\} 
	+ \nu/2 
	\end{align}
	\begin{align*}
	&\le \sum_{s=0}^\infty\P\left\{\sup_{\substack{\bma\in \B(\bma_0,\delta_0), \psi \in \M_{RK},\\ 2^{s}\ee\le d_{\bma}(\psi,\psi_{\bma})\le 2^{s+1}\ee}}
	\sqrt n \int_{\mathcal{X} \times \R}  \Big(  f_{\psi,\bma}(\bm x, y) - f_{\psi_{\bma},\bma}(\bm x, y)  \Big)  d(\P_n-P_0)(\bm x, y)\ge \sqrt n 2^{2s}\ee^2,\right\}
	+ \nu/2 \nonumber
	\\
	&\le \sum_{s=0}^\infty\P\left\{\sup_{h \in {\cal F}_{RK 2^{s+1}\ee} }
	\sqrt n \int_{\mathcal{X} \times \R}   h(\bm x, y)   d(\P_n-P_0)(\bm x, y)\ge \sqrt n 2^{2s}\ee^2\right\}
	+ \nu/2.\nonumber
	\end{align*}
	We now show that there exists a constant $C > 0$ such that with $\ee =M(\log n)n^{-1/3}$,
	\begin{align}
	\label{expectation-bound}
	&\E\left\{ \sup_{h \in {\cal F}_{RK2^{s+1}\ee} }
	\sqrt n \left|\int_{\mathcal{X} \times \R}   h(\bm x, y)   d(\P_n-P_0)(\bm x, y)\right|\right\}  \le C M^{1/2} (\log n)^{2}   n^{-1/6}   2^{(s+1)/2}\nonumber\\
	\end{align}	
	An application of Markov's inequality, together with (\ref{proof:L2-a}), then yields, with $\ee =M(\log n)n^{-1/3}$
	\begin{eqnarray*}
		\P\left\{ \sup_{\bma\in \B(\bma_0,\delta_0) }   d_{\bma}(\hat\psi_{n \bma}, \psi_{\bma})\ge\ee\right\}&\le  & \sum_{s=0}^\infty \frac{C M^{1/2} (\log n)^{2}   n^{-1/6}   2^{(s+1)/2}}{\sqrt n  2^{2s}\ee^2} +\nu/2 \\
		& =   & \sum_{s=0}^\infty  \frac{C  (\log n)^{2} n^{-1/6}   2^{(s+1)/2}}{\sqrt n 2^{2s} M^{3/2} (\log n)^{2} n^{-2/3}} +\nu/2 =  \frac{2^{1/2}C} {M^{3/2} } \sum_{s=0}^\infty  \frac{1}{2^{3s/2}} +\nu/2  \le \nu,
	\end{eqnarray*}
	for $M$ sufficiently large.  The result of Proposition \ref{prop:L_2-psi-psi_n-alpha} hence follows by showing (\ref{expectation-bound}). Using the obtained bound in (\ref{ControlExp}) with $v=2^{s+1} \ee$ and using that $ 2^{s+1} \ge 1, s \ge 0$ and $\ee \le 1$ for $n$ large enough we get some some constant $D_0 > 0$
	\begin{eqnarray*}
		&&	\E\left\{ \sup_{h \in {\cal F}_{RK2^{s+1}\ee} }
		\sqrt n \left|\int_{\mathcal{X} \times \R}   h(\bm x, y)   d(\P_n-P_0)(\bm x, y)\right|\right\} \\
		&& \lesssim (\log n)^{3/2} M^{1/2} 2^{(s+1)/2} (\log n)^{1/2} n^{-1/6}   \left(  1+  D_0  (\log n)^{3/2} \frac{1+ M^{1/2} 2^{(s+1)/2} (\log n)^{1/2} n^{-1/6}}{\sqrt n M^{3/2} 2^{3(s+1)/2} (\log n)^{3/2} n^{-1/2}}  \right)\\ 
		&& =  (\log n)^{2}  n^{-1/6}   M^{1/2} 2^{(s+1)/2} \left(  1+  D'_0  \frac{1+ M^{1/2} 2^{(s+1)/2} (\log n)^{1/2} n^{-1/6}}{ 2^{3(s+1)/2}  }  \right), \ \ \text{with $D'_0 = D_0 M^{-3/2}$} \\
		&& \le 2  (\log n)^{2}   n^{-1/6} M^{1/2}   2^{(s+1)/2}, \ \ \text{for $ s \ge 0 $ and $n$ large enough}.
	\end{eqnarray*}	
	This proves the desired result:
	$$\sup_{\bma\in \B(\bma_0,\delta_0) }   d^2_{\bma}(\hat\psi_{n \bma}, \psi_{\bma}) = \sup_{\bma \in \B(\bma_0,\delta_0) }\int\left\{\hat \psi_{n\bma}(\bma^T\bm x)-\psi_{\bm\a}(\bma^T\bm x)\right\}^2\,dG(\bm x)=O_p\left((\log n) ^2n^{-2/3}\right).$$

\end{proof}

\section{Asymptotic behavior of the simple score estimator}
\label{supB:sse}
In this section we prove Theorem $\ref{theorem:asymptotics}$ given in Section \ref{section:Asymptotics}. The proof is decomposed into three parts: In Section \ref{subsec:Appendix-existence1} we first prove the existence of a crossing of zero of $\f_n$ defined in (\ref{def:phi_n}). The proof of consistency and asymptotic normality of $\hat \bma_n$ are given in Section \ref{subsec:Appendix-consistency1} and Section \ref{subsec:Appendix-normal1}. 
\subsection{Proof of existence of a crossing of zero}
\label{subsec:Appendix-existence1}
Let $\f$ be the population version of $\f_n$ defined by
\begin{align}
\label{def:phi-population}
\f(\bmb) :=\int \left(\bm J_{\mbS}(\bmb)\right)^T\bm x\left\{y -\psi_{\bm\a}\left(\mbS(\bmb)^T\bm x\right)\right\}\,dP_0(\bm x,y),
\end{align}
where $\psi_{\bma}$ is defined by
\begin{align*}
\psi_{\bm\a}(u)  := \E\left[\psi_0\left(\bma_0^T\bm X\right) |  \bma^T\bm X = u \right] \equiv \E\left[\psi_0\left(\mbS(\bmb_0)^T\bm X\right) |  \mbS(\bmb)^T\bm X = u \right].
\end{align*}
We have the following result:
\begin{proposition}
	\label{prop:score-connection}
	$$ \f_n(\bmb)  =\f(\bmb)  +o_p(1),$$
	uniformly in $\bmb \in \mathcal{C} := \left\{\bmb \in \R^{d-1}: \mbS(\bmb)\in \B(\bma_0,\d_0)\right\}$.
\end{proposition}

\begin{proof}
	For any $\bmb \in \mathcal{C}$, we write,
	\begin{align}
	\label{decomposition:phi_n}
	\f_n(\bmb)&= \int \left(\bm J_{\mbS}(\bmb)\right)^T\bm x\left\{y -\psi_{\bm\a}\left(\mbS(\bmb)^T\bm x\right)\right\}\,d\P_n(\bm x,y) \nonumber\\
	&\qquad + \int \left(\bm J_{\mbS}(\bmb)\right)^T\bm x\left\{\psi_{\bm\a}\left(\mbS(\bmb)^T\bm x\right) -\hat \psi_{n\bm\a}\left(\mbS(\bmb)^T\bm x\right)\right\}\,d\P_n(\bm x,y) \nonumber
	\\
	&=\f(\bmb) + \int \left(\bm J_{\mbS}(\bmb)\right)^T\bm x\left\{y -\psi_{\bm\a}\left(\mbS(\bmb)^T\bm x\right)\right\}\,d(\P_n-P_0)(\bm x,y) \nonumber\\
	&\qquad + \int \left(\bm J_{\mbS}(\bmb)\right)^T\bm x\left\{\psi_{\bm\a}\left(\mbS(\bmb)^T\bm x\right) -\hat \psi_{n\bm\a}\left(\mbS(\bmb)^T\bm x\right)\right\}\,d(\P_n-P_0)(\bm x,y) \nonumber
	\\
	&\qquad + \int \left(\bm J_{\mbS}(\bmb)\right)^T\bm x\left\{\psi_{\bm\a}\left(\mbS(\bmb)^T\bm x\right) -\hat \psi_{n\bm\a}\left(\mbS(\bmb)^T\bm x\right)\right\}\,dP_0(\bm x,y) \nonumber
	\\
	&= 	\f(\bmb)+ I + II + III.
	\end{align}
	To find the rate of convergence of the term $I$ in (\ref{decomposition:phi_n})  we will use Lemma \ref{lemma: EntropyNew} in Appendix \ref{supD:entropy}.  Note first that for $ 1 \le i \le d-1$ the $i$-th component of the vector $\left(\bm J_{\mbS}(\bmb)\right)^T\bm x $ can be written as $s(\bmb)_{i1} x_1 + \ldots +  s(\bmb)_{id} x_d$, where by Assumption A8 the functions $s_{ij}$   are assumed to be uniformly bounded with partial derivatives that are also uniformly bounded on the bounded convex set $\mathcal{C}$ to which $\bmb $ belongs. If  $B_1$  is  the same constant found in Lemma \ref{lemma: EntropyNew} in Appendix \ref{supD:entropy} then the $\epsilon$-bracketing entropy is bounded above by $B_1 K_0 /\epsilon$. Applying  Lemma 3.4.3 of \cite{vdvwe:96}, Markov's inequality and Lemma \ref{lemma: EntropyNew} in Appendix \ref{supD:entropy} to each of the empirical processes corresponding to the term  $s(\bmb)_{ij} x_j$ for $1 \le j \le d$ yields  (with $D=8 M R K_0$, and $M$ is a constant bounding the sum of $s(\bmb)_{ij}$ and their partial derivatives) for $A > 0$
	\begin{align*}
	&P\Big( \vert I \vert \ge   An^{-1/2} \Big) \le \frac{ D}{A}  J_n(B_2) \left (1  +  \frac{J_n(B_2)}{\sqrt n B^2_2} \right), \  \text{where $J_n$ is defined in (\ref{Jn}) and $B_2$ is the same } \\
	& \hspace{4.5cm} \text{ constant of Lemma \ref{lemma: EntropyNew} in Appendix \ref{supD:entropy} } \\
	&\qquad \le  \frac{ D}{  A}  B_3 \left (1  +  \frac{B_3}{\sqrt n B^2_2} \right) , \  \ \text{using the inequality in (\ref{IneqJn}) with $B_3 = B_2 + 2  B^{1/2}_1  K^{1/2}_0 B^{1/2}_2$ } \\
	&\qquad \asymp  \frac{1}{A}.  
	\end{align*} 	
	This implies that $ I = O_p(n^{-1/2})$. For the last term $III$ in (\ref{decomposition:phi_n}) we get by an application of the Cauchy-Schwarz inequality and by Proposition \ref{prop:L_2-psi-psi_n-alpha} that this term is $O_p(n^{-1/3}\log n)$, i.e.
	\begin{align*}
	III&\le  \left(\int \left\|\left(\bm J_{\mbS}(\bmb)\right)^T\bm x\right\|_2^2 dG(\bm x)\right)^{1/2}\left(\int \left\{\psi_{\bm\a}\left(\mbS(\bmb)^T\bm x\right) -\hat \psi_{n\bm\a}\left(\mbS(\bmb)^T\bm x\right)\right\}^2\,dG(\bm x)\right)^{1/2}\\
	&= O_p(n^{-1/3}\log n),
	\end{align*}
	where we also use that $\left(\bm J_{\mbS}(\bmb)\right)^T\bm x$ is bounded in $L_2$ norm, a straightforward implication of Assumption A1 (boundedness of $\mcX$ and Assumption A8  (uniform boundedness of the components of the matrix $\bm J_{\mbS} $). 
	The result now follows by showing that the term $II$ is $o_p(1)$.   Consider the class of functions 
	\begin{eqnarray*}
		{\cal G}_{jRKv} &= & \bigg \{g(\bm x, y)  =  s(\bmb) x_j \left\{\psi_{\bma}\bigl(\bma^T\bm x\bigr)-\psi(\bma^T\bm x\bigr)\right\} , \nonumber\\
		&& \hspace{1.8cm} \text{such that} \ (\bma,\bmb, \psi) \in \cS_{d-1}  \times \mathcal{C} \times   \M_{RK}  \ \text{and} \  (\bm x, y) \in \mcX \times \in \R, \\
		& & \hspace{1.8cm} \text{and}  \ \sup_{\bma \in \mathcal{B}(\bma_0, \delta_0)} d_{\bma}(\psi_{\bma}, \psi)  \le v   \bigg\}
	\end{eqnarray*}
	with $s$ a function satisfying (\ref{DiffAssump}).  Then, ${\cal G}_{jRKv}   \subset \mathcal{Q}_{jRK}  - \mathcal{Q}_{jRK}$, where  $\mathcal{Q}_{jRC}$ is the same class defined in (\ref{QjRC}).  Here, we choose $K$ large enough such that $K \ge K_0$. If follows from (\ref{EntropyP0}) in the proof of Lemma \ref{lemma: EntropyNew} in Appendix \ref{supD:entropy}, that  (at the cost of increasing the constant $L$ in (\ref{EntropyP0}))
	\begin{eqnarray*}
		H_B\Big(\epsilon,  \tilde{\cal G}_{jRKv}, \Vert \cdot \Vert_{P_0}\Big)  \le \frac{L K}{\epsilon},
	\end{eqnarray*}
	where $\tilde{\cal G}_{jRKv} = (16M_0K)^{-1} {\cal G}_{jRKv}$. Also, we have for all $g \in {\cal G}_{jRKv}$
	\begin{eqnarray*}
		\Vert g \Vert_{P_0}  \le  M R v.
	\end{eqnarray*}
	Fix $\nu > 0$ and let $s_{ij}$ be the $i \times j$ entry of $ \bm J_{\mbS}(\bmb)$  for $ 1\le i \le d-1$ and $1\le j  \le d$. Also, let
	\begin{eqnarray*}
		II_{ij}=  \int_{\mcX \times \R}  s_{ij}(\bmb) x_j \left\{\psi_{\bm\a}\left(\mbS(\bmb)^T\bm x\right) -\hat \psi_{n\bm\a}\left(\mbS(\bmb)^T\bm x\right)\right\}\,d(\P_n-P_0)(\bm x,y).
	\end{eqnarray*}
	Using Lemma \ref{lemma: maxorderLSE} and Proposition  \ref{prop:L_2-psi-psi_n-alpha} there exists some constant $K_1 > 0$ large enough (and independent of $n$ ) such that with $K = K_1 \log n$  and $v = K_1 \log n \ n^{-1/3}$  we have that for $A > 0$
	\begin{align*}
	&P\left( \vert II_{ij}  \vert \ge A n^{-1/2} \right)  \\
	\qquad &=  P\left( \vert II_{ij}  \vert \ge A n^{-1/2} , \sup_{\bma \in \mathcal{B}(\bma_0, \delta_0) }\sup_{\bm x \in \mcX} \big \vert \hat{\psi}_{n\bma} (\bma^T \bm x)\big \vert \le K, \sup_{\bma \in \mathcal{B}(\bma_0, \delta_0) }d_{\bma}(\psi_{\bma}, \psi)  \le v  \right)  + \nu/2 \\
	\qquad& \lesssim    \frac{K}{A}  J_n (MRv)\left (1  +  \frac{J_n(MRv)}{\sqrt n  M^2 R^2 v^2} \right)     + \nu/2, \ \ \text{where $J_n$ is defined in (\ref{Jn})}  \\
	& \le  \frac{K}{A}   \left(MRv + 2 (MRL)^{1/2}  K^{1/2}  v^{1/2}\right) \left(1 +  \frac{MRv + 2 (MRL)^{1/2}  K^{1/2}  v^{1/2}}{\sqrt n  M^2 R^2 v^2}   \right)  + \nu/2, \  \ \text{using the inequality in (\ref{IneqJn})} \\
	& \le \frac{\tilde{M}}{A} (\log n)^2  \ n^{-1/6}  \left(1 +  \frac{1}{\log n M^2 R^2 }  \right)  + \nu/2, \  \ \text{for some constant $\tilde{M} > 0$}  \\
	& \le \nu
	\end{align*}
	for $n$ large enough. We conclude that $II_{ij}  = o_p(n^{-1/2})$ which in turn implies that $II = o_p(n^{-1/2})$.
	
\end{proof}
\medskip
\begin{proof}[Proof of Theorem \ref{theorem:asymptotics} (Existence)]
	Using Proposition \ref{prop:score-connection} we get, analogously to the development in \cite{kim_piet:18}, the relation
	\begin{align}
	\label{expansion_phi_n}
	\f_n(\bm\a)=\f'(\bm\b_0)(\bmb-\bmb_0)+R_n(\bmb),
	\end{align}
	where $R_n(\bm\a)=o_p(1)+o(\bmb-\bmb_0)$, uniformly in $\bmb \in \mathcal{C}$  and where $\f'$ is the derivative of $\f$ defined in (\ref{def:phi-population}).
	Using Lemma \ref{lemma:derivative_psi_a} in Appendix \ref{supE:AuxiliaryResults}, we get that the derivative of $\f$ at $\bmb_0$ is given by the matrix 
	\begin{align*}
	\f'(\bmb_0)=\left(\bm J_{\mbS}(\bmb_0)\right)^T\E\left[\psi_0'\left(\mbS(\bmb_0)^T\bm x\right)\,\text{Cov}(\bm X|\mbS(\bmb_0)^T\bm X)\right]\bm J_{\mbS}(\bmb_0) = \left(\bm J_{\mbS}(\bmb_0)\right)^T\bm A\bm J_{\mbS}(\bmb_0) = \bm B,
	\end{align*}
	where $\bm A$ and $\bm B$ are defined in (\ref{def_A}) and (\ref{def_B}) respectively.
	We now define, for $h>0$, the functions
	\begin{align*}
	\tilde R_{n,h}(\bmb)=\frac1{h^{d-1}}\int K_h(u_1-\b_{01})\dots K_h(u_{d-1}-\b_{d-1})\,R_n(u_1,\dots,u_{d-1})\,du_1\dots\,du_{d-1},
	\end{align*}	
	where 
	\begin{align*}
	K_h(\bm x)=h^{-1}K(x/h),\qquad x\in\R,
	\end{align*}
	letting $K$ be one of the usual smooth kernels with support $[-1,1]$.
	
	Furthermore, we define:
	\begin{align*}
	\tilde \f_{n,h}(\bmb)= \f'(\bmb_0)(\bmb-\bmb_0) +\tilde R_{nh}(\bmb).
	\end{align*}
	Clearly:
	\begin{align*}
	\lim_{h\downarrow0}\tilde \f_{n,h}(\bmb)=\f_n(\bmb)\qquad\text{and}\qquad \lim_{h\downarrow0}\tilde R_{nh}(\bmb)=R_n(\bmb),
	\end{align*}
	for each continuity point $\bmb$ of $\f_n$.
	
	We now reparametrize, defining
	\begin{align*}
	\bm\g=\f'(\bmb_0)\bmb,\qquad \bm \g_0=\f'(\bmb_0)\bmb_0.
	\end{align*}
	This gives:
	\begin{align*}
	\f'(\bmb_0)(\bmb-\bmb_0) +\tilde R_{nh}(\bmb)=\bm\g -\bm \g_0+\tilde R_{nh}\left(\bm B^{-1}\bm  \g\right),
	\end{align*}
	By (\ref{expansion_phi_n}), the mapping
	\begin{align*}
	\bm\g\mapsto \bm\g_0-R_n\left(\bm B^{-1}\bm\g\right),
	\end{align*}
	maps, for each $\eta>0$, the ball $B_{\eta}(\bm\b_0)=\{\bmb:\|\bm\b-\bm\b_0\}\le\eta\}$ into $B_{\eta/2}(\bm\b_0)=\{\bm\b:\|\bm\b-\bm\b_0\}\le\eta/2\}$ for all large $n$, with probability tending to one, where $\|\cdot\|$ denotes the Euclidean norm, implying that the {\it continuous} map
	\begin{align*}
	\bm\g\mapsto \bm\g_0-\tilde R_{nh}\left(\bm B^{-1}\bm\g\right),
	\end{align*}
	maps $B_{\eta}(\bm\g_0)=\{\bm\g:\|\bm\g-\bm\g_0\|_2\le\eta\}$ into itself for all large $n$ and small $h$. So for large $n$ and small $h$
	there is, by Brouwer's fixed point theorem, a point $\bm\g_{nh}$ such that
	\begin{align*}
	\bm\g_{nh}=\bm\g_0-\tilde R_{nh}\left(\bm B^{-1}\bm\g_{nh}\right).
	\end{align*}
	Defining $\bm\b_{nh}=\bm B^{-1}\bm\g_{nh}$, we get:
	\begin{align}
	\label{Brouwer_relation}
	\tilde\f_{n,h}(\bmb_{nh})=\f'(\bmb_0)(\bmb_{nh}-\bmb_0) +\tilde R_{nh}(\bmb_{nh})= \bm 0.
	\end{align}
	By compactness, $(\bmb_{n,1/k})_{k=1}^{\infty}$ must have a subsequence $(\bmb_{n,1/k_i})$ with a limit $\tilde{\bmb}_n$, as $i\to\infty$.
	
	Suppose that the $j$th component $\f_{nj}$ of $\f_n$ does not have a crossing of zero at $\tilde \bmb_n$. Since $\f_{nj}$ only has finitely many jump discontinuities, since there can only be discontinuities at a changing of ordering of the values $\bma^T\bm X_i$, there must be a closed ball $B_{\d}(\tilde\bmb_n)=\{\bmb:\|\bmb-\tilde\bmb_n\|\le\d\}$ of $\tilde\bmb_n$ such that $\{\bar{\f}_{nj}(\bmb):\bmb\in B_{\d}(\tilde\bmb_n)\}$ has a constant sign in the closed ball $B_{\d}$,
	say $\bar{\f}_{nj}(\bmb)>0$ for $\bmb\in \bar{B}_{\d}(\tilde\bmb_n)$. Again using that $\f_{nj}$ only has finitely many jump discontinuities, this means that
	\begin{align*}
	\bar{\f}_{n,j}(\bmb)\ge c>0,\qquad\text{for all }\bmb\in \bar{B}_{\d}(\tilde\bmb_n).
	\end{align*}
	This means that the $j$th component $\tilde \f_{n,h,j}$ of $\tilde \f_{n,h}$ satisfies
	\begin{align*}
	&\tilde \f_{n,h,j}(\bmb)=\left[ \f'(\bmb_0)(\bmb-\bmb_0)\right]_j+\tilde R_{nh,j}(\bmb)\\
	&=\frac1{h^{d-1}}\int \left\{\left[\f'(\bmb_0)(\bmb-\bmb_0)\right]_j+R_{nj}(u_1,\dots,u_{d-1})\right\}K_h(u_1-\b_{01})\dots K_h(u_{d-1}-\b_{d-1})\,du_1\dots\,du_{d-1}\\
	&\ge
	\frac1{h^{d-1}}\int \left\{\left[\f'(\bmb_0)(\bm u-\bmb_0)\right]_j+R_{nj}(u_1,\dots,u_{d-1})\right\}K_h(u_1-\b_1)\dots K_h(u_{d-1}-\b_{d-1})\,du_1\dots\,du_{d-1}
	-c/2\\
	&\ge c\,\frac1{h^{d-1}}\int K_h(u_1-\b_1)\dots K_h(u_d-\b_d)\,du_1\dots\,du_d-c/2\\
	&=c/2,
	\end{align*}
	for $\bmb\in B_{\d/2}(\tilde\bmb_n)$ and sufficiently small $h$, contradicting (\ref{Brouwer_relation}), since $\bmb_{nh}$, for $h=1/k_i$, belongs to
	$B_{\d/2}(\tilde\bmb_n)$ for large $k_i$.
\end{proof}

\subsection{Proof of consistency of the simple score estimator}
\label{subsec:Appendix-consistency1}
\begin{proof}
	Since $\hat\bmb_n$ is contained in the compact set $\mathcal{C}$, the sequence $(\hat\bmb_n)$ has a subsequence $(\hat\bmb_{n_k}=\hat\bmb_{n_k}(\omega))$, converging to an element $\bmb_*$. Let $\bma_{n_k} = \mbS(\hat \bmb_{n_k})$. If $\hat\bmb_{n_k}=\hat\bmb_{n_k}(\omega)\longrightarrow \bmb_*$, we get by continuity of the map $\mbS$ that $\bma_{n_k} \to \bma_* = \mbS(\bmb_*)$. By Proposition \ref{prop:L_2-psi-psi_n-alpha} and the fact that we solve the score equation in a ball $\B(\bma_0,\d_0)$, we also have
	\begin{align*}
	\hat \psi_{n_k,\hat\bma_{n_k}}\left( \mbS(\bmb_{n_k})^T\bm x\right)\longrightarrow  \psi_{\bma_*}(\mbS(\bmb_*)^T\bm x),
	\end{align*}
	where $\psi_{\bma}$ is defined in (\ref{def_psi_alpha}). By Proposition \ref{prop:score-connection} and the fact that in the limit, the crossing of zero becomes a root of the continuous limiting function, we get,
	\begin{align}
	&\lim_{k\to\infty}\f_{n_k}(\bmb_{n_k}) = \f(\bmb_{*}) = \bm  0
	\end{align}
	where,
	\begin{align*}
	\f(\bmb_{*})& = \int \bm J_{\mbS}(\bmb_*)^T\bm x\left\{y -\psi_{\bma_*}\left(\mbS(\bmb_*)^T\bm x\right)\right\}\,dP_0(\bm x,y)\nonumber\\
	& = \int \bm J_{\mbS}(\bmb_*)^T \bm x\left\{\psi_0\left(\mbS(\bmb_0)^T\bm x\right) -\psi_{\bma_*}\left(\mbS(\bmb_*)^T\bm x\right)\right\}\,dG(\bm x)\nonumber\\
	& = \int\bm J_{\mbS}(\bmb_*)^T \bm x\left[\psi_0\left(\mbS(\bmb_0)^T\bm x\right) -\E\left\{ \psi_0\left(\mbS(\bmb_0)^T\bm x\right)|\mbS(\bmb_*)^T\bm X=\mbS(\bmb_*)^T\bm x\right\}\right]\,dG(\bm x)\nonumber\\
	&=\E \left[\text{ Cov}\left[\bm J_{\mbS}(\bmb_*)^T\bm X  ,  \psi_0\left(\mbS(\bmb_0)^T\bmx\right)|\mbS(\bmb_*)^T\bm X\right]\right]
	\end{align*}	
	We next conclude that, 
	\begin{align*}
	0 &= (\bmb_0 - \bmb_*)^T\f(\bmb_{*}) \\
	&= \E \left[\text{ Cov}\left[(\bmb_0-\bmb_*)^T\bm J_{\mbS}(\bmb_*)^T\bm X  ,  \psi_0\left(\mbS(\bmb_*)^T\bm X + (\mbS(\bmb_0)-\mbS(\bmb_*))^T\bm X\right)|\mbS(\bmb_*)^T\bm X \right]\right]
	\end{align*}
	which can only happen if $\bmb_*$ corresponds to $\pm\bma_0$, where we use the positivity of the random variable $\text{ Cov}((\bma_0 - \bma)^T\bm X, \\\psi_0(\bma^T\bm X) | \bma^T\bm X)$ shown in Lemma \ref{lemma:positive-lagrange-covariance} in Appendix \ref{supE:AuxiliaryResults} .	Note that,
	\begin{align*}
	&\text{ Cov}\left[(\bmb_0-\bmb)^T\bm J_{\mbS}(\bmb)^T\bm X  ,  \psi_0\left(\mbS(\bmb)^T\bm X + (\mbS(\bmb_0)-\mbS(\bmb))^T\bm X \right)|\mbS(\bmb)^T\bm X= u \right]\\
	&= \text{ Cov}\left[ (\mbS(\bmb_0)-\mbS(\bmb)+ o(\bmb-\bmb_0))^T\bm X  ,  \psi_0\left(\mbS(\bmb)^T\bm X + (\mbS(\bmb_0)-\mbS(\bmb))^T\bm X\right)|\mbS(\bmb)^T\bm X = u \right] \\
	&=\text{ Cov}\left[(\bma_0-\bma)^T\bm X  ,  \psi_0(\bma^T\bm X + (\bma_0-\bma)^T\bm X)|\bma^T\bm X =u \right] + o(\bmb-\bmb_0)
	\end{align*}
	where the first term in the expression above is positive for all $\bma \in \B(\bma_0,\d_0)$ by Lemma \ref{lemma:positive-lagrange-covariance} in Appendix \ref{supE:AuxiliaryResults}, and that the covariance is zero if $\bma$ is a multiple of $\bma_0$. Since we also have the condition $\|\bma\|=1$ for solutions of the equation, the only multiple of $\bma_0$ which should be considered is therefore $-\bma_0$. But we can disregard this possibility by taking $\d_0$ sufficiently small in $\bma \in \B(\bma_0,\d_0)$.
\end{proof}

\subsection{Proof of asymptotic normality of the simple score estimator}
\label{subsec:Appendix-normal1}
We define $\f_n$ at $\hat\bmb_n$ by putting
\begin{align}
\label{estimator_crossing_def}
\f_n(\hat\bmb_n)=\bm 0.
\end{align}
Note that, with this definition, we use the representation of the components as a convex combination of the left and right limit at $\hat\bmb_n$: 
\begin{align}
\label{convex_comb_def}
\f_{n,j}(\hat\bmb_n)=\g_j\f_{n,j}(\hat\bmb_n-)+(1-\g_j)\f_{n,j}(\hat\bmb_n+)=0,
\end{align}
where $\hat\bmb_n+$ and $\hat\bmb_n-$ denote, respectively, the upper and lower limits at $\bmb_n$, taken for each component of $\bmb_n$, where $\f_{n,j}$ denotes the $j$th component of $\f_{n}$ and where we can choose $\g_j\in[0,1]$ in such a way that (\ref{convex_comb_def}) holds since we have a crossing of zero componentwise. Note that this does not change the location of the crossing of zero. Since the following asymptotic representations are also valid for one-sided limits as used in (\ref{convex_comb_def}) we can use  (\ref{estimator_crossing_def})  and assume $\f_{n}(\hat\bmb_n)=\bm 0$.

We now  first give an outline of the proof. We show:
\begin{align}
\label{asymptotic-relation}
&\bm J_{\mbS}(\bmb_0)^T \int\left\{\bm x-\E(\bm X|\mbS(\bmb_0)^T\bm  x)\right\}\left\{y-\psi_{\hat\bma_n}\left(\mbS(\hat\bmb_n)^T\bm  x\right)\right\}\,dP_0(\bm x, y)
+o_p\left(\hat\bmb_n-\bmb_0\right)\nonumber\\
&=-\bm J_{\mbS}(\bmb_0)^T\int\left\{\bm x-\E(\bm X|\mbS(\bmb_0)^T\bm  x)\right\}\left\{y-\psi_0\left(\mbS(\bmb_0)^T\bm x\right)\right\}\,d\bigl(\P_n-P_0\bigr)(\bm x, y) \nonumber\\
&\qquad+o_p\left(n^{-1/2}\right),
\end{align}
where from now on we will use the notation $\E(\bm X |\mbS(\bmb)^T\bm  x )$ to denote $\E(\bm X |\mbS(\bmb)^T\bm  X = \mbS(\bmb)^T\bm  x)$ for all $\bmb \in \mathcal{C}$ and $\bmx \in \mathcal{X}$.

Since $\hat \bmb_n \to_p \bmb_0$ and since the function $\bmb \to \psi_{\mbS(\bmb)}(\mbS(\bmb)^T\bm  x)$ has derivative  $\psi_0'(\mbS(\bmb_0)^T\bm  x)\bm J_{\mbS}(\bmb_0)^T \bigl( \bm x - \E( \bm X |\mbS(\bmb_0)^T\bm  x)\bigr)$ at $\bmb  = \bmb_0$ for all $\bm x \in \cal X$ (See Lemma \ref{lemma:derivative_psi_a} in Appendix \ref{supE:AuxiliaryResults}), and since
$$
\bm J_{\mbS}(\bmb_0)^T \int\left\{\bm x-\E(\bm X|\mbS(\bmb_0)^T\bm  x)\right\}\left\{y-\psi_{\bma_0}\left(\mbS(\bmb_0)^T\bm  x\right)\right\}\,dP_0(\bm x, y)=0,
$$
we get from (\ref{asymptotic-relation}):
\begin{align*}
&\bm B \left(\hat\bmb_n-\bmb_0\right)+o_p\left(\hat\bmb_n-\bmb_0\right)=\left\{1+o_p(1)\right\}\bm B \left(\hat\bmb_n-\bmb_0\right)\\
&=\bm J_{\mbS}(\bmb_0)^T\int\left\{\bm x-\E(\bm X|\mbS(\bmb_0)^T\bm  x)\right\}\left\{y-\psi_0\left(\mbS(\bmb_0)^T\bm x\right)\right\}\,d\bigl(\P_n-P_0\bigr)(\bm x, y)+o_p\left(n^{-1/2}\right),
\end{align*} 
where $\bm B$, defined in $(\ref{def_B})$, is nonsingular. This implies:
\begin{align}
\label{order_beta_n-beta_0}
\hat\bmb_n-\bmb_0=O_p(n^{-1/2}),
\end{align}
and
\begin{align*}
&\hat\bmb_n-\bmb_0
=\bm B ^{-1}\bm J_{\mbS}(\bmb_0)^T\int\left\{\bm x-\E(\bm X|\mbS(\bmb_0)^T\bm  x)\right\}\left\{y-\psi_0\left(\mbS(\bmb_0)^T\bm x\right)\right\}\,d\bigl(\P_n-P_0\bigr)(\bm x, y)+o_p\left(n^{-1/2}\right).
\end{align*} 
Finally, since
\begin{align*}
&\sqrt n\bm B ^{-1}\bm J_{\mbS}(\bmb_0)^T\int\left\{\bm x-\E(\bm X|\mbS(\bmb_0)^T\bm  x)\right\}\left\{y-\psi_0\left(\mbS(\bmb_0)^T\bm x\right)\right\}\,d\bigl(\P_n-P_0\bigr)(\bm x, y)
\to_d N_{d}\left(\bm 0,  \bm  \Pi \right),
\end{align*}
where $\bm \Pi$ is defined in (\ref{def_Pi}),
we conclude that,
\begin{align*}
\sqrt n (\hat \bmb_n -\bmb_0) &  \to_d  N_{d}\left(\bm 0,  \bm  \Pi \right),
\end{align*}

The asymptotic normality of the estimator $\hat\bma_n$ then follows by noting that,
\begin{align*}
\sqrt n (\hat \bma_n -\bma_0) &=  \bm J_{\mbS}(\bmb_0) \sqrt n (\hat \bmb_n -\bmb_0)+ o_p\left( \sqrt n (\hat \bmb_n -\bmb_0)\right) \to_d  N_{d}\left(\bm 0, \bm J_{\mbS}(\bmb_0)\bm  \Pi \left(\bm J_{\mbS}(\bmb_0)\right)^T\right).
\end{align*}

To prove (\ref{asymptotic-relation}) we first define the piecewise constant function $\bar E_{n, \bmb}$
\begin{eqnarray*}
	\bar E_{n, \bmb}(u)  =  \left \{
	\begin{array}{lll}
		\E\left[\bm X| \mbS(\bmb)^T\bm X= \t_{i,\bmb}\right] \ \  \ \ \ \ \ \  \ \text{ if $\psi_{\bma}(u)  > \hat\psi_{n\bma}(\tau_i)$  \ for all $u \in (\tau_i, \tau_{i+1})$}, \\
		\E\left[\bm X| \mbS(\bmb)^T\bm X= s\right] \ \ \ \  \ \  \  \ \ \ \  \ \text{ if $\psi_{\bma}(s)  = \hat\psi_{n\bma}(s)$ \ for some $s \in (\tau_i, \tau_{i+1})$}, \\
		\E\left[\bm X| \mbS(\bmb)^T\bm X= \t_{i+1,\bmb}\right]\ \ \ \ \ \ \ \text{if $\psi_{\bma}(u) < \hat\psi_{n\bma}(\tau_i)$  \ for all $u \in (\tau_i, \tau_{i+1})$},
	\end{array}
	\right.
\end{eqnarray*}
where the $\t_{i,\bmb}$ denote the sequence of jump points of the monotone LSE  $\hat\psi_{n\bma} = \hat\psi_{n\mbS(\bmb)}$. We then have
\begin{align}
\label{proof1:constant-integral-0}
\int \bar E_{n, \hat \bmb_n}\left(\mbS(\hat \bmb_n)^T\bm x \right)\left\{y-\hat \psi_{n\hat \bma_n}\left(\mbS(\hat \bmb_n)^T\bm x\right)\right\}\,d\P_n(\bm x, y) = \bm 0.
\end{align}
This follows from the fact that $\hat\psi_{n\bma}$, i.e. the minimizer of the quadratic criterion $\int_{\mcX  \times \R}  \left(  y -  \psi(\bma^T \bm x)  \right)^2 d\P_n(\bm x, y)$ over monotone functions $\psi  \in \mathcal{M}$, is the left derivative of the greatest convex minorant of the cumulative sum diagram $\{(0,0), \Big(i, \sum_{j=1}^i  Y_i^\alpha  \Big), i=1, \ldots, n \Big \}$. (See also \cite{piet_geurt:14}, p.332).   By Lemma \ref{lemma:DerBoundedBelow} in Appendix \ref{supE:AuxiliaryResults} we also know that $\psi_{\bma}'$ stays away from zero for all $\bma = \mbS(\bmb)$ in a neighborhood of $\bma_0=\mbS(\bmb_0)$. Using the same techniques as in \cite{piet_geurt:14}, we can find a constant $C>0$ such that for all $i = 1, \ldots, d $ and $u \in \mathcal{I}_{\bma}$,
\begin{align}
\label{diff:rho-rhobar}
\left| \E\left( X_i| \mbS(\bmb)^T\bm X =u\right) - \bar E_{ni, \bmb}(u)\right | \le C\left|\psi_{\bma}(u) -\hat \psi_{n\bma}(u)\right|
\end{align}
where $\bar E_{ni, \bmb}$ denotes the $i$th component of $E_{n, \bmb}$.

In the sequel, we will use $\bm J_{\mbS}(\hat\bmb_n)   = O_p(1)$, an immediate consequence of consistency of $\hat \bma_n$ and Assumption A8.    Now, as a consequence of (\ref{proof1:constant-integral-0}), we can write
\begin{align}
\label{proof1-relation1}
\f_n(\hat\bmb_n) &= \bm J_{\mbS}(\hat\bmb_n) ^T \int \left\{\bm x -\E\left(\bm X|\mbS(\hat \bmb_n)^T\bm x\right)\right\}\left\{y-\hat \psi_{n\hat \bma_n}\left(\mbS(\hat \bmb_n)^T\bm x\right)\right\}\,d\P_n(\bm x, y) \nonumber\\
& + \bm J_{\mbS}(\hat\bmb_n) ^T \int\left\{ \E\left(\bm X|\mbS(\hat \bmb_n)^T\bm x\right) - \bar E_{n, \hat \bmb_n}\left(\mbS(\hat \bmb_n)^T\bm  x\right) \right\}\left\{y-\hat \psi_{n\hat \bma_n}\left(\mbS(\hat \bmb_n)^T\bm x\right)\right\}\,d\P_n(\bm x, y) \nonumber\\
& =\bm J_{\mbS}(\hat\bmb_n) ^T  \big(I + II\big).
\end{align}
The term $II$ can be written as
\begin{align}
\label{proof1-relation1b}
&II =\int\left\{ \E\left(\bm X|\mbS(\hat \bmb_n)^T\bm x\right) - \bar E_{n, \hat \bmb_n}\left(\mbS(\hat \bmb_n)^T\bm  x\right) \right\}\left\{y-\hat \psi_{n\hat \bma_n}\left(\mbS(\hat \bmb_n)^T\bm x\right)\right\}\,d(\P_n-P_0)(\bm x, y) \nonumber\\
&\qquad + \int\left\{ \E\left(\bm X|\mbS(\hat \bmb_n)^T\bm x\right) - \bar E_{n, \hat \bmb_n}\left(\mbS(\hat \bmb_n)^T\bm  x\right) \right\}\left\{y-\psi_{\hat \bma_n}\left(\mbS(\hat \bmb_n)^T\bm x\right)\right\}\,dP_0(\bm x, y) \nonumber
\\
&\qquad + \int\left\{ \E\left(\bm X|\mbS(\hat \bmb_n)^T\bm x\right) - \bar E_{n, \hat \bmb_n}\left(\mbS(\hat \bmb_n)^T\bm  x\right) \right\}
\left\{\hat \psi_{n\hat \bma_n}\left(\mbS(\hat \bmb_n)^T\bm x\right)-\psi_{\hat \bma_n}\left(\mbS(\hat \bmb_n)^T\bm x\right)\right\}\,dP_0(\bm x, y) \nonumber
\\
& = II_a + II_b +II_c.
\end{align}
We first note that by Lemma \ref{lemma:BoundedVar} in Appendix \ref{supE:AuxiliaryResults}, the functions $ u \mapsto  \E( X_i| \mbS( \bmb)^T\bm X =u)$ are uniformly bounded by $R$  for all $\bmb \in \mathcal{C}$ and  $i \in \{1, \ldots, d \}$. Also, they admit  a bounded variation, with a total variation that is uniformly bounded for all $\bmb \in \mathcal{C}$ and  $i \in \{1, \ldots, d \}$.  By definition of $\bar E_{n, \bmb}$ its $i$-th component, $\bar E_{ni, \bmb}$ is also uniformly bounded by $R$ and has a finite total variation which cannot exceed the total variation of $ u \mapsto  \E( X_i| \mbS( \bmb)^T\bm X =u)$.  Using Lemma \ref{lemma:BoundedVar2} of Appendix \ref{supE:AuxiliaryResults} , we can find two monotone functions   $f_1$ and $f_2$ such that $u \mapsto \E\left(\bm X|\mbS(\bmb)^T\bm X  = u\right) - \bar E_{n, \bmb}(u) = f_2(u)  -   f_1( u)$ with $f_1, f_2  \in \mathcal{M}_{RC_1}$ for some constant $C_1 > 0$. Also,  we know that $ \hat \psi_{n\hat \bma_n}  \in \mathcal{M}_{R K}$ with $K = K_1 \log n$  with increasing probability as $n \to \infty$ provided that $K_1 > 0$ is chosen large enough.  Noting that for any bounded increasing functions $f_1, f_2, f_3$  we have that $(f_2 - f_1) f_3$ is again bounded and has a bounded variation, it follows that the class  of functions, $\mathcal{F}_a$ say,  involved in term $II_a$ is included in $\mathcal{H}_{RK'v}$   defined  in Lemma \ref{lemma: BernEntropyGeneral} in Appendix \ref{supD:entropy} . Here, the constant $K'   = K_2 \log n$  for some large enough constant $K_2 > 0$, and  $v = C_2  (\log n)^2 n^{-1/3}$ for some constant $C_2 > 0$ using (\ref{diff:rho-rhobar})  and Proposition \ref{prop:L_2-psi-psi_n-alpha}.   Using Lemma \ref{lemma: BernEntropyGeneral} in Appendix \ref{supD:entropy}, we can show that (when the event $ \hat \psi_{n\hat \bma_n}  \in \mathcal{M}_{R K}$ occurs) 
\begin{eqnarray*}
	H_B\big(\epsilon,\widetilde{\mathcal{F}}_a, \Vert \cdot \Vert_{B,P_0} \big)   \le \frac{B_1}{\epsilon},  \ \ \text{for some constant $B_1> 0$,}
\end{eqnarray*}
with $\widetilde{\mathcal{F}}_a  = \tilde{D}^{-1}\mathcal{F}_a$ with $\tilde{D}  \asymp K'  = K_2 \log n$.  Also, for any element $\tilde{f}  = \tilde{D}^{-1}  f \in \widetilde{\mathcal{F}}$ we have that 
\begin{eqnarray*}
	\Vert \tilde{f} \Vert_{B, P_0}  \le  B_2 \tilde D^{-1} v =  C_2  (\log n) n^{-1/3}  = \delta_n,  \ \ \text{for some constant $C_2 > 0$}.
\end{eqnarray*}	
Let $II_{a, i}$ be the term corresponding to $i$-th component of $\bm X$. Using Markov's inequality we have for a fixed $A > 0$, $\nu > 0$  and $n$ large enough that 
\begin{align*}
P&\left( \vert II_{a, i}  \vert \ge    A n^{-1/2}  \right)   =  P\left( \vert II_{a, i}  \vert \ge    A n^{-1/2} , \sup_{\bma \in \mathcal{B}(\bma_0, \delta_0) }\sup_{\bm x \in \mcX} \big \vert \hat{\psi}_{n\bma} (\bma^T \bm x)\big \vert \le K \right)  + \nu/2 \\
& \lesssim     \frac{\tilde{D}}{A}  J_n(\delta_n) \left (1  +  \frac{J_n(\delta_n)}{\sqrt n \delta^2_n} \right)     + \nu/2,   \ \ \text{where $J_n$ is defined in (\ref{Jn})}  \\ 
& \lesssim  \frac{\log n}{A}  B_2 \delta^{1/2}_n  \left(1 +  \frac{B_2}{\sqrt n \delta^{3/2}_n}  \right) + \nu/2  , \  \ \text{for some constant $B_2 > 0$, using the inequality in (\ref{IneqJn})}  \\
& \lesssim  \frac{1}{A}  (\log n)^{3/2}n^{-1/6}  \left(1 +  \frac{B_3}{(\log n)^{3/2}}  \right)  + \nu/2, \ \ \text{with  $B_3 = B_2 C^{-3/2}_2$}  \\
& \le \nu
\end{align*}
for $n$ large enough.   We conclude that $II_{a, i}  = o_p(n^{-1/2})$ which in turn implies that 
$$ II_a = o_p(n^{-1/2}).$$
We turn  now to $II_b$. 
Using Lemma \ref{lemma:derivative_psi_a} in Appendix \ref{supE:AuxiliaryResults} and a Taylor expansion of $\bmb \mapsto \psi_{\bma}\left(\mbS(\bmb)^T\bm x\right)$ we get,
\begin{align}
\label{taylor-expansion}
\psi_{\bma}\left(\mbS(\bmb)^T\bm x\right) &= \psi_0\left(\mbS(\bmb_0)^T\bm x\right) + (\bmb - \bmb_0)^T\left[\bm J_{\mbS}(\bmb_0)^T\left( \bm x - \E(\bm X |\mbS(\bmb_0)^T\bm X =\mbS(\bmb_0)^T\bm x )\right)\psi_0'\left(\mbS(\bmb_0)^T\bm x\right)\right]\nonumber\\
& \qquad+ o(\bmb - \bmb_0),
\end{align}
so that

\begin{align*}
II_b&= \bm J_{\mbS}(\hat\bmb_n)^T \int\left\{ \E\left(\bm X|\mbS(\hat \bmb_n)^T\bm x\right) - \bar E_{n, \hat \bmb_n}\left(\mbS(\hat \bmb_n)^T\bm  x\right) \right\}\left\{\psi_0\left(\mbS( \bmb_0)^T\bm x\right)- \psi_{\hat \bma_n}\left(\mbS(\hat \bmb_n)^T\bm x\right)\right\}\,dP_0(\bm x, y)\nonumber\\
& =-\left\{1+o_p(1) \right\} (\hat\bmb_n - \bmb_0)^T\bm J_{\mbS}(\bmb_0)^T \int\left\{ \E\left(\bm X|\mbS(\bmb_0)^T\bm x\right) - \bar E_{n, \hat \bmb_n}\left(\mbS(\hat \bmb_n)^T\bm  x\right) \right\}\\
&\qquad\qquad\qquad\qquad\qquad\cdot\left[\bm J_{\mbS}(\bmb_0)^T\left( \bm x - \E(\bm X |\mbS(\bmb_0)^T\bm X =\mbS(\bmb_0)^T\bm x )\right)\psi_0'\left(\mbS(\bmb_0)^T\bm x\right)\right]
\,dP_0(\bm x, y)\\
&=o_p\left(\hat\bmb_n - \bmb_0\right),
\end{align*} 
using (\ref{diff:rho-rhobar}) and the consistency of $\hat \bmb_n$.  

We next consider the term $II_c$.  Using uniform boundedness of $\bm J_{\mbS}$ on $\mathcal{C}$ and the inequality in (\ref{diff:rho-rhobar}) it follows that  
\begin{eqnarray*}
	\Vert II_c \Vert &\lesssim  &  \int \left\{\psi_{\hat \bma_n}\left(\hat \bma_n^T\bm x\right)- \hat \psi_{n\hat \bma_n}\left(\hat \bma_n^T\bm x\right)\right\}^2dG(\bm x)  \\
	& = & O_p((\log n)^2 n^{-2/3})  = o_p(n^{-1/2})
\end{eqnarray*}
uniformly in $\beta \in \mathcal{C}$.  
We conclude that (\ref{proof1-relation1}) can be written as
\begin{align}
\label{proof1-relation2}
\f_n(\hat\bmb_n) &= \bm J_{\mbS}(\hat\bmb_n)^T \int \left\{\bm x -\E\left(\bm X|\mbS(\hat \bmb_n)^T\bm x\right)\right\}\left\{y-\hat \psi_{n\hat \bma_n}\left(\mbS(\hat \bmb_n)^T\bm x\right)\right\}\,d\P_n(\bm x, y) \nonumber
\\
&\qquad+ o_p\left(n^{-1/2}+ (\hat \bmb_n- \bmb_0)\right) \nonumber\\
&= \bm J_{\mbS}(\hat\bmb_n) ^T \int \left\{\bm x -\E\left(\bm X|\mbS(\hat \bmb_n)^T\bm x\right)\right\}\left\{y- \psi_{\hat \bma_n}\left(\mbS(\hat \bmb_n)^T\bm x\right)\right\}\,d\P_n(\bm x, y)  \nonumber\\
&\qquad +  \bm J_{\mbS}(\hat\bmb_n) ^T \int \left\{\bm x -\E\left(\bm X|\mbS(\hat \bmb_n)^T\bm x\right)\right\}\left\{\psi_{\hat \bma_n}\left(\mbS(\hat \bmb_n)^T\bm x\right)-\hat \psi_{n\hat \bma_n}\left(\mbS(\hat \bmb_n)^T\bm x\right)\right\}\,d\P_n(\bm x, y) \nonumber\\
&\qquad +  o_p\left(n^{-1/2}+ (\hat \bmb_n- \bmb_0)\right) \nonumber\\
& = I_a + I_b + o_p\left(n^{-1/2}\right), 
\end{align}

We show below that $I_b  = o_p\left(n^{-1/2} + (\hat \bma_n -\bma_0)\right)$ 
such that the limiting distribution of the score estimator follows from the analysis of the term $I_a$  which can be rewritten as
\begin{align}
\label{proof1-relation2_Ia}
I_a &= \bm J_{\mbS}(\hat\bmb_n) ^T \int \left\{\bm x -\E\left(\bm X|\mbS(\hat \bmb_n)^T\bm x\right)\right\}\left\{y-  \psi_{\hat \bma_n}\left(\mbS(\hat \bmb_n)^T\bm x\right)\right\}\,d(\P_n-P_0)(\bm x, y) \nonumber\\
&\qquad+ \bm J_{\mbS}(\hat\bmb_n) ^T \int \left\{\bm x -\E\left(\bm X|\mbS(\hat \bmb_n)^T\bm x\right)\right\}\left\{y- \psi_{\hat \bma_n}\left(\mbS(\hat \bmb_n)^T\bm x\right)\right\}\,dP_0(\bm x, y)
\end{align}
where we recall that $ \psi_{\bma}(u) =  \E\left(\psi_0(\bma^T \bm X | \bma^T X  = u \right)$.   For the second term on the right-hand side of (\ref{proof1-relation2_Ia}) we have by (\ref{taylor-expansion})
\begin{align}
\label{proof1-relation2_Ia1}
&\bm J_{\mbS}(\hat\bmb_n) ^T \int \left\{\bm x -\E\left(\bm X|\mbS(\hat \bmb_n)^T\bm x\right)\right\}\left\{y- \psi_{\hat \bma_n}\left(\mbS(\hat \bmb_n)^T\bm x\right)\right\}\,dP_0(\bm x, y)\nonumber\\
&= -\Biggl\{ \bm J_{\mbS}(\bmb_0) ^T \int \psi_0'\left(\mbS(\bmb_0)^T\bm x\right)\left\{\bm x -\E\left(\bm X|\mbS(\bmb_0)^T\bm x\right)\right\}\left\{\bm x -\E\left(\bm X|\mbS(\bmb_0)^T\bm x\right)\right\}^T\,dP_0(\bm x, y)\nonumber\\
&\qquad\qquad\qquad\qquad\qquad\qquad\qquad\qquad\qquad\qquad\qquad\qquad\qquad\qquad\qquad\qquad \quad \times \bm J_{\mbS}(\bmb_0)\Biggr\}(\hat \bmb_n -\bmb_0)\nonumber\\
&\quad +o_p(\hat \bmb_n- \bmb_0).
\end{align}
For the first term on the right-hand side of (\ref{proof1-relation2_Ia}) we have that
\begin{align}
\label{proof1-relation2_Ia2}
&\bm J_{\mbS}(\hat\bmb_n) ^T \int \left\{\bm x -\E\left(\bm X|\mbS(\hat \bmb_n)^T\bm x\right)\right\}\left\{y- \psi_{\hat \bma_n}\left(\mbS(\hat \bmb_n)^T\bm x\right)\right\}\,d(\P_n-P_0)(\bm x, y) \nonumber\\
&= \bm J_{\mbS}(\bmb_0) ^T \int \left\{\bm x -\E\left(\bm X|\mbS(\bmb_0)^T\bm x\right)\right\}\left\{y- \psi_0\left(\mbS(\bmb_0)^T\bm x\right)\right\}\,d(\P_n-P_0)(\bm x, y) +o_p(n^{-1/2}) + o_p(\hat \bmb_n - \bmb_0).
\end{align}
Indeed,  since  this amounts to showing that 
\begin{align}
A&= \bm \left(J_{\mbS}(\hat\bmb_n)  - J_{\mbS}(\bmb_0) \right)^T \int \left\{\bm x -\E\left(\bm X|\mbS(\hat \bmb_n)^T\bm x\right)\right\}\left\{y- \psi_{\hat \bma_n}\left(\mbS(\hat \bmb_n)^T\bm x\right)\right\}\,d(\P_n-P_0)(\bm x, y)\nonumber \\
&= o_p(\hat \bmb_n - \bmb_0), \label{A}  \\
B&= \int \bigg(  \E\left(\bm X|\mbS(\hat \bmb_n)^T\bm x\right)   - \E\left(\bm X|\mbS(\bmb_0)^T\bm x\right) \bigg)  (y - \psi_0(\bma_0^T \bm x))  d(\P_n - P_0)(\bm x, y)  = o_p(n^{-1/2} ) \label{B}
\end{align}
and 
\begin{align}
&C= \int \bigg(  \bm x  - \E\left(\bm X|\mbS(\hat \bmb_n)^T\bm x\right) \bigg)  \bigg(\psi_0(\bma_0^T \bm x)   - \psi_{\hat \bma_n}\left(\mbS(\hat \bmb_n)^T\bm x\right)\bigg)   d(\P_n - P_0)(\bm x, y)  
=o_p(n^{-1/2}). \label{C}
\end{align}
We start by proving (\ref{A}). Using again that $u \mapsto \E\left(X_i|\mbS(\hat \bmb_n)^T\bm X = u \right)$ is a bounded function with a uniformly bounded total variation, and that $x_i$ is a fixed (and deterministic) function,  we can show that the class of functions involved in $A$, $\mathcal{F}_A$ say, satisfies $\mathcal{F}_A \subset x_i\mathcal{H}_{RC_1 v}  +  \mathcal{H}_{RC_1 v}  $ with $v$ and $C_1$ are some constants that are independent of $n$ (since $\psi_{\bma}$, $\bm X$ and $u \mapsto E[\bm X | \bma^T X = u]$ are all bounded by constants independent of $n$).  
Now it follows by Lemma \ref{lemma: BernEntropyGeneral} in Appendix \ref{supD:entropy} that $H_B\big(\epsilon, \widetilde{\mathcal{H}}_{RC_1v}, \Vert \cdot \Vert_{B, P_0} \big) \lesssim 1/\epsilon$  with $\widetilde{\mathcal{H}}_{RC_1v} = (16M_0 C_1)^{-1} \mathcal{H}_{RC_1v} $  and $\Vert \tilde{h}  \Vert_{B, P_0}  \lesssim    C_2$ for some constant $C_2 > 0$ that is independent of $n$ for all $\tilde{h} \in \widetilde{\mathcal{H}}_{RC_1v} $. Hence, using  arguments similar to those of the proof of $II_a = o_p(n^{-1/2})$ we can  show that 
\begin{eqnarray*}
	\int \left\{\bm x -\E\left(\bm X|\mbS(\hat \bmb_n)^T\bm x\right)\right\}\left\{y- \psi_{\hat \bma_n}\left(\mbS(\hat \bmb_n)^T\bm x\right)\right\}\,d(\P_n-P_0)(\bm x, y) = O_p(n^{-1/2}).
\end{eqnarray*}
Using a Taylor expansion of $\bm J_{\mbS}(\bmb)$ around $\bmb_0$ gives the desired rate in (\ref{A}).   

Now we turn to  term $B$ in (\ref{B}).  Fix $\nu > 0$ and $i \in \{1, \ldots, d \}$. Using consistency of $\hat \bmb_n$ and Lemma \ref{lem: continuity} in Appendix \ref{supE:AuxiliaryResults}, then for all $\eta> 0$  there exists $n$ large enough such that  
\begin{eqnarray*}
	\Big \vert \E\left( X_i|\mbS(\hat \bmb_n)^T\bm x\right)   - \E\left(X_i|\mbS(\bmb_0)^T\bm x\right)  \Big \vert \le \eta,
\end{eqnarray*}
with probability at least $1- \nu/2$.  Thus, for $L > 0$ we have that for $n$ large enough
\begin{eqnarray*}
	&&P(|B_i|> L  n^{-1/2})\\
	&& = P\left( \Big \vert \int \bigg(  \E\left(X_i|\mbS(\hat \bmb_n)^T\bm x\right)   - \E\left(X_i|\mbS(\bmb_0)^T\bm x\right) \bigg)  (y - \psi_0(\bma_0^T \bm x))  d(\P_n - P_0)(\bm x, y) \Big \vert  > L  n^{-1/2}\right) \\
	&& \le \nu/2  +  \frac{1}{L}  E[\Vert \mathbb{G}_n\Vert_{\mathcal{F'}} ], \ \ \text{where $\mathcal{F'} $ is defined in (\ref{F'})} \\
	&& \le \nu/2  + \frac{C_1}{L}  \eta \  \ \text{for some constant $C_1 > 0$,}
\end{eqnarray*}
where $B_i$ denotes the $i$th component of $B$ defined in (\ref{B}) and  where we have used the result of Lemma \ref{CorBernsteinfixed}.  Choosing $\eta$ such that $\eta \le \nu L C^{-1}_1/2$ gives the claimed rate of convergence in (\ref{B}).

To establish the convergence rate of $C$, we first note that, for $ i \in \{1, \ldots, d \}$, we have that $\bm x \mapsto \E( X_i|\mbS(\hat \bmb_n)^T\bm X=\mbS( \hat \bmb_n)^T \bm x) \psi_{\hat \bma_n}(\mbS( \hat \bmb_n)^T \bm x)$ belongs to the class $\mathcal{G}_{RC_1} - \mathcal{G}_{RC_1}$ for some constant $C_1 > 0$ where $\mathcal{G}_{RK}$ was defined in (\ref{GRK}). This follows from using again the fact that the function $u \mapsto E\left( X_i|\mbS( \bmb)^T\bm X  = u\right)$ is uniformly bounded and has a uniform total variation for all $\beta \in \mathcal{C}$, that $\psi_{\bma}$ is a bounded monotone function, and the fact that $(f_1  - f_2)f_3$ is a bounded function with bounded total variation for any increasing and bounded functions $f_1, f_2$ and $f_3$, where we again use Lemma \ref{lemma:BoundedVar2} in Appendix \ref{supE:AuxiliaryResults} to write the function $u \mapsto E\left( X_i|\mbS( \bmb)^T\bm X  = u\right)$ as the difference $f_1-f_2$. Note now that both $\bm x  \mapsto x_i$ and $\bm x \mapsto \psi_0(\mbS( \bmb_0)^T \bm x)$ are fixed and bounded functions, and that the order bracketing entropy of a class does not get altered after multiplication its members by such functions (similarly for addition). It follows from Lemma \ref{lemma:EntropyDiff} and Lemma \ref{lemma:EntropyGRK} in Appendix \ref{supD:entropy}, that the  $\epsilon$-bracketing entropy of the class of functions involved in term $C$
with respect to $\Vert \cdot \Vert_{P_0}$  is bounded above by $B/ \epsilon$ for some constant $B$.

Furthermore, using consistency of $\hat \bma_n$ and Lemma \ref{lem: continuity} of Appendix \ref{supE:AuxiliaryResults}, we can find for any fixed $\nu > 0$ an $\eta > 0$ such that $\sup_{\bmx}| \psi_0(\bma_0^T\bmx)  - \psi_{\hat \bma_n}(\hat \bma_n^T\bmx) | \le\eta$ with probability at least $1-\nu/2$ for $n$ large enough. Hence, at the cost of increasing the constant $B$, both the $\Vert \cdot \Vert_{\infty}$ and $\Vert \cdot \Vert_{P_0}$ norms of the functions of the class involved in term $C$ are bounded above by $B \eta$.  Using Markov's inequality and Lemma 3.4.2 of \cite{vdvwe:96} it follows that for all $L > 0$
\begin{align*}
&P(|C_{i}|> L  n^{-1/2})\\
&=P\left(\left \vert \int \bigg(  x_i  - \E\left(\bm X_i|\mbS(\hat \bmb_n)^T\bm x\right) \bigg)  \bigg(\psi_0(\mbS( \bmb_0)^T \bm x)  -  \psi_{\hat \bma_n}\left(\mbS(\hat \bmb_n)^T\bm x\right)  \bigg)   d(\P_n - P_0)(\bm x, y)  \right \vert \ge L n^{-1/2} \right) \\
& \le \nu/2 + \frac{1}{L}  J_n(B \eta) \left(1 +   B\eta \frac{J_n(B\eta)}{\sqrt n  B^2 \eta^2}\right)  \le  \nu/2   +  \frac{1}{L} \left(B_1\eta^{1/2} +    \frac{B_1}{B} \frac{1}{\sqrt n} \right)  \le \nu
\end{align*}
taking $\eta$ small enough  and $n$ large enough.  We conclude that $C = o_p(n^{-1/2})$.  Now we come back to term $I_b$ given by
\begin{align*}
I_b &=  \bm J_{\mbS}(\hat\bmb_n) ^T \int \left\{\bm x -\E\left(\bm X|\mbS(\hat \bmb_n)^T\bm x\right)\right\}\left\{\psi_{\hat \bma_n}\left(\mbS(\hat \bmb_n)^T\bm x\right)-\hat \psi_{n\hat \bma_n}\left(\mbS(\hat \bmb_n)^T\bm x\right)\right\}\,d\P_n(\bm x, y). 
\end{align*}
Note first that 
\begin{align}
\label{term:I_bprime}
I_b &=  \bm J_{\mbS}(\hat\bmb_n) ^T \int \left\{\bm x -\E\left(\bm X|\mbS(\hat \bmb_n)^T\bm x\right)\right\}\left\{\psi_{\hat \bma_n}\left(\mbS(\hat \bmb_n)^T\bm x\right)-\hat \psi_{n\hat \bma_n}\left(\mbS(\hat \bmb_n)^T\bm x\right)\right\}\,d(\P_n- P_0)(\bm x, y) \nonumber \\
& =  \bm J_{\mbS}(\hat\bmb_n) ^T  I'_b,
\end{align}
since 
\begin{eqnarray*}
	&& \int \left\{\bm x -\E\left(\bm X|\mbS(\hat \bmb_n)^T\bm x\right)\right\}\left\{\psi_{\hat \bma_n}\left(\mbS(\hat \bmb_n)^T\bm x\right)-\hat \psi_{n\hat \bma_n}\left(\mbS(\hat \bmb_n)^T\bm x\right)\right\} dP_0(\bm x, y)  \\
	&& =  \E\left[\left(\bm X - \E\left(\bm X|\mbS(\hat \bmb_n)^T\bm X\right)\right) \left(\psi_{\hat \bma_n}\left(\mbS(\hat \bmb_n)^T\bm X\right)-\hat \psi_{n\hat \bma_n}\left(\mbS(\hat \bmb_n)^T\bm X \right) \right) \right] \\
	&& = \E\left[ \E \left[\left(\bm X - \E\left(\bm X|\mbS(\hat \bmb_n)^T\bm X\right)\right) | \mbS(\hat \bmb_n)^T\bm X\right] \left(\psi_{\hat \bma_n}\left(\mbS(\hat \bmb_n)^T\bm X\right)-\hat \psi_{n\hat \bma_n}\left(\mbS(\hat \bmb_n)^T\bm X \right) \right) \right] \\
	&& =\bm 0.
\end{eqnarray*}
Let $\mathcal{F}_b$ denote the class of functions involved in term $I'_b$ defined in (\ref{term:I_bprime}), where in the definition of this class we consider the event where $\hat \psi_{n \hat\bma_n}$ is bounded.  Given the arguments used recurrently above we can directly state that the $\epsilon$-bracketing entropy of this class is no larger than $A_1 \log n/\epsilon$ for some constant $A_1 > 0$ with increasing probability. Also, the $\Vert \cdot \Vert_{\infty}$ and $\Vert \cdot \Vert_{P_0}$ norms of the members of the class $\mathcal{F}_b$ are respectively bounded above with increasing probability by $A_1 \log n $ and $A_1 \log n  \ n^{-1/3}  = \eta_n$ at the cost of taking a larger $A_1$. For a fixed $\nu > 0$ and $L > 0$ we have for $ i \in \{1, \ldots, d \}$,using Lemma 3.4.2 of \cite{vdvwe:96}, 
\begin{align*}
&P \left( \left \vert \int \left\{\bm x_i -\E\left(\bm X_i|\mbS(\hat \bmb_n)^T\bm x\right)\right\}\left\{\psi_{\hat \bma_n}\left(\mbS(\hat \bmb_n)^T\bm x\right)-\hat \psi_{n\hat \bma_n}\left(\mbS(\hat \bmb_n)^T\bm x\right)\right\}\,d(\P_n- P_0)(\bm x, y)  \right \vert > L n^{-1/2} \right) \\
&\le \nu/2 +  \frac{A_2}{L} (\log n)^{1/2} \eta^{1/2}_{n} \left( 1+  \frac{A_2   (\log n)^{1/2} \eta^{1/2}_{n}}{\sqrt{n} \eta^{2}_n}(\log n)\right), \  \ \text{for some constant $A_2 > 0$}  \\
& \le \nu/2 +  \frac{A_2}{L} (\log n)^{1/2} \eta^{1/2}_{n} \left( 1+  \frac{A_2   (\log n)^{3/2}}{\sqrt{n} \eta^{3/2}_n} \right), \  \ \text{for some constant $A_2 > 0$}  \\
&  \lesssim \nu/2 +  \frac{A_2}{L}  (\log n)  n^{-1/6}  \left(1 +   \frac{A_2}{A_1^{3/2}}  \right)  \le \nu ,
\end{align*}
for $n$ large enough.    
This implies that $I_b = o_p(n^{-1/2})$. We conclude by (\ref{proof1-relation2_Ia1}), (\ref{proof1-relation2_Ia2}), (\ref{A}), (\ref{B}), (\ref{C}) and Definition (\ref{estimator_crossing_def}) that,
\begin{align*}
\bm B\left(\hat\bmb_n-\bmb_0\right)&=\int \left(\bm J_{\mbS}(\bmb_0)\right)^T\left\{\bm x-\E(\bm X|\mbS(\bmb_0)^T\bm x)\right\}\left\{y-\psi_0\bigl(\mbS(\bmb_0)^T\bm x\bigr)\right\}\,d\bigl(\P_n-P_0\bigr)(\bm x, y) \nonumber\\
&\qquad +o_p\left(n^{-1/2}+   \| \hat\bmb_n-\bmb_0  \|  \right),
\end{align*}
where 
\begin{align*}
\bm B= \left(\bm J_{\mbS}(\bmb_0)\right)^T \E\Bigl[\psi_0'(\mbS(\bmb_0)^T\bm X)\,\text{ Cov}(\bm X|\mbS(\bmb_0)^T\bm X)\Bigr]  \left(\bm J_{\mbS}(\bmb_0)\right).
\end{align*}
We get,
\begin{align*}
\sqrt n \left(\hat\bmb_n-\bmb_0\right)&= \sqrt n\bm B^{-1}\int \left(\bm J_{\mbS}(\bmb_0)\right)^T\left\{\bm x-\E(\bm X|\mbS(\bmb_0)^T\bm x)\right\}\left\{y-\psi_0\bigl(\mbS(\bmb_0)^T\bm x\bigr)\right\}\,d\bigl(\P_n-P_0\bigr)(\bm x, y) \nonumber\\
&\qquad +o_p\left(1+\sqrt n \|\hat\bmb_n-\bmb_0\|\right)\nonumber\\
\to_d N(\bm 0,\bm  \Pi),
\end{align*}
where 
\begin{align*}
\bm \Pi = \bm B^{-1} \left(\bm J_{\mbS}(\bmb_0)\right)^T \bm \Sigma \, \bm J_{\mbS}(\bmb_0) \, \bm B^{-1} \in \R^{(d-1)\times(d-1)}.
\end{align*}
The asymptotic limiting distribution of the single index score estimator $\hat \bma_n$ now follows by an application of the Delta method and we conclude that
\begin{align*}
\sqrt n (\hat \bma_n -\bma_0) & =  J_{\mbS}(\bmb_0) \sqrt n (\hat \bmb_n -\bmb_0)+ o_p\left( \sqrt n (\hat \bmb_n -\bmb_0)\right) \to_d  N_{d}\left(\bm 0,  J_{\mbS}(\bmb_0)\bm  \Pi \left(J_{\mbS}(\bmb_0)\right)^T\right).
\end{align*}
Finally, the result of Theorem \ref{theorem:asymptotics} follows by Lemma \ref{lemma:Moore-Penrose}. This completes the proof.

\section{Asymptotic distribution of the efficient score estimator}
\label{supC:ese}
In this section we prove $(iii)$ of Theorem \ref{theorem:asymptotics-efficient} on the asymptotic normality of the efficient score estimator $\tilde\bma_n$. The proofs of existence and consistency of $\tilde\bma_n$, given in $(i)$ and $(ii)$ of Theorem \ref{theorem:asymptotics-efficient} follow the same lines as the corresponding proofs for the simple score estimator $\hat \bma_n$ given in Sections \ref{subsec:Appendix-existence1} and \ref{subsec:Appendix-existence1} and are omitted.

\textbf{Proof of asymptotic normality:}
Let $\t_i$ denote the sequence of jump points of the monotone LSE  $\hat\psi_{n\bma}$. We introduce the piecewise constant function $\bar{\rho}_{n, \bmb}$ defined  for $u \in [\tau_i, \tau_{i+1})$  as
\begin{eqnarray*}
	\bar {\rho}_{n, \bmb}(u)  =  \left \{
	\begin{array}{lll}
		\E[\bm X| \mbS(\bmb)^T\bm X= \t_i]\psi_{\bma}'(\t_i) \ \  \ \ \ \ \ \ \textrm{ if $\psi_{\bma}(u)  > \hat\psi_{n\bma}(\tau_i)$  \ for all $u \in (\tau_i, \tau_{i+1})$}, \\
		\E[\bm X| \mbS(\bmb)^T\bm X= s]\psi_{\bma}'(s) \ \ \ \  \ \  \  \ \ \ \textrm{ if $\psi_{\bma}(s)  = \hat\psi_{n\bma}(s)$ \ for some $s \in (\tau_i, \tau_{i+1})$}, \\
		\E[\bm X| \mbS(\bmb)^T\bm X= \t_{i+1}]\psi_{\bma}'(\t_{i+1})\ \ \ \textrm{if $\psi_{\bma}(u) < \hat\psi_{n\bma}(\tau_i)$  \ for all $u \in (\tau_i, \tau_{i+1})$}. 
	\end{array}
	\right.
\end{eqnarray*}
We can write,
\begin{align}
\label{proof2-relation1}
&\xi_{nh}(\tilde\bmb_n) \nonumber\\
&=  \bm J_{\mbS}(\tilde\bmb_n) ^T \int \left\{\bm x \tilde{\psi}_{nh,\bma}'\left(\mbS(\tilde \bmb_n)^T\bm x \right)-\E\left(\bm X|\mbS(\tilde \bmb_n)^T\bm x\right)\psi_{\tilde \bma_n}'\left(\mbS(\tilde \bmb_n)^T\bm x \right)\right\}\left\{y-\hat \psi_{n\tilde \bma_n}\left(\mbS(\tilde \bmb_n)^T\bm x\right)\right\}\,d\P_n(\bm x, y) \nonumber\\
& +  \bm J_{\mbS}(\tilde\bmb_n) ^T \int\left\{ \E\left(\bm X|\mbS(\tilde \bmb_n)^T\bm x\right)\psi_{\tilde \bma_n}'\left(\mbS(\tilde \bmb_n)^T\bm x \right) - \bar \rho_{n, \tilde \bmb_n}\left(\mbS(\tilde \bmb_n)^T\bm  x\right) \right\}\left\{y-\hat \psi_{n\tilde \bma_n}\left(\mbS(\tilde \bmb_n)^T\bm x\right)\right\}\,d\P_n(\bm x, y) \nonumber\\
& = J + JJ,
\end{align}
using,
\begin{align*}
\int \bar {\rho}_{n, \tilde \bmb_n}\left(\mbS(\tilde \bmb_n)^T\bm x \right)\left\{y-\hat \psi_{n\tilde \bma_n}\left(\mbS(\tilde \bmb_n)^T\bm x\right)\right\}\,d\P_n(\bm x, y) = \bm 0.
\end{align*}

The term $JJ$ can be written as
\begin{align}
\label{proof2-relation1b}
JJ &=\bm J_{\mbS}(\tilde\bmb_n) ^T \int\left\{ \E\left(\bm X|\mbS(\tilde \bmb_n)^T\bm x\right)\psi_{\tilde \bma_n}'\left(\mbS(\tilde \bmb_n)^T\bm x \right) - \bar \rho_{n, \tilde \bmb_n}\left(\mbS(\tilde \bmb_n)^T\bm  x\right) \right\}\nonumber\\
&\qquad\qquad\qquad\qquad\qquad\qquad\qquad\qquad\qquad\cdot\left\{y-\hat \psi_{n\tilde \bma_n}\left(\mbS(\tilde \bmb_n)^T\bm x\right)\right\}\,d(\P_n-P_0)(\bm x, y) \nonumber\\
&\qquad+\bm J_{\mbS}(\tilde\bmb_n) ^T \int\left\{ \E\left(\bm X|\mbS(\tilde \bmb_n)^T\bm x\right)\psi_{\tilde \bma_n}'\left(\mbS(\tilde \bmb_n)^T\bm x \right) - \bar \rho_{n, \tilde \bmb_n}\left(\mbS(\tilde \bmb_n)^T\bm  x\right) \right\}\nonumber\\
&\qquad\qquad\qquad\qquad\qquad\qquad\qquad\qquad\qquad\cdot\left\{y- \psi_{\tilde \bma_n}\left(\mbS(\tilde \bmb_n)^T\bm x\right)\right\}\,dP_0(\bm x, y) \nonumber\\
&\qquad+\bm J_{\mbS}(\tilde\bmb_n) ^T \int\left\{ \E\left(\bm X|\mbS(\tilde \bmb_n)^T\bm x\right)\psi_{\tilde \bma_n}'\left(\mbS(\tilde \bmb_n)^T\bm x \right) - \bar \rho_{n, \tilde \bmb_n}\left(\mbS(\tilde \bmb_n)^T\bm  x\right) \right\}\nonumber\\
&\qquad\qquad\qquad\qquad\qquad\qquad\qquad\qquad\qquad\cdot\left\{\psi_{\tilde \bma_n}\left(\mbS(\tilde \bmb_n)^T\bm x\right)- \hat \psi_{n\tilde \bma_n}\left(\mbS(\tilde \bmb_n)^T\bm x\right)\right\}\,dP_0(\bm x, y) \nonumber\\
& = JJ_a + JJ_b +JJ_c,
\end{align}
We first note that by Assumption A10, the functions $u\mapsto \psi_{\bma}'(u):= \psi_{\mbS(\bmb)}'(u)$ are uniformly bounded and have a total variation that is uniformly bounded for all $\bmb \in \mathcal{C}$. This also implies, using Lemma \ref{lemma:BoundedVar}, that the functions $u \mapsto \E\left( X_i| \mbS( \bmb)^T\bm X =u\right)\psi_{\bma}'(u)$ have a bounded variation for all $\bmb \in \mathcal{C}$. Using the same arguments as those for term $II_a$ defined in (\ref{proof1-relation1b}) in the proof of Theorem \ref{theorem:asymptotics}, it easily follows that,
$$ JJ_a = o_p(n^{-1/2}).$$
We next consider the term $JJ_b$. By Lemma \ref{lemma:DerBoundedBelow} we know that $\psi_{\bma}'$ stays away from zero for all $\mbS(\bmb)$ in a neighborhood of $\mbS(\bmb_0)$. Using the same techniques as in \cite{piet_geurt:14}, we can find a constant $K>0$ such that for all $i = 1, \ldots, d $ and $u \in \mathcal{I}_{\bma}$,
\begin{align}
\label{diff:E-Ebar}
\left| \E\left( X_i| \mbS(\bmb)^T\bm X =u\right)\psi_{\bma}'(u) - \bar \rho_{ni, \bmb}(u)\right | \le K\left|\psi_{\bma}(u) -\hat \psi_{n\bma}(u)\right|,
\end{align}
where $\bar\rho_{ni, \bmb}$ denotes the $i$th component of $\rho_{n, \bmb}$. This implies that the difference $\E\left( X_i| \mbS(\bmb)^T\bm X =u\right)\psi_{\bma}'(u) - \bar \rho_{ni, \bmb}(u)$ converges to zero for all  $u \in \mathcal{I}_{\bma}$. Using Lemma \ref{lemma:derivative_psi_a} and a Taylor expansion of $\bmb \mapsto \psi_{\bma}\left(\mbS(\bmb)^T\bm x\right)$ we get,
\begin{align*}
\psi_{\bma}\left(\mbS(\bmb)^T\bm x\right) &= \psi_0\left(\mbS(\bmb_0)^T\bm x\right) + (\bmb - \bmb_0)^T\left[\bm J_{\mbS}(\bmb_0)^T\left( \bm x - \E(\bm X |\mbS(\bmb_0)^T\bm X =\mbS(\bmb_0)^T\bm x )\right)\psi_0'\left(\mbS(\bmb_0)^T\bm x\right)\right]\nonumber\\
& \qquad+ o(\bmb - \bmb_0),
\end{align*}
such that
\begin{align*}
JJ_b &= \bm J_{\mbS}(\tilde\bmb_n) ^T \int\left\{ \E\left(\bm X|\mbS(\tilde \bmb_n)^T\bm x\right)\psi_{\tilde \bma_n}'\left(\mbS(\tilde \bmb_n)^T\bm x \right) - \bar \rho_{n, \tilde \bmb_n}\left(\mbS(\tilde \bmb_n)^T\bm  x\right) \right\}\nonumber\\
&\qquad\qquad\qquad\qquad\qquad\qquad\qquad\cdot\left\{\psi_0\left(\mbS( \bmb_0)^T\bm x\right)- \psi_{\tilde \bma_n}\left(\mbS(\tilde \bmb_n)^T\bm x\right)\right\}\,dP_0(\bm x, y)\\
& =o_p\left(\tilde\bmb_n - \bmb_0\right).
\end{align*} 
For the therm $JJ_c$, we get by an application of the Cauchy-Schwarz inequality together with the uniform boundedness of $\bm J_{\mbS}$, Proposition \ref{prop:L_2-psi-psi_n-alpha} and (\ref{diff:E-Ebar}) that,
\begin{align*}
JJ_c &\le \bm J_{\mbS}(\tilde\bmb_n) ^T \Biggl( \int\left\{ \E\left(\bm X|\mbS(\tilde \bmb_n)^T\bm x\right)\psi_{\tilde \bma_n}'\left(\mbS(\tilde \bmb_n)^T\bm x \right) - \bar \rho_{n, \tilde \bmb_n}\left(\mbS(\tilde \bmb_n)^T\bm  x\right) \right\}^2\,dP_0(\bm x, y)\Biggr)^{1/2}\\
&\qquad\qquad\qquad \cdot \Biggl(\int  \left\{\psi_{\tilde \bma_n}\left(\mbS(\tilde \bmb_n)^T\bm x\right)- \hat \psi_{n\tilde \bma_n}\left(\mbS(\tilde \bmb_n)^T\bm x\right)\right\}^2\,dP_0(\bm x, y)\Biggr)^{1/2} \\
& \lesssim  \int \left\{\psi_{\tilde \bma_n}\left(\tilde \bma_n^T\bm x\right)- \hat \psi_{n\tilde \bma_n}\left(\tilde \bma_n^T\bm x\right)\right\}^2dG(\bm x) = O_p\left((\log n)^2 n^{-2/3}\right) =o_p(n^{-1/2}).
\end{align*}
We conclude that (\ref{proof2-relation1}) can be written as
\begin{align}
\label{proof2-relation2}
&\xi_{nh}(\tilde\bmb_n) \nonumber\\
&=  \bm J_{\mbS}(\tilde\bmb_n) ^T \int \left\{\bm x \tilde{\psi}_{nh,\tilde\bma_n}'\left(\mbS(\tilde \bmb_n)^T\bm x \right)-\E\left(\bm X|\mbS(\tilde \bmb_n)^T\bm x\right)\psi_{\tilde \bma_n}'\left(\mbS(\tilde \bmb_n)^T\bm x \right)\right\}\nonumber\\
&\qquad\qquad\qquad\qquad\qquad\qquad\qquad\qquad\qquad\qquad\qquad\qquad\cdot\left\{y-\hat \psi_{n\tilde \bma_n}\left(\mbS(\tilde \bmb_n)^T\bm x\right)\right\}\,d\P_n(\bm x, y) \nonumber\\
& \qquad +  o_p\left(n^{-1/2}+ (\tilde \bmb_n- \bmb_0)\right)\nonumber
\end{align}
\begin{align}
&=  \bm J_{\mbS}(\tilde\bmb_n) ^T \int \left\{\bm x \tilde{\psi}_{nh,\tilde\bma_n}'\left(\mbS(\tilde \bmb_n)^T\bm x \right)-\E\left(\bm X|\mbS(\tilde \bmb_n)^T\bm x\right)\psi_{\tilde \bma_n}'\left(\mbS(\tilde \bmb_n)^T\bm x \right)\right\}\nonumber\\
&\qquad\qquad\qquad\qquad\qquad\qquad\qquad\qquad\qquad\qquad\qquad\qquad\cdot\left\{y- \psi_{\tilde \bma_n}\left(\mbS(\tilde \bmb_n)^T\bm x\right)\right\}\,d\P_n(\bm x, y) \nonumber\\
&\qquad  +   \bm J_{\mbS}(\tilde\bmb_n) ^T \int \left\{\bm x \tilde{\psi}_{nh,\tilde\bma_n}'\left(\mbS(\tilde \bmb_n)^T\bm x \right)-\E\left(\bm X|\mbS(\tilde \bmb_n)^T\bm x\right)\psi_{\tilde \bma_n}'\left(\mbS(\tilde \bmb_n)^T\bm x \right)\right\}\nonumber\\
&\qquad\qquad\qquad\qquad\qquad\qquad\qquad\qquad\cdot\left\{\psi_{\tilde \bma_n}\left(\mbS(\tilde \bmb_n)^T\bm x\right)-\hat \psi_{n\tilde \bma_n}\left(\mbS(\tilde \bmb_n)^T\bm x\right)\right\}\,d(\P_n-P_0)(\bm x, y) \nonumber\\
&\qquad  +   \bm J_{\mbS}(\tilde\bmb_n) ^T \int \left\{\bm x \tilde{\psi}_{nh,\tilde\bma_n}'\left(\mbS(\tilde \bmb_n)^T\bm x \right)-\E\left(\bm X|\mbS(\tilde \bmb_n)^T\bm x\right)\psi_{\tilde \bma_n}'\left(\mbS(\tilde \bmb_n)^T\bm x \right)\right\}\nonumber\\
&\qquad\qquad\qquad\qquad\qquad\qquad\qquad\qquad\cdot\left\{\psi_{\tilde \bma_n}\left(\mbS(\tilde \bmb_n)^T\bm x\right)-\hat \psi_{n\tilde \bma_n}\left(\mbS(\tilde \bmb_n)^T\bm x\right)\right\}\,dP_0(\bm x, y) \nonumber\\
&\qquad +  o_p\left(n^{-1/2}+ (\tilde \bmb_n- \bmb_0)\right)\nonumber\\
&= J_a + J_b + J_c +  o_p\left(n^{-1/2}+ (\tilde \bmb_n- \bmb_0)\right).
\end{align}
We first consider the term $J_b$. By Assumption A10, Lemma \ref{lemma:BoundedVar} and Lemma \ref{lemma:derivative_kernel} we get that the functions $u\mapsto \E\left(\bm X|\mbS(\bmb)^T\bm x = u\right)\psi_{\tilde \bma_n}'\left(u \right)$ and $u\mapsto \tilde{\psi}_{nh,\tilde\bma_n}'(u) $ have a uniformly bounded total variation for all $\bmb \in \mathcal{C}$. 
Using similar arguments as for the term $I_b$ defined in (\ref{proof1-relation2}) we get for $A>0$ and $\nu >0$ that
$$ P(|J_b| \ge  An^{-1/2}) \le \nu,$$
for $n$ large enough and we conclude that $J_b = o_p(n^{-1/2})$.
For the term $J_c$ we get,
\begin{align*}
J_c &=\bm J_{\mbS}(\tilde\bmb_n) ^T \int \left\{\bm x -\E\left(\bm X|\mbS(\tilde \bmb_n)^T\bm x\right)\right\}\tilde{\psi}_{nh,\tilde\bma_n}'\left(\mbS(\tilde \bmb_n)^T\bm x \right)\nonumber\\
&\qquad\qquad\qquad\qquad\qquad\qquad\qquad\qquad\cdot\left\{\psi_{\tilde \bma_n}\left(\mbS(\tilde \bmb_n)^T\bm x\right)-\hat \psi_{n\tilde \bma_n}\left(\mbS(\tilde \bmb_n)^T\bm x\right)\right\}\,dP_0(\bm x, y) \\
&+\bm J_{\mbS}(\tilde\bmb_n) ^T \int \left\{ \tilde{\psi}_{nh,\tilde\bma_n}'\left(\mbS(\tilde \bmb_n)^T\bm x \right)- \psi_{\tilde \bma_n}'\left(\mbS(\tilde \bmb_n)^T\bm x \right) \right\}\E\left(\bm X|\mbS(\tilde \bmb_n)^T\bm x\right)\nonumber\\
&\qquad\qquad\qquad\qquad\qquad\qquad\qquad\qquad\cdot\left\{\psi_{\tilde \bma_n}\left(\mbS(\tilde \bmb_n)^T\bm x\right)-\hat \psi_{n\tilde \bma_n}\left(\mbS(\tilde \bmb_n)^T\bm x\right)\right\}\,dP_0(\bm x, y) \\
&=\bm J_{\mbS}(\tilde\bmb_n) ^T \int \left\{ \tilde{\psi}_{nh,\tilde\bma_n}'\left(\mbS(\tilde \bmb_n)^T\bm x \right)- \psi_{\tilde \bma_n}'\left(\mbS(\tilde \bmb_n)^T\bm x \right) \right\}\E\left(\bm X|\mbS(\tilde \bmb_n)^T\bm x\right)\nonumber\\
&\qquad\qquad\qquad\qquad\qquad\qquad\qquad\qquad\cdot\left\{\psi_{\tilde \bma_n}\left(\mbS(\tilde \bmb_n)^T\bm x\right)-\hat \psi_{n\tilde \bma_n}\left(\mbS(\tilde \bmb_n)^T\bm x\right)\right\}\,dP_0(\bm x, y) 
\end{align*}
Furthermore, let  $H_{\bmb}$ be the distribution function of the random variable $\mbS( \bmb)^T\bm X$ and let $\E(\bm X|u)$ denote the conditional expectation of $\bm X$ given $\mbS(\bmb)^T\bm X = u$, then 
\begin{align*}
&\int \left\{ \tilde{\psi}_{nh,\tilde\bma_n}'\left(u \right)- \psi_{\tilde \bma_n}'\left(u \right) \right\}\E\left(\bm X|u\right)\left\{\psi_{\tilde \bma_n}(u)-\hat \psi_{n\tilde \bma_n}\left(u\right)\right\}\,dH_{\tilde \bmb_n}(u)\\
& =\int \left\{ \frac1h\int K\left(\{u-v\}/h\right)\,d \hat\psi_{n\tilde\bma_n}(v)- \psi_{\tilde \bma_n}'\left(u \right) \right\}\E\left(\bm X|u\right)\left\{\psi_{\tilde \bma_n}(u)-\hat \psi_{n\tilde \bma_n}\left(u\right)\right\}\,dH_{\tilde \bmb_n}(u)\\
& =\int \left( \frac1{h^2}\int K'\left(\{u-v\}/h\right)\left\{  \hat\psi_{n\tilde\bma_n}(v)- \psi_{\tilde \bma_n}\left(v \right)\right\}dv  \right)\E\left(\bm X|u\right)\left\{\psi_{\tilde \bma_n}(u)-\hat \psi_{n\tilde \bma_n}\left(u\right)\right\}\,dH_{\tilde \bmb_n}(u)\\
& \qquad+\int \left( \frac1{h}\int K\left(\{u-v\}/h\right)\psi_{\tilde \bma_n}'\left(v \right)dv - \psi_{\tilde \bma_n}'\left(u \right) \right)\E\left(\bm X|u\right)\left\{\psi_{\tilde \bma_n}(u)-\hat \psi_{n\tilde \bma_n}\left(u\right)\right\}\,dH_{\tilde \bmb_n}(u).
\end{align*}
The last term on the right hand side is $O_p\left(n^{-2/7-1/3}\right)= o_p\left(n^{-1/2}\right)$. This follows by an application of the Cauchy-Schwarz inequality since
\begin{align*}
\left\{\int \left( \frac1{h}\int K\left(\{u-v\}/h\right)\psi_{\tilde\bma_n}'\left(v \right)dv - \psi_{\tilde\bma_n}'\left(u \right) \right)^2 \,dH_{\tilde \bmb_n}(u)\right\}^{1/2} = O_p\left(n^{-2/7}\right),
\end{align*}
and
\begin{align*}
\left\{ \int \left(\psi_{\tilde \bma_n}(u)-\hat \psi_{n\tilde \bma_n}\left(u\right)\right)^2 \,dH_{\tilde \bmb_n}(u)\right\}^{1/2} = O_p\left(n^{-1/3}\right).
\end{align*}
The first term on the right hand side is $O_p\left(n^{1/7-2/3}\right)= o_p\left(n^{-1/2}\right)$ using that for small $h$
\begin{align*}
&\int \left( \frac1{h^2}\int K'\left(\{u-v\}/h\right)\left\{  \hat\psi_{n\tilde\bma_n}(v)- \psi_{\tilde\bma_n}\left(v \right)\right\}dv  \right)\E\left(\bm X|u\right)\left\{\psi_{\tilde \bma_n}(u)-\hat \psi_{n\tilde \bma_n}\left(u\right)\right\}\,dH_{\tilde \bmb_n}(u)\\
&\lesssim \frac 1h \int \left(\psi_{\tilde \bma_n}(u)-\hat \psi_{n\tilde \bma_n}\left(u\right)\right)^2 \,dH_{\tilde \bmb_n}(u).
\end{align*}
We conclude that (\ref{proof2-relation2}) can be written as,
\begin{align}
\label{proof2-relation3}
&\xi_{nh}(\tilde\bmb_n) =  \bm J_{\mbS}(\tilde\bmb_n) ^T \int \left\{\bm x \tilde{\psi}_{nh,\tilde\bma_n}'\left(\mbS(\tilde \bmb_n)^T\bm x \right)-\E\left(\bm X|\mbS(\tilde \bmb_n)^T\bm x\right)\psi_{\tilde \bma_n}'\left(\mbS(\tilde \bmb_n)^T\bm x \right)\right\}\nonumber\\
&\qquad\qquad\qquad\qquad\qquad\cdot\left\{y- \psi_{\tilde \bma_n}\left(\mbS(\tilde \bmb_n)^T\bm x\right)\right\}\,d\P_n(\bm x, y) +o_p\left(n^{-1/2} +(\tilde \bmb_n -\bmb_0)\right) \nonumber\\
&=  \bm J_{\mbS}(\tilde\bmb_n) ^T \int\bm x \left\{ \tilde{\psi}_{nh,\tilde\bma_n}'\left(\mbS(\tilde \bmb_n)^T\bm x \right)-\psi_{\tilde \bma_n}'\left(\mbS(\tilde \bmb_n)^T\bm x \right)\right\}\left\{y- \psi_{\tilde \bma_n}\left(\mbS(\tilde \bmb_n)^T\bm x\right)\right\}\,d(\P_n-P_0)(\bm x, y) \nonumber\\
&\quad + \bm J_{\mbS}(\tilde\bmb_n) ^T \int\bm x \left\{ \tilde{\psi}_{nh,\tilde\bma_n}'\left(\mbS(\tilde \bmb_n)^T\bm x \right)-\psi_{\tilde \bma_n}'\left(\mbS(\tilde \bmb_n)^T\bm x \right)\right\}\left\{y- \psi_{\tilde \bma_n}\left(\mbS(\tilde \bmb_n)^T\bm x\right)\right\}\,dP_0(\bm x, y) \nonumber\\
&\quad+\bm J_{\mbS}(\tilde\bmb_n) ^T \int \left\{\bm x -\E\left(\bm X|\mbS(\tilde \bmb_n)^T\bm x\right)\right\}\psi_{\tilde \bma_n}'\left(\mbS(\tilde \bmb_n)^T\bm x \right)\left\{y- \psi_{\tilde \bma_n}\left(\mbS(\tilde \bmb_n)^T\bm x\right)\right\}\,d(\P_n-P_0)(\bm x, y) \nonumber\\
&\quad + \bm J_{\mbS}(\tilde\bmb_n) ^T \int \left\{\bm x -\E\left(\bm X|\mbS(\tilde \bmb_n)^T\bm x\right)\right\}\psi_{\tilde \bma_n}'\left(\mbS(\tilde \bmb_n)^T\bm x \right)\left\{y- \psi_{\tilde \bma_n}\left(\mbS(\tilde \bmb_n)^T\bm x\right)\right\}\,dP_0(\bm x, y)\nonumber\\
&\quad +o_p\left(n^{-1/2} +(\tilde \bmb_n -\bmb_0)\right)\nonumber\\
&= JJJ_a + JJJ_b+JJJ_c + JJJ_d +o_p\left(n^{-1/2} +(\tilde \bmb_n -\bmb_0)\right).
\end{align}
We consider  $JJJ_a$ first and note that by Assumption A10 and Lemma \ref{lemma:derivative_kernel}, the functions $\psi_{\bma}'$ and $\tilde \psi_{nh,\bma}'$ have a uniformly bounded total variation. By an application of Lemma \ref{lemma:BoundedVar2} we can write the difference $\tilde \psi_{nh,\bma}' - \psi_{\bma}'$ as the difference of two monotone functions, say $f_1,f_2 \in \mathcal{M}_{RC_1}$ for some constant $C_1 >0$. This implies that the class of functions 
$${\cal F}_1 = \Biggl\{ f(\bm x, y ) :=\{ \tilde{\psi}_{nh,\bma}'\left(\mbS( \bmb)^T\bm x \right)-\psi_{ \bma}'\left(\mbS( \bmb)^T\bm x \right)\}\{y- \psi_{ \bma}\left(\mbS( \bmb)^T\bm x\right)\}, (\bm x, y, \bmb) \in \mathcal{X}\times \R \times \mathcal{C}  \Biggr\},$$ 
is contained in the class $\mathcal{H}_{RC_1v}$ where $v \asymp h^{-1}\log n n^{-1/3}$ (See the proof of Lemma \ref{lemma:derivative_kernel}). By Lemma \ref{lemma: BernEntropyGeneral} and the fact that the order bracketing entropy of a class does not get altered after multiplication with the fixed and bounded function $\bm x \mapsto x_i$ we get that the class of functions involved with the term $JJJ_a$, say ${\cal F}_a$, satisfies
\begin{align*}
H_B\big(\epsilon,{\mathcal{F}_a}, \Vert \cdot \Vert_{B,P_0} \big)   \lesssim \frac{1}{\epsilon} \qquad \text{ and } \qquad \|f \|_{_{B,P_0}} \lesssim v
\end{align*}
Using again an application of Markov's inequality, together with Lemma 3.4.3 of \cite{vdvwe:96} we conclude that for $A>0$
\begin{align*}
P(|JJJ_a| > An^{-1/2}) \lesssim  v^{1/2}  = h^{-1/2}(\log n)^{1/2} n^{-1/6},
\end{align*}
which can be made arbitrarily small for $n$ large enough and $h\asymp n^{-1/7}$. We conclude that
$$JJJ_a = o_p(n^{-1/2}).$$
Using similar arguments as for the term $JJ_b$ defined in (\ref{proof2-relation1b}) we also get, 
$$ JJJ_b = o_p\left(\tilde \bmb_n -\bmb_0\right). $$
The result of Theorem \ref{theorem:asymptotics-efficient} follows by noting that, using the same techniques as for the term $I_a$ in (\ref{proof1-relation2_Ia2}), we get
\begin{align*}
JJJ_c &= \left(\bm J_{\mbS}(\bmb_0) \right)^T \int \left\{\bm x -\E\left(\bm X|\mbS(\bmb_0)^T\bm x\right)\right\}\psi_{0}'\left(\mbS(\bmb_0)^T\bm x \right)\left\{y- \psi_{0}\left(\mbS(\bmb_0)^T\bm x\right)\right\}\,d(\P_n-P_0)(\bm x, y)\\
&\quad+ o_p(n^{-1/2})+ o_p(\hat \bmb_n -\bmb_0),
\end{align*}
and that by a Taylor expansion  of $\bmb \mapsto \psi_{\bma}\left(\mbS(\bmb)^T\bm x\right)$ we get,
\begin{align*}
JJJ_d &= -\Biggl\{\left(\bm J_{\mbS}(\bmb_0) \right)^T \Biggl(\int \left(\psi_{0}'\left(\mbS(\bmb_0)^T\bm x \right)\right)^2
\cdot \left\{\bm x -\E\left(\bm X|\mbS(\bmb_0)^T\bm x\right)\right\} \left\{\bm x -\E\left(\bm X|\mbS(\bmb_0)^T\bm x\right)\right\}^T\,dP_0(\bm x, y)\Biggr)\\ 
&\qquad\qquad\qquad\qquad\qquad\qquad\bm \times J_{\mbS}(\bmb_0)\Biggr\}(\tilde \bmb_n -\bmb_0)  + o_p(\tilde \bmb_n -\bmb_0).
\end{align*}
The rest of the proof follows the same line as the proof of asymptotic normality of the simple score estimator defined in Theorem \ref{theorem:asymptotics} and is omitted.

\section{Entropy results}
\label{supD:entropy}

\begin{lemma}\label{EntropyF1F2}
	Fix $\epsilon > 0$, and consider $\mathcal{F}_1$  a class of functions defined on $\mathcal{X} \times \RR$ bounded by some constant $A> 0$  and equipped by the $L_2$ norm $\Vert \cdot \Vert_{P_0}$ with respect to $P_0$.  Also, let $\mathcal{F}_2$ be another class of continuous functions defined on a bounded set $\mathcal{C} \subset \RR^{d-1}$  such that $\mathcal{F}_2$ is equipped by the supremum norm $\Vert \cdot \Vert_{\infty}$, and bounded by some constant $B > 0$.  Moreover assume that $H_B(\epsilon, \mathcal{F}_1, \Vert \cdot \Vert_{P_0}) < \infty$ and $H_B(\epsilon, \mathcal{F}_2, \Vert \cdot \Vert_{\infty}) < \infty$. Consider 
	\begin{eqnarray*}
		\mathcal{F}  = \mathcal{F}_1 \mathcal{F}_2  = \Big \{ f(\bm x)  =  f_{\bmb}(\bm x, y)= f_1(\bm x, y)  f_2(\bmb): (\bm x, y, \bmb)  \in  \mathcal{X} \times \mathbb{R} \times \mathcal{C} \Big\}.
	\end{eqnarray*}
	Then there exists some constant $B > 0$ such that 
	\begin{eqnarray*}
		H_B(\epsilon, \mathcal{F}, \Vert \cdot \Vert_{P_0}) \le H_B(B \epsilon, \mathcal{F}_1, \Vert \cdot \Vert_{P_0})    +   H_B(B \epsilon, \mathcal{F}_2, \Vert \cdot \Vert_{\infty}).
	\end{eqnarray*}
	
\end{lemma}

\medskip

\begin{proof}
	Let $f  = f_1 f_2  \in \mathcal{F}$ for some pair $(f_1, f_2) \in \mathcal{F}_1 \times \mathcal{F}_2$. For $\epsilon > 0$  consider the $(f^L_1,f^U_1)$ and  $(f^L_2,f^U_2)$ $\epsilon$-brackets with respect to $\Vert \cdot \Vert_{P_0}$  for $f_1$ and $f_2$. Note that since $\mathcal{F}_1$ and $\mathcal{F}_2$ are bounded by $M  = \max(A, B)$  we can always assume that $ -M \le f^L_i  \le  f^U_i \le M$ for $i \in \{1, 2\}$.   As we deal with a product of two functions, construction of a bracket for $f$ requires considering different sign cases for a given pair $(\bm x, \bmb)$:
	\begin{enumerate}
		\item $0 \le f^L_1(\bm x)$ and $0 \le f^L_2(\bmb)$,
		\item  $0 \le f^L_1(\bm x)$, $f^L_2(\bmb)  < 0$ and $f^U_2(\bmb)  \ge 0$,
		\item $ f^L_1(\bm x) \le 0$, $f^U_1(\bm x) \ge 0$ and $0 \le f^L_2(\bmb)$,
		\item $f^U(\bm x) \le 0$, $f^L(\bmb) \ge 0$,
		\item $f^L(\bm x) \ge 0$, $f^U(\bmb) \ge 0$,
		\item  $ f^L_1(\bm x) \le 0$, $f^U(\bm x) \ge 0$, $f^L_2(\bmb)  \le 0$ and $f^U_2(\bmb)  \ge 0$,
		\item $ f^L_1(\bm x) \le 0$, $f^U(\bm x) \ge 0$  and $f^U_2(\bmb)  \le 0$,
		\item $f^U(\bm x) \le 0$, $f^L(\bmb) \le 0$ and $f^U(\bmb) \ge 0$,
		\item  $f^U_1(\bm x) \le 0$ and $f^U_2(\bmb) \le 0$.
		
	\end{enumerate}
	We can assume without loss of generality that each one these cases occur for all $x \in \mcX$ and $\bmb \in \mathcal{C}$ since the general case can be handled by considering the $9$ different subsets of $\mathcal{X} \times \mathcal{C}$. In the proof, we will restrict ourselves to making the calculations explicit for cases 1 and 2 since the remaining cases can be handled very similarly. Then,    $f^L_1 f^L_2  \le f \le  f^U_1 f^U_2 $. Also, we have that 
	\begin{eqnarray*}
		f^U_1 f^U_2   -  f^L_1 f^L_2  = \Big(f^U_1  - f^L_1\Big)  f^U_2  +  f^L_1  \Big(  f^U_2  - f^L_2 \Big).
	\end{eqnarray*}
	Recall that $M  = \max(A, B)$. Then,  it follows that 
	\begin{eqnarray*}
		\int_{\mathcal{X}} \Big(f^U_1 f^U_2   -  f^L_1 f^L_2\Big)^2  dP_0&  \le   & 2  M \  \Big( \int_{\mathcal{X}} \Big(f^U_1  - f^L_1\Big)^2 dP_0(\bm x)  +  \Vert f^U_2  - f^L_2  \Vert^2_\infty  \Big)  \\
		& \le &  4M \epsilon^2.
	\end{eqnarray*}
	This in turn implies that $H_B(\epsilon, \mathcal{F}, \Vert \cdot \Vert_{P_0}) \le H_B(C \epsilon, \mathcal{F}_1, \Vert \cdot \Vert_{P_0})   + H_B(C \epsilon, \mathcal{F}_2, \Vert \cdot \Vert_{\infty})$ with $C = (2M)^{-1}$.   Now we consider case 2.   It is not difficult to show that  
	\begin{eqnarray*}
		f^L_2  f^U_1   \le  f  \le  f^U_1  f^U_2. 
	\end{eqnarray*}
	Hence,
	\begin{eqnarray*}
		\int_{\mathcal{X}}  \Big(f^U_1  f^U_2  -  f^L_2  f^U_1  \Big)^2  dP_0  \le A^2  \Vert  f^U_2   -  f^L_2 \Vert^2_\infty \le A^2 \epsilon^2
	\end{eqnarray*}
	and we can take $C =  A^{-1}$. 
\end{proof}

\medskip

\begin{lemma}\label{lemma:EntropyDiff}
	Let $\mathcal{F}$ be a class of functions satisfying $H_B(\epsilon, \mathcal{F}, \Vert \cdot \Vert_{P_0}) < \infty$ for every $\epsilon \in (0, \epsilon_0) $  for some given $\epsilon_0 > 0$. If $\mathcal{D}  =  \mathcal{F} - \mathcal{F}$ the class of all differences of elements of $\mathcal{F}$, then
	\begin{eqnarray*}
		H_B(\epsilon, \mathcal{D}, \Vert \cdot \Vert_{P_0})   \le 2 H_B(\epsilon/2, \mathcal{F}, \Vert \cdot \Vert_{P_0}).
	\end{eqnarray*}
	
\end{lemma}

\medskip

\begin{proof}
	Let $\epsilon \in (0, \epsilon_0)$ and $d  = f_2  - f_1 $ denote an element in $\mathcal{D}$  with $(f_1, f_2) \in \mathcal{F}^2$.  Also, let $(f^L_1, f^U_1)$ and $(f^L_2, f^U_2)$  $\epsilon$-brackets for $f_1$ and $f_2$.   Define $d^L =  f^L_2  - f^U_1$ and $ d^U =  f^U_2  - f^L_1$. It is clear that $(d^L, d^U)$ is a bracket for $d$. Furthermore, we have that   
	\begin{align*}
	&\int_{\mcX} \left(  d^U(\bm x, y)  -  d^L(\bm x, y)  \right)  ^2  dP_0(\bm x, y) \\
	& \qquad\le   2 \left \{  \int_{\mcX} \left( f^U_1(\bm x, y)  -  f^L_1 (\bm x, y) \right)^2  dP_0(\bm x, y)  +   \int_{\mcX} \left( f^U_2(\bm x, y)  -  f^L_2(\bm x, y) \right)^2  dP_0(\bm x, y) \right \}   \\
	&\qquad\le  4 \epsilon^2.
	\end{align*}
	Thus, 
	\begin{eqnarray*}
		\exp\Big(H_B(2 \epsilon, \mathcal{D}, \Vert \cdot \Vert_{P_0}) \Big) \le \exp \Big(H_B(\epsilon, \mathcal{F}, \Vert \cdot \Vert_{P_0})\Big)^2,
	\end{eqnarray*}
	which is equivalent to the statement of the lemma. 
\end{proof}

\medskip

Consider the class  $\mathcal{G}_{RK}$ defined as 
\begin{eqnarray}\label{GRK}
\mathcal{G}_{RK}   = \Big \{ g:  g(\bm x) = g_{\bma}(\bm x)= \psi(\bma^T \bm x), \bm x \in \mathcal{X}, (\psi, \bma) \in \mathcal{M}_{RK} \times \mathcal{B}(\bma_0, \delta_0)   \Big \}.
\end{eqnarray}
where $\mathcal{M}_{RK}$ is the same class defined in (\ref{MRK}).

\medskip

\begin{lemma}\label{lemma:EntropyGRK}
	There exists $A > 0$ such that for $\epsilon \in  (0, K)$ we have that 
	\begin{eqnarray*}
		H_B(\epsilon, \mathcal{G}_{RK}, \Vert \cdot \Vert_{P_0}  )  \le \frac{A K}{\epsilon}. 
	\end{eqnarray*}
\end{lemma}

\medskip

\begin{proof}
	See the proof of Lemma 4.9 in \cite{balabdaoui2016}. 
\end{proof}

\medskip

\begin{lemma}\label{lemma: BernEntropyGeneral}
	For some constants $C > 0$ and $\delta > 0$ consider the class of functions 
	\begin{eqnarray*}
		\mathcal{D}_{RC\delta}  = \Big \{ d:  d= f_{1,\bma}  - f_{2,\bma},  \ (f_{1, \bma}, f_{2, \bma})  \in \mathcal{G}^2_{RC}, \Vert d(\alpha^T \cdot )\Vert_{P_0}   \le \delta \ \textrm{for all $\bma \in \mathcal{B}(\bma_0, \delta_0)$} \Big \}.  
	\end{eqnarray*}
	Let $\mathcal{H}_{RCv}$ be a class of functions  such that 
	\begin{eqnarray}\label{HRCv}
	\mathcal{H}_{RCv}  =  \Big \{ h: h(\bm x, y)  =  y d_1(\bma^T  \bm x )  - d_2(\bma^T  \bm x),  \  (\bm x, y, \bma) \in \mathcal{X} \times \RR \times  \mathcal{B}(\bma_0, \delta_0) , (d_1, d_2)  \in \mathcal{D}^2_{RC v}   \Big \}  
	\end{eqnarray}
	where $C \ge K_0 \vee 1$. 
	\medskip
	Then,  for all $\epsilon \in (0, C)$  we have that 
	\begin{eqnarray*}
		H_B\Big(\epsilon,  \widetilde{\mathcal{H}},  \Vert \cdot \Vert_{B, P_0} \Big)  \le  H_B\Big(\epsilon \tilde{C}^{-1}, \mathcal{H}_{RCv},  \Vert \cdot \Vert_{P_0}\Big) \le \frac{\tilde{C} C}{\epsilon}   \asymp \frac{1}{\epsilon},
	\end{eqnarray*}
	and that
	\begin{eqnarray*}
		\Vert \tilde{h} \Vert_{B,P_0}  \lesssim \tilde{D}^{-1} v  ,
	\end{eqnarray*}
	where 
	\begin{eqnarray}\label{constants}
	A'  = A\Big(2(a_0M_0  +1)\Big)^{-1/2},  \ \ \tilde{D} = 16 M_0 C  \  \ \textrm{and} \ \ \tilde{C}  = \frac{1}{8M_0} \left( 2 a_0   +  \frac{1}{2}  e^{(2M_0)^{-1}}  \right)^{1/2} \ \frac{1}{C},
	\end{eqnarray}
	with $a_0, M_0$ the same constants from Assumption A6, $A$  the same constant in Lemma \ref{lemma:EntropyGRK}, and $\widetilde{\mathcal{H}}   {=}  \mathcal{H}_{RCv} \tilde{D}^{-1}$. 
	
\end{lemma}

\medskip

\begin{proof}
	Consider $(d^L_1, d^U_1)$ and  $(d^L_2, d^U_2)$ to be  $\epsilon$-brackets of the functions $\bm x \mapsto d_1(\bma^T  \bm x)$ and $\bm x \mapsto d_2(\bma^T  \bm x )$ and some $\bma \in \mathcal{B}(\bma_0, \delta_0)$.  It follows from Lemma  4.9 of \cite{balabdaoui2016} and Lemma \ref{lemma:EntropyDiff} that there exists some constant $A > 0$ such that 
	\begin{eqnarray*}
		H_B\Big(\epsilon,  \mathcal{D}_{RC}, \Vert \cdot \Vert_{P_0} \Big)  \le \frac{A C}{\epsilon}.
	\end{eqnarray*}
	Define now
	\begin{eqnarray*}
		h^L(\bm x, y) =  
		\begin{cases}
			y d^L_1(\bm x)  - d^U_2(\bm x), \  \ \textrm{if $y \ge 0 $}  \\
			y d^U_1(\bm x)  - d^U_2(\bm x), \  \ \textrm{if $y < 0 $} 
		\end{cases}
	\end{eqnarray*}
	and 
	\begin{eqnarray*}
		h^U(\bm x, y) =  
		\begin{cases}
			y d^U_1(\bm x)  - d^L_2(\bm x), \  \ \textrm{if $y \ge 0 $}  \\
			y d^L_1(\bm x)  - d^L_2(\bm x), \  \ \textrm{if $y < 0 $}. 
		\end{cases}
	\end{eqnarray*}
	Note first that $(h^L, h^U)$ is a bracket for $h(\bm x, y)  =  y    d_1(\bma^T  \bm x )  - d_2(\bma^T  \bm x)$. Next we compute the size of this bracket with respect to $\Vert \cdot \Vert_{P_0}$.  We have that 
	\begin{eqnarray*}
		\int_{\mathcal{X}  \times \RR} \Big( h^U(\bm x, y)  -  h^L(\bm x,y)  \Big)^2 dP_0(\bm x, y)  & \le  &  2 \ \Big \{ \int_{\mathcal{X}  \times \RR}    y^2 \Big(d^U_1(\bm x)   -  d^L_1(\bm x) \Big)^2 dP_0(\bm x, y)  +  \int_{\mathcal{X}}  \Big(d^U_2(\bm x)   -  d^L_2\bm x) \Big)^2    dG(\bm x)  \Big \}   \\
		& = & 2    \Big \{ 2 a_0   \  \int_{\mathcal{X}}  \Big(d^U_1(\bm x)   -  d^L_1(\bm x) \Big)^2    dG(\bm x)   +  \int_{\mathcal{X}}  \Big(d^U_2(\bm x)   -  d^L_2\bm x) \Big)^2    dG(\bm x)  \Big \}  \\
		& \le &  2 \big( 2 a_0  +1 \big)   \epsilon^2  ,
	\end{eqnarray*}
	where $a_0$ is the same constant of Assumption A6.   It follows that  
	\begin{eqnarray*}
		H_B\Big(\epsilon,\mathcal{H}, \Vert \cdot \Vert_{P_0}  \Big)  \le \frac{\tilde{A} C}{\epsilon},
	\end{eqnarray*}
	with $\tilde{A} = A \Big(2( 2a_0  +1)\Big)^{-1/2} $ and $A$ is the same constant of Lemma \ref{lemma:EntropyGRK}.    Let now $D > 0$ be some constant to be determined later.  For a given $h \in \mathcal{H}_{RK^2v}$, we consider $\tilde{h}   =  D^{-1}  h$ which admits $[D^{-1} h^L, D^{-1}  h^U]$ as bracket. We will compute the size  of this bracket with respect to the Bernstein norm. By definition of the latter we can write for any function $h$ such that $h^k$ is $P_0$ integrable that 
	\begin{eqnarray*}
		\Vert h \Vert^2_{B, P_0}     =  2 \sum_{k=2}^\infty  \frac{1}{k!}  \vert h \vert^k dP_0.
	\end{eqnarray*}
	Thus, using this and convexity of the function $x \mapsto \vert x \vert^{k}$  for all $k \ge 2$ it follows that
	\begin{align*}
	&\Vert  D^{-1} h^U   - D^{-1}  h^L  \Vert^2_{B, P_0}  =  2 \sum_{k=2}^\infty \frac{1}{k! D^k}    \int_{\mathcal{X} \times \RR}  \Big \vert y \big(d^U_1(\bm x)  -  d^L_1(\bm x)\big)  + d^U_2(\bm x)  -  d^L_2(\bm x)  \Big \vert^k dP_0(\bm x, y)\\
	&\qquad\le   2 \sum_{k=2}^\infty \frac{2^{k-1}}{k! D^k}   \bigg \{  \int_{\mathcal{X} \times \RR}   \vert y \vert^k \big(d^U_1(\bm x)  -  d^L_1(\bm x)\big)^k  dP_0(\bm x, y) + \int_{\mathcal{X} \times \RR}  \big (d^U_2(\bm x)  -  d^L_2(\bm x)  \big)^k dP_0(\bm x, y) \bigg \}.
	\end{align*}
	Using Assumption A7  and the fact that $\vert d^L_i \vert \le K^2$ and  $\vert d^U_i \vert \le 2 C$  for $ i \in \{1, 2\}$ (an assumption that one can always make in  constructing brackets for a bounded class)  we can write 
	\begin{eqnarray*}
		\Vert  D^{-1} h^U   - D^{-1}  h^L  \Vert^2_{B, P_0}  &  \le &   \sum_{k=2}^\infty \frac{1}{k!} \left(\frac{2}{D} \right)^k  \bigg \{ a_0 M^{k-2}_0  k!  (4C)^{k-2}  \int_{\mathcal{X}} \big(d^U_1(\bm x)  -  d^L_1(\bm x)\big)^2  dP_0(\bm x, y)   \\
		&& \hspace{1.8cm} +  (4C)^{k-2}   \int_{\mathcal{X}} \big(d^U_1(\bm x)  -  d^L_1(\bm x)\big)^2  dP_0(\bm x, y) \bigg \}  \\
		&  = & \left(\frac{2}{D} \right)^2  \bigg \{ a_0  \sum_{k=2}^\infty \left(\frac{8M_0 C}{D}\right)^{k-2}  +   \sum_{k=2}^\infty  \frac{1}{k!} \left(\frac{8C}{D}\right)^{k-2} \bigg \} \ \epsilon^2  \\
		& \le &  \left(\frac{2}{D} \right)^2  \bigg \{ a_0  \sum_{k=0}^\infty \left(\frac{8M_0 C}{D}\right)^{k}  +  \frac{1}{2} \sum_{k=0}^\infty  \frac{1}{k!} \left(\frac{8C}{D}\right)^{k} \bigg \} \ \epsilon^2,  \  \ \textrm{using $k! \ge 2 (k-2)!$}.
	\end{eqnarray*}
	Taking $D =\tilde{D} =  16 M_0 C$ yields
	\begin{eqnarray*}
		\Vert  \tilde{D}^{-1} h^U   - \tilde{D}^{-1}  h^L  \Vert^2_{B, P_0}   \le \left(\frac{2}{\tilde{D}} \right)^2  \left( 2 a_0   +  \frac{1}{2}  e^{(2M_0)^{-1}}  \right)\ \epsilon^2,
	\end{eqnarray*}
	which in turn implies that
	\begin{eqnarray*}
		\Vert  \tilde{D}^{-1} h^U   - \tilde{D}^{-1}  h^L  \Vert_{B, P_0}   \le  \frac{1}{8M_0} \left( 2 a_0   +  \frac{1}{2}  e^{(2M_0)^{-1}}  \right)^{1/2} \ \frac{1}{C}  \ \epsilon.
	\end{eqnarray*}
	This completes the proof of the first claim about the entropy bound of the class $\widetilde{\mathcal{H}}$ with $\tilde{D}$ defined as above. Now for a given element $\tilde{h} \in \widetilde{\mathcal{H}}$ we calculate
	\begin{eqnarray*}
		\Vert \tilde h \Vert^2_{B, P_0} & = & 2 \sum_{k=2}^\infty \frac{1}{\tilde{D}^k} \frac{1}{k!} \int_{\mathcal{X} \times \RR}  \big \vert y d_1(\bma^T \bm x)  - d_2(\bma^T \bm x) \big \vert^k  dP_0(\bm x, y)\\
		& \le &  2 \sum_{k=2}^\infty \frac{2^{k-1}}{\tilde{D}^k} \frac{1}{k!} \int_{\mathcal{X} \times \RR}  \left \{ \big \vert y  \big \vert^k  \big \vert d_1(\bma^T \bm x)  \big \vert^k + \big \vert d_2(\bma^T \bm x)  \big \vert^k  dP_0(\bm x, y)\right \}  \\
		& \le  &  2 \sum_{k=2}^\infty \frac{2^{k-1}}{\tilde{D}^k} \frac{1}{k!} (2C)^{k-2}  \left \{ a_0 M^{k-2}_0  k!   \int_{\mathcal{X} \times \RR}  \big \vert d_1(\bma^T \bm x)  \big \vert^2   dP_0(\bm x, y)+  \int_{\mathcal{X} \times \RR}   \big \vert d_2(\bma^T \bm x)  \big \vert^2  dP_0(\bm x, y)   \right \} \\
		& \le & \left(\frac{2}{\tilde{D}}\right)^2 \left\{  a_0 \sum_{k=2}^\infty \left(\frac{8M_0 C}{\tilde{D}}\right)^{k-2}   +   \sum_{k=2}^\infty \frac{1}{k!} \left(\frac{8C}{\tilde{D}} \right)^{k-2}\right \}  \  v^2 \ \  \textrm{using the definition of the class}\\
		& \le & \left(\frac{2}{\tilde{D}} \right)^2  \left( 2 a_0   +  \frac{1}{2}  e^{(2M_0)^{-1}}  \right) \ v^2 \  \ \textrm{using arguments as above,} 
	\end{eqnarray*}
	implying that 
	\begin{eqnarray*}
		\Vert \tilde h \Vert_{B, P_0}  \le 2 \left( 2 a_0   +  \frac{1}{2}  e^{(2M_0)^{-1}}  \right)^{1/2}   \frac{1}{\tilde{D}} \ v \lesssim \tilde{D}^{-1} v,
	\end{eqnarray*} 
	as claimed.
\end{proof}

\bigskip

Recall that $\mcX$ is the support of the covariates $X_i, \ i=1, \ldots, n$. Let us denote by $\mcX_j, \ j =1, \ldots, d$ the set of the $j$-th projection of $\bm x \in \mcX$. Also, consider some function $s$ that $d-1$ times continuously differentiable on a convex and bounded set $\mathcal{C} \in \RR^{d-1}$ with a nonempty interior such that there exists $M > 0$ satisfying 
\begin{eqnarray}\label{DiffAssump2}
\max_{k. \le d-1} \sup_{\bmb \in \mathcal{C} } \vert D^{k} s(\bmb)  \vert  \le M
\end{eqnarray}
where $k = (k_1, \ldots, k_d)$  with $k_j$ an integer $\in \{0, \ldots, d-1 \}$, $k. = \sum_{i=1}^{d-1} k_i$ and 
\begin{eqnarray*}
	D^k \equiv \frac{\partial^{k.} s(\bmb)}{\partial \beta_{k_1} \ldots  \partial \beta_{k_d}}.
\end{eqnarray*}
Consider now the class 
\begin{eqnarray}\label{QjRC}
\mathcal{Q}_{jRC}  =\Big \{q_j(\bm x, y)  = s(\bm \beta)  x_j (y - \psi(\bma^T \bm x)), \  (\bma , \bmb, \psi) \in \mathcal{B}(\alpha_0, \delta_0)  \times  \mathcal{C} \times \mathcal{M}_{RC}  \ \textrm{and} \  (x_j, y) \in \mcX_j \times \in \RR    \Big \}. 
\end{eqnarray}
Define    
\begin{eqnarray*}
	\widetilde{\mathcal{Q}}_{jRC}  =  \Big \{\tilde{q}_j: \tilde{q}_j= q_j \tilde{D}^{-1},  q_j \in \mathcal{Q}_{RC}   \Big\},
\end{eqnarray*}
where $\tilde{D}   > 0$ is some appropriate constant.

\medskip
\begin{lemma}\label{lemma: EntropyNew}
	Let $\epsilon \in (0,1)  $ and  $C  \ge \max(1, 2M_0,  M e^{-1/4} 2^{-1/2} R^{-1}, 2a^{1/2}_0 e^{-1/2})  $. Then, there exist some constant $B_1 > 0$  and $B_2$  depending on $a_0$, $M_0$,  and $R$   such that 
	\begin{eqnarray*}
		H_B\Big(\epsilon, \widetilde{\mathcal{Q}}_{jRC}, \Vert \cdot \Vert_{B, P_0} \Big)  \le \frac{B_1  C}{  \epsilon}, \ \  \Vert \tilde{q}_j \Vert_{B, P_0}  \le B_2,
	\end{eqnarray*}
	if $\tilde{D} = 8 MR C$ where $a_0$ and $M_0$ are the same positive constants in Assumption A6, and $M$ is from (\ref{DiffAssump2}).
\end{lemma}

\medskip

\begin{proof}
	Fix $ j \in \{1, \ldots, d \}$. The proof  of this lemma uses similar techniques as in showing Lemma \ref{lemma: BernEntropyGeneral}.  Let $(g^L, g^U)$ be $\epsilon$-brackets for the class $\mathcal{G}_{RC}$.   Using the result of  Lemma \ref{lemma:EntropyGRK}  we know that there are at most $N \le \exp(A C/\epsilon)$ such brackets covering $\mathcal{G}_{RC}$ for some constant $A > 0$. Define 
	\begin{eqnarray}\label{kLkU}
	\Big(k^L(\bm x, y), k^U(\bm x, y)  \Big)  =  \left \{
	\begin{array}{ll}
	\Big( x_j  (y -  g^L(\bm x) ),  x_j  (y -  g^U(\bm x) )  \Big), \  \textrm{if $x_j \ge 0 $  }  \\
	\Big( x_j  (y -  g^U(\bm x) ),  x_j  (y -  g^L(\bm x) )  \Big), \  \textrm{if $x_j < 0 $  }.
	\end{array}
	\right. 
	\end{eqnarray}
	Then, the collection of all possible pairs $(q^L, q^U)$ form brackets for the class of functions 
	\begin{eqnarray*}
		\mathcal{K}_{jRC} \equiv \Big \{k_j(\bm x, y)  =  x_j (y - \psi(\bma^T \bm x)), \  (\bma , \psi) \in \mathcal{B}(\alpha_0, \delta_0) \times \mathcal{M}_{RC}  \ \textrm{and} \  (x_j,\bm x,  y) \in \mcX_j \times \mcX \times \RR    \Big \}.
	\end{eqnarray*}
	Furthermore we have that 
	\begin{eqnarray*}
		\Vert k^U   - k^L \Vert^2_{P_0}  &=  &  \int_{\mcX }  x^2_j  \big(g^U(\bm x)   - g^L(\bm x) \big)^2 dG(\bm x) \\
		& \le   &  \Vert \bm x \Vert^2_2   \  \int_{\mcX } \big(g^U(\bm x)   - g^L(\bm x) \big)^2 dG(\bm x) \le R^2 \epsilon^2.
	\end{eqnarray*}
	This implies that 
	\begin{align*}
	H_B\Big(\epsilon, \mathcal{K}_{jRC}, \Vert \cdot \Vert_{P_0} \Big)  \le \frac{ARC}{\epsilon},
	\end{align*}
	where $A$ is the same constant of  Lemma \ref{lemma:EntropyGRK}.  Furthermore, the assumption in (\ref{DiffAssump})  implies that the function $s$ belongs to $C^{d-1}_{\tilde{M}}$ as defined in Section 2.7 in \cite{vdvwe:96},  with $\tilde{M} = 2M$. Using now Theorem 2.7.1 of \cite{vdvwe:96} it follows that there exists some constant $B > 0$ such that 
	\begin{eqnarray*}
		\log N\Big(\epsilon, C^{d-1}_{\tilde{M}}, \Vert \cdot \Vert_\infty \Big)  \le B  \left(\frac{1}{\epsilon}\right)^{d/(d-1)} \le    \frac{B}{\epsilon}.
	\end{eqnarray*}
	This also implies that 
	\begin{eqnarray*}
		H_B\Big(\epsilon, C^{d-1}_{\tilde{M}}, \Vert \cdot \Vert_\infty \Big) = \log N\Big(\epsilon/2, C^{d-1}_{\tilde{M}}, \Vert \cdot \Vert_\infty \Big)  \le  \frac{2 B}{\epsilon}.
	\end{eqnarray*}
	Indeed, for an arbitrary $s \in  C^{d-1}_{\tilde{M}}$ there exists $s_i, i \in \{1, \ldots, N \}$,  with $ N = N\big(\epsilon/2, C^{d-1}_{\tilde{M}}, \Vert \cdot \Vert_\infty \big)$, such that $ \Vert s - s_i \Vert_\infty \le \epsilon/2$. The claim follows from noting that $(s_i - \epsilon/2, s_i + \epsilon/2)$ is an $\epsilon$-bracket for  $ C^{d-1}_{\tilde{M}}$ with respect to $\Vert \cdot \Vert_\infty$.  Using Lemma \ref{EntropyF1F2} it follows that there exists some constant $L > 0$ such that 
	\begin{align}\label{EntropyP0}
	H_B\Big(\epsilon, \mathcal{Q}_{jRC}, \Vert \cdot \Vert_{P_0} \Big)  \le     L \left(\frac{1}{\epsilon}  +\frac{C}{\epsilon} \right)  \le    \frac{2L C}{\epsilon},
	\end{align}
	using that $C \ge 1$, $d-1 \ge 1$ and $\epsilon \in (0,1)$.  Consider now a constant $D > 0$, and $(q^L, q^U)$ and $\epsilon$-bracket.   From the proof of Lemma \ref{EntropyF1F2} we know that we can restrict attention to the case for example to case 1 assumed to occur for all $(\bm x, \bmb) \in \mcX \times \mathcal{C}$.  In such that we have $q^L = s^L  k^L $ and $q^U = s^U k^U$ where $(s^L, s^U)$ is an $\epsilon$-bracket for $C^1_{\tilde{M}}$equipped with $\Vert \cdot \Vert_\infty$, where the expression of $(k^L, k^U)$ is given in (\ref{kLkU}).  We can now write 
	\begin{eqnarray*}
		\Vert D^{-1}  q^U - D^{-1}  q^L \Vert^2_{B, P_0} &= & 2 \sum_{k=2}^\infty \frac{1}{k!} \frac{1}{D^k} \int_{\mcX \times \RR} \big \vert  s^U k^U   -  s^L k^L \big \vert^k   dP_0 \\
		& \le &   \sum_{k=2}^\infty \frac{2^k}{k!} \frac{1}{D^k} \int_{\mcX \times \RR}  \bigg \{ \big \vert s^U  \big (k^U -  k^L \big) \vert^k    +  \big  \vert k^L \big( s^U - s^L \big) \big \vert^k \bigg \} dP_0 \\
	\end{eqnarray*}
	with 
	\begin{eqnarray*}
		\int_{\mcX \times \RR}  \big \vert s^U  \big (k^U -  k^L \big) \big \vert^k dP_0  \le M^k (2RC)^{k-2}  \int_{\mcX \times \RR} \big (k^U -  k^L \big)^2 dP_0   = M^2 (2MC R)^{k-2}  \epsilon^2 
	\end{eqnarray*}
	where we used the fact that $\vert s \vert \le M $ by assumption of the lemma  (implying that we can constructs brackets $(s^L, s^U)$ satisfying the same property), and $k^U - k^U =  x_j (g^U- g^L) \le 2R C$. Also, if we assume without loss of generality that $x_j \ge 0$ is satisfied for all $\bm x \in \mcX$ we have that 
	\begin{eqnarray*}
		\int_{\mcX \times \RR}  \big  \vert k^L \big( s^U - s^L \big) \big \vert^k  dP_0  &\le &  (2M)^{k-2}  \int_{\mcX \times \RR}  \big  \vert x_j (y - g^L(\bm x))  \vert^k dP_0(\bm x, y)  \times \epsilon^2\\
		& \le &  (2M)^{k-2}  R^k 2^{k-1}  \int_{\mcX \times \RR}  \Big \{ \vert y \vert^k + \left \vert g^L(\bm x ) \right \vert ^k \Big \}  dP_0(\bm x, y)  \times \epsilon^2 \\
		& \le &  (2M)^{k-2}  R^k 2^{k-1}  \Big (  a_0  M^{k-2}_0  k!  +  C^k \Big)  \epsilon^2.
	\end{eqnarray*}
	Putting these inequalities together and after some algebra we get  
	\begin{align*}
	&\Vert D^{-1}  q^U - D^{-1}  q^L \Vert^2_{B, P_0}\\
	&\le  \left(\frac{1}{2} \left( \frac{2M}{D}\right)^{2} e^{4MCR/D}  +  \left(\frac{2RC}{D}\right)^2  e^{8MCR/D}   + 2 a_0  \left( \frac{2R}{D}\right)^{2} \frac{1}{1 - 8MM_0R/D}) \right) \epsilon^2.
	\end{align*}
	Now let us choose $\tilde{D} = D \ge \max(16 M M_0 R, 8 MR C)$. In particular, we can assume that $C$ is large enough so that $\max(16 M M_0 R, 8 MR C)  = 8 MR C  = \tilde{D}$ (or equivalently $C \ge 2 M_0$).   Then, $4MCR/\tilde{D} = 1/2$, $8MCR/\tilde{D} = 1/4 $, and $8MM_0R/\tilde{D} = M_0/C \le 1/2$. Therefore,  
	\begin{eqnarray*}
		\Vert \tilde{D}^{-1}  q^U - \tilde{D}^{-1}  q^L \Vert^2_{B, P_0}  &\le  & \left(\frac{1}{2} \left( \frac{2M}{\tilde{D}}\right)^{2} e^{1/2}  +  \left(\frac{2RC}{\tilde{D}}\right)^2  e   + 4 a_0  \left( \frac{2R}{\tilde{D}}\right)^{2}  \right)  \ \epsilon^2\\
		& = &  \left(2  M^2 e^{1/2}  + 4 R^2 e  \ C^2   + 16 a_0   R^2 \right)  \  \frac{1}{\tilde{D}^2}  \ \epsilon^2  \\
		&  \le &  \frac{\tilde{A} C^2}{\tilde{D}^2}  \  \epsilon^2 =  \frac{\tilde{A}}{64 M^2 R^2} \epsilon^2, 
	\end{eqnarray*}
	if $C$ is large enough, where $\tilde{A} = 2M^2 e^{1/2} + 4 R^2 e  + 16 a_0 R^2$.   It follows that we can find some constant $\tilde{L} > 0$ such that 
	\begin{align*}
	\Vert \tilde{D}^{-1}  q^U -\tilde{D}^{-1}  q^L \Vert_{B, P_0} &\le  \tilde{L} \epsilon.
	\end{align*}
	This in turn implies that
	\begin{align*}
	H_B\big(\tilde{L} \epsilon, \widetilde{Q}_{jRC}, \Vert \cdot \Vert_{B, P_0}\big)  &  \le    H_B\big(\epsilon, {Q}_{jRC}, \Vert \cdot \Vert_{P_0}\big)  \lesssim   \frac{ 2MC}{\epsilon}  ,
	\end{align*}
	using (\ref{EntropyP0}). Hence,  we can find a constant $B_1 > 0$ such that 
	$$H_B\big( \epsilon, \widetilde{Q}_{jRC}, \Vert \cdot \Vert_{B, P_0}\big)  \le \frac{B_1 C}{  \epsilon}.  $$
	Now we turn to computing an upper bound for $\Vert \tilde{q}_j \Vert_{B, P_0}$. We have 
	\begin{eqnarray*}
		\Vert \tilde{q}_j\Vert^2_{B, P_0}  &= & 2 \sum_{k=2}^\infty \frac{1}{k!}  D^{-k} \int_{\mcX \times \RR}  \vert s(\bmb) \vert^k \big \vert x_j \big( y - \psi(\bma^T \bm x) \big)  \big \vert^k dP_0(\bm x, y) \\
		& \le &  \sum_{k=2}^\infty \frac{1}{k!}  2^k D^{-k} (RM)^k \int_{\mcX \times \RR}  \Big \{\big \vert y \big \vert^k  +  \big \vert \psi(\bma^T \bm x)  \big \vert^k \Big \} dP_0(\bm x, y) \\
		& \le &  \sum_{k=2}^\infty \frac{1}{k!} 2^k D^{-k} (RM)^k \Big(a_0 M^{k-2}_0  k!  +  C^k \Big)  \\
		& \le &  a_0 \left(\frac{2MR}{D}\right)^2 \sum_{k=2}^\infty \left(\frac{2RMM_0}{D}\right)^{k-2}  + \frac{1}{2} \left(\frac{2MRC}{D}\right)^2 \sum_{k=2}^\infty \frac{1}{(k-2)!}  \left( \frac{2RMC}{D}  \right)^{k-2}\\
		&\le  &a_0\left(\frac1{2C}\right)^2 \sum_{k=0}^\infty \left(\frac1{4}\right)^k + \frac 12 \left(\frac1{2}\right)^2 \sum_{k=0}^\infty \frac1 {k!}\left(\frac1{2}\right)^k   \\
		&\le &  a_0\left(\frac1{2C}\right)^2 \frac 34 + \frac 12 \left(\frac1{2}\right)^2 e^{1/2} \qquad 
		\textrm{if $D= 4MR C$ and $C \ge \max(1,2 M_0)$}.
	\end{eqnarray*} 
	The proof of the lemma is complete if we write $B_2 = (3a_0/16+  e^{1/2}/8)^{1/2}$. 
	
\end{proof}

\bigskip

In the next lemma, we consider a given a class of functions $\mathcal{F}$ which admits a bounded bracketing entropy with respect to  $\Vert \cdot \Vert_{P_0}$ for $\epsilon \in (0, 1]$. Suppose also that there exists $D > 0$ such that $\Vert f \Vert_\infty \le D$  and $\delta > 0$ such that $\Vert f \Vert_{P_0}   \le \delta$  for all $f \in \mathcal{F}$. Then we can derive an upper bound for the bracketing entropy for the class
\begin{eqnarray}\label{tildeF}
\widetilde{\mathcal{F}}  & = & \Big \{\tilde{f}:  \tilde{f}(\bm x,y)  =  (4M_0 D)^{-1} f(\bm x)  \Big (y - \lambda \psi_0(\bma^T_0  \bm x)  \Big),  (\bm x, y) \in \mathcal{X} \times \RR \  \textrm{and}  \  f \in \mathcal{F}  \Big \}  \notag \\
&&
\end{eqnarray}
with respect to the Bernstein norm. Here, $M_0$ is the same constant from Assumption A6 and $ \tilde{D}$ is a positive constant that will be determined below.

\medskip

\begin{lemma}\label{Bernsteinfixed}
	Let $\mathcal{F}$ be a class of functions satisfying the conditions above. Then,
	\begin{eqnarray*}
		H_B(\epsilon, \widetilde{\mathcal{F}},  \Vert \cdot \Vert_{B, P_0})  \le  H_B(\epsilon \tilde{D}^{-1}, \mathcal{F},  \Vert \cdot \Vert_{P_0}), \  \ \textrm{and} \  \  \Vert \tilde{f} \Vert_{B, P_0}  \le \tilde{D} \delta
	\end{eqnarray*}
	where 
	\begin{eqnarray}\label{tildeD}
	\tilde{D}  =\left(\frac{ a_0}{2 M^2_0}   +  \frac{\lambda^2 K^2_0}{8M^2_0} e^{\lambda K_0  (2M_0)^{-1}} \right)  ^{1/2}  D^{-1} ,
	\end{eqnarray}
	and $a_0, M_0$ are the same constants from Assumption A6.
\end{lemma}

\bigskip

\begin{proof}
	Let $(L, U)$ be an $\epsilon$-bracket for $\mathcal{F}$ with respect to $\Vert \cdot \Vert_{P_0}$. Consider the class 
	$$
	\mathcal{F}' = \Big \{ f': f'(\bm x, y)=f(\bm x)  \Big (y - \lambda \psi_0(\bma_0^T\bm x)  \Big),  (\bm x, y) \in \mathcal{X} \times \RR \  \textrm{and}  \  f \in \mathcal{F}  \Big \}. 
	$$
	Then for $f' \in \mathcal{F}'$ we have 
	\begin{eqnarray*}
		L(\bm x)   (y - \lambda \psi_0(\bma^T_0 \bm x) )  \le f'(\bm x, y)  \le   U(\bm x)   (y - \lambda \psi_0(\bma^T_0  \bm x) ),  \  \textrm{if   $y - \lambda \psi_0(\bma^T_0\bm   x)   \ge 0 $ \  or }  \\
		U(\bm x)   (y - \lambda \psi_0(\bma^T_0  \bm x) )  \le f'(\bm x, y) \le   L(\bm x)   (y - \lambda \psi_0(\bma^T_0 \bm x) ),  \  \textrm{if   $y - \lambda \psi_0(\bma^T_0 \bm x)   < 0 $ }.
	\end{eqnarray*}
	Let $(L', U')$ denote the new bracket.  Using the definition of the Bernstein norm, convexity of  $x \mapsto x^k,  \ k \ge 2$ and $\Vert \psi_0 \Vert_\infty \le K_0$  we have that  
	\begin{eqnarray*}
		&& \big \Vert( U'  -  L')  (4M_0 D)^{-1} \big \Vert^2_{B, P_0} \\
		&& =   2  \sum_{k=2}^\infty  \frac{(4M_0 D)^{-k}}{k!}  \int_{\mathcal{X} \times \RR}  (U(\bm x)   - L(\bm x))^k   \vert y - \lambda \psi_0(\bma_0^T\bm x)  \vert^k  dP_0(\bm x, y) \\
		&& \le 2  \sum_{k=2}^\infty  \frac{(4M_0 D)^{-k}}{k!}  \int_{\mathcal{X} \times \RR}  (U(\bm x)   - L(\bm x))^k   2^{k-1} \Big( \vert y  \vert^k + \lambda^k  \vert  \psi_0(\bma_0^T\bm x) \vert^k  \Big) dP_0(\bm x, y) \\
		&& \le   \sum_{k=2}^\infty  \frac{1}{2^k D^k M^k _0k!}   \int_{\mathcal{X} \times \RR}  (U(\bm x)   - L(\bm x))^k   \big (\vert y  \vert^k +  \lambda^k K^k_0  \big)  dP_0(\bm x, y) \\
		&&  \le    \sum_{k=2}^\infty  \frac{1}{ 2^k D^k M^k_0 k!}   \int_{\mathcal{X} \times \RR}  (U(\bm x)   - L(\bm x))^k   (a_0  k!  M^{k-2}_0 + \lambda^k  K^k_0 )  d P_0(\bm x, y) \\
		&& \le   \frac{a_0}{4M^2_0  D^2}  \sum_{k=2}^\infty  \frac{1}{2^{k-2}}  \int_{\mathcal{X} \times \RR} (U(\bm x)   - L(\bm x))^2   g(\bm x) dx     + \frac{\lambda^2 K^2_0}{4D^2M^2_0} \sum_{k=2}^\infty \left(\frac{\lambda K_0}{2M_0}\right)^{k-2} \frac{1}{k!}  \int_{\mathcal{X} \times \RR} (U(\bm x)   - L(\bm x))^2   g(\bm x) d\bm x      \\
		&& \le \frac{a_0}{2M^2_0  D^2}   \epsilon^2  +  \frac{\lambda^2 K^2_0}{8D^2M^2_0} \sum_{k=2}^\infty \frac{1}{(k-2)!} \left(\frac{\lambda K_0}{2M_0}\right)^{k-2}   \ \epsilon^2 \\
		&&  \le   \left(\frac{ a_0}{2 M^2_0  D^2}   +  \frac{\lambda^2 K^2_0}{8D^2M^2_0} e^{\lambda K_0  (2M_0)^{-1}} \right)  \epsilon^2 =   \tilde{D}^2  \ \epsilon^2.
	\end{eqnarray*}
	This implies that 
	\begin{align*}
	H_B\Big(\epsilon \tilde{D}, \widetilde{\mathcal{F}}, \Vert \cdot  \Vert_{B, P_0}  \Big) \le H_B\Big(\epsilon , \mathcal{F}, \Vert \cdot  \Vert_{P_0}  \Big),
	\end{align*}
	or equivalently
	\begin{align*}
	H_B\Big(\epsilon , \mathcal{F}', \Vert \cdot  \Vert_{B, P_0}  \Big) \le H_B\Big(\epsilon \tilde{D}^{-1}, \mathcal{F}, \Vert \cdot  \Vert_{P_0}  \Big).
	\end{align*}
	Using similar calculations we can write  
	\begin{eqnarray*}
		\Vert  \tilde{f} \Vert^2_{B, \PP}   &=  & 2  \sum_{k=2}^\infty  \frac{1}{(4M_0D)^k}  \frac{1}{k!}    \int_{\mathcal{X} \times \RR}   \vert f(\bm x) \vert^k  \vert y - \lambda \psi_0(\bma_0^T\bm x)  \vert^k  dP_(\bm x, y)  \\
		& \le &  \frac{1}{D^2}  \sum_{k=2}^\infty  \frac{1}{(2M_0)^k}  \frac{1}{k!}    \int_{\mathcal{X} \times \RR}   f(\bm x)^2 (a_0  k!  M_0^{k-2} + \lambda^k  K^k_0 )   g(\bm x) d \bmx  \\
		&  \le &    \left(  \frac{a_0}{4M_0^2 D^2}  \sum_{k=2}^\infty  \frac{1}{2^{k-2}}      +   \frac{\lambda^2 K^2_0}{8D^2  M_0^2}  \sum_{k=2}^\infty \left(\frac{\lambda K_0}{2M_0} \right)^{k-2} \frac{1}{(k-2)!}  \right)  \int_{\mathcal{X} \times \RR}  f(\bm x)^2   g(\bm x)  d\bm x \\
		& \le &  \tilde{D}^2  \delta^2,
	\end{eqnarray*}
	which completes the proof. 
	
\end{proof}

\bigskip

In the next corollary, we consider the class 
\begin{eqnarray*}
	\mathcal{F} = \Big \{x \mapsto f_{\bma}(\bm x)  = E_{i,\bma_0}(\bma^T_0 \bm x) -  E_{i,\bma}(\bma^T\bm x), \  \bm x \in \mathcal{X}, \bma \in \mathcal{B}(\bma_0, \delta)   \Big \},
\end{eqnarray*}
where $E_{i,\bma}(u) =\E\left\{X_i | \bma^T\bmX = u \right\}$ for $i \in \{1, \ldots, d \}$ and $\delta \in (0, \delta_0)$.  Using the same arguments in the proof of Lemma \ref{lem: continuity} with $f(\bm x)= x_i$ it follows that for all $x \in \mathcal{X}$  and $\bma, \bma' \in \mathcal{B}(\bma_0, \delta)$
\begin{eqnarray*}
	\vert f_{\bma'}(\bm x)  -   f_{\bma}(\bm x)  \vert \le  M \Vert \bma' - \bma \Vert ,
\end{eqnarray*} 
for the same constant $M$ of that lemma. Now, we can apply Theorem 2.7.11 of \cite{vdvwe:96} to conclude that 
\begin{eqnarray*}
	N_B\big(2 \epsilon , \mathcal{F}, \Vert \cdot \Vert_{P_0} \big)  \le N\big(\epsilon, \mathcal{B}(\bma_0, \delta), \Vert \cdot \Vert \big),
\end{eqnarray*}  
where $N\big(\epsilon,\mathcal{B}(\bma_0, \delta), \Vert \cdot \Vert \big)$ is the $\epsilon$-covering number for  $\mathcal{B}(\bma_0, \delta)$ with respect to the norm $\Vert \cdot \Vert$ which is of order $(\delta / \epsilon)^d$  for $\epsilon \in (0, \delta)$. Hence, using the inequality $\log(\bm x) \le x$ for $x > 0$ we can find a constant $M' > 0$  depending on $d$  such that 
\begin{eqnarray*}
	H_B\big(\epsilon, \mathcal{F}, \Vert \cdot \Vert_{P_0} \big)  \le \frac{M' \delta } {\epsilon}. 
\end{eqnarray*}  
Furthermore, there exists $\tilde{M} >0$ such that $\Vert f \Vert_\infty \le \tilde{M} \delta$ and $\Vert f \Vert_{P_0} \le \tilde{M} \delta$.

\medskip

\begin{lemma}\label{CorBernsteinfixed}
	Let $\mathcal{F}$ be the class of functions as above and consider the related class 
	\begin{eqnarray}\label{F'}
	\mathcal{F}' = \Big \{f': f'(\bm x, y)= f(\bm x)   \Big (y - \lambda \psi_0(\bma^T_0  \bm x)  \Big), (\bm x, y)  \in \mcX \times \RR, f \in \mathcal{F}  \Big\}.
	\end{eqnarray}
	Then, 
	\begin{eqnarray*}
		E[\Vert \mathbb{G}_n \Vert_{\mathcal{F}'  } ]  \lesssim  \delta.
	\end{eqnarray*}
	
\end{lemma}

\medskip

\begin{proof} Note that  for any function $f' \in \mathcal{F}'$ and constant $C > 0$ we have that $\mathbb{G}_n(f' C^{-1})  = C^{-1}  \mathbb{G}_n f'$ implying that  $ \Vert \mathbb{G}_n \Vert_{\mathcal{F}' } = 4 M_0 \tilde{M} \delta   \Vert \mathbb{G}_n \Vert_{\widetilde{\mathcal{F}}}$, where 
	\begin{eqnarray*}
		\widetilde{\mathcal{F}}  =\Big \{\tilde{f}: \tilde{f}(\bm x, y)=  (4M_0 \tilde{M}\delta)^{-1} f'(\bm x, y), f' \in \mathcal{F}'  \Big\}.
	\end{eqnarray*}
	Note also that the constant $\tilde{D}$ in  Lemma \ref{Bernsteinfixed} is given by $\tilde{D}  \asymp \delta^{-1}$, where $\tilde{D}$  depends on $\tilde{M}$, $a_0$, $M_0$ and $K_0$. Also, using the entropy calculations along with Lemma \ref{Bernsteinfixed} we can show easily that 
	\begin{eqnarray*}
		H_B(\epsilon, \widetilde{\mathcal{F}}, \Vert \cdot \Vert_{B, P_0}  )  \lesssim  \frac{1}{\epsilon},
	\end{eqnarray*}
	and that  $\Vert \tilde{f} \Vert_{B, P_0}  \lesssim 1$. Using Lemma 3.4.3 of \cite{vdvwe:96} it follows that   there exists some constant $B > 0$ such that 
	\begin{eqnarray*}
		E[\Vert \mathbb{G}_n \Vert_{\widetilde{\mathcal{F}}  } ] \lesssim J_n  \left(1 +   \frac{J_n}{\sqrt{n}  B^2} \right),
	\end{eqnarray*}
	with $J_n = \int_0^B \sqrt{1  +  B/\epsilon} d\epsilon$. Hence, $E[\Vert \mathbb{G}_n \Vert_{\widetilde{\mathcal{F}}}] \lesssim 1$ and   $E[\Vert \mathbb{G}_n \Vert_{\mathcal{F}'  } ]  \lesssim \delta$ as claimed.
\end{proof}

\section{Auxiliary results}
\label{supE:AuxiliaryResults}
\begin{proof}[Proof of Lemma \ref{lemma:Moore-Penrose}]
	We have:
	\begin{align*}
	\left(\bm J_{\mbS}(\bmb_0)\right)^T{\bm A}\bm J_{\mbS}(\bmb_0)\left\{\left(\bm J_{\mbS}(\bmb_0)\right)^T{\bm A}\bm J_{\mbS}(\bmb_0)\right\}^{-1}\left(\bm J_{\mbS}(\bmb_0)\right)^T{\bm A}\bm J_{\mbS}(\bmb_0)=\left(\bm J_{\mbS}(\bmb_0)\right)^T{\bm A}\bm J_{\mbS}(\bmb_0).
	\end{align*}
	In the parametrization that we consider, the columns of $\bm J_{\mbS}(\bmb_0)$ are orthogonal to $\bm\a_0$. We can therefore extend the matrix $\bm J_{\mbS}(\bmb_0)$ with a last column $\bma_0$ to a square nonsingular matrix $\bar{\bm J}_{\mbS}(\bmb_0)$. This leads to the equality
	\begin{align*}
	\left(\bm {\bar J}_{\mbS}(\bmb_0)\right)^T{\bm A}\bm J_{\mbS}(\bmb_0)\left\{\left(\bm J_{\mbS}(\bmb_0)\right)^T{\bm A}\bm J_{\mbS}(\bmb_0)\right\}^{-1}\left(\bm J_{\mbS}(\bmb_0)\right)^T{\bm A}\bar{\bm J}_{\mbS}(\bmb_0)=\left(\bar{\bm J}_{\mbS}(\bmb_0)\right)^T\bm A\bar{\bm J}_{\mbS}(\bmb_0).
	\end{align*}
	Multiplying on the left by $\left(\left(\bar{\bm J}_{\mbS}(\bmb_0)\right)^T\right)^{-1}$ and on the right by  $\bar{\bm J}_{\mbS}(\bmb_0)^{-1}$, we get:
	\begin{align}
	\label{property1}
	{\bm A}\bm J_{\mbS}(\bmb_0)\left\{\left(\bm J_{\mbS}(\bmb_0)\right)^T{\bm A}\bm J_{\mbS}(\bmb_0)\right\}^{-1}\left(\bm J_{\mbS}(\bmb_0)\right)^T{\bm A}=\bm A.
	\end{align}
	This shows that $\bm J_{\mbS}(\bmb_0)\left\{\left(\bm J_{\mbS}(\bmb_0)\right)^T{\bm A}\bm J_{\mbS}(\bmb_0)\right\}^{-1}\left(\bm J_{\mbS}(\bmb_0)\right)^T$ is a generalized inverse of $\bm A$.
	
	To complete the proof and show that it is indeed the Moore-Penrose inverse of $\bm A$, we first note that
	\begin{align}
	\label{property2}
	&\bm J_{\mbS}(\bmb_0)\left\{\left(\bm J_{\mbS}(\bmb_0)\right)^T{\bm A}\bm J_{\mbS}(\bmb_0)\right\}^{-1}\left(\bm J_{\mbS}(\bmb_0)\right)^T{\bm A} \bm J_{\mbS}(\bmb_0)\left\{\left(\bm J_{\mbS}(\bmb_0)\right)^T{\bm A}\bm J_{\mbS}(\bmb_0)\right\}^{-1}\left(\bm J_{\mbS}(\bmb_0)\right)^T \nonumber \\
	&=\bm J_{\mbS}(\bmb_0)\left\{\left(\bm J_{\mbS}(\bmb_0)\right)^T{\bm A}\bm J_{\mbS}(\bmb_0)\right\}^{-1}\left(\bm J_{\mbS}(\bmb_0)\right)^T.
	\end{align}
	Furthermore, 
	\begin{align*}
	&\left(\bm A \bm J_{\mbS}(\bmb_0)\left\{\left(\bm J_{\mbS}(\bmb_0)\right)^T{\bm A}\bm J_{\mbS}(\bmb_0)\right\}^{-1}\left(\bm J_{\mbS}(\bmb_0)\right)^T\right)^T\\
	&=\bm J_{\mbS}(\bmb_0)\left\{\left(\bm J_{\mbS}(\bmb_0)\right)^T{\bm A}\bm J_{\mbS}(\bmb_0)\right\}^{-1}
	\left(\bm J_{\mbS}(\bmb_0)\right)^T\bm A^T\\
	&=\bm J_{\mbS}(\bmb_0)\left\{\left(\bm J_{\mbS}(\bmb_0)\right)^T{\bm A}\bm J_{\mbS}(\bmb_0)\right\}^{-1}
	\left(\bm J_{\mbS}(\bmb_0)\right)^T\bm A,
	\end{align*}
	where the last equality holds since $\bm A$ is symmetric, being a covariance matrix. We have to show that
	\begin{align}
	\label{Hermitian1}
	&\bm J_{\mbS}(\bmb_0)\left\{\left(\bm J_{\mbS}(\bmb_0)\right)^T{\bm A}\bm J_{\mbS}(\bmb_0)\right\}^{-1}
	\left(\bm J_{\mbS}(\bmb_0)\right)^T\bm A\nonumber\\
	&=\bm A	\bm J_{\mbS}(\bmb_0)\left\{\left(\bm J_{\mbS}(\bmb_0)\right)^T{\bm A}\bm J_{\mbS}(\bmb_0)\right\}^{-1}
	\left(\bm J_{\mbS}(\bmb_0)\right)^T.
	\end{align}
	Multiplying on the left by $(\bm J_{\mbS}(\bmb_0))^T$ and on the right by $\bm J_{\mbS}(\bmb_0)$, we get:
	\begin{align*}
	&(\bm J_{\mbS}(\bmb_0))^T\bm J_{\mbS}(\bmb_0)\left\{\left(\bm J_{\mbS}(\bmb_0)\right)^T{\bm A}\bm J_{\mbS}(\bmb_0)\right\}^{-1}
	\left(\bm J_{\mbS}(\bmb_0)\right)^T\bm A\bm J_{\mbS}(\bmb_0)\\
	&=(\bm J_{\mbS}(\bmb_0))^T\bm J_{\mbS}(\bmb_0)\\
	&=(\bm J_{\mbS}(\bmb_0))^T\bm A	\bm J_{\mbS}(\bmb_0)\left\{\left(\bm J_{\mbS}(\bmb_0)\right)^T{\bm A}\bm J_{\mbS}(\bmb_0)\right\}^{-1}
	\left(\bm J_{\mbS}(\bmb_0)\right)^T\bm J_{\mbS}(\bmb_0),
	\end{align*}
	and (\ref{Hermitian1}) follows by the orthogonality relation of the columns of $\bm J_{\mbS}(\bmb_0)$ with $\bma_0$ in the same way as before, replacing the matrix $\bm J_{\mbS}(\bmb_0)$ by $\bm{\bar J}_{\mbS}(\bmb_0)$ in the outer factors of the equality relation.
	
	In a similar way we obtain:
	\begin{align}
	\label{Hermitian2}
	&\left(\bm J_{\mbS}(\bmb_0)\left\{\left(\bm J_{\mbS}(\bmb_0)\right)^T{\bm A}\bm J_{\mbS}(\bmb_0)\right\}^{-1}\left(\bm J_{\mbS}(\bmb_0)\right)^T \bm A\right)^T\nonumber\\
	&=\bm J_{\mbS}(\bmb_0)\left\{\left(\bm J_{\mbS}(\bmb_0)\right)^T{\bm A}\bm J_{\mbS}(\bmb_0)\right\}^{-1}
	\left(\bm J_{\mbS}(\bmb_0)\right)^T\bm A.
	\end{align}	
	Since the matrix $\bm J_{\mbS}(\bmb_0)\left\{\left(\bm J_{\mbS}(\bmb_0)\right)^T{\bm A}\bm J_{\mbS}(\bmb_0)\right\}^{-1}\left(\bm J_{\mbS}(\bmb_0)\right)^T$ satisfies properties (\ref{property1}), (\ref{property2}),(\ref{Hermitian1}) and (\ref{Hermitian2}), the matrix satisfies the four properties which define the Moore-Penrose pseudo-inverse matrix of $\bm A$. This completes the proof of Lemma \ref{lemma:Moore-Penrose}.
\end{proof}
\begin{remark}
{ The same proof holds for showing that the Moore-Penrose inverse $\tilde {\bm A}$ is given by 
		\begin{align*}
		\bm J_{\mbS}(\bmb_0) \left\{\left(\bm J_{\mbS}(\bmb_0)\right)^T{\tilde {\bm A}}\bm J_{\mbS}(\bmb_0)\right\}^{-1}\left(\bm J_{\mbS}(\bmb_0)\right)^T.
		\end{align*}
}
\end{remark}

\begin{lemma}[Derivative $\bma \mapsto \psi_{\bma}(\bma^T\bm x)$]
	\label{lemma:derivative_psi_a}
	$$\frac{\partial }{\partial  \a_j}\psi_{\bma}(\bma ^T\bm x)\Bigm|_{\bma =\bma_0} = \left( x_j - E( X_j |\bm\alpha_0^T\bm X=\bm \a_0^T\bm x)\right)\psi_0'(\bm \a_0^T\bm x),$$
	and
	\begin{align*}
	\frac{\partial }{\partial  \b_j}\psi_{\bma}(\bma ^T\bm x)\Bigm|_{\bma =\bma_0} &= \frac{\partial }{\partial  \b_j}\psi_{\mbS(\bmb)}(\mbS(\bmb)^T\bm x)\Bigm|_{\bmb =\bmb_0} 
	\\
	&= \left(\bm J_{\mbS}(\bmb_0)^T\right)_j\left( \bm x - E(\bm X |\mbS(\bmb_0)^T\bm X =\mbS(\bmb_0)^T\bm x )\right)\psi_0'\left(\mbS(\bmb_0)^T\bm x\right),
	\end{align*}	
	where $\left(\bm J_{\mbS}(\bmb_0)^T\right)_j$ denotes the $j$th row of $\bm J_{\mbS}(\bmb_0)^T$.
\end{lemma}
\begin{proof}
	We assume without loss of generality that the first component $\a_1$ of $\bma$ is not equal to zero. Denote the conditional density of $(X_2,\ldots,X_d)^T$ given $\bm\alpha^T\bm X= u$ by $h_{\bma}(\cdot|u)$	
	Using the change of variables $t_1 = \bm\a^T\bm x$, $t_j = x_j$ for $j=2,\ldots,d$, the function $\psi_{\bm \a}$ can be written as
	\begin{align*}
	&\psi_{\bm\alpha}(\bma^T\bm x)  = \E[\psi_0(\bm\alpha^T_0 \bm X) | \bma^T\bm X = \bma^T\bm x] \\
	&= \int \psi_0\left(\frac{\a_{01}}{\a_1}(\bma^T\bm x -\a_2\tilde x_2 - \ldots - \a_d \tilde x_d) + \sum_{j=2}^{d}\a_{0j}\tilde x_j\right) h_{\bma}(\tilde x_2,\ldots, \tilde x_d|\bma^T\bm x)\prod_{j=2}^{d}d\tilde x_j,
	\end{align*}
	with partial derivatives w.r.t. $\a_j$ for$j=2,\ldots,d$ given by,
	\begin{align*} 
	&\frac{\partial }{\partial  \a_j}\psi_{\bm\alpha}(\bma^T\bm x)=\frac{\partial }{\partial \a_j} \E[\psi_0(\bm\alpha^T_0 \bm X) | \bm\alpha^T\bm X = \bma^T\bm x]\\
	&= \int \frac{\a_{01}}{\a_1} (x_j -\tilde x_j)\psi_0'\left(\frac{\a_{01}}{\a_1}(\bma^T\bm x -\a_2\tilde x_2 - \ldots - \a_d \tilde x_d) + \sum_{j=2}^{d}\a_{0j}\tilde x_j\right) h_{\bma}(\tilde x_2,\ldots, \tilde x_d|\bma^T\bm x)\prod_{j=2}^{d}d\tilde x_j\\
	&\qquad + \int \psi_0\left(\frac{\a_{01}}{\a_1}(\bma^T\bm x -\a_2\tilde x_2 - \ldots - \a_d \tilde x_d) + \sum_{j=2}^{d}\a_{0j}\tilde x_j\right) \frac{\partial }{\partial  \a_j} h_{\bma}(\tilde x_2,\ldots, \tilde x_d|\bma^T\bm x)\prod_{j=2}^{d}d\tilde x_j,
	\end{align*}	
	which is at $\bm \a =\bm \a_0$ equal to
	\begin{align*}
	\frac{\partial }{\partial \a_j}\psi_{\bm\alpha}(\bma^T\bm x)\Bigm |_{\bm \a =\bm \a_0} 
	&= \int (x_j-\tilde x_j)\psi_0'\left(\bma_0^T\bm x \right)  h_{\bma_0}(\tilde x_2,\ldots, \tilde x_d|\bma_0^T\bm x)\prod_{j=2}^{d}d\tilde x_j\\
	&= \psi_0'(\bma_0^T\bm x) \left\{ x_j -\E(X_j |\bma_0^T\bm X=\bma_0^T\bm x )\right\}.
	\end{align*}
	For the  partial derivatives w.r.t. $\a_1$ we have,
	\begin{align*}
	&\frac{\partial }{\partial  \a_1}\psi_{\bm\alpha}(\bma^T\bm x)\\
	&\quad = \int \left\{\frac{\a_{01}}{\a_1}x_1-\frac{\a_{01}}{\a_1^2}(\bma^T\bm x -\a_2 \tilde x_2 - \ldots - \a_d \tilde x_d) \right\}\psi_0'\left(\frac{\a_{01}}{\a_1}(\bma^T\bm x -\a_2 \tilde x_2 - \ldots - \a_d \tilde x_d) + \sum_{j=2}^{d}\a_{0j} \tilde x_j\right)\\
	&\qquad\qquad\qquad\qquad\qquad\qquad\qquad\qquad\qquad\qquad\qquad\qquad\qquad\qquad\qquad\qquad\qquad  h_{\bma}(\tilde x_2,\ldots, \tilde x_d|\bma^T\bm x)\prod_{j=2}^{d}d \tilde x_j\\
	& + \int \psi_0\left(\bma^T\bm x + (\a_{01} -\a_1)\frac{\bma^T\bm x- \a_2 \tilde x_2 - \ldots - \a_d  \tilde x_d}{\a_1} + \sum_{j=2}^{d}(\a_{0j}- \a_j) \tilde x_j\right) \frac{\partial }{\partial  \a_1}h( \tilde x_2,\ldots,  \tilde x_d | \bma^T\bm x)\prod_{j=2}^{d}d \tilde x_j,
	\end{align*}	
	and,
	\begin{align*}
	\frac{\partial }{\partial \a_1}\psi_{\bm\alpha}(\bma^T\bm x)\Bigm |_{\bm \a =\bm \a_0} 
	= \psi_0'(\bma_0^T\bm x) \left\{ x_1 -E(X_1 |\bma_0^T\bm X=\bma_0^T\bm x )\right\}.
	\end{align*}
	This proves the first result of Lemma \ref{lemma:derivative_psi_a}. The proof for the second results follows similarly and is omitted.
\end{proof}	

\begin{lemma}
	\label{lemma:positive-lagrange-covariance}
	Let $\bar\f$ be defined by 
	\begin{align}
	\label{def:phi0}
	\bar \f(\bma) &=\int \bm x\left\{y -\psi_{\bm\a}(\bm\a^T\bm x)\right\}\,dP_0(\bm x,y)
	= \int \bm x\left\{\psi_0(\bma_0^T\bm x)-\psi_{\bm\a}(\bm\a^T\bm x)\right\}\,dG(\bm x),
	\end{align}	
	then we have for each $\bma \in \B(\bma_0, \delta_0)$,
	\begin{align*}
	\bar\f(\bma) = \E \left[\text{\rm Cov}\left[\bm X  ,  \psi_0(\bma^T\bm X + (\bma_0-\bma)^T\bm X)|\bma^T\bm X \right]\right].
	\end{align*}
	Moreover, 
	\begin{align*}
	\bma^T \bar\f(\bma) = 0,
	\end{align*}
	and,
	\begin{align*}
	(\bma_0-\bma)^T\bar\f(\bma) =\E \left[\text{\rm Cov}\left[(\bma_0-\bma)^T\bm X  ,  \psi_0(\bma^T\bm X + (\bma_0-\bma)^T\bm X)|\bma^T\bm X \right]\right] \ge 0,
	\end{align*}
	and $\bma_0$ is the only value such that the above equation holds uniform in $\bma \in \B(\bma_0, \delta_0)$.
\end{lemma}
\begin{proof}
	We have,
	\begin{align}
	\label{proof:cov-term1}
	\bar\f(\bma) &=\int \bm x\left\{y -\psi_{\bm\a}(\bm\a^T\bm x)\right\}\,dP_0(\bm x,y)
	= \int \bm x\left\{\psi_0(\bma_0^T\bm x)-\psi_{\bm\a}(\bm\a^T\bm x)\right\}\,dG(\bm x)\nonumber\\
	&= \int \bm x\left[\psi_0(\bma_0^T\bm x)- \E\left\{ \psi_0(\bma_0^T\bm X)|\bma^T\bm X=\bma^T\bm x\right\}\right]\,dG(\bm x)\nonumber\\
	&= \E \left[\text{\rm Cov}\left[\bm X  ,  \psi_0(\bma_0^T\bm X)|\bma^T\bm X \right]\right],
	\end{align}	
	and
	\begin{align*}
	\bma^T \int \bm x\left[\psi_0(\bma_0^T\bm x)- \E\left\{ \psi_0(\bma_0^T\bm X)|\bma^T\bm X=\bma^T\bm x\right\}\right]\,dG(\bm x) = \E \left[\text{\rm Cov}\left[\bma^T\bm X  ,  \psi_0(\bma_0^T\bm X)|\bma^T\bm X \right]\right] =\bm 0.
	\end{align*}
	We next note that,
	\begin{align*}
	(\bma_0 -\bma)^T\bar\f(\bma)&= \E \left[\text{\rm Cov}\left[(\bma_0 -\bma)^T\bm X  ,  \psi_0(\bma_0^T\bm X)|\bma^T\bm X \right]\right]\\&
	=\E \left[\text{\rm Cov}\left[(\bma_0 -\bma)^T\bm X  ,  \psi_0(\bma^T\bm X + (\bma_0-\bma)^T\bm X)|\bma^T\bm X \right]\right],
	\end{align*}
	which is positive by the monotonicity of $\psi_0$. This can be seen as follows. 		Using Fubini's theorem, one can prove that for any random variables $X$ and $Y$ such that $XY,X$ and $Y$ are integrable, we have
	\begin{align*}
	\text{\rm Cov}\left\{X,Y \right\} = \E XY-\E X\E Y=\int\{\P(X\ge s,Y\ge t)-\P(X\ge s)\P(Y\ge t)\}\,ds\,dt.
	\end{align*}
	Denote $Z_1 =(\bma_0 -\bma)^T\bm X$ and $Z_2 =\psi_0( u + (\bma_0 -\bma)^T\bm X) = \psi_0(u+Z_1)$, then, using monotonicity of the function $\psi_0$, we have
	\begin{align*}
	\P(Z_1\ge z_1,Z_2\ge z_2) &= \P(Z_1 \ge \max\{z_1, \tilde z_2 \}) \ge \P(Z_1 \ge \max\{z_1, \tilde z_2 \})\P(Z_1 \ge \min\{z_1, \tilde z_2 \})\\
	& = \P(Z_1 \ge z_1)\P(Z_2 \ge z_2 \},
	\end{align*}
	where 
	\begin{align*}
	\tilde z_2 = \psi_0^{-1}(z_2) - u = \inf\{t\in \mathbb{R}: \psi_0(t)\ge z_2\}-u. 
	\end{align*}
	We conclude that,
	\begin{align*}
	&\text{\rm Cov}\left\{(\bma_0 -\bma)^T\bm X,\psi_0(\bm\a^T\bm X + (\bm\a_0-\bm \a )'\bm X)|\bm\a^T\bm X = u \right\} \\
	&\qquad = \int\{\P(Z_1\ge z_1,Z_2\ge z_2)-\P(Z_1\ge z_1)\P(Z_2\ge z_2)\}\,ds\,dt \ge 0,
	\end{align*}
	and hence the first part of the Lemma follows. We next prove the uniqueness of the parameter $\bma_0$. We start by assuming that, on the contrary, there exists $\bma_1 \ne \bma_0 $ in $\B(\bma_0, \delta_0)$ such that
	$$(\bma_0 - \bma)^T\bar\f(\bma)\ge 0 \qquad \text{ and } (\bma_1 - \bma)^T\bar\f(\bma)\ge 0 \qquad \text{for all } \bma \in \B(\bma_0, \delta_0),$$
	and we consider the point $\bma \in \B(\bma_0, \delta_0)$ such that 
	$$
	|\a_j - \a_{j0}| = |\a_j - \a_{j1}| \qquad \qquad \text{for }j= 1,\ldots,d.
	$$
	For this point, we have,
	$$(\bma_0 - \bma)^T\bar\f(\bma)=-(\bma_1 - \bma)^T\bar\f(\bma) \qquad \text{for all } \bma \in \B(\bma_0, \delta_0),$$
	which is not possible since both terms should be positive. This completes the proof of Lemma \ref{lemma:positive-lagrange-covariance}.
\end{proof}

\bigskip

\begin{lemma}\label{lem: continuity}
	Let $f  : \mcX  \to  \RR^k, \ k \le d$  be a differentiable function on $\mathcal{X}$  such that there exists a constant $M > 0$ satisfying  $\Vert f \Vert_\infty \le M$.     Then, under the assumptions A1 and A5 we can find  a constant $\tilde{M} > 0$ such that for all $\bma \in \mathcal{B}(\bma_0, \delta_0)$ we have that 
	\begin{eqnarray*}
		\sup_{ \bm x \in \mcX}  \left\| \E[f(\bm X) |  \bma^T \bm X= \bma^T \bm x ]  -  \E[f(\bm X) |  \bma^T _0 \bm X= \bma^T_0\bm x ] \right\|  \le \tilde{M} \Vert \bma - \bma_0  \Vert.
	\end{eqnarray*}
	
\end{lemma}

\medskip

\begin{proof}
	We can assume without loss of generality that $\alpha_{0,1} \ne 0$ where $\alpha_{0,1}$  is the first component of $\bma_0$.  At the cost of taking a smaller $\delta_0$, we can further assume that $\tilde{\bma}_1  \ne 0$  for all $\bma \in \mathcal{B}(\bma_0, \delta_0)$.  Consider the change of variables  $t_1 = \bma^T \bm X$, $t_i = x_i$ for $i=1, \ldots, d$. Then, the density of $(\bma^T \bm X, X_2, \ldots, X_d)$ is given by 
	\begin{eqnarray*}
		g_{(\bma^T \bm X, X_2, \ldots, X_d)} (t_1, \ldots, t_d)  = g\left(\frac{1}{\alpha_1} \Big(t_1 - \alpha_2 t_2 - \ldots - \alpha_d t_d), t_2, \ldots, t_d\right)  \frac{1}{\alpha_1}.
	\end{eqnarray*}
	Then, for $i=2, \ldots, d$, the conditional density $ h_{\bma}(\cdot | u)$  of  the $(d-1)$-dimensional vector $(X_2, \ldots, X_d)$ given that $\bma^T  \bm X = u$  is equal to  
	\begin{eqnarray}\label{halpha}
	\frac{g\left( \frac{u - \alpha_2 x_2 - \ldots - \alpha_d x_d}{\alpha_1}, x_2, \ldots, x_d\right) }{\int g\Big(\frac{u- \alpha_2 t_2 - \ldots - \alpha_d t_d}{\alpha_1}, t_2, \ldots, t_d\Big) \prod_{j=2}^d dt_j}   &:=  &  h_{\bma}( x_2, \ldots, x_d|u) ,
	\end{eqnarray}
	where the domain of integration in the denominator is the set $\{(x_2, \ldots, x_d):  (\bm x, \mathcal{X} )\}$. Note that $X_1  =  (\bma^T\bm X - \alpha_2  X_2  - \ldots - \alpha_d X_d)/\alpha_1$. Thus,  for $(\bm x, \mcX)$ we have that
	\begin{eqnarray}\label{ExpExpec}
	&& \E[f(\bm X) | \bma^T\bm X  = \bma^T\bm x ]   =\E[f(X_1,X_2, \ldots, X_d)  | \bma^T\bm X = \bma^T\bm x]  \notag\\
	& &=  \E\left[f\left(\frac{\bma^T\bm X  - \alpha_2  X_2  - \ldots - \alpha_d X_d}{\alpha_1}, X_2, \ldots, X_d \right) |  \  \bma^T\bm X  =\bma^T\bm x \right]  \notag \\
	& & =  \int f\left(\frac{\bma^T\bm x  - \alpha_2  \bar x_2  - \ldots - \alpha_d \bar x_d}{\alpha_1}, \bar x_2, \ldots, \bar x_d \right) h_{\bma}(\bar x_2, \ldots, \bar x_d|\bma^T\bm x)  \prod_{j=2}^d d\bar x_j   \notag.
	\end{eqnarray}
	Note now that function
	\begin{eqnarray*}
		\bma \mapsto  h_{\bma}(\bar x_2, \ldots, \bar x_d| \bma^T\bm x)  =  \frac{g\left(\frac{\bma^T\bm x - \alpha_2 \bar x_2 - \ldots - \alpha_d \bar x_d}{\alpha_1}, \bar x_2, \ldots, \bar x_d\right) }{\int g\Big(\frac{\bma^T\bm x- \alpha_2 t_2 - \ldots - \alpha_d t_d}{\alpha_1}, t_2, \ldots, t_d\Big) \prod_{j=2}^d dt_j}, 
	\end{eqnarray*}
	is continuously  differentiable on $\mathcal{B}(\bma_0, \delta_0)$. This follows from assumptions A1 and A5 together with Lebesgue dominated convergence theorem which allows us to differentiate the density $g$  under the integral sign.
Also, for $i>1$,
\begin{align*}
&\frac{\partial}{\partial\a_i}h_{\bma}(\bar x_2,\ldots,\bar x_d|\bma^T\bmx)
=\frac{\partial}{\partial\a_i}\frac{g\left(\frac{\bma^T\bmx-\sum_{j=2}^d \a_j\bar x_j}{\a_1},\bar x_2,\dots,\bar x_d\right)}{\int g\left(\frac{\bma^T\bmx-\sum_{j=2}^d \a_j t_j}{\a_1},t_2,\dots,t_d\right)\,t_2\dots dt_d}\\
&=\frac{\left(x_i-\bar x_i\right)\partial_1g\left(\frac{\bma^T\bmx-\sum_{j=2}^d \a_j\bar x_j}{\a_1},\bar x_2,\dots,\bar x_d\right)}{\a_1\int g\left(\frac{\bma^T\bmx-\sum_{j=2}^d \a_j t_j}{\a_1},t_2,\dots,t_d\right)\,t_2\dots dt_d}\\
&\qquad-\frac{g\left(\frac{\bma^T\bmx-\sum_{j=2}^d \a_j\bar x_j}{\a_1},\bar x_2,\dots,\bar x_d\right)\int \{x_i-t_i\}\partial_1g\left(\frac{\bma^T\bmx-\sum_{j=2}^d \a_jt_j}{\a_1},t_2,\dots,t_d\right)\,dt_2\dots dt_d}
{\a_1\left\{\int g\left(\frac{\bma^T\bmx-\sum_{j=2}^d \a_jt_j}{\a_1},t_2,\dots,t_d\right)\,dt_2\dots dt_d\right\}^2}
\end{align*}
where $\partial_1$ denotes differentiation w.r.t.\ the first argument of the function. Assumptions A1 and A5 now allow us to find a constant $D > 0$ depending on $R$, $\underline{c}_0$, $\bar{c}_0$ and $\bar{c}_1$ such that 
	\begin{eqnarray*}
		\Big \Vert  \frac{\partial h_{\bma}}{\partial \alpha_i} \Big  \Vert_\infty \le D,
	\end{eqnarray*}
	for $i=2, \ldots, d$.

Consider now the function $\alpha \mapsto  E[f(X) | \bma^T\bm X  = \bma^T\bm x ] $. Using the assumptions of the lemma we conclude for $i>1$:
\begin{align*}
&\frac{\partial  \E[f(X) | \bma^T\bm X  = \bma^T\bm x ] }{\partial \alpha_i}  \\
&=\frac{\partial}{\partial\a_i}\int f\left(\frac{\bma^T\bm x  - \alpha_2  \bar x_2  - \ldots - \alpha_d \bar x_d}{\alpha_1}, \bar x_2, \ldots, \bar x_d \right) h_{\bma}(\bar x_2, \ldots, \bar x_d|\bma^T\bm x)  \prod_{j=2}^d d\bar x_j \\
&=\int \frac{\partial}{\partial\a_i}f\left(\frac{\bma^T\bm x  - \alpha_2  \bar x_2  - \ldots - \alpha_d \bar x_d}{\alpha_1}, \bar x_2, \ldots, \bar x_d \right) h_{\bma}(\bar x_2, \ldots, \bar x_d|\bma^T\bm x)  \prod_{j=2}^d d\bar x_j\\
&\qquad+\int f\left(\frac{\bma^T\bm x  - \alpha_2  \bar x_2  - \ldots - \alpha_d \bar x_d}{\alpha_1}, \bar x_2, \ldots, \bar x_d \right) \frac{\partial}{\partial\a_i}h_{\bma}(\bar x_2, \ldots, \bar x_d|\bma^T\bm x)  \prod_{j=2}^d d\bar x_j.
\end{align*}
Furthermore, we have that 
	\begin{eqnarray*}
		\sup_{(\bm x, \mcX)}  \Big \vert  \frac{\partial  \E[f(\bm X) | \bma^T\bm X  = \bma^T\bm x ] }{\partial \alpha_i}  \Big \vert   \le  MD \int  \prod_{j=2}^d dx_j = M',
	\end{eqnarray*}
for all $i>1$  and $(\bm x, \mathcal{X})$ and $\bma \in \mathcal{B}(\bma_0, \delta)$. By a symmetry argument, the result also holds for $i=1$.
The results now follow using  the mean value theorem:
	\begin{eqnarray*}
		\Big \vert \E[f(\bm X) | \bma^T\bm X  = \bma^T\bm x ]- \E[f(\bm X) | \bma^T_0 \bm X  = \bma^T_0 \bm x ]  \Big  \vert
		\le \sup_{\|\tilde\bma-\bma_0\|\le\|\bma-\bma_0\|} \left\| \frac{\partial  \E[f(\bm X) | \widetilde{\bma}^T \bm X  = \widetilde{\bma}^T \bm x ] }{\partial\tilde\bma} \right\| \|\bma-\bma_0\|. 
	\end{eqnarray*}
This gives the result. 
\end{proof}

\medskip

\begin{lemma}
	\label{lemma:BoundedVar}
	Denote for $i \in \{1, \ldots, d \}$ the $i$th component of the function $u\mapsto   \E[\bm X  |  \bma^T\bm X = u]$ by  $E_{i,\bma}$. Then  $E_{i,\bma}$ has a total bounded variation. Furthermore, there exists a constant $B > 0$ such that for all $\bma \in \B(\bma_0, \delta_0)$ 
	\begin{align*}
	\Vert E_{i,\bma} \Vert_\infty \le B, \ \ \textrm{and} \ \   \int_{\mathcal{I}_{\bma}}  \vert E'_{i,\bma}(u)  \vert du  \le B,
	\end{align*}
	where  $\mathcal{I}_{\bma} =  \{ \bma^T\bm x: \bm x \in \mcX\}$.
\end{lemma}

\begin{proof}
	Since $\mcX  \subset \mathcal{B}(0, R)$, it is clear that $\Vert E_{i,\bma} \Vert_\infty \le R$.   As above let us assume without loss of generality that the first component of $\bma_0$ is not equal to $0$. At the cost of taking a smaller $\delta_0$, we can further assume that $\tilde{\alpha}_1  \ne 0$  for all $\bma \in \B(\bma_0, \delta_0)$. We known that for $i=2, \ldots, d$
	\begin{align*}
	E_{i, \bma}(u)  = \int  x_i   h_{\bma}(x_2, \ldots, x_d|u)  dx_2  \ldots dx_d, 
	\end{align*}
	where integration is done over the set $\{(x_2, \ldots, x_d): (\bm x, \mcX)\}$ and $u \in \mathcal{I}_{\bma} \subset (a_0 - \delta_0 R,  b_0 + \delta_0 R)$ and where $h_{\bma}$ denotes conditional density of $(X_2,\ldots,X_d)'$ given $\bma^T\bm X= u$, defined in (\ref{halpha}).  Using assumptions A1 and A5 along with the Lebesgue dominated convergence theorem we are allowed to write
	\begin{align*}
	E'_{i, \bma}(u)   = \int  x_i    \frac{\partial}{\partial u} h_{\bma}(x_2, \ldots, x_d|u)  \  dx_2  \ldots dx_d.
	\end{align*}
	Straightforward calculations yield that
	\begin{align*}
	&\frac{\partial}{\partial u} h_{\bma}(x_2, \ldots, x_d|u) \\
	&   =  \frac{ \frac{\partial g}{\partial x_1} \left(\frac{u - \alpha_2 x_2 - \ldots - \alpha_d x_d}{\alpha_1}, x_2, \ldots, x_d\right) }{\alpha_1 \left(\int g\Big(\frac{u- \alpha_2 t_2 - \ldots - \alpha_d t_d}{\alpha_1}, t_2, \ldots, t_d\Big) \prod_{j=2}^d dt_j\right)}  \\
	&  \ -   \frac{g\left(\frac{1}{\alpha_1} \Big(u - \alpha_2 x_2 - \ldots - \alpha_d x_d), x_2, \ldots, x_d\right) \int \frac{\partial g}{\partial x_1} \Big(\frac{u- \alpha_2 t_2 - \ldots - \alpha_d t_d}{\alpha_1}, t_2, \ldots, t_d\Big) \prod_{j=2}^d dt_j }{\alpha_1 \left(\int g\Big(\frac{u- \alpha_2 t_2 - \ldots - \alpha_d t_d}{\alpha_1}, t_2, \ldots, t_d\Big) \prod_{j=2}^d dt_j\right)^2}.
	\end{align*}
	Thus, we can find constant $C > 0$ depending only on $\vert \alpha_{0,1} \vert$, $\underline{c}_0$, $\underline{c}_1$, $\bar{c}_1$  and $R$ such that   $\int \vert E'_{i, \bma}(u)    \vert   du \le C$  for all $\alpha   \in \B(\bma_0, \delta_0) $.   Now $B = \max(R, C)$ gives the claimed inequalities. If  $i =1$, then  
	\begin{align*}
	E_{1, \bma}(u)  =  \frac{1}{\alpha_1}  \Big(u - \alpha_j \sum_{j=2}^d   E_{j, \bma}(u)  \Big), \ \ \textrm{and} \ \  e'_{1, \bma}(u)  =  \frac{1}{\alpha_1}  \Big(1 - \alpha_j \sum_{j=2}^d    e'_{j, \bma}(u)  \Big).
	\end{align*}
	for $u \in I_{\bma} $. We conclude again that the claimed inequalities are true at the cost of increasing the constant $B$ obtained above.   
\end{proof}

\bigskip

\begin{lemma}\label{lemma:BoundedVar2}
	Let $f $ be a function defined on some interval $[a, b]$ such that  
	\begin{eqnarray*}
		\Vert f \Vert_\infty \le M, \ \  V(f, [a, b])  : = \sup_{a =x_0< x_1 \ldots < x_n = b} \sum_{j=1}^n  \vert f(x_j)  - f(x_{j-1}) \vert  \le M,
	\end{eqnarray*}
	for some finite constant $M > 0$.  Then,  there exist  two non-decreasing functions $f_1$ and $f_2$  on $[a, b]$   such that   $\Vert f_1 \Vert_\infty, \Vert f_2  \Vert_\infty  \le2 M$  and  $ f = f_2  -f_1$.  
\end{lemma}

\medskip

\begin{proof}  The fact that   $f= f_2  -f_1$   with $f_1$  and $f_2$ non-decreasing on $[a, b]$ follows from the well-known Jordan's decomposition. Furthermore, we can take $f_1(\bm x)  = V(f, [a, x])$ and $f_2(\bm x) =  f(\bm x)-  f_1(\bm x)$for $(\bm x, [a, b]$.   By assumption, $\Vert f_1  \Vert_\infty \le M \le 2M$ and $\Vert f_2 \Vert \le \Vert f \Vert_\infty + \Vert f_1 \Vert_\infty \le 2 M$.  
\end{proof}

\bigskip

\begin{lemma}
	\label{lemma:DerBoundedBelow}
	Suppose that  Assumptions A4-A5 hold, then we can find a constant $C > 0$ such that for all $\bma$  close enough to $\bma_0$ we have that 
	\begin{align*}
	\psi_{\bma}'(u)  >  C,
	\end{align*}
	for all $ u \in \mathcal{I}_{\bma}$. 
\end{lemma}
\begin{proof}
	Without loss of generality we can assume (following the argument at the end of the proof) that there exist a constant $C_0>0$ such that $\psi'_0  \ge C_0 $ on $(a_0 - \delta_0 R, b_0 + \delta_0  R)$ and  we assume again that $a_1\ne 0$ for all $\bma \in \mathcal{B}(\bma_0, \delta_0)$. By calculations similar to the calculations made in the proof of Lemma \ref{lemma:derivative_psi_a}, we get
	\begin{align*}
	\psi_{\bm\alpha}(u)
	&= \frac{\a_{01}}{\a_1}\int \psi_0'\left(\frac{\a_{01}}{\a_1}(u -\a_2 \tilde x_2 - \ldots - \a_d \tilde x_d) + \sum_{j=2}^{d}\a_{0j} \tilde x_j\right) h_{\bma} ( \tilde x_2,\ldots, \tilde x_d|u)\prod_{j=2}^{d}d \tilde x_j\\
	& + \int \psi_0\left(\frac{\a_{01}}{\a_1}(u -\a_2 \tilde x_2 - \ldots - \a_d \tilde x_d) + \sum_{j=2}^{d}\a_{0j} \tilde x_j\right) \frac{\partial }{\partial  u}h(\tilde x_2,\ldots,  \tilde x_d|u)\prod_{j=2}^{d}d \tilde x_j.
	\end{align*}	
	Now, a Taylor expansion of $\alpha_{i}$ in the neighborhood of $\alpha_{0, i}$ and using that $\alpha_{0,1}/\alpha_1 = 1-\epsilon_1/\alpha_{0,1} + o(\epsilon_1)$ yields
	\begin{eqnarray*}
		&&\psi_0\left(\frac{\alpha_{0,1}}{\alpha_1} \left(u - \alpha_2 x_2  - \cdots - \alpha_d x_d  \right)  +\alpha_{0,2} x_2  + \cdots  + \alpha_{0,d} x_d  \right)  \\
		&& = \psi_0\left (u  - \frac{\epsilon_1}{\alpha_{0,1}}  (u - \epsilon_2 x_2 - \ldots - \epsilon_d x_d)    + o(\epsilon_1) \right)\\
		&& = \psi_0(u)  - \frac{\epsilon_1}{\alpha_{0,1}}(u - \epsilon_2 x_2 - \ldots - \epsilon_d x_d)   \psi'_0(u)   +  o(\epsilon_1)  \\
		&&  = \psi_0(u)  -  \frac{\epsilon_1}{\alpha_{0,1}} u  \psi'_0(u)  +  o(\Vert \bma - \bma_0 \Vert). 
	\end{eqnarray*}
	Using the Lebesgue dominated convergence theorem and the fact that $h_{\bma}(\tilde x_2, \ldots, \tilde x_d|  u)$ is a conditional density it follows that 
	\begin{align*}
	\int \psi_0\left(\frac{\a_{01}}{\a_1}(u -\a_2 \tilde x_2 - \ldots - \a_d \tilde x_d) + \sum_{j=2}^{d}\a_{0j} \tilde x_j\right) \frac{\partial }{\partial  u}h(\tilde x_2,\ldots,  \tilde x_d|u)\prod_{j=2}^{d}d \tilde x_j,
	= o(\bma -\bma_0),
	\end{align*}
	such that there exist positive constants $C_1>0$ and $C_2>0$ such that
	\begin{align*}
	\psi_{\bma}'(u)   \ge  C_1   - {C_2\epsilon_1}    +  o(\bma -\bma_0)  \ge C > 0,
	\end{align*}
	provided that $\Vert \bma -\bma_0\Vert $ is small enough. 
	
	We now argue why the condition $\psi'_0  \ge C_0 $ on $(a_0 - \delta_0 R, b_0 + \delta_0  R)$ can be made: 
	
	Suppose first that there exists $R > 0$ such that $\mathcal{X} \subset \mathcal{B}(0, R)$ is bounded. For any given $a > 0$ we can consider the following monotone single index model   
	\begin{align}
	\label{model2}
	Y +  a \bma^T_0  \bm X = \psi_0(\bma^T_0  \bm X) + a \bma^T_0  \bm X  + \epsilon.
	\end{align}
	If $\psi_0$ is differentiable then the related link function, $\phi_0(t) = \psi_0(t) + a t$ is clearly monotone, with a first derivative satisfying $\phi'_0 \ge a > 0$. Replacing $Y$ by $Y +  a \bma^T_0  \bm X$ and $\psi_0$ by $\phi_0$  does not effect the arguments for establishing the limit distribution of the score estimator since the added terms are filtered throughout. Furthermore,  Assumption A6 can be shown to remain  valid for the new response variable $ Y +  a \bma^T_0  \bm X$. Using convexity of the function $t \mapsto t^m, m \ge 2$ we have that
	\begin{eqnarray*}
		E[\vert Y +  a \bma^T_0  \bm X\vert ^m | X=x]  &\le  & 2^{m-1} \left( m! c_0 M^{m-2}_0  + a^m \vert \bma^T_0  \bm x\vert)^m \right) \\
		&\le &   2^{m-1} \left( m! c_0 M^{m-2}_0  + a^m R^m  \right)  \\
		&= & m!  \left( 2 c_0 (2M)^{m-2}  + a^2R^2  (2 a R)^{m-2} \frac{2}{m!} \right) \\
		&\le  &   m!  \ (2  M \vee 2 aR)^{m-2}  (2c_0 + a^2R^2), \ \ \text{for $m \ge 2$} \\
		& =  & m!  \tilde{c}_0 \tilde{M}^{m-2}_0  
	\end{eqnarray*}
	with   $\tilde{c}_0   = 2c_0 + a^2R^2$ and $\tilde{M}_0 = 2  M \vee 2aR$, where $R, c_0$ and $M_0$ are specified in Assumption A1 and A6. So all assumptions remain valid under model (\ref{model2}) and the theory stated in Section \ref{section:Asymptotics} will still follow.   In the case $\mathcal{X}$ is not bounded, then we can take  any differentiable, bounded and monotone transformation of $a \bma^T_0X$ which gives a strictly positive first derivative. For example, we can easily check that the choice $ (e^{\bma^T_0  \bm X} +1)/(e^{\bma^T_0  \bm X} +2)$ for example works. 
	
\end{proof}

\begin{lemma}
	\label{lemma:derivative_kernel}
	If $h\asymp n^{-1/7}$, then there exists a constant $B>0$ such that for all $\bma \in \B(\bma_0,\d)$
	\begin{align*}
	\|\psi_{nh,\bma}'\|_\infty\le B \qquad\text{ and } \qquad \int_{\mathcal{I}_{\bma}} |\psi_{nh,\bma}''(u)|du \le B,
	\end{align*}
	where $\mathcal{I}_{\bma} = \{\bma^T\bmx : \bmx \in \mathcal{X} \}$
\end{lemma}
\begin{proof}
	Using integration by parts and Proposition \ref{prop:L_2-psi-psi_n-alpha}, we have for all $u\in \mathcal{I}_{\bma}$ 
	\begin{align*}
	\psi_{nh,\bma}'(u) &= \frac1h \int K\left(\frac{u-x}{h}\right)d\hat\psi_{n\bma}(x)\\
	&= \frac1h \int K\left(\frac{u-x}{h}\right)\psi_{\bma}'(x)dx + \frac1{h^2} \int K'\left(\frac{u-x}{h}\right)(\hat\psi_{n\bma}(x) - \psi_{\bma}(x)) dx \\
	&= \frac1h \int K\left(\frac{u-x}{h}\right)d\psi_{\bma}(x) + \frac1{h} \int K'\left(w\right)(\hat\psi_{n\bma}(u+hw) - \psi_{\bma}(u+hw)) dw
	\\
	& = \psi_{\bma}'(u) + O(h^2) + O_p(h^{-1}\log n n^{-1/3}) = \psi_{\bma}'(u)  + o_p(1).
	\end{align*}
	This proves the first part of Lemma \ref{lemma:derivative_kernel}. For the second part, we get by a similar calculation that,
	\begin{align*}
	\psi_{nh,\bma}''(u) &= \frac1h \int K\left(\frac{u-x}{h}\right)\psi_{\bma}''(x)dx + \frac1{h^2} \int K''\left(w\right)(\hat\psi_{n\bma}(u+hw) - \psi_{\bma}(u+hw)) dw
	\\
	& = \frac1h \int K\left(\frac{u-x}{h}\right)\psi_{\bma}''(x)dx  + O_p(h^{-2}\log n n^{-1/3}).
	\end{align*}
	Since $h^{-2}\log n n^{-1/3} = o(1)$ for $h\asymp n^{-1/7}$, the second part follows by Assumption A10.
\end{proof}